\documentclass[reqno,oneside]{amsart}
\usepackage[foot]{amsaddr}
\usepackage{amsmath}
\usepackage{amsthm}
\usepackage{amssymb}
\usepackage{bbm}
\usepackage{bm}
\usepackage{stmaryrd}
\usepackage{tikz}
\usepackage{mathtools}
\counterwithin{equation}{section}
\usepackage{mathrsfs}

\usepackage{pgfplots}
\pgfplotsset{compat=1.18}
\usepackage{subcaption}
\usepackage{float}
\usepackage{graphicx}
\usepackage{caption}
\usepackage[bottom]{footmisc}
\usepackage{url}
\usepackage[utf8]{inputenc}
\usepackage[T1]{fontenc}
\usepackage{enumitem}
\usepackage[normalem]{ulem}
\usepackage{booktabs}
\usepackage{algorithm}
\usepackage[noend]{algpseudocode}
\algrenewcommand\algorithmicindent{1.5em}
\algrenewcommand\alglinenumber[1]{\footnotesize #1.\enspace}
\usepackage[numbers,sort]{natbib}

\usepackage[hidelinks]{hyperref}
\usepackage[stretch=10]{microtype}
\usepackage{cleveref}
\crefformat{equation}{(#2#1#3)}

\algdef{SE}[SUBALG]{Indent}{EndIndent}{}{\algorithmicend}%
\algtext*{Indent}
\algtext*{EndIndent}

\algnewcommand{\LineComment}[1]{\State \texttt{/* #1 */}}
\def\algbackskip{\hskip-\ALG@thistlm}
\calclayout

\usepackage[framemethod=TikZ]{mdframed}
\usepackage{xcolor}
\usepackage{siunitx}

\theoremstyle{plain}
\newtheorem{theorem}{Theorem}[section]

\newtheorem{corollary}[theorem]{Corollary}

\newtheorem{lemma}[theorem]{Lemma}

\newtheorem{proposition}[theorem]{Proposition}

\theoremstyle{definition}
\newtheorem{definition}{Definition}
\newtheorem{remark}[theorem]{Remark}
\newtheorem*{remark*}{Remark}

\newcommand{\eg}{{e.g.,}\ }
\newcommand{\ie}{{i.e.,}\ }

\newcommand{\nref}[2]{\hyperref[#1]{#2}} 

\newcommand{\negsp}{\hspace{-0.25em}}

\newcommand{\cadlag}{c\`adl\`ag{}}

\newcommand{\bc}[1]{{\text{\fontsize{5}{5}\selectfont{$($}{\hspace*{-0.3pt}}\fontsize{5}{4}\selectfont$#1$\hspace*{-0.3pt}\fontsize{5}{5}\selectfont{$)$}}}} 

\newcommand{\bk}[1]{{\text{\fontsize{5}{5}\selectfont{$[$}{\hspace*{-0.3pt}}\fontsize{5}{4}\selectfont$#1$\hspace*{-0.3pt}\fontsize{5}{5}\selectfont{$]$}}}} 

\newcommand{\de}{\vcentcolon=}
\newcommand{\ed}{=\vcentcolon}

\newcommand{\dH}{d_{\rm{H}}}

\renewcommand{\inf}{\mathop{\mathrm{inf}\vphantom{\mathrm{sup}}}}

\newcommand{\intd}{{\rm d}}
\newcommand{\dd}[1]{{\textstyle{\frac{\rm{d}}{{\rm{d}}#1}}}}

\newcommand{\Area}{\operatorname{Area}}

\newcommand{\TV}{{\rm{TV}}}

\newcommand{\tmix}{t_{\rm{mix}}}
\newcommand{\Unif}[1]{\operatorname{Unif}#1}

\newcommand{\Var}{\operatorname{Var}}
\newcommand{\iid}{{i.i.d.}\ }

\newcommand{\defbold} [1]{\expandafter\newcommand\csname b#1\endcsname{\mathbf{#1}}}
\newcommand{\defcal} [1]{\expandafter\newcommand\csname c#1\endcsname{\mathcal{#1}}}
\newcommand{\defbb} [1]{\expandafter\newcommand\csname bb#1\endcsname{\mathbb{#1}}}
\newcommand{\defscr} [1]{\expandafter\newcommand\csname s#1\endcsname{\mathscr{#1}}}
\newcommand{\deffrak} [1]{\expandafter\newcommand\csname f#1\endcsname{\mathfrak{#1}}}
\newcommand{\defbracketed} [1]{\expandafter\newcommand\csname bc#1\endcsname{\bc{#1}}}
\newcommand{\deftilde} [1]{\expandafter\newcommand\csname ti#1\endcsname{\tilde{#1}}}
\newcommand{\defoverline} [1]{\expandafter\newcommand\csname ol#1\endcsname{\overline{#1}}}
\newcommand{\defvec} [1]{\expandafter\newcommand\csname vec#1\endcsname{\vec{#1}}}

\expandafter\defbold{A}
\expandafter\defbold{B}
\expandafter\defbold{C}
\expandafter\defbold{D}
\expandafter\defbold{E}
\expandafter\defbold{F}
\expandafter\defbold{G}
\expandafter\defbold{H}
\expandafter\defbold{I}
\expandafter\defbold{J}
\expandafter\defbold{K}
\expandafter\defbold{L}
\expandafter\defbold{M}
\expandafter\defbold{N}
\expandafter\defbold{O}
\expandafter\defbold{P}
\expandafter\defbold{Q}
\expandafter\defbold{R}
\expandafter\defbold{S}
\expandafter\defbold{T}
\expandafter\defbold{U}
\expandafter\defbold{V}
\expandafter\defbold{W}
\expandafter\defbold{X}
\expandafter\defbold{Y}
\expandafter\defbold{Z}

\expandafter\defcal{A}
\expandafter\defcal{B}
\expandafter\defcal{C}
\expandafter\defcal{D}
\expandafter\defcal{E}
\expandafter\defcal{F}
\expandafter\defcal{G}
\expandafter\defcal{H}
\expandafter\defcal{I}
\expandafter\defcal{J}
\expandafter\defcal{K}
\expandafter\defcal{L}
\expandafter\defcal{M}
\expandafter\defcal{N}
\expandafter\defcal{O}
\expandafter\defcal{P}
\expandafter\defcal{Q}
\expandafter\defcal{R}
\expandafter\defcal{S}
\expandafter\defcal{T}
\expandafter\defcal{U}
\expandafter\defcal{V}
\expandafter\defcal{W}
\expandafter\defcal{X}
\expandafter\defcal{Y}
\expandafter\defcal{Z}

\expandafter\defbb{A}
\expandafter\defbb{B}
\expandafter\defbb{C}
\expandafter\defbb{D}
\expandafter\defbb{E}
\expandafter\defbb{F}
\expandafter\defbb{G}
\expandafter\defbb{H}
\expandafter\defbb{I}
\expandafter\defbb{J}
\expandafter\defbb{K}
\expandafter\defbb{L}
\expandafter\defbb{M}
\expandafter\defbb{N}
\expandafter\defbb{O}
\expandafter\defbb{P}
\expandafter\defbb{Q}
\expandafter\defbb{R}
\expandafter\defbb{S}
\expandafter\defbb{T}
\expandafter\defbb{U}
\expandafter\defbb{V}
\expandafter\defbb{W}
\expandafter\defbb{X}
\expandafter\defbb{Y}
\expandafter\defbb{Z}

\expandafter\defscr{A}
\expandafter\defscr{B}
\expandafter\defscr{C}
\expandafter\defscr{D}
\expandafter\defscr{E}
\expandafter\defscr{F}
\expandafter\defscr{G}
\expandafter\defscr{H}
\expandafter\defscr{I}
\expandafter\defscr{J}
\expandafter\defscr{K}
\expandafter\defscr{L}
\expandafter\defscr{M}
\expandafter\defscr{N}
\expandafter\defscr{O}
\expandafter\defscr{P}
\expandafter\defscr{Q}
\expandafter\defscr{R}
\expandafter\defscr{S}
\expandafter\defscr{T}
\expandafter\defscr{U}
\expandafter\defscr{V}
\expandafter\defscr{W}
\expandafter\defscr{X}
\expandafter\defscr{Y}
\expandafter\defscr{Z}

\expandafter\deffrak{A}
\expandafter\deffrak{B}
\expandafter\deffrak{C}
\expandafter\deffrak{D}
\expandafter\deffrak{E}
\expandafter\deffrak{F}
\expandafter\deffrak{G}
\expandafter\deffrak{H}
\expandafter\deffrak{I}
\expandafter\deffrak{J}
\expandafter\deffrak{K}
\expandafter\deffrak{L}
\expandafter\deffrak{M}
\expandafter\deffrak{N}
\expandafter\deffrak{O}
\expandafter\deffrak{P}
\expandafter\deffrak{Q}
\expandafter\deffrak{R}
\expandafter\deffrak{S}
\expandafter\deffrak{T}
\expandafter\deffrak{U}
\expandafter\deffrak{V}
\expandafter\deffrak{W}
\expandafter\deffrak{X}
\expandafter\deffrak{Y}
\expandafter\deffrak{Z}

\expandafter\defbracketed{A}
\expandafter\defbracketed{B}
\expandafter\defbracketed{C}
\expandafter\defbracketed{D}
\expandafter\defbracketed{E}
\expandafter\defbracketed{F}
\expandafter\defbracketed{G}
\expandafter\defbracketed{H}
\expandafter\defbracketed{I}
\expandafter\defbracketed{J}
\expandafter\defbracketed{K}
\expandafter\defbracketed{L}
\expandafter\defbracketed{M}
\expandafter\defbracketed{N}
\expandafter\defbracketed{O}
\expandafter\defbracketed{P}
\expandafter\defbracketed{Q}
\expandafter\defbracketed{R}
\expandafter\defbracketed{S}
\expandafter\defbracketed{T}
\expandafter\defbracketed{U}
\expandafter\defbracketed{V}
\expandafter\defbracketed{W}
\expandafter\defbracketed{X}
\expandafter\defbracketed{Y}
\expandafter\defbracketed{Z}

\expandafter\deftilde{A}
\expandafter\deftilde{B}
\expandafter\deftilde{C}
\expandafter\deftilde{D}
\expandafter\deftilde{E}
\expandafter\deftilde{F}
\expandafter\deftilde{G}
\expandafter\deftilde{H}
\expandafter\deftilde{I}
\expandafter\deftilde{J}
\expandafter\deftilde{K}
\expandafter\deftilde{L}
\expandafter\deftilde{M}
\expandafter\deftilde{N}
\expandafter\deftilde{O}
\expandafter\deftilde{P}
\expandafter\deftilde{Q}
\expandafter\deftilde{R}
\expandafter\deftilde{S}
\expandafter\deftilde{T}
\expandafter\deftilde{U}
\expandafter\deftilde{V}
\expandafter\deftilde{W}
\expandafter\deftilde{X}
\expandafter\deftilde{Y}
\expandafter\deftilde{Z}

\expandafter\defoverline{A}
\expandafter\defoverline{B}
\expandafter\defoverline{C}
\expandafter\defoverline{D}
\expandafter\defoverline{E}
\expandafter\defoverline{F}
\expandafter\defoverline{G}
\expandafter\defoverline{H}
\expandafter\defoverline{I}
\expandafter\defoverline{J}
\expandafter\defoverline{K}
\expandafter\defoverline{L}
\expandafter\defoverline{M}
\expandafter\defoverline{N}
\expandafter\defoverline{O}
\expandafter\defoverline{P}
\expandafter\defoverline{Q}
\expandafter\defoverline{R}
\expandafter\defoverline{S}
\expandafter\defoverline{T}
\expandafter\defoverline{U}
\expandafter\defoverline{V}
\expandafter\defoverline{W}
\expandafter\defoverline{X}
\expandafter\defoverline{Y}
\expandafter\defoverline{Z}

\expandafter\defvec{A}
\expandafter\defvec{B}
\expandafter\defvec{C}
\expandafter\defvec{D}
\expandafter\defvec{E}
\expandafter\defvec{F}
\expandafter\defvec{G}
\expandafter\defvec{H}
\expandafter\defvec{I}
\expandafter\defvec{J}
\expandafter\defvec{K}
\expandafter\defvec{L}
\expandafter\defvec{M}
\expandafter\defvec{N}
\expandafter\defvec{O}
\expandafter\defvec{P}
\expandafter\defvec{Q}
\expandafter\defvec{R}
\expandafter\defvec{S}
\expandafter\defvec{T}
\expandafter\defvec{U}
\expandafter\defvec{V}
\expandafter\defvec{W}
\expandafter\defvec{X}
\expandafter\defvec{Y}
\expandafter\defvec{Z}

\expandafter\defbracketed{a}
\expandafter\defbracketed{b}
\expandafter\defbracketed{c}
\expandafter\defbracketed{d}
\expandafter\defbracketed{e}
\expandafter\defbracketed{f}
\expandafter\defbracketed{g}
\expandafter\defbracketed{h}
\expandafter\defbracketed{i}
\expandafter\defbracketed{j}
\expandafter\defbracketed{k}
\expandafter\defbracketed{l}
\expandafter\defbracketed{m}
\expandafter\defbracketed{n}
\expandafter\defbracketed{o}
\expandafter\defbracketed{p}
\expandafter\defbracketed{q}
\expandafter\defbracketed{r}
\expandafter\defbracketed{s}
\expandafter\defbracketed{t}
\expandafter\defbracketed{u}
\expandafter\defbracketed{v}
\expandafter\defbracketed{w}
\expandafter\defbracketed{x}
\expandafter\defbracketed{y}
\expandafter\defbracketed{z}

\expandafter\deftilde{a}
\expandafter\deftilde{b}
\expandafter\deftilde{c}
\expandafter\deftilde{d}
\expandafter\deftilde{e}
\expandafter\deftilde{f}
\expandafter\deftilde{g}
\expandafter\deftilde{h}
\expandafter\deftilde{i}
\expandafter\deftilde{j}
\expandafter\deftilde{k}
\expandafter\deftilde{l}
\expandafter\deftilde{m}
\expandafter\deftilde{n}
\expandafter\deftilde{o}
\expandafter\deftilde{p}
\expandafter\deftilde{q}
\expandafter\deftilde{r}
\expandafter\deftilde{s}
\expandafter\deftilde{t}
\expandafter\deftilde{u}
\expandafter\deftilde{v}
\expandafter\deftilde{w}
\expandafter\deftilde{x}
\expandafter\deftilde{y}
\expandafter\deftilde{z}

\expandafter\defoverline{a}
\expandafter\defoverline{b}
\expandafter\defoverline{c}
\expandafter\defoverline{d}
\expandafter\defoverline{e}
\expandafter\defoverline{f}
\expandafter\defoverline{g}
\expandafter\defoverline{h}
\expandafter\defoverline{i}
\expandafter\defoverline{j}
\expandafter\defoverline{k}
\expandafter\defoverline{l}
\expandafter\defoverline{m}
\expandafter\defoverline{n}
\expandafter\defoverline{o}
\expandafter\defoverline{p}
\expandafter\defoverline{q}
\expandafter\defoverline{r}
\expandafter\defoverline{s}
\expandafter\defoverline{t}
\expandafter\defoverline{u}
\expandafter\defoverline{v}
\expandafter\defoverline{w}
\expandafter\defoverline{x}
\expandafter\defoverline{y}
\expandafter\defoverline{z}

\expandafter\defvec{a}
\expandafter\defvec{b}
\expandafter\defvec{c}
\expandafter\defvec{d}
\expandafter\defvec{e}
\expandafter\defvec{f}
\expandafter\defvec{g}
\expandafter\defvec{h}
\expandafter\defvec{i}
\expandafter\defvec{j}
\expandafter\defvec{k}
\expandafter\defvec{l}
\expandafter\defvec{m}
\expandafter\defvec{n}
\expandafter\defvec{o}
\expandafter\defvec{p}
\expandafter\defvec{q}
\expandafter\defvec{r}
\expandafter\defvec{s}
\expandafter\defvec{t}
\expandafter\defvec{u}
\expandafter\defvec{v}
\expandafter\defvec{w}
\expandafter\defvec{x}
\expandafter\defvec{y}
\expandafter\defvec{z}

\def\b1{\mathbf{1}}
\def\bb1{\mathbbm{1}}

\newcommand{\Lc}[1]{L^{\mkern-2mu \bc{#1}}}
\newcommand{\tiLc}[1]{\tilde{L}^{\mkern-2mu \bc{#1}}}
\renewcommand{\vecv}{\vec{v}\mkern 2mu}
\newcommand{\hatn}{\hat{n}}
\renewcommand{\vecw}{\vec{w}\mkern 2mu}
\newcommand{\Wu}{\cW^{\mathfrak{u}}_1}
\newcommand{\bin}{\operatorname{bin}}
\newcommand{\imix}{i_{\rm{mix}}}

\newcommand{\nb}{m}

\newcommand{\nn}{n}

\title{Recovering semipermeable barriers from reflected Brownian motion}
\author{Alexander Van Werde} 
\author{Jaron Sanders}
\address{Eindhoven University of Technology, Department of Mathematics and Computer Science\\
\href{mailto:a.van.werde@tue.nl}{a.van.werde@tue.nl},  \href{mailto:jaron.sanders@tue.nl}{jaron.sanders@tue.nl}}

\begin{document}
\begin{abstract}
    We study the recovery of one-dimensional semipermeable barriers for a stochastic process in a planar domain.
    The considered process acts like Brownian motion when away from the barriers and is reflected upon contact until a sufficient but random amount of interaction has occurred, determined by the permeability, after which it passes through. 
    Given a sequence of samples, we wonder when one can determine the location and shape of the barriers. 

    This paper identifies several different recovery regimes, determined by the available observation period and the time between samples, with qualitatively different behavior.     
    The observation period $T$ dictates if the full barriers or only certain pieces can be recovered, and the sampling rate significantly influences the convergence rate as $T\to \infty$. 
    This rate turns out polynomial for fixed-frequency data, but exponentially fast in a high-frequency regime.

    Further, the environment's impact on the difficulty of the problem is quantified using interpretable parameters in the recovery guarantees, and is found to also be regime-dependent.
    For instance, the curvature of the barriers affects the convergence rate for fixed-frequency data, but becomes irrelevant when $T\to \infty$ with high-frequency data.

    The results are accompanied by explicit algorithms, and we conclude by illustrating the application to real-life data. 
\end{abstract}
\maketitle 
\vspace{-1em}
\section{Introduction}
Obstructions in the environment can shape the movements of stochastic processes, and uncovering this can lead to important scientific insights.
To name a concrete example, barriers to animal mobility, such as roads or natural features, are crucial knowledge for ecological research \cite{beyer2016you,sawyer2013framework}. 
For another, the interaction of molecules with obstacles, revealed by single-particle tracking, has led to numerous advancements in subcellular biology \cite{hofling2013anomalous,paviolo2020nanoscale,albrecht2016nanoscopic} such as the significant discovery that cell membranes are not homogeneous, but rather compartmentalized by barriers which hinder lateral diffusions \cite{kusumi2005paradigm,sadegh2017plasma}.

As a model for scenarios such as these, we study \emph{reflected Brownian motion with semipermeable barriers} in a setting where the process takes values in a planar domain and where the barriers are smooth, closed curves. 
This process acts like Brownian motion when away from the barriers and reflects on contact until a random time allows it to pass through, after which it reflects on the other side until the next random time. 
The barriers hence hinder its movements by temporarily constraining it to a single side; see \Cref{fig: Simulations} and \Cref{def: ReflectedBrownianMotion}.

\begin{figure}[tb]
    \centering
    \begin{subfigure}{.49\textwidth}
        \centering 
        \includegraphics[width=.8\linewidth]{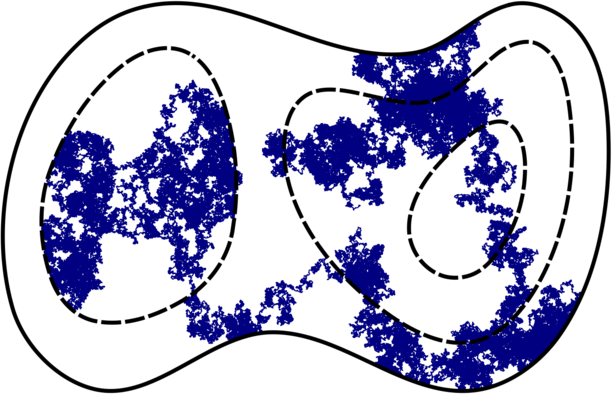}
    \end{subfigure}%
    \hfill 
    \begin{subfigure}{.49\textwidth}
        \centering
        \includegraphics[width=.8\linewidth]{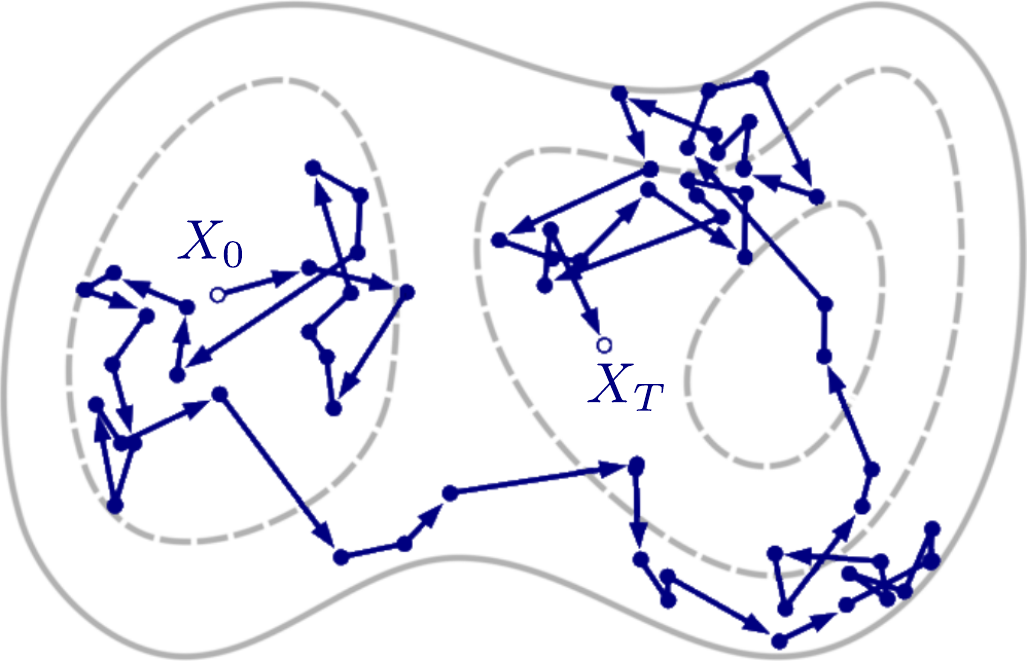}
    \end{subfigure}
    \caption{
    A simulated sample path of a reflected Brownian motion with semipermeable barriers, and the same data observed at a finite sampling rate. Our goal is to recover the underlying barriers given a finite number of samples.  
    }
    \label{fig: Simulations}
\end{figure}

This process was introduced to the mathematical literature relatively recently in 2016 by Lejay \cite{lejay2016snapping} and has since been used to describe various interface phenomena \cite{lejay2018monte,forien2019gene,erhard2021slow,zhao2024voter,slkezak2021diffusion,bressloff2022probabilistic,schumm2023numerical,bressloff2023renewal,mandrekar2016brownian}.  
It is also known as snapping out Brownian motion and should not be confused with skew Brownian motion \cite{harrison1981skew,lejay2006constructions}. 
The difference is that skew Brownian motion has permeable as opposed to semipermeable barriers, meaning that the barriers may be instantly traversed for the latter process, but with a possible directional bias.

The current paper studies the statistical problem which aims to recover the location of the barriers based on a sequence of samples $\{X_{it}: i=0,1,\ldots,\lfloor T/t\rfloor \}$.   
Our goal is to construct estimators and to fundamentally understand what affects the problem's difficulty.
This problem lies at the intersection of two key disciplines of modern statistics: set estimation which studies the statistical recovery of regions and their boundaries based on related observations \cite{brunel2018methods,cuevas2009set}, and statistical inference for stochastic processes \cite{bishwal2007parameter,kutoyants2013statistical}.  

The closest related work is that of Cholaquidis, Fraiman, Lugosi, and Pateiro--L\'opez \cite{cholaquidis2016set} who estimated a domain $D\subseteq \bbR^d$ and its boundary $\partial D$ based on a trajectory of a (classical) reflected Brownian motion; see also \cite{cholaquidis2021level,cholaquidis2023home} for recent extensions with drift or restricted sampling.
However, \cite{cholaquidis2016set} mostly focuses on recovering the domain $D$ itself, viewing the boundary as a corollary, and the methods are not suitable for estimating semipermeable barriers. 
To our knowledge, ours is the first mathematical work focused on recovering barrier locations given samples from a stochastic process.   

Also related to barriers is the estimation of the bias parameter for skew Brownian motion \cite{bardou2010statistical,lejay2019two,lejay2014brownian,lejay2018estimation,barahona2016simulation,bai2022bayesian} and of discontinuous diffusion or drift coefficients \cite{lejay2018statistical,lejay2020maximum,su2015quasi,karr2017point}. 
We however desire the locations, not the permeabilities which would be more analogous.
Estimating permeabilities presumably requires the locations, so our results give a starting point for future work in that direction.      

\subsection{Summary of main results}\label{sec: Summary}
We identify fundamental regimes with different notions of recovery, construct explicit estimators for the barrier locations in \Cref{alg: KernelDiscontinuity,alg: Improve,alg: HighFrequency}, and quantify what factors affect the problem's difficulty in terms of interpretable parameters.

One class of such factors arises from the environment. 
For example, the barriers' geometry matters as recovery grows challenging if these may squiggle arbitrarily wildly.
Indeed, dependence on geometric assumptions is a topic of significant interest in set estimation, and it is generally necessary to rule out pathological examples; see \cite[\S 3.2]{cuevas2009set}.  
We do this by assuming smoothness and quantify the geometry using the curvature.
To completely capture all environmental features so that guarantees depending only on the parameters can be established naturally also requires various other quantities including the barriers' permeabilies to name one; see \Cref{sec: Parameters}. 

Another important class of factors arises from the nature of the available data, determined by the intersample time $t$ and the length of the observation period $T$. 
Indeed, these quantities are typically constrained in applications. 
The observation period could be limited by budgetary considerations or if the problem informs a decision in the near future, and the intersample time by technical limitations of the measurement procedure.

The parameters $T$ and $t$ have a pronounced effect in our results, giving fundamental regimes where qualitatively different guarantees become possible.
The regimes may be summarized as follows; see also \Cref{fig: RecoveryTypes} and \Cref{tab: Recovery}. 

\begin{figure}[tb]
    \centering
    \begin{subfigure}{.49\textwidth}
        \centering
        \textbf{\ \ Partial recovery}
        \par\bigskip
        \includegraphics[width=.8\linewidth]{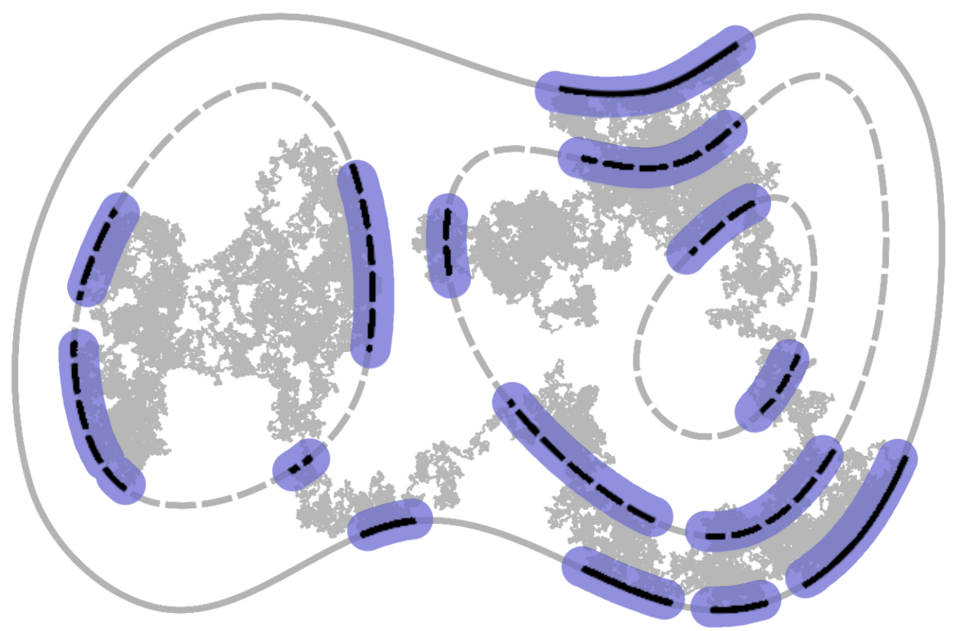}
    \end{subfigure}%
    \hfill 
    \begin{subfigure}{.49\textwidth}
        \centering 
        \textbf{\ \ Complete recovery}
        \par\bigskip
        \includegraphics[width=.8\linewidth]{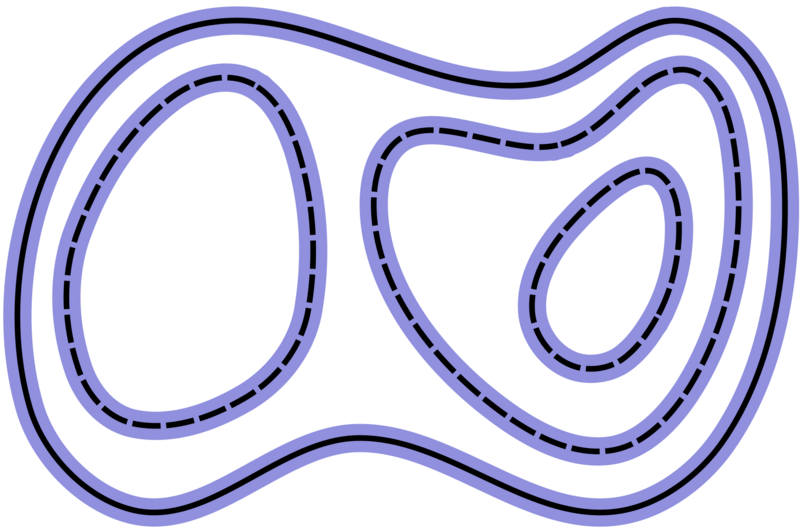}
    \end{subfigure}%
    \caption{Visualization for our recovery notions. 
    Partial recovery aims to recover the parts of the barrier which were hit by the continuous-time process based on discrete-time samples. 
    Complete recovery is more demanding and asks to recover the barriers completely.}
    \label{fig: RecoveryTypes}
\end{figure}
\begin{description}[leftmargin = 1.4em,font=\normalfont\itshape]
    \item[Complete recovery from fixed-frequency data]
    To start, we consider a regime where the observation period $T$ is large and where the intersample time $t$ is fixed and not too large. 
    \Cref{thm: Main_GlobalRecovery} then shows that the barriers can be recovered completely with uniform approximation error.

    When the approximation error is measured with respect to the Hausdorff distance, our convergence rate as $T\to \infty$ has order $T^{-2/3}$ up to a logarithmic factor. 
    The exponent $2/3$ is here expected to be optimal as this is the minimax rate for the related problem of boundary estimation given \iid samples from a planar domain \cite{aamari2023minimax}.      

    The dominant factors affecting the constants occurring in the rate are the curvature of the barriers and the mixing properties of the process.  
    The permeability of the barriers does not appear in the rate. 
    Rather, it determines, together with other factors, what constraint on the intersample time $t$ is required for our performance guarantee. 

    \item[Partial recovery from high-frequency data] 

    If $T$ is fixed, not necessarily large, then the process may not have explored the full domain so complete recovery is impossible. 
    However, there are practical situations where the observation time is limited. 
    Even if complete recovery is too much to ask, one may wonder if some partial information can be recovered. 

    Indeed, \Cref{thm: Main_PathDependentRecovery} shows that those pieces of the barriers which were hit by the continuous process $(X_s)_{s\leq T}$ can be recovered based on the observed discrete samples, and the convergence rate for typical points has order $\sqrt{t}$ up to a logarithmic factor as $t\to 0$.
    This is expected to be optimal as $\sqrt{t}$ is the order of the spacing between samples. 
    
    The word ``typical'' is not redundant: the analysis suggests that there can be points whose recovery is exceptionally difficult because the process spent little time near them. 
    \Cref{cor: Xuniform} implies that these are still recovered, but with a worse approximation error.
    \item[Complete recovery from high-frequency data] One may finally wonder about the setting when $T$ is large \emph{and} $t$ is small. 
    This is studied in \Cref{thm: ExponentialConvergence} for the special case where there is only the outer impermeable barrier, and a striking phenomenon is found. 
    
    While the rate in the fixed-frequency regime was polynomial as $T\to \infty$, an exponentially fast rate of order $\exp(-c\sqrt{T})$ can be achieved if one has high-frequency data.
    Moreover, this rate only depends on the geometry of the domain through its size, the curvature of the barrier being asymptotically irrelevant.   
\end{description}
For practitioners, these regimes can help to calibrate the expectations for what is possible with available data. 
Whether the full barriers or only certain pieces can be recovered depends on the observation period, and the approximation error may depend significantly on the time between samples. 
The effect of the environment is also heavily regime-dependent.
\begin{table}[h]
    \renewcommand{\arraystretch}{1.25}
    \begin{tabular}{|c!{\color{lightgray}\vrule}c!{\color{lightgray}\vrule}c!{\color{lightgray}\vrule}c!{\color{lightgray}\vrule}c|} 
        \hline 
        Result&$t$&$T$&Recovery type&Rate for approximation error
        \\ 
        \hline 
        \Cref{thm: Main_GlobalRecovery}&fixed&large&Complete&$T^{-2/3}\ln(T)^{2/3}$ uniformly as $T\to \infty$
        \\   
        \Cref{thm: Main_PathDependentRecovery}&small&fixed&Partial&$\sqrt{t}\ln(T/t)$ for typical points as $t\to 0$\\ 
        \Cref{thm: ExponentialConvergence} &small&large&Complete&$\exp(-c\sqrt{T})$ uniformly as $T\to \infty$\\ 
        \hline 
    \end{tabular} 
    \caption{Summary of our main results; see \Cref{sec: Results} for rigorous statements which also make the dependence on environmental factors explicit.}
    \label{tab: Recovery}
\end{table}

\subsection{Proof techniques}\label{sec: ProofTech}
The vast majority of previous work on Brownian motion with semipermeable barriers is only in one-dimensional settings \cite{lejay2006constructions,forien2019gene,erhard2021slow,zhao2024voter,slkezak2021diffusion,mandrekar2016brownian,bressloff2023renewal} with a few notable exceptions in work of Lejay \cite{lejay2018monte} and Schumm and Bressloff \cite{bressloff2022probabilistic,schumm2023numerical}. 
A key challenge is hence that our setting is a planar domain. 
Specifically, the main difficulty is that the barriers can be arbitrary smooth closed curves.
If we could pretend as if they were straight lines, then everything would be analytically tractable. 

We considered multiple approaches for this and put effort into simplifying these for transparent proofs.
For instance, one initial idea was to use a local diffeomorphism to rectify the barriers and then apply It\^o calculus, similar to \cite[\S 2.2]{lejay2018statistical}. 
However, while this may work, quantifying environmental dependence grows cumbersome because It\^o calculus gives non-constant drift and diffusion terms making the dynamics only ``nearly'' one-dimensional.
By depending on the arbitrary local diffeomorphism, this introduces irrelevant information to the computation.

Alternatively, one could rectify using a biholomorphism to a half-space via the Riemann mapping theorem and apply conformal invariance.  
This is more elegant as it then suffices to consider only a time change. 
However, the dependence of the biholomorphism on the rectified barrier is not necessarily transparent, so controlling the time change is still delicate. 

The method which we present here is based on coupling techniques instead. 
We first show in \Cref{lem: BarrierLocallyStraight} that the barriers can be locally approximated by a straight line, and subsequently approximate $X_t$ by a (classical) reflected Brownian motion reflecting on this line in \Cref{sec: XtYtApprox}.
Crucially, the error in the coupling of these processes and the timescale on which it is valid can be established transparently.
This allows reducing the proofs of \Cref{thm: Main_GlobalRecovery,thm: Main_PathDependentRecovery} to computations for one-dimensional reflected Brownian motion.
The latter computations are not trivial, but they can be done in a clear manner. 

The proof of \Cref{thm: ExponentialConvergence} is fairly short once one has the correct ideas and does not require the coupling.
Rather, it combines a general-purpose result of Matthews \cite{matthews1988covering} on covering times with a result by Chen and Friedman \cite{chen2011asymptotic} on narrow escape problems. 
This actually implicitly relies on the biholomorphism approach mentioned above, as this is used in \cite{chen2011asymptotic}. 
That this makes the dependence on the curvature of the barrier less explicit does not matter here because it turns out to be asymptotically irrelevant anyways.

\subsection{Application to animal movement data}
We give a case study with animal movement data in \Cref{sec: PotentialApplications}. 
The considered dataset consists of reindeer movements and was made publicly available by Loe et al.\ \cite{loe2016behavioral}.    
The recovery algorithm\footnote{Source code is made available at \url{https://github.com/Alexander-Van-Werde/Brownian-barriers.git}.} here successfully reveals both impermeable and semipermeable natural barriers which are validated using a satellite image. 
Impermeable barriers arise at the coastline and at steep slopes bounding the valley wherein the animals live, while rivers give semipermeable barriers which are occasionally crossed.

This is simply an illustrative example, but it may be noted that having a natural model with a notion of barrier permeability could indeed be useful in ecological research. 
This could, for instance, allow rigorously studying how much different types of roads impede movements.    
Future work about the estimation of the permeability parameter would hence be valuable. 
We refer to \cite{ringbauer2018estimating,beyer2016you,remon2018estimating,paquette2009statistical} and the references therein for methods in current use whose properties have mainly been investigated using simulations so far. 

Let us further note that animal movement data was a key motivation for the work of Cholaquidis, Fraiman, Lugosi, and Pateiro--López \cite{cholaquidis2016set} which we mentioned earlier. 
Specifically, they were concerned with estimating the \emph{home range}, which is the area where the animal does all its normal daily activities.  
Previous analyses had assumed i.i.d.\ data, but that assumption grew unrealistic as technological advances made high-frequency tracking possible, which is what motivated the move to process-based models; see \cite[\S 1.2]{cholaquidis2016set}.

Our theoretical results also shed new light on that problem as the special case of our setting where one only has the outer impermeable barrier. 
In particular, \Cref{thm: ExponentialConvergence,thm: Main_PathDependentRecovery} are highly relevant as they suggest that the problem may qualitatively change as tracking frequency improves, at least if the Brownian model is a sufficient approximation for reality.  
These are novel findings as the rates proved in \cite[\S 4]{cholaquidis2016set} were all polynomial with the same exponent as for \iid data, so without the exponential acceleration, nor did they consider that partial recovery may be possible when complete recovery is impossible.

\subsection{Structure of this paper}
The results are stated rigorously in \Cref{sec: Results_top}. 
General preliminaries for the arguments are presented in 
\Cref{sec: PrelimProofTech}, and the proofs are given in \Cref{sec: ProofsPropsFixedFreq,sec: ProofPathDependent,sec: ProofExponential}. 
We conclude in \Cref{sec: PotentialApplications} with the case study using animal movement data.

Supplementary details are given in the appendices. 
If desired, the algorithms and case study can also be read independently of the rigorous results which are our main focus; see \Cref{sec: PrelimGlobal,sec: AlgPathDependent} as well as \Cref{sec: PotentialApplications}.

\section{Results}\label{sec: Results_top}

\subsection{Notation and definitions}\label{sec: Def_ReflectedBrownianMotion}\label{sec: DomainsAndBarriers}
Fix an open planar set $D_0 \subseteq \bbR^2$ whose closure $D$ is connected, bounded, and has a smooth boundary $\partial D$.
Additionally, consider $m\geq 0$ smooth curves $B_1,\ldots,B_{\nb}\subseteq D_0$ and denote $B_0 \de \partial D$. 
In what follows, the $B_i$ with $i\geq 1$ will be semipermeable barriers and $B_0$ will be an impermeable barrier, keeping the process in $D$.  

To simplify the presentation and proofs, let us adopt some additional assumptions.
First, assume that $B_i \cap B_j = \emptyset$ for every $i\neq j$ so that there are no intersection points which would require separate consideration.  
Second, assume that the $D$ is simply connected so that $B_0$ is a single connected curve.
Finally, assume that the $B_i$ are simple closed curves so that there are no self-intersections or endpoints. 
Then, in particular, the Jordan curve theorem yields that $\bbR^2\setminus B_i$ has precisely two connected components of which only one is bounded.

For every $i\leq {\nb}$ let $\vec{n}_i:B_i \to \bbR^2$ be the unique vector field which is orthogonal to $B_i$, points towards the bounded component, and has unit length.   
We say that $x\in \bbR^2$ is \emph{on the positive side of $B_i$} if $x$ lies in the closure of the bounded component of $\bbR^2 \setminus B_i$. 
Similarly, the \emph{negative side of $B_i$} refers to the closure of the unbounded component. 
Note that $x$ is then both on the positive and negative sides when $x\in B_i$.

The process will be driven by randomness from some auxiliary processes.
To drive the movements when away from all barriers, we consider a $\bbR^2$-valued Wiener process $W_t$. 
Further, fix scalars $\lambda_i^+,\lambda_i^->0$ for every $i\geq 1$, specifying the permeability of the two sides of each barrier.
Then, to regulate the random event where the process switches sides, we consider continuous-time \cadlag\ Markov chains $s_i(t)$ taking values in $\{+1,-1 \}$ with transition rate $\lambda_i^+$ (resp.\ $\lambda_i^-$) from $-1$ to $+1$ (resp.\ from $+1$ to $-1$). 
Finally, let $s_0(t) \de +1$ for all $t\geq 0$. 

Processes satisfying the following requirements exist and are unique given an initial condition; see \Cref{prop: ProcessExistsAndIsUnique}. 
\pagebreak[2]
\begin{definition}\label{def: ReflectedBrownianMotion}
    Let $X_t$ and $\Lc{i}_t$ be continuous stochastic processes which take values in $D$ and $\bbR_{\geq 0}$, respectively.
    Then, $X_t$ is called \emph{a reflected Brownian motion with semipermeable barriers $B_i$ and local times $\Lc{i}_t$} if the following properties hold with probability one: 
    \begin{enumerate}[leftmargin=2.5em, label = (\roman*)]
        \item\label{item: Def_ReflectedBrownianMotion_i} The following stochastic differential equation is satisfied:  
        \begin{align} 
            \intd X_t =  \intd W_t +  \sum_{i=0}^{\nb} s_i(\Lc{i}_t) \vec{n}_i(X_t) \bb1\{X_t \in B_i \}\intd \Lc{i}_t.\nonumber
        \end{align}
        \item\label{item: Def_ReflectedBrownianMotion_ii} For every $i \leq {\nb}$, the process $\Lc{i}_t$ is nondecreasing, satisfies $\Lc{i}_0 =0$, and increases at time $t$ if and only if $X_t \in  B_i$. 
        That is,  
        \begin{align} 
            \Lc{i}_t = \int_0^t \bb1\{X_r\in B_i \}\, \intd \Lc{i}_r.\nonumber 
        \end{align}
        \item\label{item: Def_ReflectedBrownianMotion_iii} For every $i \leq {\nb}$, if $s_i(\Lc{i}_t) = +1$ then $X_t$ is on the positive side of $B_i$. 
        Similarly, if $s_i(\Lc{i}_t) = -1$ then $X_t$ is on the negative side of $B_i$.    
    \end{enumerate}
\end{definition}
\begin{remark}
    Readers who are unfamiliar with the theory of local times may benefit from a more definite description. 
    The local time at the $i$th barrier admits the following expression:  
    \begin{align} 
        \Lc{i}_t  = \lim_{\varepsilon \to 0} \frac{1}{2\varepsilon} \int_0^t \bb1\bigl\{ \exists y \in B_i: \Vert X_r - y \Vert < \varepsilon \bigr\}\, \intd r, \label{eq: LocalTime}
    \end{align}
    almost surely; see \Cref{sec: ExpressionLocal} for a proof. 
    One can hence interpret $\Lc{i}_t$ as measuring how much $X_r$ interacted with barrier $B_i$ by time $t$. 

    The interpretation for the terms $s_i(\Lc{i}_t)$ in \Cref{def: ReflectedBrownianMotion} is correspondingly that these change sign randomly as the process interacts with a barrier, depending on its permeability and the amount of interaction.
    Upon such a sign change, \cref{item: Def_ReflectedBrownianMotion_i,item: Def_ReflectedBrownianMotion_iii} express that $X_t$ passes to the other side of the barrier and that the direction of reflection flips to match this. 
\end{remark}  
\begin{remark}
    Some clarifications are required to avoid ambiguity. 
    First, let it be understood that there is no dependence between the considered auxiliary processes except for that which is necessary for \Cref{def: ReflectedBrownianMotion} to be well-defined. 
    That is, we assume that the $s_i$ are conditionally independent given $(s_i(0))_{i=1}^{\nb}$, and unconditionally independent of $\{W_t:t\geq 0\}$.

    Second, let us note that the law of the process $X_t$ not only depends on the initial condition $X_0$, but also on $s_i(0)$ if $X_0 \in B_i$. 
    For brevity, however, we will typically suppress this from the notation. 
    So, for instance, one should understand that ``for every initial condition $x_0 \in D$'' really means ``for every initial condition $x_0 \in D$ and every $s_i(0) \in \{+1,-1 \}$ with $x_0$ on the positive (resp.\ negative) side of $B_i$ if $s_i(0) = +1$ (resp.\ $s_i(0) = -1$)''.     
    
    Finally, the term ``smooth'' here always means $C^\infty$-smooth.
    It would of course also be interesting to extend our results to settings with very little regularity, but we do not wish to focus on such issues in the current work.
\end{remark}

\subsection{Parameters quantifying the difficulty of the environment}\label{sec: Parameters}
The permeability of the barriers can be quantified by the following parameter: 
\begin{align} 
    \lambda_{\max} &\de \max\bigl\{\lambda_i^*: i\in \{1,\ldots,{\nb} \},\,  *\in \{+,-\} \bigr\}.\label{eq: Def_lambda_minmax} 
\end{align}
If $\lambda_{\max}$ is large, then some barriers have little influence and the process is allowed to cross them quickly. 
Such barriers are naturally more difficult to detect.

The following two parameters are of a geometric nature, related to the difficulty of the configuration of barriers. 
For every $i\leq {\nb} $ let $k_i:B_i \to \bbR_{\geq 0}$ be the \emph{unsigned curvature} of $B_i$.
Denote $\kappa$ for the greatest achieved value: 
\begin{align} 
    \kappa \de \max\bigl\{ k_i(x): i\in \{0,1,\ldots,{\nb} \},\, x \in B_i \bigr\}. \label{eq: Def_kappa}
\end{align} 
A bound on this parameter allows one to avoid scenarios where the curves squiggle wildly. 
It is further useful to have a parameter which ensures some minimal spacing between the barriers. 
For $x\in \bbR^2$ let $\sB(x,r)$ denote the open ball of radius $r>0$ and set 
\begin{align} 
    \rho &\de \sup\bigl\{r \geq 0: \sB(x,r')\cap (\cup_{i=0}^{\nb} B_i)  \text{ is connected }\forall r' \leq r,\, \forall x\in D\bigr\}. \label{eq: Def_rho}
\end{align}
This parameter serves to avoid pathological scenarios where distinct barriers lie extremely close to each other, or scenarios where a single barrier doubles over on itself with a U-bend.

We finally consider two parameters whose necessity arises from the fact that it is impossible to recover a barrier with which the process had no interactions.
Let $\pi$ be the stationary distribution of $X_t$.
That is, the probability measure given by 
$
    \pi(E) = \lim_{t\to \infty} \bbP(X_t \in E)
$ for every measurable $E\subseteq D$.  
Then, with $\Area(\cdot)$ the Lebesgue measure on $\bbR^2$, the following quantifies if there are infrequently visited regions:  
\begin{align} 
    \pi_{\min} \de \inf\bigl\{\pi(E)/\Area(E) :\text{measurable }E\subseteq D \text{ with }\Area(E)>0 \bigr\}.\label{eq: Def_pimin} 
\end{align}
Finally, even if the process visits a region frequently in the limit $t\to \infty$, it could take a long time between visits to the region. 
To quantify this, consider the \emph{mixing time}, 
\begin{align} 
    \tmix \de \inf\bigl\{t\geq 0 :  \lvert \bbP(X_t \in E \mid X_0 = x_0) - \pi(E) \rvert \leq 1/4,\  \forall x_0 \in D,\, E\subseteq D\bigr\}.\label{eq: Def_tmix}
\end{align}
This parameter could be large, for instance, if one of the barriers has low permeability and hence traps the process on the side of $X_0$ for a long time.

\subsection{Results}\label{sec: Results}
\subsubsection{Complete recovery from fixed-frequency data}\label{sec: Results_UniformRecovery}
To start, we consider the regime where the observation period $T$ is large and the time between samples $t$ is fixed.
It will also be assumed that $t$ is not \emph{too} large. 
The latter assumption is required for our proofs because the permeability of the barriers will dilute their effect over time.  

The goal is \emph{complete recovery}, meaning that we desire an estimator $\hat{B}\subseteq \bbR^2$ such that every point of $\cup_{i=0}^{\nb} B_i$ is close to some point in $\hat{B}$ and conversely.
In other words, we quantify the approximation error in terms of the \emph{Hausdorff distance}: 
\begin{align} 
    \dH(A,B) \de \max\Bigl\{\sup_{a\in A}\inf_{b\in B}\Vert a - b \Vert,\ \sup_{b\in B}\inf_{a\in A}\Vert a - b \Vert\Bigr\}\label{eq: Def_dH} 
\end{align} 
where $A, B\subseteq \bbR^2$ are compact sets and $\Vert \cdot \Vert$ is the Euclidean norm.

The following result establishes that it is possible to achieve complete recovery given a sufficiently long sample path.  
Further---and this is the main content really---it is made explicit how the performance depends on the interpretable parameters from \Cref{sec: Parameters}.    
\begin{theorem}\label{thm: Main_GlobalRecovery}
    There exist absolute constants $c_1,c_2,c_3>0$ such that the following holds for every $\eta \in (0,1)$, every $t\leq c_1 \min\{1/\kappa^2, 1/\lambda_{\max}^2,\rho^2\}$, $\varepsilon \leq c_2 \kappa t$, and $T>0$. 

    Suppose that one is given samples $\{X_{jt}: j=0,1,\ldots,\lfloor T/t\rfloor \}$ from a sample path starting in stationarity, \ie with $X_0\sim \pi$.
    Assume that the observation time satisfies 
    \begin{align} 
        T \geq  c_3\frac{\tmix}{ \pi_{\min}} 
        \sqrt{\frac{\kappa}{\varepsilon^3}} \ln\Bigl(\frac{\Area(D)}{\eta}\sqrt{\frac{\kappa}{\varepsilon^3 }} \Bigr). \label{eq: Main_GlobalRecovery_n} 
    \end{align}
    Then, there exists an explicit algorithm whose output $\hat{B}$ satisfies the following guarantee:   
    \begin{align} 
        \bbP\bigl(\dH(\hat{B}, \cup_{i=0}^{\nb} B_i) \leq \varepsilon\bigr) \geq 1 - \eta. \label{eq: Main_GlobalRecovery_dH}
    \end{align}           
\end{theorem}
The proof is given in \Cref{sec: ProofsPropsFixedFreq} as is the algorithm. 
Rewriting \eqref{eq: Main_GlobalRecovery_n}, the error $\varepsilon$ decays at a rate of order $T^{-2/3} \ln(T)^{2/3}$ as $T\to \infty$. 
The exponent $2/3$ on $T^{-2/3}$ is optimal for recovering the boundary of a $C^2$-smooth set given \iid samples \cite{aamari2023minimax}, and we expect that it will also be optimal here as only bounds on the curvature are assumed.
It could be interesting future research to rigorously prove such an optimality result. 
Another interesting question is whether algorithms which also assume bounds on the barriers' higher-order derivatives may achieve better rates of convergence.

\begin{remark}\label{rem: effective_obs}
    It often occurs in practice that one is given a collection of sequences, say $\{X_{jt}^{\bk{i}}:j\leq \lfloor T_i/t \rfloor \}$ with $i=1,\ldots,r$. 
    Assuming that the paths and their initial conditions are independent, the same algorithms and guarantees then remain applicable with $T$ replaced by the \emph{effective observation period} $ \sum_{i=1}^r T_i$.       
\end{remark}

\subsubsection{Partial recovery from high-frequency data}\label{sec: Results_PathDependentRecovery}
If $T$ is fixed, potentially smaller than the mixing time of the process, then it can not be guaranteed that the process interacted with all parts of the barrier.
Consequently, complete recovery may be impossible.

We here show that the next best thing is possible: one can recover those pieces which were hit by the continuous-time process $(X_s)_{s\leq T}$ using only the discrete samples. 
More precisely, let  
\begin{align} 
    \cX_i \de B_i\cap \{X_t: t\in [0,T]  \}. \label{eq:SafeSun}
\end{align} 
The following result then shows that ``typical'' points in $\cup_{i=0}^{\nb} \cX_i$ can be recovered up to an approximation error which is small for high-frequency data: 
\begin{theorem}\label{thm: Main_PathDependentRecovery}
    For every $\eta \in (0,1)$ there exist $c_1,c_2,c_3>0$ depending only on $\eta$ such that for every $T>0$ and initial condition $x_0\in D$ the following holds.  
     
    Suppose that one is given samples $\{X_{jt}: j = 0,1,\ldots,\lfloor T/t \rfloor\}$ with $t>0$ sufficiently small to satisfy $t \leq c_1 T$ and 
    $ 
    \ln(T/t)^4  t \leq c_2\min\{T, 1/\kappa^2, 1/\lambda_{\max}^2,\rho^2 \}.
    $ 
    Then, there exists an explicit algorithm whose output $\hat{\cX}$ satisfies the following guarantees: 
    \begin{description}
        \item[(1) Nonoccurrence of false positives] 
        Every point in $\hat{\cX}$ is close to a barrier with high probability. 
        More precisely,    
        \begin{align} 
            \bbP\bigl(\inf\{\Vert p - y \Vert: y\in \cup_{i=0}^{\nb} B_i\} \leq c_3 \ln(T/t) \sqrt{t}\, \text{ for all }\, p\in \hat{\cX} \bigr)\geq 1- \eta.\label{eq: HighFrequency_FalsePositives} 
        \end{align}   
        \item[(2) Typical points are recovered] Let $\tau$ be a stopping time for the filtration generated by $X_t$, $\Lc{i}_t$, and $s_i(\Lc{i}_t)$.
        Then, if $\tau$ is such that $X_\tau \in \cup_{i=0}^{\nb} B_i$ almost surely, 
        \begin{align} 
            \bbP\bigl(\inf\{\Vert X_\tau - p \Vert: p\in \hat{\cX} \} \leq c_3 \ln(T/t)\sqrt{t}\text{ or }\tau >T \bigr) \geq 1 - \eta. \label{eq: HighFrequency_RandomPoint}
        \end{align}    
    \end{description}
\end{theorem}

Let us emphasize that the algorithm does not require prior knowledge of the stopping time $\tau$. 
As a concrete example, one can take $\tau \de \inf\{t\geq 0: X_t \in \cup_{i=0}^{\nb} B_i \}$. 
Then, the result shows that the first hit of a barrier is recovered with high probability. 
Similarly, taking $\tau \de \inf\{t \geq 0: s_i(\Lc{i}_t) \neq s_i(0) \text{ for some }i\leq {\nb} \}$ shows that the first location where $X_t$ switches sides for some barrier is recovered with high probability.

The combination of \eqref{eq: HighFrequency_FalsePositives} and \eqref{eq: HighFrequency_RandomPoint} can be interpreted as stating that the approximation error has order $\ln(T/t)\sqrt{t}$ for most points, but there can be exceptional points in $\cup_{i=0}^{\nb} \cX_i$ which are particularly difficult to recover. 
Those points are still fairly close to $\hat{\cX}$ but may suffer from a worse approximation error:  
\begin{corollary}\label{cor: Xuniform}
    For every $\varepsilon >0$ there exist $c_1,c_2 >0$ depending on $\varepsilon, \eta$, $\min\{1/\kappa, 1/\lambda_{\max},\rho\}$, and $T$ such that for $t\leq c_1$ and $\ln(T/t)^4 t \leq c_2$ the estimator of \Cref{thm: Main_PathDependentRecovery} satisfies  
        \begin{align} 
            \bbP\bigl(\inf\{\Vert x - p \Vert: p\in \hat{\cX} \} \leq \varepsilon\, \text{ for all }\, x \in \cup_{i=0}^{\nb} \cX_i \bigr) \geq 1 - \eta.  \label{eq: HighFrequency_Uniform}
        \end{align} 
\end{corollary}
The proofs are given in \Cref{sec: ProofPathDependent} and come with an explicit algorithm. 
The nature of the difficult points is there also made explicit: they occur when the process only spends a small amount of time in the neighborhood. 
It would be interesting future research to study these points in more detail as this would be an essential ingredient to determine an optimal rate of convergence for $\dH(\hat{\cX}, \cup_{i=0}^{\nb} \cX_i)$.

\subsubsection{Complete recovery from high-frequency data}\label{sec: Results_UniformHighFreq}
Finally, we consider a regime where $T$ is large and $t$ is small.
Specifically, we study the behavior when one first takes the limit $t\to 0$ and subsequently lets $T$ grow.  
For proof-technical reasons, we focus on the case $m=0$ where there is only the outer impermeable barrier which bounds the domain.

By \eqref{eq: HighFrequency_FalsePositives} and \eqref{eq: HighFrequency_Uniform}, one can recover $\cX_0$ with an arbitrarily small approximation error and without false positives if the sampling rate is sufficiently high. 
Hence, it remains to determine the rate of convergence of $\cX_0$ to $B_0$ when $T$ tends to infinity:
\begin{theorem}\label{thm: ExponentialConvergence}
    Suppose that $m=0$ and fix some initial condition $x_0 \in D$. 
    Then, for every $\eta \in (0,1)$ there exists a positive constant $c>0$ depending only on $\eta$ such that   
    \begin{align} 
        \liminf_{T\to \infty}\bbP\Bigl( \dH(\cX_0,B_0) \leq \exp(-\sqrt{c T/\Area(D)  } )\Bigr) \geq 1-\eta.\label{eq:KindMom}
    \end{align}
\end{theorem}

We observe a pronounced qualitative difference between the fixed-frequency regime and the high frequency regime: the polynomial rate of convergence in \Cref{thm: Main_GlobalRecovery} is replaced by an exponential rate in \Cref{thm: ExponentialConvergence}.
Further, convergence rate in the high-frequency regime only depends on the size of the domain whereas in the fixed-frequency regime it also depends on the curvature of the barrier among other things.

While not covered by the current analysis, we expect that the qualitative conclusion of \Cref{thm: ExponentialConvergence} of an exponential rate remains valid in the general case $m>0$. 
Such a generalization would be relevant future research as the quantitative formulation would clarify how exactly the environment enters the problem.
For instance, one possible guess is that $1/\Area(D)$ may have to be replaced by a quantity involving the stationary distribution.

\section{General preliminaries}\label{sec: PrelimProofTech}  
\Cref{sec: GeneralPurposePrelim} shows that there indeed exists a unique process satisfying \Cref{def: ReflectedBrownianMotion}. 
\Cref{sec: Local} provides the coupling technique mentioned in \Cref{sec: ProofTech}.

\subsection{Well-definedness}
\label{sec: GeneralPurposePrelim} 
Consider a $D$-valued random variable $X_0$ which is on the positive (resp.\ negative) side of $B_i$ if $s_i(0) = +1$ (resp.\ $s_i(0) = -1$). 
Assume that $(X_0, s_1,\ldots,s_{\nb})$ is independent of $W$ and that $X_0$ is conditionally independent of $(s_i)_{i=1}^{\nb}$ given $(s_i(0))_{i=1}^{\nb}$.   
  
It will be convenient to ensure that null sets are measurable. 
Hence, let it be understood that the $\sigma$-algebra generated by a family of random variables $\{Y_i:i\in \cI \}$ refers to the $\bbP$-completion\footnote{Formally, if $(\Omega, \sF, \bbP)$ is the implicit probability space where upon the processes $\{W_t:t\geq 0 \}$ and $\{s_i(t):t\geq 0\}$ are defined, then the \emph{$\bbP$-completion} of a $\sigma$-algebra $\sG \subseteq \sF$ is here defined to be the smallest $\sigma$-algebra containing all sets in $\sG$ and all $\sF$-null sets.} of the least $\sigma$-algebra with respect to which all $Y_i$ are measurable.  
\begin{proposition}\label{prop: ProcessExistsAndIsUnique} 
    Let $\cG_t$ denote the $\sigma$-algebra generated by the random variables $X_0$, $W_r$ with $0 \leq r \leq t$, and $s_i(r)$ with $0\leq r < \infty$ and $0\leq i \leq {\nb}$.
    Then, there exist $\cG_t$-adapted processes $X_t$ and $\Lc{i}_t$ satisfying \Cref{def: ReflectedBrownianMotion} with initial condition $X_0$. 
    
    Moreover, pathwise uniqueness holds: if $(\tiX_t, \tiLc{0}_{t},\ldots,\tiLc{{\nb}}_t)$ is another $\cG_t$-adapted processes with $\tiX_0 =X_0$ which satisfies \Cref{def: ReflectedBrownianMotion} with respect to the same $W_t$ and $s_i$, then $\tiX_t = X_t$ and $\tiLc{i}_t = \Lc{i}_t $ for every $t\geq 0$, almost surely.      
\end{proposition}

Let $X_t$ and $\Lc{i}_t$ be the unique $\cG_t$-adapted processes satisfying \Cref{def: ReflectedBrownianMotion} with initial condition $X_0$.
The stochastic differential equation in \Cref{def: ReflectedBrownianMotion} then suggests that $X_t$ and $\Lc{i}_t$ only depend on $s_j$ through $s_j(r)$ for $ r\leq \Lc{j}_t$, and that $X_t$ is Markovian. 
Indeed: 
\begin{proposition}\label{prop: Fadapted}
    For any $t\geq 0$ and $\ell_0,\ell_1,\ldots,\ell_{\nb} \geq 0$, let $\cF_{t,\ell_0,\ldots,\ell_{\nb}}$ denote the $\sigma$-algebra generated by the random variables $X_0$, $W_r$ with $0\leq r\leq t$, and $s_i(z_i)$ with $0\leq z_i \leq \ell_i $.  
    Then, the event $\{\Lc{i}_t \leq \ell_i,\, \forall i\leq {\nb}\}$ is in $\cF_{t,\ell_0,\ldots,\ell_{\nb}}$.   
    
    Further, let $\cF_t$ be the $\sigma$-algebra of events $E$ with $E \cap \{\Lc{i}_t \leq \ell_i,\, \forall i\leq {\nb}\} \in \cF_{t,\ell_0,\ldots,\ell_{\nb}}$ for every $\ell_0,\ell_1,\ldots,\ell_{\nb} \geq 0$. 
    Then, the processes $X_t$, $\Lc{i}_t$, and $s_i(\Lc{i}_t)$ are $\cF_t$-adapted.  
\end{proposition}
\begin{proposition}\label{prop: X_Markov}
    The process $X_t$ satisfies the weak Markov property. 
    That is, for any measurable $E \subseteq D$ and fixed $0 < u \leq t$, almost surely  
    \begin{align} 
        \bbP(X_t \in E \mid \cF_u) = \bbP(X_t \in E \mid X_u).  
    \end{align}
    Moreover, the process $\sX_t \de (X_t, (s_i(\Lc{i}_t)_{i=0}^{\nb}))$ satisfies the strong Markov property. 
    That is, for every $\cF_t$-stopping time $\tau$, every fixed $u \geq 0$, and every measurable $E \subseteq D \times \{-1,+1 \}^{{\nb}+1}$, almost surely
    \begin{align} 
        \bbP(\sX_{\tau + u} \in E \mid \cF_\tau) = \bbP(\sX_{\tau + u} \in E \mid \sX_\tau) 
    \end{align} 
    where $\cF_{\tau}$ is the $\sigma$-algebra consisting of events $\cE$ with $\cE \cap \{\tau \leq t \}\in \cF_t$ for all $t\geq 0$. 
\end{proposition}
\begin{remark}
    The strong Markov property is not satisfied by $X_t$ itself when ${\nb}\geq 1$.
    Indeed, consider the stopping time $\tau\de \inf\{t\geq 0: X_t \in B_i\}$ for some $i\geq 1$.
    Then, one can not infer from $X_\tau$ on what side of the barrier $X_{\tau + \varepsilon}$ should live.  
\end{remark}
The proofs of \Cref{prop: ProcessExistsAndIsUnique,prop: Fadapted,prop: X_Markov} are deferred to \Cref{apx: Proofs_GeneralPurposePrelim}.

\subsection{Local approximation technique for the barriers and process}\label{sec: Local}
We show in \Cref{sec: LocalApproxBarrier} that the barriers can locally be approximated by straight lines. 
Subsequently, we approximate the processes $X_t$ and $\Lc{i}_t$ in \Cref{sec: XtYtApprox,sec: LtApprox}, respectively.   
\subsubsection{Local approximation of the barriers}\label{sec: LocalApproxBarrier}
Consider a parameter $\delta >0$, and introduce
\begin{align} 
    r_\delta \de \delta\min\{1/\kappa, 1/\lambda_{\max},\rho \}. \label{eq: Def_R_eps}
\end{align} 
Recall from \eqref{eq: Def_kappa} that $\kappa$ is an upper bound on the curvature of the barriers. 
The relevance of $\delta\kappa^{-1}$ is hence that this is the scale at which the barriers are well-approximated by straight lines. 
The following elementary result, visualized in \Cref{fig: ZoomedStraight}, makes this precise. 

\begin{lemma}\label{lem: BarrierLocallyStraight}
    Suppose that $\delta < 1/2$. 
    Then, for every $x_0 \in \bbR^2$, there is at most one barrier $B_i$ which intersects $\sB(x_0,r_{\delta})$. 
    Moreover, if such a $B_i$ exists, then there exists a unit vector $ \hatn \in \bbR^2$ depending on $x_0$ and $B_i$ such that the following properties hold:
    \begin{enumerate}[leftmargin = 2em, label = (\arabic*)]
        \item\label{item:GhostlyQuip} There exists some $\mathfrak{c}\in \bbR$ such that $\lvert \langle  y, \hatn\rangle - \mathfrak{c} \rvert <  4\delta r_\delta$ for every $y\in B_i\cap \sB(x_0, r_\delta)$.    
        \item\label{item:IdleBat} One has $\Vert \vec{n}_i(y)-  \hatn \Vert < 2\delta$ for every $y\in B_i \cap \sB(x_0, r_\delta)$.
        \item\label{item:JollyDragon} Every point $z \in \sB(x_0, r_\delta)$ on the positive side of $B_i$ satisfies $ \langle  z, \hatn\rangle > \mathfrak{c} -  4\delta r_\delta$. 
        Similarly, every $z\in \sB(x_0, r_\delta)$ the negative side of $B_i$ satisfies $ \langle  z, \hatn\rangle < \mathfrak{c}  +  4\delta r_\delta$.    
    \end{enumerate}  
\end{lemma}  
\begin{figure}[t]
    \includegraphics[width = 0.75\textwidth]{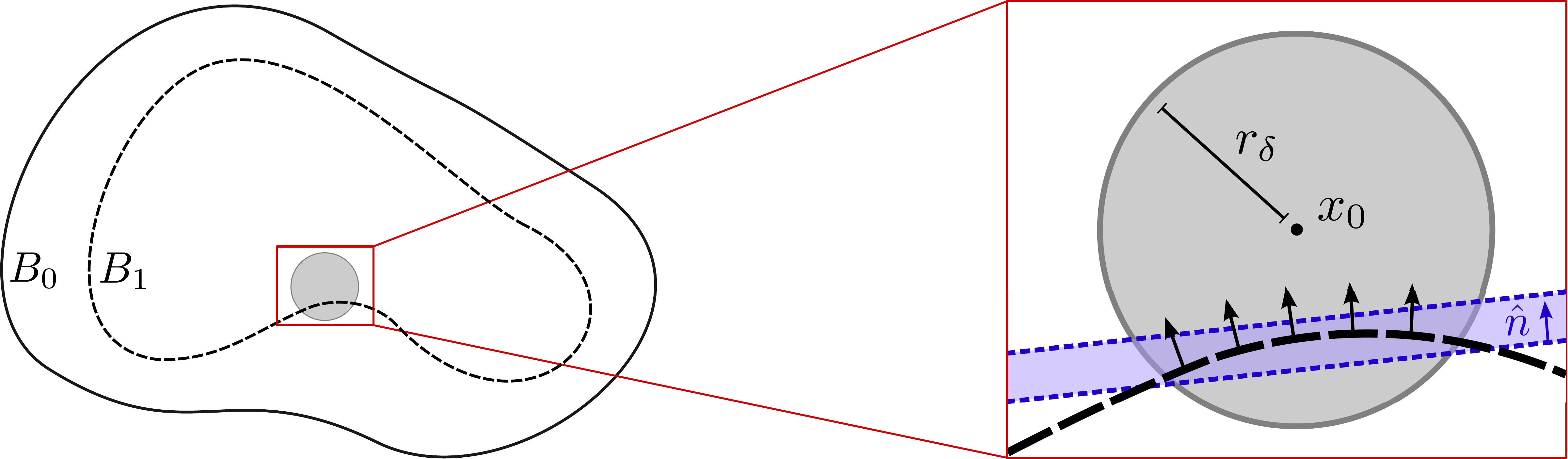}
    \caption{Visualization of \Cref{lem: BarrierLocallyStraight}.
    Locally in $\sB(x_0,r_{\delta})$, the barrier is contained in a strip of points $y$ with $\langle y, \hatn \rangle \approx \mathfrak{c}$, and the normal vectors $\vec{n}_i$ are well-approximated by $\hatn$. 
    }
    \label{fig: ZoomedStraight}
\end{figure}
This lemma follows from the Frenet--Serret formulas \cite[p.70--74]{dineen2014multivariate} which control the rate of change of the tangent and normal to $B_i$ in terms of curvature; see \Cref{apx: ProofBarrierLocallyStraight}.
\begin{remark}\label{rem: LocalStraight}
   \Cref{lem: BarrierLocallyStraight} would remain true if $r_\delta$ was replaced by $\min\{\delta /\kappa, \rho \}$. 
   However, $\lambda_{\max}$ will be relevant for some of the forthcoming argument following \Cref{cor: TauBound} 
   The current formulation serves to minimize notational burden, as $1/\lambda_{\max}$ will become a relevant scale of distance from \Cref{cor: TauBound} onwards.      
\end{remark}

\subsubsection{Coupling with a process reflecting on a straight line}\label{sec: XtYtApprox}
Let $\mathfrak{c}\in \bbR$ and $ \hatn \in \bbR^2$ be as in \Cref{lem: BarrierLocallyStraight} and define a straight line by   
\begin{align} 
    A^+ &\de \{y \in \bbR^2: \langle y, s_i(0) \hatn\rangle = s_i(0)\mathfrak{c} + 4\delta r_\delta \},\label{eq:SmartOrb}  
\end{align}
Then, we will compare $X_t$ with a process $Y^+_t$ which reflects on $A^+$ instead of $B_i$.

We have to specify an initial condition.   
In the ideal case, we take $Y_0^+ = x_0$. 
However, this is not always possible because we need the initial condition to lie on the side of $A^+$ towards which $s_i(0) \hatn$ points. 
This motivates the following definition:  
\begin{align} 
    y_0^+ &\de 
    \begin{cases}
        x_0 + (s_i(0)\mathfrak{c} + 4\delta r_\delta  - \langle  x_0, s_i(0) \hatn\rangle) s_i(0) \hatn& \text{ if }\langle x_0,  s_i(0) \hatn \rangle \leq s_i(0) \mathfrak{c} + 4\delta r_\delta ,\label{eq:RoyalLobster}\\ 
        x_0 &\text{ otherwise}.  
    \end{cases}
\end{align}
In other words, when $x_0$ lies on the side of $A^+$ towards which $s_i(0) \hatn$ points we take $y_0^{+} = x_0$, and otherwise we project onto $A^+$. 

Let $Y_t^+$ be a reflected Brownian motion with reflection on $A^+$, local time $K_t^+$, and initial condition $y_0^+$. 
That is, $Y^+_0 =y_0^+$ and the following stochastic differential equation holds: 
\begin{align} 
    \intd Y_t^+ = \intd W_t +  s_i(0) \hatn\, \intd K_t^+ \label{eq:MadRug}
\end{align} 
subject to the usual characterizing conditions\footnote{That is, $Y_t^+$ is continuous with values in $\{y\in \bbR^2:\langle y ,s_i(0) \hatn \rangle \geq s_i(0)\mathfrak{c} + 4\delta r_\delta  \}$, and $K_t^+$ is a continuous and nondecreasing process with $K_0^+ = 0$ which only increases when $Y_t^+ \in A^+$.\label{note: UsualConditions}}.
Crucially, it should here be understood that the $W_t$ in \eqref{eq:MadRug} is the same Wiener process as drives $X_t$ in \Cref{def: ReflectedBrownianMotion}; see \Cref{fig: ZoomedBM}. 

The following results show that $Y_t^+$ provides a good approximation to $X_t$ when $t$ is sufficiently small. 
More precisely, consider the stopping time defined by 
\begin{align} 
    \tau \de \inf\bigl\{t\geq 0: X_t \not\in \sB(x_0, r_\delta)\ \text{ or }\ s_i(\Lc{i}_t) \neq s_i(0) \text{ for some }i \leq {\nb}\bigr\}.\label{eq: Def_tau} 
\end{align}
Then, a good approximation will hold when $t\leq \tau$:
\begin{lemma}\label{lem: X_UpperLower_Y}
    Suppose that $\delta < 1/2$ and consider some $x_0\in D$ and $i\leq {\nb}$ with $\sB(x_0,r_\delta) \cap B_i \neq \emptyset$. 
    Then, for every $t\geq 0$,  
    \begin{align} 
        \langle Y_{\min\{t,\tau \}}^+  , s_i(0) \hatn \rangle - 8\delta r_\delta \leq \langle X_{\min\{t,\tau \}}, s_i(0) \hatn \rangle \leq \langle Y_{\min\{t,\tau \}}^+  , s_i(0) \hatn \rangle.\label{eq:SafeMouse}  
    \end{align} 
\end{lemma}

\begin{figure}[t]
    \centering
    \begin{subfigure}{.5\textwidth}
      \centering
      \includegraphics[width=.63\linewidth]{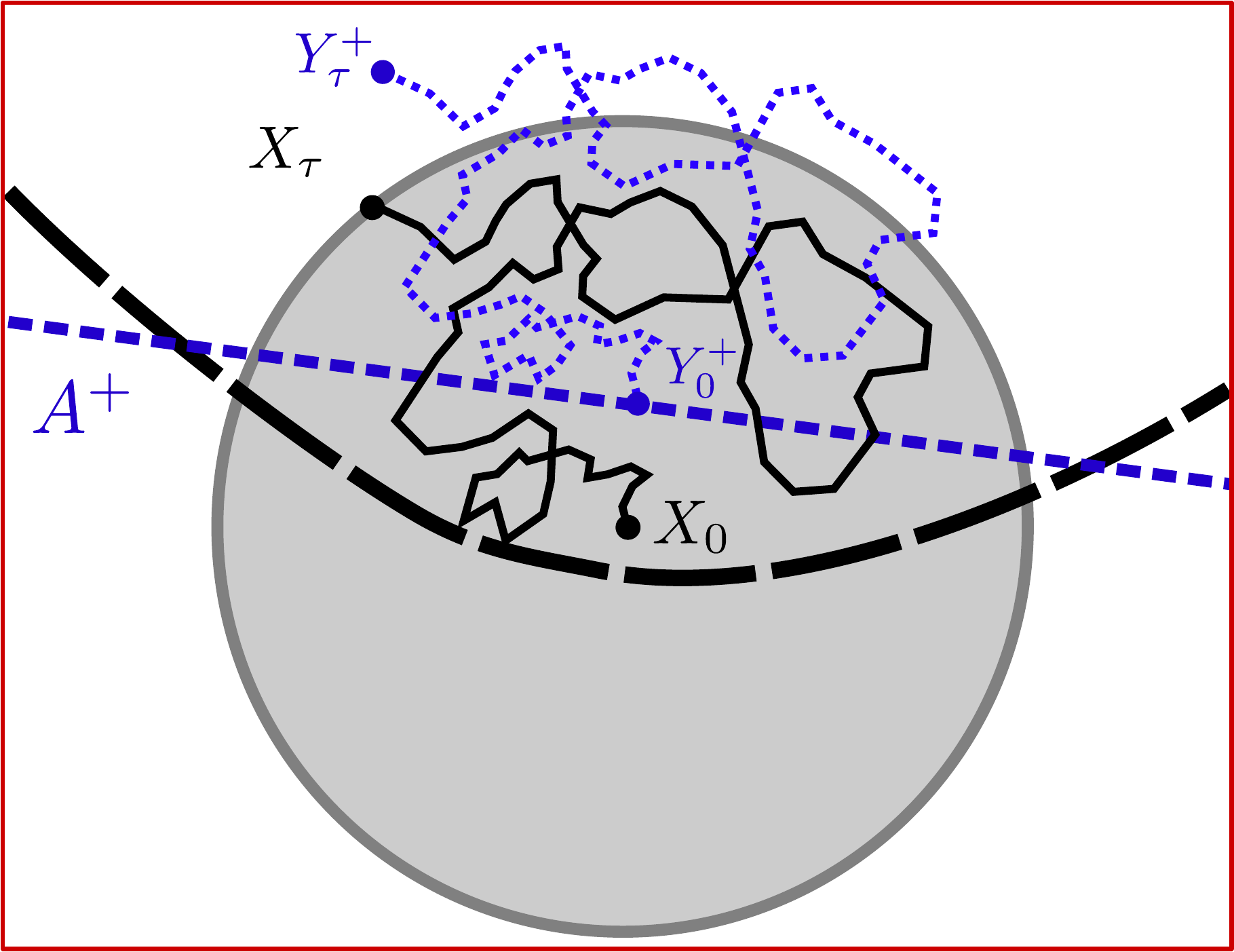}
    \end{subfigure}%
    \begin{subfigure}{.5\textwidth}
      \centering
      \includegraphics[width=.63\linewidth]{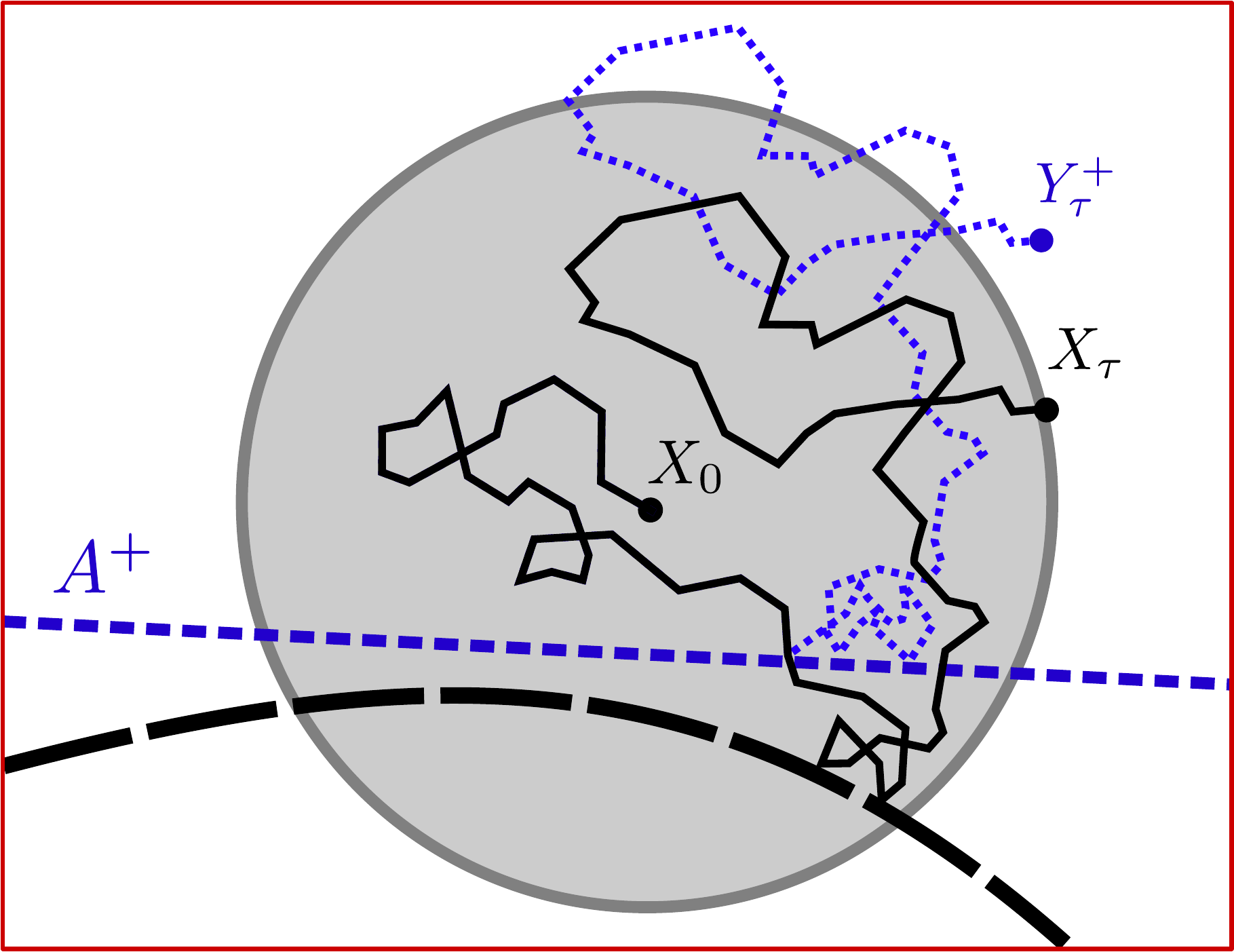}
    \end{subfigure}
    \caption{Visualization of the process $Y_t^+$ from \Cref{sec: XtYtApprox}.
    The figures on the left and right correspond to the first and second case of \eqref{eq:RoyalLobster}, respectively.
    }
    \label{fig: ZoomedBM}
\end{figure}
\begin{proof}
    Recall that $X_t$ and $Y_t^+$ are driven by the same Wiener process and recall the stochastic differential equations from \Cref{def: ReflectedBrownianMotion} and \eqref{eq:MadRug}.
    Then, using that $\Lc{j}_{t}=0$ for $t\leq \tau$ and $j\neq i$ since $\sB(x_0,r_\delta)$ does not intersect $B_j$, it holds for every $t\geq 0$ that 
    \begin{align} 
        X_{\min\{t,\tau \}}{}&{} - Y_{\min\{t,\tau \}}^+ \label{eq:SmartNut} \\ 
        &= x_0 - y_0^+ + {\textstyle \int_{0}^{\min\{t,\tau \}}} s_i(0) \vec{n}_i(X_r)\bb1\{X_r \in B_i \} \intd \Lc{i}_r - s_i(0) \hatn K_{\min\{t,\tau \}}^+.\nonumber    
    \end{align} 
    In particular, the function $f(t)\de \langle X_{\min\{t,\tau \}} - Y_{\min\{t,\tau \}}^+, s_i(0) \hatn \rangle$ can only increase at times when $X_{\min\{t,\tau \}} \in B_i$ and $t<\tau$.
    Using \cref{item:GhostlyQuip} from \Cref{lem: BarrierLocallyStraight} and using that $Y^+$ takes values on the side of $A^+$ towards which $s_i(0)\hat{n}$ points, it holds at such times that   
    \begin{align} 
        \langle X_{\min\{t,\tau \}},  s_i(0) \hatn \rangle < s_i(0)\mathfrak{c} + 4\delta r_\delta \ \text{ and }\ \langle Y_{\min\{t,\tau \}}^+,  s_i(0) \hatn \rangle\geq s_i(0)\mathfrak{c} + 4\delta r_\delta.
    \end{align}
    Hence, the function $f(t)$ can only increase when $f(t)<0$. 
    Since $f(0)\leq 0$ by definition of $y_0^+$ it follows that $f(t)\leq 0$ for all $t$. 
    The latter yields the upper bound in \eqref{eq:SafeMouse}.    

    We next consider the lower bound.  
    Note that $\langle \vec{n}_i(x), \hat{n} \rangle >0$ for every $x\in B_i \cap \sB(x_0,r_\delta)$ by \cref{item:IdleBat} from \Cref{lem: BarrierLocallyStraight}. 
    Hence, considering \eqref{eq:SmartNut}, the function $f(t)$ can only decrease at times when $Y_t^+ \in A^+$ and $t< \tau$. 
    By \eqref{eq:SmartOrb}, we have $Y_t^+ \in A^+$ if and only if $\langle Y_t^+, s_i(0) \hatn \rangle = s_i(0)\mathfrak{c} + 4\delta r_\delta$.
    Further, for every $t< \tau$ we have that $X_{t}$ lies in $\sB(x_0, r_\delta)$ on the side of $B_i$ towards which $s_i(0) \hatn$ points.
    Hence, also using \cref{item:JollyDragon} from \Cref{lem: BarrierLocallyStraight}, 
    \begin{align} 
        \langle X_{\min\{t,\tau \}}, s_i(0) \hatn \rangle > s_i(0)\mathfrak{c} - 4\delta r_\delta.\label{eq:GhostlyIce}
    \end{align}
    It follows that $f(t)$ can only decrease when $f(t) > -8\delta r_\delta$. 
    Combining this with the fact that $f(0) \leq -8\delta r_\delta$ by definition of $y_0^+$ yields $f(t)\leq -8\delta r_\delta$ for all $t$, proving the lower bound.  
\end{proof}

\begin{lemma}\label{lem: X_horizontal_Y}
    Assume that $\delta < 1/2$ and consider some $x_0\in D$ and $i\leq {\nb}$ with $\sB(x_0,r_\delta) \cap B_i \neq \emptyset$. 
    Let $\hatn^{\perp}\in \bbR^2$ be a unit vector orthogonal to $ \hatn$. 
    Then, for every $t\geq 0$,
    \begin{align} 
        \bigl\lvert \langle X_{\min\{t,\tau \}}   , \hatn^{\perp} \rangle  \rvert - \langle Y_{\min\{t,\tau \}}^+   , \hatn^{\perp} \rangle  \bigr\rvert\leq 2\delta \Lc{i}_{\min\{t,\tau \}} .\label{eq:HauntedMom}
    \end{align} 
\end{lemma}
\begin{proof}
    Recall from \eqref{eq:RoyalLobster} that $x_0 - y_0^+$ is a multiple of $ \hatn$. 
    Taking the inner product with $\hatn^{\perp}$ of \eqref{eq:SmartNut} and using that \cref{item:IdleBat} in \Cref{lem: BarrierLocallyStraight} implies that $\lvert \langle \vec{n}_i(x), \hatn^{\perp} \rangle \rvert \leq 2\delta$ for every $x\in B_i\cap \sB(x_0, r_\delta)$ then yields \eqref{eq:HauntedMom}. 
\end{proof}

\subsubsection{Local approximation for \texorpdfstring{$\Lc{i}_t$}{Lit} and consequences for the stopping time \texorpdfstring{$\tau$}{tau}}\label{sec: LtApprox}
To make effective use of \Cref{lem: X_UpperLower_Y,lem: X_horizontal_Y} it would be desirable to know that $t \leq \tau$ with high probability.
To this end, we start by studying $\Lc{i}_t$: 
\begin{lemma}\label{lem: TwoSidedLBound}
    Suppose that $\delta < 1/2$ and consider some $x_0\in D$ and $i\leq {\nb}$ with $\sB(x_0,r_\delta) \cap B_i \neq \emptyset$. 
    Then, with $ \hatn$ as in \Cref{lem: BarrierLocallyStraight}, for all $t\geq 0$ 
    \begin{align} 
        \Lc{i}_{\min\{t,\tau \}} &\leq (1-2\delta)^{-1}  \bigl\langle  X_{\min\{t,\tau \}} - X_0 - W_{\min\{t,\tau \}} ,s_i(0) \hatn\bigr\rangle,\label{eq:BraveLobster}  \\  
        \Lc{i}_{\min\{t,\tau \}} &\geq (1+2\delta)^{-1} \bigl\langle   X_{\min\{t,\tau \}} - X_0 - W_{\min\{t,\tau \}} ,s_i(0)  \hatn\bigr\rangle. \label{eq:OldGoose}
    \end{align}
    In particular, with $K_t^+$ as in \Cref{sec: XtYtApprox},  
    \begin{align} 
        (1+2\delta)^{-1}(K_{\min\{t,\tau \}}^+ - 8\delta r_\delta)   \leq \Lc{i}_{\min\{t,\tau \}} \leq (1-2\delta)^{-1}(K_{\min\{t,\tau \}}^+ + 8\delta r_\delta). \label{eq:VividFish}
    \end{align}
\end{lemma}
\begin{proof}
    By \Cref{lem: BarrierLocallyStraight}, $B_i$ is the only barrier which intersects $\sB(x_0,r_\delta)$ and \eqref{eq: Def_tau} implies that $s_i(\Lc{i}_t) = s_i(0)$ for all $t < \tau$. 
    Hence, by the stochastic differential equation in \Cref{def: ReflectedBrownianMotion}, 
    \begin{align} 
        X_{\min\{t,\tau \}} - X_0 = W_{\min\{t,\tau \}} + \int_{0}^{\min\{t,\tau \}}  s_i(0) \vec{n}_i(X_s) \bb1\{X_s \in B_i \}\intd \Lc{i}_s. \label{eq:LuckyInk}   
    \end{align} 
    \Cref{lem: BarrierLocallyStraight} yields that $\Vert  \hatn - \vec{n}_i(y) \Vert \leq 2\delta$ for every $y\in B_i \cap \sB(x_0, r_\delta)$. 
    Since $ \hatn$ is a unit vector, it follows that $\langle \vec{n}_i(y), \hatn \rangle\in [1- 2\delta, 1 + 2\delta]$ for every $y\in B_i \cap \sB(x_0, r_\delta)$. 
    The bounds in \eqref{eq:BraveLobster} and \eqref{eq:OldGoose} now follow by rearranging \eqref{eq:LuckyInk}, taking the inner product with $s_i(0) \hatn$, and using \cref{item: Def_ReflectedBrownianMotion_ii} from \Cref{def: ReflectedBrownianMotion}.  

    Further, using \eqref{eq:SafeMouse} at times $0$ and $t$, 
    \begin{align} 
        \langle X_{\min\{t,\tau\}} - X_0   ,s_i(0)\hatn \rangle  \leq \langle Y^+_{\min\{t,\tau\}} - y_0^+   ,s_i(0)\hatn \rangle   +  8\delta r_\delta.  
    \end{align}
    Note that $\langle Y_{\min\{t,\tau\}}^+ - y_0^+ - W_{\min \{t, \tau \}},s_i(0)\hatn \rangle = K_{\min\{t,\tau \}}^+$ by \eqref{eq:MadRug} to conclude the upper bound in \eqref{eq:VividFish}. 
    The lower bound proceeds similarly.   
\end{proof} 

\begin{lemma}\label{lem: LtauUpper}
    Suppose that $\delta \leq 1/4$ and consider some $x_0\in D$ and $i\leq {\nb}$ with $\sB(x_0,r_\delta) \cap B_i \neq \emptyset$. 
    Then, for every $t\geq 0$, 
    \begin{align} 
        K_{\min\{t,\tau\}}^+  &\leq \sup\{\langle- W_s, s_i(0)\hatn \rangle: s\leq t \}, \label{eq:ZombieBall}\\ 
        \Lc{i}_{\min\{t,\tau\}} &\leq 2\sup\{\langle- W_s, s_i(0)\hatn \rangle: s\leq t \} + 16\delta r_\delta,  \label{eq:RedPop}\\     
        \sup_{s\leq \min\{t,\tau\}}\Vert X_s - X_0  \Vert &\leq 3\sup_{s\leq t}\Vert W_s \Vert + 16\delta r_\delta.\label{eq:ZombieBat}
    \end{align}
\end{lemma}
\begin{proof}
    Since $Y_t^+$ reflects on a straight line, it is a classical that the local time admits an explicit formula; see \eg \cite[Proposition 1]{anderson1976small}.
    Specifically, the unique process satisfying \eqref{eq:MadRug} and the characterization of \cref{note: UsualConditions} on page \pageref{note: UsualConditions} is  
    \begin{align} 
        K_t^+ = \max\bigl\{0 ,\sup\{s_i(0)\mathfrak{c}+4\delta r_\delta  - \langle y_0^+ + W_s  , s_i(0)\hatn \rangle: s\leq t\} \bigr\}. \label{eq:OldJar}
    \end{align} 
    In particular, since $\langle y_0^+, s_i(0)\hatn \rangle \geq  s_i(0)\mathfrak{c} + 4\delta r_\delta $ we have \eqref{eq:ZombieBall}.
    It was here also used that the right-hand side of \eqref{eq:ZombieBall} is a nondecreasing function to simplify by letting $s$ range up to $t$ instead of $\min\{t,\tau \}$.   
    Now, \eqref{eq:VividFish} with the assumption that $\delta \leq 1/4$ implies \eqref{eq:RedPop}.
    Further, since $\Vert X_{\min\{t,\tau\}} - X_0 \Vert \leq  \Vert W_{\min\{t,\tau \}} \Vert + \Lc{i}_{\min\{t,\tau \}}$ by \eqref{eq:LuckyInk}, we then also have \eqref{eq:ZombieBat}.  
\end{proof}

\begin{corollary}\label{cor: TauBound}
    There exist absolute constants $\delta_0, C >0$ such that the following holds. 
    Suppose that $\delta \leq \delta_0$. 
    Then, for every $x_0\in D$ and $t\geq 0$, with $\tau$ as in \eqref{eq: Def_tau},  
    \begin{align} 
        \bbP(\tau \leq t) \leq \delta +  C(t/r_\delta^2). \label{eq:CozyXray}     
    \end{align}
\end{corollary}
\begin{proof}
    This result is stated for \emph{arbitrary} $x_0\in D$. 
    If $\sB(x_0,r_\delta) \cap (\cup_{i=0}^{\nb} B_i) = \emptyset$, then $\tau \leq t$ if and only if $\sup_{s\leq t} \Vert W_t \Vert \geq r_\delta$ since $\Lc{i}_t$ can only increase when $X_t \in B_i$; recall \eqref{eq: Def_tau} and \Cref{def: ReflectedBrownianMotion}. 
    Then, 
    $ 
        \bbP(\tau \leq t) \leq \bbP(\sup_{s\leq t}\Vert W_s \Vert  \geq r_\delta )
    $
    from which \eqref{eq:CozyXray} follows readily, and one does not require the first term; see also \eqref{eq:GhostlyHotel} below.
    
    Now assume that $\sB(x_0,r_\delta) \cap B_i \neq \emptyset$ for some $i\leq {\nb}$. 
    Then, by the law of total probability, 
    \begin{align} 
        \bbP\Bigl(\tau \leq t\Bigr) &\leq \bbP\Bigl(\sup_{s\leq t}\Vert W_s \Vert  \geq r_\delta /6  \Bigr) +  \bbP\Bigl(\tau \leq t \text{ and }  \sup_{s\leq t}\Vert W_s \Vert  < r_\delta /6  \Bigr)\label{eq:CalmBus}. 
    \end{align}
    Here, using the scaling principle together with Markov's inequality\footnote{Sharper bounds are also available, of course, but this elementary estimate suffices for our purposes.},  
    \begin{align} 
        \bbP\Bigl(\sup_{s\leq t}\Vert W_s \Vert  \geq r_\delta /6   \Bigr) = \bbP\Bigl(\sup_{s\leq 1}\Vert W_s \Vert^2  \geq \bigl(r_\delta /(6\sqrt{t} )\bigr)^2   \Bigr) \leq C (t/r_{\delta}^2)\label{eq:GhostlyHotel}
    \end{align}
    for some sufficiently large absolute constant $C>0$. 

    Regarding the second term in \eqref{eq:CalmBus}, the definition \eqref{eq: Def_tau} yields that $\tau \leq t$ if and only if $\sup_{s \leq \min\{t,\tau \}}\Vert X_s  - X_0 \Vert \geq r_{\delta}$ or $E \leq \Lc{i}_{\min\{t,\tau \}}$ with $E \de \inf\{\ell >0: s_i(\ell) \neq s_i(0) \}$. 
    To bound $\Vert X_s  - X_0 \Vert$, we can here use \eqref{eq:RedPop} from \Cref{lem: LtauUpper} and take  $\delta_0$ sufficiently small so that $r_\delta/2 + 16\delta r_\delta < r_\delta$. 
    Then, using \eqref{eq:ZombieBat} and some direct calculations to bound $\Lc{i}_{\min\{t,\tau \}}$,
    \begin{align} 
        \bbP\Bigl(\tau \leq t \text{ and }  \sup_{s\leq t}\Vert W_s \Vert  < r_\delta /6  \Bigr) \leq  \bbP\Bigl(E \leq  \frac{5}{6} r_\delta \Bigr). 
    \end{align}
    If $i = 0$, then it was defined that $s_i(t) = 1$ for all $t\geq 0$ and hence $\bbP(E \leq r_\delta) = 0$ because $E = \infty$. 
    Now consider $i\geq 1$. 
    Then, $s_i$ is a continuous-time Markov chain on $\{-1,1 \}$ with transition rates $\lambda_i^+$ and $\lambda_i^-$. 
    Hence, $E$ is exponentially distributed with mean $ 1/\lambda_i^+$ or $1/\lambda_i^-$, depending on $s_i(0)$. 
    In particular,  
    \begin{align} 
        \bbP\Bigl(E \leq \frac{5}{6}r_\delta\Bigr) \leq \bbP\Bigl(E \leq r_\delta\Bigr) \leq 1 - \exp\bigl(-\lambda_{\max} r_\delta  \bigr) \leq \lambda_{\max} r_\delta \leq \delta  \label{eq:VagueXray} 
    \end{align} 
    where we used \eqref{eq: Def_lambda_minmax} and \eqref{eq: Def_R_eps}. 
    Combine  \eqref{eq:CalmBus}--\eqref{eq:VagueXray} to conclude. 
\end{proof}  

The quality of the coupling is now immediate.
Indeed, combining \Cref{lem: X_UpperLower_Y,lem: X_horizontal_Y} with \eqref{eq:RedPop} from \Cref{lem: LtauUpper}, the distance between $X_{\min\{t,\tau \}}$ and $Y_{\min\{t,\tau \}}^+$ is at most $C\delta(r_{\delta} +\sup_{s\leq t} \Vert W_s \Vert)$ for some $C>0$. 
We may here assume that $\tau > t$ and $\sup_{s\leq t} \Vert W_s \Vert < r_\delta$ with high probability; recall \Cref{cor: TauBound} and \eqref{eq:GhostlyHotel}. 
This gives the following:
\begin{corollary}\label{cor: X_uniform_Y}
    There exist absolute constants $\delta_0, C_1, C_2 >0$ such that the following holds. 
    Suppose that $\delta \leq \delta_0$ and consider some $x_0\in D$ and $i\leq {\nb}$ with $\sB(x_0,r_\delta) \cap B_i \neq \emptyset$.
    Then, for every $t\geq 0$, with $Y^+_t$ as in \Cref{sec: XtYtApprox},
    \begin{align} 
        \bbP\Bigl( \sup_{s \leq t}\Vert X_s - Y_s^+ \Vert > C_1\delta r_\delta \Bigr) \leq \delta + C_2 \bigl(t/r_\delta^2\bigr).\label{eq:SafePizza} 
    \end{align}
    Moreover, for every $x_0 \in D$ with $\sB(x_0, r_\delta) \cap (\cup_{i=0}^{\nb} B_i)  =  \emptyset$, 
    \begin{align} 
        \bbP\Bigl( \sup_{s \leq t}\Vert X_s - (x_0 + W_s) \Vert > 0 \Bigr) \leq C_2 \bigl(t/r_\delta^2\bigr). \label{eq:BlueCity}
    \end{align} 
\end{corollary}

\begin{remark}\label{rem: ExactCenter}
    That $X_t$ starts from the exact center of the ball $\sB(x_0, r_\delta)$ is merely a notational convenience and is not essential for \eqref{eq:SafePizza}. 
    That is, if one instead considers a process starting from some $\tix_0$ on the same side as $x_0$ and with $\Vert x_0 - \tix_0 \Vert \leq r_\delta/2$, then the proof remains valid up to modifications to the absolute constants due to the fact that a stronger bound on $\sup_{s\leq t} \Vert W_s \Vert$ is then be required in \eqref{eq:CalmBus}.
\end{remark}

\section{Proof of \texorpdfstring{\Cref{thm: Main_GlobalRecovery}}{Theorem}}\label{sec: ProofsPropsFixedFreq}

Algorithms for the fixed-frequency regime are given in \Cref{sec: PrelimGlobal} and the proofs of consistency are given in \Cref{sec: ProofContinuity,sec: Proof_Discontinuity,sec: ProofGlobalConclude}.  
Let us note that, while our theoretical performance guarantees are only when $t\leq c \min\{1/\kappa^2, 1/\lambda_{\max}^2, \rho^2 \}$ for some small $c>0$, these algorithms are explicitly designed to also cope when the time between samples is large.
It would be interesting future work to also quantify the performance in such regimes. 
\subsection{Algorithms}\label{sec: PrelimGlobal}
The key concept captured by reflected Brownian motion with semipermeable barriers is that $X_t$ typically stays on the same side of barriers as $X_0$.
This suggests a discontinuity for the transition kernel $\bbP(X_t \in \cdot \mid X_0 = x_0)$ when $x_0$ crosses a barrier. 
Our estimation procedures will exploit this discontinuity.

For a notion of continuity, we will rely on the \emph{Wasserstein distance with truncation level $\mathfrak{u}>0$}.
That is, the metric for probability distributions on $\bbR^2$ defined by   
\begin{align} 
    \Wu (P,Q) \de \inf_{X \sim P, Y \sim Q} \bbE\bigl[d_{\mathfrak{u}}(X,Y)\bigr] \quad \text{ with } \quad d_{\mathfrak{u}}(x,y) \de \min\{\Vert x-y \Vert , \mathfrak{u}\}\label{eq: Def_W1u}    
\end{align}
where the infimum runs over couplings with marginal distributions $P$ and $Q$. 
The truncation is here mainly a technical convenience: it avoids worries regarding outliers in the observed data which could, in principle, lead to false positives. 

We further require an empirical estimator for the transition kernel. 
To this end, suppose that we are given a subset $S \subseteq \bbR^2$, set $\cI_S \de \{0 \leq i \leq \lfloor T/t\rfloor -1 : X_{it}\in S \}$, and define 
\begin{align} 
    \hat{P}_S \de   
    \begin{cases}
                \text{the zero measure }& \text{ if }\cI_{S} = \emptyset\\
                (\# \cI_{S})^{-1} \sum_{i\in \# \cI_{S}} \delta_{X_{(i+1)t}} & \text{ else.}
    \end{cases}
    \label{eq: Def_hatP}
\end{align}
Having the case $\cI_S = \emptyset$ in mind, let us also extend $\Wu$ by defining that $\Wu(P,o) \de \mathfrak{u}$ and $\Wu(o,o) \de 0$ when $o$ is the zero measure and $P$ is a probability measure. 

\Cref{alg: KernelDiscontinuity} discretizes the domain into small boxes and detects a discontinuity by comparing the associated transition kernels.
This estimation procedure is consistent: 
\begin{proposition}\label{prop: Main_GlobalRecovery_Algorithm}
    There exist absolute constants $c_1,\ldots,c_5 >0$ such that the following holds for every $\eta \in (0,1)$, $t\leq c_1 \min\{1/\kappa^2, 1/\lambda_{\max}^2 ,\rho^2\}$, $\varepsilon \leq c_2 \sqrt{t}$, and $T>0$. 

    Assume that $X_0\sim\pi$ starts in stationarity and that the observation time satisfies 
    \begin{align} 
        T \geq  c_3\frac{\tmix}{ \varepsilon^2 \pi_{\min}} \ln\Bigl(\frac{1}{\eta}\frac{\Area(D)}{\varepsilon^2} \Bigr)\label{eq:HighUser} 
    \end{align}
    Then, the output of \Cref{alg: KernelDiscontinuity} with sensitivity threshold $\mathfrak{s} \de c_4\sqrt{t}$, discretization scale $\epsilon \de \varepsilon/(3\sqrt{2})$, and truncation level $\mathfrak{u} \de c_5 \sqrt{t}$ satisfies $\bbP(\dH(\hat{B}, \cup_{i=0}^{\nb} B_i) \leq \varepsilon) \geq 1-\eta$. 
\end{proposition}

{\small
    \begin{algorithm} [ht]
        \caption{
        }
        \label{alg: KernelDiscontinuity}
    \begin{flushleft}
        \textbf{INPUT:} Observed data sequence $\{X_{it}: i=0,\ldots,\lfloor T/t \rfloor \}$.\\
        \hphantom{\textbf{INPUT:}} Sensitivity threshold $\mathfrak{s}>0$, discretization scale $\epsilon>0$, and truncation level $\mathfrak{u} >0$.
    \end{flushleft}
    \begin{minipage}{0.72\textwidth}
    \begin{flushleft} 
        \textbf{OUTPUT:} Subset $\hat{B}\subseteq \bbR^2$ approximating $\cup_{i=0}^{\nb} B_i$.   
    \end{flushleft}
        \begin{algorithmic}[1]
            \State $\hat{B} \gets \emptyset$.
            \ForAll {integers $j,k \in \bbZ$}
            \ForAll {integers $-2 \leq h_j,h_k \leq 2$} 
                \State $\cS(j+ h_j,k + h_k)$ 
                \Statex \hfill  $\gets [(j+h_j)\epsilon, (j+ h_j + 1) \epsilon]\times [(k+h_k)\epsilon, (k+h_k+1) \epsilon]$ 
            \EndFor 
            \If {$\Wu (\hat{P}_{\cS(j,k)}, \hat{P}_{\cS(j+h_j , k + h_k)}) \geq \mathfrak{s}$ for some $\lvert h_j \rvert, \lvert h_k \rvert \leq 2$}
                \State $\hat{B} \gets \hat{B} \cup \cS(j,k)$
            \EndIf
            \EndFor  
        \end{algorithmic}
    \end{minipage}%
    \hfill 
    \begin{minipage}{0.24\textwidth}
        \centering 
        \includegraphics[width = 0.95\textwidth]{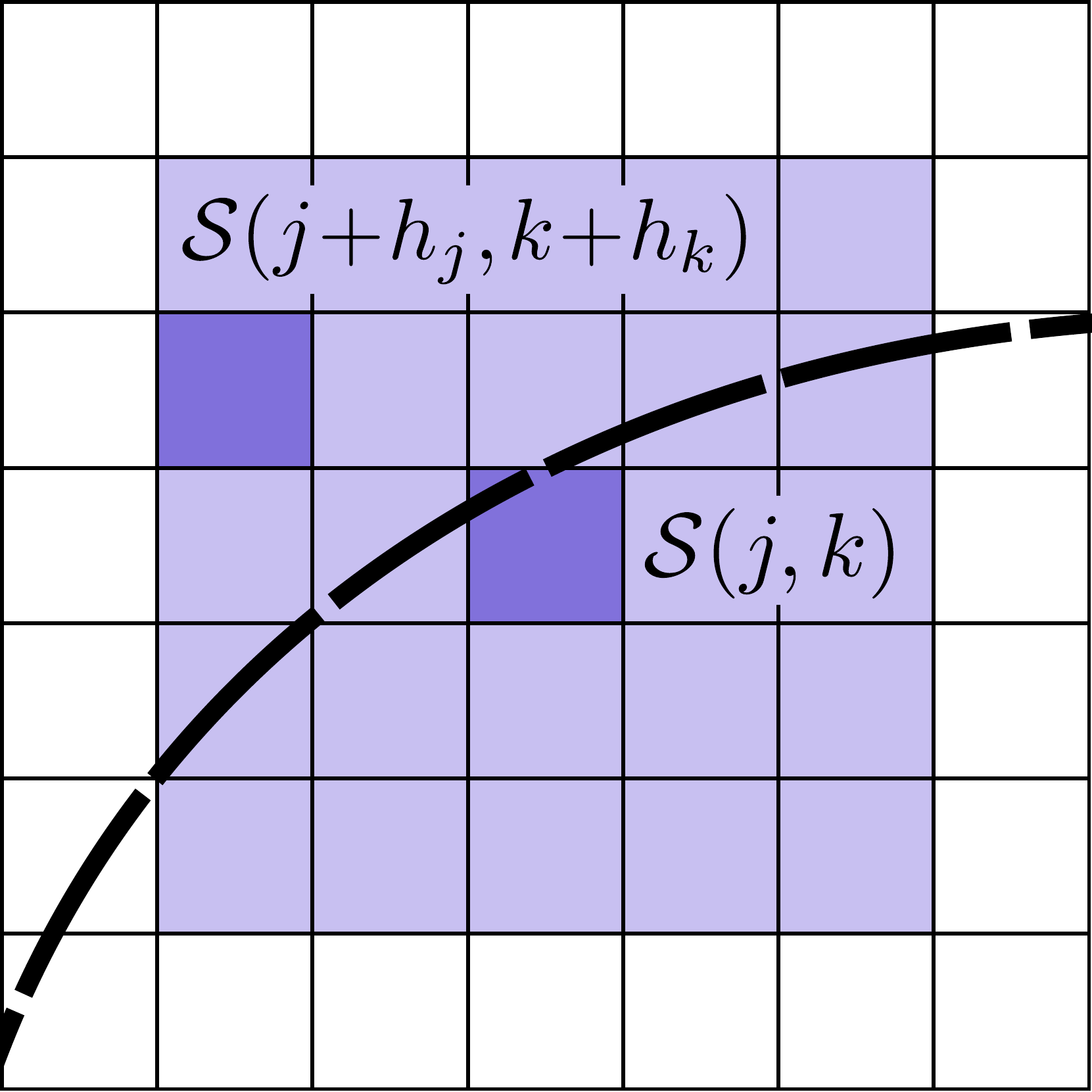}
    \end{minipage}
    \hfill 
    \end{algorithm}
}
\pagebreak[3]
The performance of \Cref{alg: KernelDiscontinuity} is satisfactory if one merely desires a rough estimate with $\varepsilon$ fixed.
Indeed, \eqref{eq:HighUser} then only asks that $T$ is large relative to $\tmix/\pi_{\min}$, and a condition of this type is necessary; recall \Cref{sec: Parameters}. 
However, if a precise estimate with $\varepsilon \to 0$ is desired, then the condition in \eqref{eq:HighUser} that $T$ should be of order $\geq \ln(1/\varepsilon)/\varepsilon^{2}$ is suboptimal.
To achieve the rate $\ln(1/\varepsilon)/\varepsilon^{3/2}$ claimed in \Cref{thm: Main_GlobalRecovery} we next iteratively refine the estimate.  

\Cref{alg: Improve} uses an initial estimate of the barriers to estimate their direction (line \ref{ln: If}) and subsequently looks for discontinuities in the transition kernel along parallel rectangular regions whose direction approximately matches that of the barrier (lines \ref{ln: ForAll} to \ref{ln: B'Up}).
Each application of the latter algorithm improves the error by a constant factor:
\begin{proposition}\label{prop: Main_GlobalRecovery_Improvement}
    There exist absolute constants $c_1,\ldots, c_7>0$ such that the following holds for every $\eta \in (0,1)$, $t\leq c_1 \min\{1/\kappa^2, 1/\lambda_{\max}^2,\rho^2\}$, $\varepsilon \leq c_2 \kappa t$, and $T>0$. 
    Assume that $X_0 \sim \pi$, assume that $T$ satisfies \eqref{eq: Main_GlobalRecovery_n} with respect to $c_3$, and define 
    \begin{align} 
        \mathfrak{s} \de c_4\sqrt{t},\ \ \ell \de c_5\sqrt{\varepsilon /\kappa},\ \ \epsilon \de \kappa\ell^2,\ \text{ and }\ \mathfrak{u} = c_6\sqrt{t}.\label{eq:UnripeNose}
    \end{align}

    Then, there exists an event $\cE$ with $\bbP(\cE) \geq 1 -\eta$ such that the following holds whenever $\cE$ occurs: for every $\sE >0$ satisfying $\varepsilon \leq \sE \leq  c_7\ell$ and every $\fB \subseteq \bbR^2$ with $\dH(\fB, \cup_{i=0}^{\nb}B_i) \leq \sE$ the output $\fB'$ of \Cref{alg: Improve} with the foregoing parameters satisfies 
    $
        \dH(\fB', \cup_{i=0}^{\nb} B_i) \leq \sE/2. 
    $  
\end{proposition}

{\small 
    \begin{algorithm} [ht]
        \caption{
        }
        \label{alg: Improve}
    \begin{minipage}{0.7\textwidth}
        \begin{flushleft}
            \textbf{INPUT:} Observed data sequence $\{X_{it}: i=0,\ldots,\lfloor T/t\rfloor \}$.\\
            \hphantom{\textbf{INPUT:}} Parameters $\mathfrak{s},\epsilon,\mathfrak{u}, \ell, \sE >0$ and initial estimate $\fB\subseteq \bbR^2$.\\ 
            \textbf{OUTPUT:} Improved estimate $\fB' \subseteq \bbR^2$ for $\cup_{i=0}^{\nb} B_i$.   
        \end{flushleft}
            \begin{algorithmic}[1]
                \State $\fB' \gets \emptyset$. \label{ln: B'}
                \ForAll {integers $j,k \in \bbZ$ and $0 \leq \nn \leq \lfloor 2\pi \ell / \epsilon \rfloor$}  
                \State $p_{j,k} \gets (j\epsilon, k \epsilon)$
                \State $\vec{w}_{\nn} \gets (\cos({\nn}\epsilon/\ell), \sin({\nn}\epsilon/\ell))$,\ $\vec{w}_{\nn}^{\perp} \gets (-\sin({\nn}\epsilon/\ell), \cos({\nn}\epsilon/\ell))$
                \If {$\max_{\pm \in \{+,-\}}\inf_{ y\in  \fB}\Vert p_{j,k} \pm \ell \vec{w}_{\nn} - y \Vert \leq \sE + 2\epsilon$} \label{ln: If}
                \ForAll {integers $-2 \leq h \leq 2$} \label{ln: ForAll}
                    \State $\cR(j,k,{\nn},h)\gets \{x \in \bbR^2:\lvert  \langle x- p_{j,k}, \vec{w}_{\nn}^{\perp}  \rangle -h\epsilon\rvert\leq  \epsilon/2,$ 
                    \Statex \hfill $\lvert \langle x- p_{j,k},\ \vec{w}_{\nn}  \rangle \rvert \leq \ell/10\}$
                \EndFor 
                \If {$\Wu (\hat{P}_{\cR(j,k,{\nn},0)}, \hat{P}_{\cR(j,k,{\nn},h)}) \geq \mathfrak{s}$ for some $\lvert h \rvert \leq 2$}
                    \State $\fB' \gets \fB' \cup \{p_{j,k}\}$ \label{ln: B'Up}
                \EndIf  
                \EndIf
                \EndFor  
            \end{algorithmic}
    \end{minipage}
    \hfill 
    \begin{minipage}{0.25\textwidth}
        \centering 
        \includegraphics[width = 0.95\textwidth]{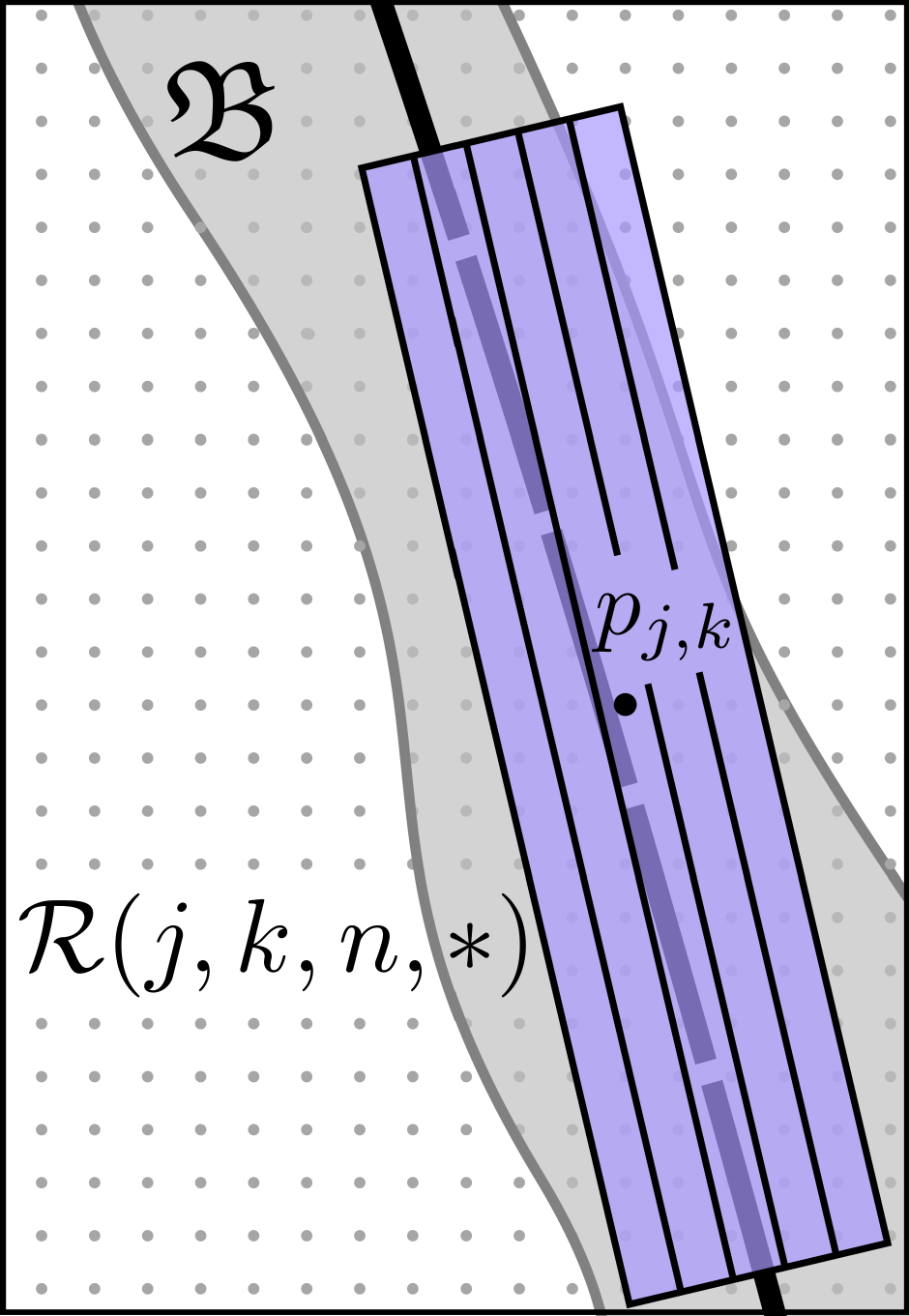}
    \end{minipage}
    \end{algorithm}
}
\Cref{thm: Main_GlobalRecovery} then follows by using \Cref{alg: KernelDiscontinuity} to produce an initial estimate and subsequently refining with \Cref{alg: Improve} until $\sE/2 \leq \varepsilon$. 
In fact, it suffices to apply \Cref{alg: Improve} $\lceil C\ln(1/\varepsilon\kappa)\rceil$-many times for some $C>0$; see \Cref{apx: Proof_Main_GlobalRecovery} for detailed computations. 
It remains to prove \Cref{prop: Main_GlobalRecovery_Algorithm,prop: Main_GlobalRecovery_Improvement}. 
This is done in \Cref{sec: ProofGlobalConclude}. 

\begin{remark}\label{rem: Performance}
    The improved performance of \Cref{alg: Improve} is ultimately because the convergence rate of $\hat{P}_S$ depends on the size of the set $S$; see \Cref{lem: Concentration_hatPS}.  
    A square box with all points at distance $\leq \varepsilon$ from a barrier has to have an area of order $\leq \varepsilon^2$, while a rectangular region can have area of order $\varepsilon \times \sqrt{\varepsilon/\kappa}  = \sqrt{\varepsilon^3/\kappa}$ if it is oriented parallel to the barrier.
\end{remark}
\begin{remark}
    Computation of Wasserstein distance can be accomplished in practice by reducing to a linear program \cite[\S 2.3]{peyre2019computational}.
\end{remark}
\begin{remark}
    An alternative algorithm which can work for arbitrarily large intersample spacing $t$ is to search for discontinuities in an empirical estimator for the stationary distribution.  
    However, performance will then only be good if the stationary distribution has a significant discontinuity on the semipermeable barriers. 
    This will not always be the case.  
\end{remark}
\subsection{Transition kernel behaves continuously on fixed side of the barriers}\label{sec: ProofContinuity}
To avoid false positives where a barrier is detected though there is none, we show that the transition kernel behaves continuously as a function of the initial condition so long as the latter does not cross any barrier. 

Note that the following estimate remains valid when the initial condition is arbitrarily close to a barrier. 
This is important to ensure that there is no limit to the resolution which \Cref{alg: KernelDiscontinuity,alg: Improve} can achieve. 
\begin{proposition}\label{prop: Continuity}
    For every $\zeta > 0$ there exist constants $c_1,c_2>0$ depending only on $\zeta$ such that for every $t\leq c_1 \min\{1/\kappa^2, 1/\lambda_{\max}^2,\rho^2 \}$, every $x_0, \tix_0 \in D$ which are on the same side of each barrier and satisfy $\Vert x_0 - \tix_0 \Vert  \leq c_2\sqrt{t}$, and every truncation level $\mathfrak{u}>0$,    
    \begin{align} 
        \Wu \Bigl(\bbP\bigl(X_t \in \cdot \mid X_0 = x_0\bigr), \bbP\bigl(X_t \in \cdot \mid X_0 = \tix_0\bigr)\Bigr)  \leq \zeta \bigr(\sqrt{t} + \mathfrak{u}\bigl). \label{eq:LoudYarn}
    \end{align}
\end{proposition}
\begin{proof}
    Consider processes $X_t$ and $\tiX_t$ satisfying \Cref{def: ReflectedBrownianMotion} with $X_0 = x_0$ and $\tiX_0 = \tix_0$ and driven by the same $W_t$ and $s_i$; see \Cref{fig: BX_Btilde} for a visualization.
    Then, recalling the definition of truncated Wasserstein distance from \eqref{eq: Def_W1u} and using this coupling, it suffices to show that  
    \begin{align} 
        \bbP(\Vert X_t - \tiX_t \Vert  > \zeta \sqrt{t}) \leq \zeta.\label{eq:WittyYarn} 
    \end{align} 
    To show this, we rely on the tools from \Cref{sec: Local} with 
    $
        \delta \de c_3\sqrt{t}/\min\{1/\kappa, 1/\lambda_{\max} ,\rho \}  
    $ where $c_3$ is a large $\zeta$-dependent constant whose value will be fixed further on.

    We start with the case where $\sB(x_0, r_\delta/2)\cap B_i \neq \emptyset$ for some $i\leq {\nb}$. 
    Then, in particular, $\sB(x_0,r_\delta)\cap B_i \neq \emptyset$. 
    Let $\delta_0, C_1, C_2$ be the constants from \Cref{cor: X_uniform_Y}. 
    The assumption on $t$ yields that $\delta \leq \delta_0$ if $c_1$ is sufficiently small to ensure that $c_3\sqrt{c_1} \leq \delta_0$.     
    Hence, 
    \begin{align} 
        \bbP(\Vert X_t - Y_t^+ \Vert > C_1\delta r_\delta ) \leq \delta + C_2 (t/r_\delta^2) \quad \text{ if }c_3\sqrt{c_1} \leq \delta_0.\label{eq:EmptyHotel} 
    \end{align} 
    Recall the definition of $r_\delta$ from \eqref{eq: Def_R_eps}, the definition of $\delta$ after \eqref{eq:WittyYarn}, and the assumption regarding $\Vert x_0 - \tilde{x}_0 \Vert$.
    It follows that $\Vert x_0 - \tix_0 \Vert \leq r_\delta/2$ whenever $c_2 < c_3/2$ is sufficiently small. 
    Hence, by \Cref{cor: X_uniform_Y} and \Cref{rem: ExactCenter}, possibly modifying $\delta_0, C_1,$ and $C_2$,    
    \begin{align} 
        \bbP(\Vert \tiX_t - \tiY_t^+ \Vert > C_1\delta r_\delta ) \leq \delta + C_2 (t/r_\delta^2) \quad \text{ if }\ c_3\sqrt{c_1} \leq \delta_0\ \text{ and }\ c_2 < c_3/2 \label{eq:GreenLog} 
    \end{align} 
    where $\tiY_t^+$ reflects on the same line $A^+$ as $Y_t^+$, but now with initial condition $\tiy_0^+$ defined as in \eqref{eq:RoyalLobster} with $x_0$ replaced by $\tix_0$. 
    (\Cref{rem: ExactCenter} is relevant as we want to keep the same reflection line as for $Y_t^+$, also when $\tilde{x}_0 \neq x_0$.)

    By \eqref{eq:MadRug} and the analogous stochastic differential equation for $\tiY_t^+$, 
    \begin{align} 
        Y_t^+ - \tiY_t^+ = y_0^+ - \tiy_0^+ + s_i(0)\hatn (K_t^ + -  \tiK_t^+).  \label{eq:CozyNose} 
    \end{align} 
    Since $\tiK_t^+$ is a nondecreasing process, the function $f(t) \de \langle Y_t^+ - \tiY_t^+, s_i(0)\hatn \rangle$ can only increase when $K_t^+$ does so.
    The latter occurs if and only if $Y_t^+ \in A^+$ in which case $f(t)\leq 0$ since $\tiY_t^+$ is defined to live on the side of $A^+$ towards which $s_i(0)\hatn$ points. 
    Similarly, $f(t)$ can only decrease if $f(t) \geq 0$.
    It follows that $\lvert f(t) \rvert$ is a nonincreasing function, and hence 
    \begin{align} 
        \lvert \langle Y_t^+ - \tiY_t^+, \hatn \rangle  \rvert \leq \lvert \langle y_0^+ - \tiy_0^+, \hatn \rangle  \rvert. 
    \end{align}
    Let $\hatn^{\perp}$ be a unit vector which is orthogonal to $\hatn$. 
    Then, \eqref{eq:CozyNose} yields that $\langle Y_t^+ - \tiY_t^+, \hatn^{\perp} \rangle = \langle y_0^+ - \tiy_0^+ , \hatn^{\perp}\rangle$ for all $t \geq 0$. 
    Now, using Pythagoras' theorem and the triangle inequality, 
    \begin{align} 
        \Vert Y_t^+ - \tiY_t^+ \Vert \leq \Vert y_0^+ - \tiy_0^+ \Vert \leq \Vert y_0^+ - x_0 \Vert + \Vert \tiy_0^+ - \tix_0 \Vert + \Vert x_0 - \tix_0 \Vert.\label{eq:RipeAnt} 
    \end{align}
    Using \cref{item:JollyDragon} from \Cref{lem: BarrierLocallyStraight} in the definition \eqref{eq:RoyalLobster} of $y_0^+$, it holds that $\Vert y_0^+ - x_0 \Vert \leq 8\delta r_\delta$. 
    Similarly, also using that $\tilde{x}_0$ is on the same side of the barrier as $x_0$ and inside $\sB(x_0,r_{\delta})$, we have $\Vert \tiy_0^+ - \tix_0 \Vert \leq 8\delta r_\delta$.   
    Further, recall that $\Vert x_0 -\tix_0 \Vert \leq c_2 \sqrt{t}$ by assumption.

    Hence, if the complementary events for \eqref{eq:EmptyHotel} and \eqref{eq:GreenLog} occur, then 
    \begin{align} 
        \Vert X_t - \tiX_t \Vert &\leq \Vert X_t - Y_t^+ \Vert + \Vert \tiX_t - \tiY_t^+ \Vert + \Vert Y_t - \tiY_t^+ \Vert \leq 2(C_1 + 8)\delta r_\delta + c_2 \sqrt{t}.  \label{eq:ElegantCrow}
    \end{align}
    Recall the definition of $\delta$ from the paragraph after \eqref{eq:WittyYarn} and recall the assumed upper bound on $t$. 
    It follows that $\delta \leq c_3\sqrt{c}_1$ and $r_\delta = c_3\sqrt{t}$. 
    Hence, combining \eqref{eq:ElegantCrow} with the probabilities from \eqref{eq:EmptyHotel} and \eqref{eq:GreenLog}, 
    \begin{align} 
        \bbP\bigl(\Vert X_t - \tiX_t \Vert > \bigl(2(C_1+8)c_3^2{}&{}\sqrt{c}_1 + c_2\bigr)\sqrt{t} \bigr)\label{eq:DarkGoose}\\ 
        &\leq 2c_3\sqrt{c}_1 + 2C_2/c_3^2   \text{ if } c_3\sqrt{c_1} \leq \delta_0 \text{ and }c_2 < c_3/2.\nonumber 
    \end{align}  
    This implies \eqref{eq:WittyYarn} if one takes $c_3$ sufficiently large to ensure that $2C_2/c_3^2 \leq \zeta/2$ and subsequently takes $c_1$ and $c_2$ sufficiently small to ensure that the conditions in \eqref{eq:DarkGoose} are satisfied, and that $2c_3\sqrt{c}_1< \zeta/2$ and $2(C_1+8)c_3^2\sqrt{c}_1 + c_2\leq \zeta$.  

    It remains to consider the case where $\sB(x_0, r_\delta/2 )\cap B_i = \emptyset$ for all $i\leq {\nb}$. 
    If $\sB(\tix_0, r_\delta/2 )\cap B_i \neq \emptyset$, then the preceding proof applies.
    Assume that $\sB(\tix_0, r_\delta/2 )\cap (\cup_{i=0}^{\nb} B_i)= \emptyset$.
    Then, \eqref{eq:BlueCity} in \Cref{cor: X_uniform_Y} yields that $X_t = x_0 + W_t$ and $\tiX_t = \tix_0 + W_t$ with high probability.
    Hence, \eqref{eq:WittyYarn} follows similarly to the foregoing.  
    This concludes the proof. 
\end{proof}

\begin{figure}[t]
    \includegraphics[width = 0.85\textwidth]{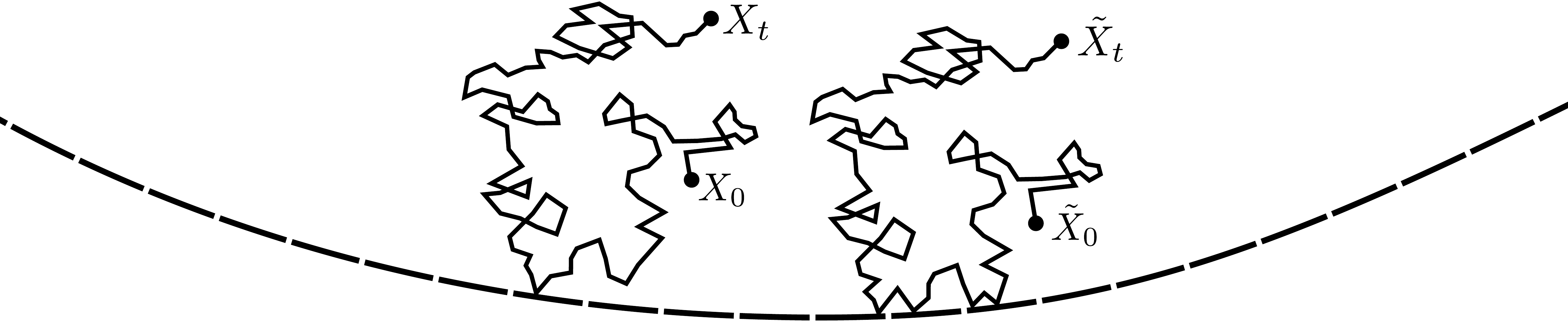}
    \caption{
        Visualization for the coupling used in the proof of \Cref{prop: Continuity}. 
    }
    \label{fig: BX_Btilde}
\end{figure}

\subsection{Transition kernel has a discontinuity when crossing barriers}\label{sec: Proof_Discontinuity}
Recall that the intuitive reason why we expect a discontinuity is that $X_t$ typically stays on the same side of all barriers as $X_0$ when $t$ is sufficiently small. 
The following result makes this rigorous: 
\begin{lemma}\label{lem: SameSideOfBarrier}
    For every $\eta \in (0,1)$ there exists a constant $c >0$ depending only on $\eta$ such that for every $t \leq c \min\{1/\kappa^{2}, 1/\lambda_{\max}^2,\rho^2\}$ and $x_0\in D$,
    \begin{align} 
        &\bbP(s_i(\Lc{i}_t) = s_i(0), \forall i\leq {\nb} \mid X_0 = x_0) \geq 1 -\eta. \label{eq:ValidJinx} 
    \end{align}
\end{lemma}
\begin{proof}
    This is immediate from \Cref{cor: TauBound} since $s_i(\Lc{i}_t) = s_i(0)$ when $\tau >t$; recall \eqref{eq: Def_tau}. 
    More specifically, let $\delta_0,C>0$ be as in \Cref{cor: TauBound} and set $\delta \de \min\{\delta_0, \eta/2\}$. 
    Then, the right-hand side of \eqref{eq:CozyXray} is $\leq \eta$ if $c$ is sufficiently small to ensure that $C(c/\delta^2) \leq \eta/2$. 
\end{proof}

\Cref{lem: SameSideOfBarrier} is highly suggestive of a discontinuity upon $x_0$ crossing $B_i$ but does not yet formally imply one. 
In principle, it is still possible that most probability mass is located on the barrier or very close to it.
The following result shows that this does not occur:
\begin{lemma}\label{lem: DistantBoundary}
    For every $\eta \in (0,1)$ there exist constants $c_1,c_2 >0$ depending only on $\eta$ such that for every $t\leq c_1 \min\{ 1/\kappa^2, 1/\lambda_{\max}^2,\rho^2 \}$ and $x_0\in D$,
    \begin{align} 
        \bbP(\forall y\in \cup_{i=0}^{\nb} B_i: \Vert X_t - y \Vert \geq c_2\sqrt{t} \mid X_0 = x_0) \geq 1 - \eta. 
    \end{align}
\end{lemma}
Similarly to \Cref{prop: Continuity}, this again follows from the approximation technique of \Cref{sec: Local}. 
Indeed, the process $Y_t^+$ can easily be shown to be at distance of order $\sqrt{t}$ from all barriers since its law is analytically tractable; see \Cref{apx: ProofDistantBoundary} for detailed computations.

We now get discontinuity when the initial condition crosses a barrier. 
Further, this discontinuity persists if the initial conditions are allowed to be random: 
\begin{corollary}\label{cor: Discontinuity}
    There exist absolute constants $c_1,c_2,c_3>0$ such that for every $t\leq c_1 \min\{1/\kappa^2, 1/\lambda_{\max}^2,\rho^2 \}$ and every truncation level $\mathfrak{u} \geq c_2\sqrt{t}$,  
    \begin{align} 
        \Wu \bigl(\bbP(X_t \in \cdot), \bbP(\tiX_t \in \cdot ) \bigr) \geq c_3\sqrt{t}   
    \end{align} 
    where $X_t$ and $\tiX_t$ are processes satisfying \Cref{def: ReflectedBrownianMotion} for which there is a barrier $B_i$ with $X_0$ on the positive side and $\tiX_0$ on the negative side almost surely.  
\end{corollary}
\begin{proof}
    Consider an arbitrary coupling $(X_t , \tiX_t)$. 
    Then, by \Cref{lem: SameSideOfBarrier} with $\eta = 1/4$ and the union bound, if $c_1$ is sufficiently small,   
    \begin{align} 
        \bbP(X_t \text{ and } \tiX_t \text{ are on different sides of }B_i) \geq 1/2.\label{eq:ZenForce}   
    \end{align} 
    By \Cref{lem: DistantBoundary} with $\eta = 1/4$, there exists $c_2>0$ such that when $c_1$ is sufficiently small, 
    \begin{align} 
        \bbP(X_t \text{ is at distance } \geq c_2\sqrt{t}\text{ from }B_i ) \geq 1 - 1/4.\label{eq:PoliteOwl}  
    \end{align} 
    The combination of the events described in \eqref{eq:ZenForce} and \eqref{eq:PoliteOwl} implies that $\Vert X_t - \tiX_t \Vert \geq c_2\sqrt{t}$. 
    Hence, by the law of total expectation, we have with $c_3 \de c_2/4$ that for every $\mathfrak{u} \geq c_2\sqrt{t}$,  
    \begin{align} 
        \bbE[\min\{\Vert X_t - \tiX_t \Vert, \mathfrak{u} \} ] \geq c_2\sqrt{t}\bbP(\Vert X_t - \tiX_t \Vert \geq c_2\sqrt{t}) \geq  c_3\sqrt{t}. 
    \end{align}
    Recall the definition of truncated Wasserstein distance from \eqref{eq: Def_W1u} and use that the coupling $(X_t, \tiX_t)$ is arbitrary to conclude the proof. 
\end{proof}

\subsection{Proof of \texorpdfstring{\Cref{prop: Main_GlobalRecovery_Algorithm,prop: Main_GlobalRecovery_Improvement}}{Propositions}}\label{sec: ProofGlobalConclude}
The behavior established for the transition kernel in \Cref{prop: Continuity} and \Cref{cor: Discontinuity} persists in the empirical estimator from \eqref{eq: Def_hatP}: 
\begin{lemma}\label{lem: Concentration_hatPS}
    For every $\alpha, \beta >0$  there exist constants $c_1, c_2, c_3 >0$ depending only on $\alpha$ and $\beta$ such that the following holds for every $t\leq c_1\min\{ 1/\kappa^2, 1/\lambda_{\max}^2,\rho^2 \}$ and $\mathfrak{u} \leq \alpha \sqrt{t}$.

    Assume that $X_0\sim \pi$ starts in stationarity.
    Then, for every $S\subseteq D$ with $\Area(S)>0$ and diameter $\sup_{a,b\in S}\Vert a-b \Vert \leq \sqrt{t}$, denoting $P_{S} \de \bbP(X_t\in \cdot \mid X_0 \in S)$,     
    \begin{align} 
        \bbP\bigl(\Wu(\hat{P}_S, P_S) \leq \beta \sqrt{t}\bigr) \geq 1 - c_2 \exp\bigl(-c_3 \pi_{\min} \Area(S) T/\tmix \rceil \bigr).
    \end{align} 
\end{lemma}
This follows by using a Markovian Bernstein inequality from \cite{paulin2015concentration} to show that the number of transitions from $S$ to a fixed small target region concentrates, and subsequently taking a union bound over a large but fixed number of target regions; see \Cref{apx: Concentration}.

\Cref{prop: Main_GlobalRecovery_Algorithm} regarding the performance of \Cref{alg: KernelDiscontinuity} now follows readily. 
First, using \Cref{lem: Concentration_hatPS} with the lower bound on $T$ from \eqref{eq:HighUser}, it may be ensured that $\hat{P}_{\cS(j,k)} \approx P_{\cS(j,k)}$ for all boxes $\cS(j,k)\subseteq D$.  
If $\cS(j,k)$ is at a sufficient distance from all barriers so that $\cS(j+h_j,k+h_k)$ is on the same side for every $\lvert h_j \rvert, \lvert h_k \rvert \leq 2$, then \Cref{prop: Continuity} implies that continuity holds for the estimated transition kernel and \Cref{alg: KernelDiscontinuity} will not add $\cS(j,k)$ to $\hat{B}$. 
Conversely, \Cref{cor: Discontinuity} yields a discontinuity when $\cS(j,k)$ is sufficiently close to some barrier $B_i$, in which case \Cref{alg: KernelDiscontinuity} adds $\cS(j,k)$ to $\hat{B}$. 
Detailed computations making this outline rigorous are given in \Cref{apx: ProofGlobalRecovery}.  

The proof idea for \Cref{prop: Main_GlobalRecovery_Improvement} is similar, although more technical in execution.
Fix some small $\beta >0$ and let $\cE$ be the event where concentration occurs: 
\begin{align} 
    \cE \de \{\omega: \Wu(\hat{P}_{\cR(j,k,{\nn},h)}, P_{\cR(j,k,{\nn},h)}) \leq \beta \sqrt{t},\ \ \forall j,k,{\nn},h \text{ with }  \cR(j,k,{\nn},h) \subseteq D\}.\label{eq:ZippyImp} 
\end{align} 
\Cref{lem: Concentration_hatPS} with the union bound then ensures that $\cE$ holds with high probability, so it remains to show that \Cref{alg: Improve} improves the error by a constant factor whenever $\cE$ occurs.

The idea is again to rely on \Cref{prop: Continuity} and \Cref{cor: Discontinuity}. 
In order for those results to be applicable, however, it has to be shown that the orientation of the barriers is estimated sufficiently precisely in \Cref{alg: Improve} to ensure that the rectangular regions $\{\cR(j,k,{\nn},h): \lvert h \rvert  \leq 2\}$ all remain on a fixed side when $p_{j,k}$ is at distance $>\sE/2$ from the barriers, or that they change sides when $p_{j,k}$ is close to a barrier.
The latter follows by using the assumption that $\dH(\fB,\cup_{i=0}^{\nb} B_i)\leq \sE$ together with the fact that the curvature of the barriers is bounded.  
Specifically, this follows from the variant of \Cref{lem: BarrierLocallyStraight} outlined in \Cref{rem: LocalStraight}.  
Details for the proof of \Cref{prop: Main_GlobalRecovery_Improvement} are given in \Cref{apx: Proof_Main_GlobalRecovery_Improvement}.

\section{Proof of \texorpdfstring{\Cref{thm: Main_PathDependentRecovery}}{Theorem} and \texorpdfstring{\Cref{cor: Xuniform}}{Corollary}}\label{sec: ProofPathDependent}
An algorithm for the high-frequency regime is given in \Cref{sec: AlgPathDependent} and the proofs of consistency are given in \Cref{sec: ConcTransition,sec: Main_PathDependentRecovery_Algorithm}.

\subsection{Algorithm}\label{sec: AlgPathDependent}
Consider a point $p\in \bbR^2$ and a unit vector $\vecv\in \bbR^2$.  
Then, with $\vec{v}^{\perp}$ a unit vector orthogonal to $\vecv$ and $\ell >0$ a parameter to be chosen later, we define a parallel rectangular strip and half-space by 
\begin{align} 
    R_+(p,\vecv) &\de \bigl\{q\in \bbR^2:\langle q-p, \vecv  \rangle\in [\sqrt{t}, 2\sqrt{t}],\, \langle q-p, \vec{v}^{\perp}  \rangle \in [-\ell, \ell] \bigr\},\label{eq: Def_Splus}\\
    H_{-}(p,\vecv) &\de \bigl\{q\in \bbR^2:\langle q-p, \vecv  \rangle < -\sqrt{t} \bigr\}. \label{eq: Def_Smin} 
\end{align}
It is intuitive that $X_t$ will then only rarely transition from $R_+(p,\vecv)$ to $H_{-}(p,\vecv)$ if a barrier separates these sets.
Correspondingly, \Cref{alg: HighFrequency} employs a grid search over $p$ and $\vecv$ and detects a barrier if $R_+(p,\vecv)$ was visited many times but there were few transitions to $H_{-}(p,\vecv)$. 
This algorithm has the desired properties:
\begin{proposition}\label{prop: Main_PathDependentRecovery_Algorithm}
    Adopt the notation of \Cref{thm: Main_PathDependentRecovery}. 
    Then there are constants $c_4,\ldots,c_7$ depending only on $\eta$ such that the performance guarantees in \Cref{thm: Main_PathDependentRecovery} and \Cref{cor: Xuniform} are satisfied if $\hat{\cX}$ is the output of \Cref{alg: HighFrequency} with parameters 
    \begin{align} 
        \mathfrak{s} \de c_4\ \ \epsilon \de c_5 \sqrt{t}, \ \ \ell \de c_6\ln(T/t)\sqrt{t} \ \text{ and }\  \mathfrak{n}_0 \de c_{7}\ln(T/t). 
    \end{align}
\end{proposition}

{
\small 
\begin{algorithm} [ht]
    \caption{}
    \label{alg: HighFrequency}
\begin{minipage}{0.6\textwidth}
    \begin{flushleft}
    \textbf{INPUT:} Observed data sequence $\{X_{it}: i=0,\ldots, \lfloor T/t \rfloor\}$.\\
    \hphantom{\textbf{INPUT:}}  Parameters $\mathfrak{s}, \epsilon, \ell, \mathfrak{n}_0>0$.\\ 
    \textbf{OUTPUT:} Subset $\hat{\cX}\subseteq \bbR^2$ approximating $\cup_{i=0}^{\nb} \cX_i$.   
\end{flushleft}
    \begin{algorithmic}[1]
        \State $\hat{\cX} \gets \emptyset$.
        \ForAll {integers $j,k \in \bbZ$ and $0 \leq n \leq \lfloor 2\pi \ell /\epsilon \rfloor$}
        \State $p_{j,k} \gets (j\epsilon, k \epsilon)$;\quad  $\vec{v}_n \gets (\cos(n\epsilon / \ell), \sin(n\epsilon / \ell))$  
        \State $N_{j,k,n} \gets \#\{0 \leq i \leq \lfloor T/t \rfloor - 1:X_{it} \in R_+(p_{j,k},\vec{v}_n) \}$ 
        \EndFor 
    \end{algorithmic}
\end{minipage}
\hfill
\begin{minipage}{0.36\textwidth}
    \centering 
    \includegraphics[width = 1\textwidth]{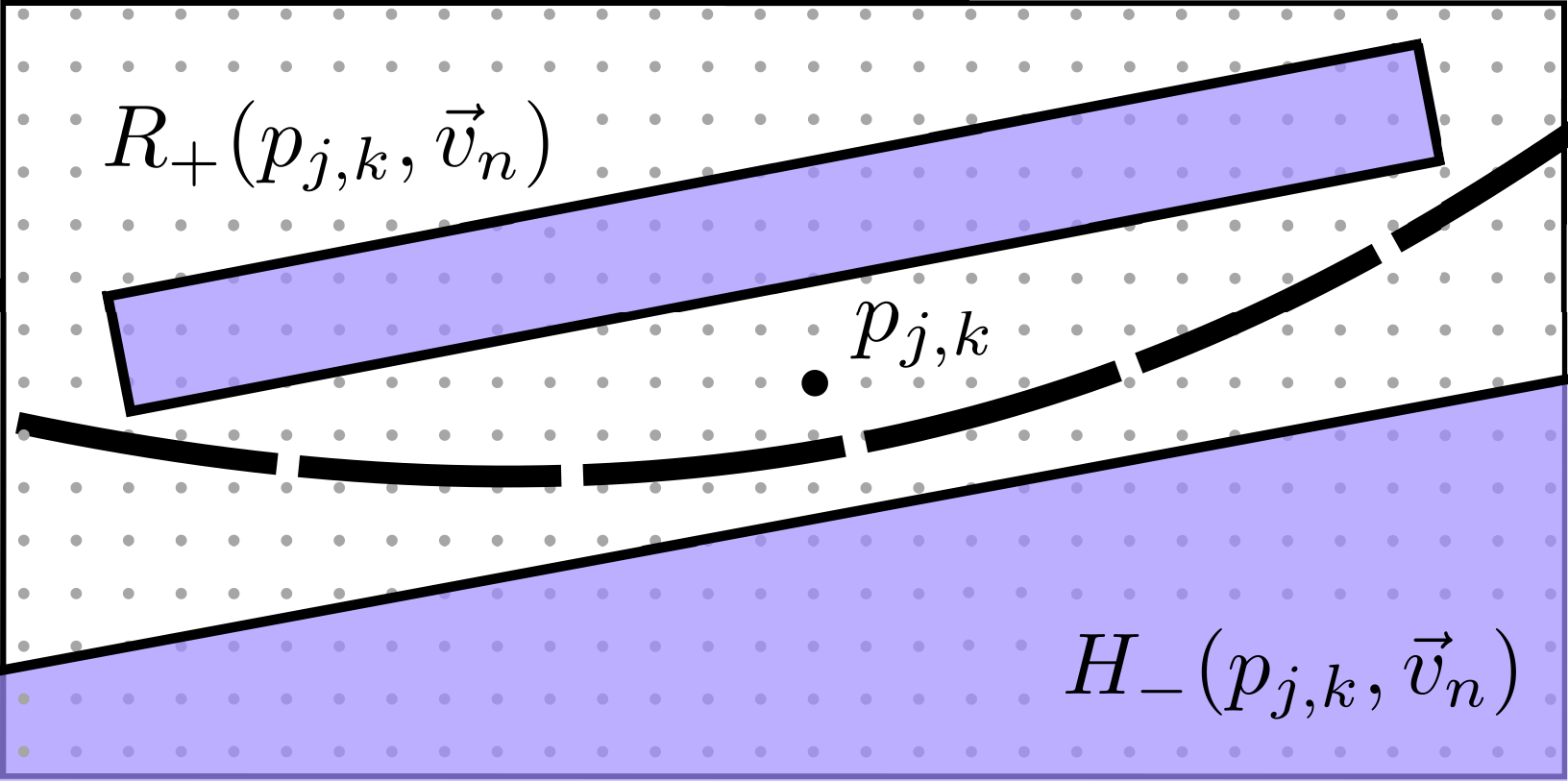}
\end{minipage}
    \begin{algorithmic}[1]
        \setcounter{ALG@line}{4}
        \Indent 
        \State $M_{j,k,n} \gets \#\{0\leq i \leq \lfloor T/t \rfloor - 1:X_{it} \in R_+(p_{j,k},\vec{v}_n),\, X_{(i+1)t} \in H_{-}(p_{j,k},\vec{v}_n)  \}$ 
        \If {$N_{j,k,n}\geq \mathfrak{n}_0$ and $M_{j,k,n}/ N_{j,k,n} < \mathfrak{s}$}
        \State $\hat{\cX} \gets \hat{\cX}\cup \{p_{j,k} \}$ 
        \EndIf    
        \EndIndent
    \end{algorithmic}
\end{algorithm}
}

\begin{remark}
    \Cref{alg: HighFrequency} with appropriate parameters can also be shown to be consistent in the fixed-frequency setting of \Cref{sec: PrelimGlobal}, but will be more sensitive to the sampling rate. 
    Specifically, it is intuitive that  \Cref{alg: KernelDiscontinuity,alg: Improve} should typically continue working when $t\gg \min\{1/\kappa^2, 1/\lambda_{\max}^2,\rho^2\}$, although a greater observation period may be needed to compensate for the discontinuity in the transition kernel becoming less pronounced. 

    On the other hand, \Cref{alg: HighFrequency} will generally stop working altogether when $t$ is large, no matter how great the observation period, as the linear approximation of the barriers fails at the spatial scale which is then necessary.  
\end{remark}
\subsection{Concentration for the number of transitions and number of visits}\label{sec: ConcTransition}
Random fluctuations can cause deviations from the expected behavior in the transition or visit counts used in \Cref{alg: HighFrequency}.
This could potentially result in false positives or negatives. 

The following result will be used to control such fluctuations despite the fact that the considered time frame does not necessarily exceed the mixing time of the Markov chain. 
The proof amounts to an application of Hoeffding's inequality and is deferred to \Cref{apx: ProofDriftExponential}.   
\begin{lemma}\label{lem: DriftExponential}
    Let $Y_1,Y_2,\ldots$ be a sequence of $\{0,1 \}$-valued random variables, not necessarily independent or identically distributed, such that 
    $
        \bbP(Y_n = 1 \mid Y_1,\ldots,Y_{n-1}) \geq q$ 
    for some $q>0$ and every $n\geq 1$. 
    Then, for every $\eta \in (0,q)$ and $n_0 \geq 1$,  
    \begin{align} 
        \bbP\Bigl(\sum_{i=1}^{n}Y_i \geq n(q-\eta)  \  \text{ for all }\  n \geq n_0  \Bigr) \geq 1-2\eta^{-2} \exp\Bigl(-2\eta^2 n_0 \Bigr).  \label{eq:BlueOtter} 
    \end{align}    
\end{lemma}

\subsubsection{Transitions from \texorpdfstring{$R_+(p,\vecv)$}{R(p,v)} to \texorpdfstring{$H_{-}(p,\vecv)$}{H(p,v)}} 
The following result shows that the number of transitions behaves ``as expected'' if $R_+(p,\vecv)$ was visited sufficiently many times. 
This explains the role of the parameter $\mathfrak{n}_0$ in \Cref{alg: HighFrequency}: it reduces false positives by pruning locations which were not visited enough and could hence suffer from large fluctuations.   

\begin{corollary}\label{cor: NumTransitions}
    There exist absolute constants $c_1,\ldots,c_7 >0$ such that the following holds for every $\sqrt{t}\leq \ell \leq c_1 (t \min\{1/\kappa^2, 1/\lambda_{\max}^2,\rho^2\})^{1/4}$, $n_0\geq 1$, and initial condition $x_0\in D$.
    
    Given some $T>0$, a point $p\in \bbR^2$, and a unit vector $\vecv \in \bbR^2$, denote 
    \begin{align} 
        N(p,\vecv) &\de \#\{0 \leq i \leq \lfloor T/t \rfloor-1: X_{it} \in R_+(p,\vecv)  \} \label{eq:VividImp},\\  
        M(p,\vecv) &\de \#\{0 \leq i \leq \lfloor T/t \rfloor - 1: X_{it} \in R_+(p,\vecv) , X_{(i+1)t} \in H_{-}(p,\vecv) \}.\label{eq:ValidOrb} 
    \end{align}
    Then, if $p$ satisfies $\Vert p - y \Vert\geq c_2 \ell$ for every $y \in \cup_{i=0}^{\nb} B_i$,
    \begin{align} 
        \bbP\bigl(M(p,\vecv) / N(p,\vecv) < c_3 \text{ and }  N(p,\vecv) \geq n_0  \bigr) \leq c_4\exp(- c_5 n_0).\label{eq:TallSax}  
    \end{align}
    On the other hand, if $\Vert p - y \Vert \leq c_6 \sqrt{t}$ and $\min_{\pm \in \{+,- \}}\Vert \vecv \pm  \vec{n}_i(y) \Vert \leq c_7\sqrt{t}/\ell$ for some $y \in B_i$,
    \begin{align} 
        \bbP\bigl(M(p,\vecv) / N(p,\vecv) \geq c_3 \text{ and }  N(p,\vecv) \geq n_0  \bigr) \leq c_4\exp(- c_5 n_0).  \label{eq:ZippyCity}
    \end{align}
\end{corollary}

The proof is given in \Cref{apx: ProofNumTransition} and relies on an application of \Cref{lem: DriftExponential} with $Y_n$ the random variable indicating if the $n$th visit to $R_+(p,\vecv)$ was followed by a transition to $H_{-}(p,\vecv)$. 
Such transitions are unlikely in the setting of \eqref{eq:ZippyCity} where the sets are separated by a barrier, but not unlikely in the setting of \eqref{eq:TallSax} where there is no barrier in the way.

\subsubsection{Many visits to \texorpdfstring{$R_+(p,\vecv)$}{S+(p,v)}}\label{sec: ManyVisits}
\Cref{cor: NumTransitions} will suffice to avoid false positives where some $p_{j,k}$ far away from all barriers is added to $\hat{\cX}$ by \Cref{alg: HighFrequency}.
However, an algorithm which does not result in false positives is only useful if it also returns some true positives where points on the barrier are recovered. 
To this end, it has to be shown that the count $N(p,\vecv)$ from \eqref{eq:VividImp} is large with high probability. 

Again, the tools from \Cref{sec: Local} can be used to approximate the barrier with a straight line.
This leads to the following preliminary reduction; see \Cref{apx: ReductionToBrownianCount} for the proof.
\begin{lemma}\label{lem: NewManyHits}
    For every $\zeta <1/2$ and $\eta \in (0,1)$ there exist  $c_1,\ldots, c_5>0$ such that the following holds. 
    Consider $t,\ell >0$, $x_0 \in B_i$, $p\in \bbR^2$, and a unit vector $\vecv \in \bbR^2$ with 
    \begin{align} 
        \sqrt{t} \leq \ell &\leq c_1 (t \min\{1/\kappa^2, 1/\lambda_{\max}^2,\rho^2 \})^{1/4}, \qquad &\ell \leq c_2 \min\{ 1/\kappa, 1/\lambda_{\max},\rho \}, \label{eq:NewMoistKnot}\\ 
        \Vert x_0 - p \Vert &\leq c_3\sqrt{t},  &\Vert \vecv - s_i(0) \vec{n}_i(x_0) \Vert \leq c_4 \sqrt{t}/\ell. \nonumber    
    \end{align}  
    Then, for every integer $J\geq 1$ with $J\leq c_5 \ell^2/t$, every $n_0 \geq 0$, and every $0 \leq \gamma < 1$,  
    \begin{align} 
        \bbP(\#\{0 \leq j \leq J:{}&{} X_{(\gamma +j)t} \in R_+(p,\vecv)   \} \geq n_0 \mid X_0 = x_0 )  \label{eq:NewQuickArm}  
        \geq \bbP(N_+ \geq n_0) - \eta  
    \end{align}
    where, with $(\cB_t)_{t\geq 0}$ a one-dimensional Brownian motion, 
    \begin{align} 
        N_+ \de \# \{0 \leq j \leq J:  \lvert \cB_{\gamma +j} \rvert  \in [1 + \zeta, 2-\zeta] \}.\label{eq:NewNormalLid} 
    \end{align} 
\end{lemma}

The relevance of the parameter $\gamma$ in \eqref{eq:NewQuickArm} for the proof is of a technical nature, arising in the arguments for \eqref{eq: HighFrequency_RandomPoint}. 
In particular, the case $\gamma \neq 0$ is used to deal with the fact that the stopping time does not necessarily take values in $t \bbZ$ so that we may incur a non-integer multiple of $t$ as a time shift after using strong Markovianity; see \eqref{eq:VividBox} in \Cref{apx: Proof_Main_PathDependentRecovery_Algorithm}. 
 
It remains to show that the random variable $N_+$ from \eqref{eq:NewNormalLid} is typically large.
This is the content of the following result, and may be the most interesting part of the proof of \Cref{thm: Main_PathDependentRecovery} as the arguments give insight on the key difficulties for partial recovery; see also \Cref{rem: tau_heavy}.  
\begin{lemma}\label{lem: NplusLarge}
    For every $\zeta < 1/2$ there exist $c_1,c_2,c_3 >0$ such that the following holds for every $\eta\in (0,1)$, $J \geq c_1\ln(c_2/\eta)^2/\eta^2$, and $0\leq \gamma < 1$.
    With $N_+$ as in \eqref{eq:NewNormalLid},
    \begin{align} 
        \bbP(N_+  \geq  c_3 \eta \sqrt{J}) \geq 1 - \eta.    \label{eq:ZenSquid}
    \end{align}
\end{lemma}
\begin{proof}
    The idea in the following argument is to exploit that there is a nonzero probability for a contribution to the count $N_+$ each time that the process $\cB_t$ moves from $0$ to $2$. 
    This reduces us to studying the number of times $\cB_t$ moves from $0$ to $2$. 
 
    More precisely, define a sequence of stopping times $0 \ed \sigma_1 < \tau_1 < \sigma_2 < \tau_2 < \cdots$ as well as $\{0,1 \}$-valued random variables $Y_1,Y_2,\ldots$ by  
    \begin{align} 
        \tau_i \de {}&{}\inf\{t >\sigma_i:  \cB_t  = 2\},\qquad \sigma_{i+1} \de \inf\{t>\tau_i: \cB_t = 0 \},\label{eq:EmptySwan}\\  
        &Y_i \de \bb1 \{\exists j \in \bbZ \cap [\sigma_i, \tau_i - \gamma]: \lvert \cB_{\gamma + j} \rvert \in [1 + \zeta, 2-\zeta] \}.\label{eq:SmartPen} 
    \end{align}
    Let it here be understood that $[\sigma_i, \tau_i - \gamma] = \emptyset$ if $\tau_i -\gamma < \sigma_i$. 
    At any rate, if $i_0\geq 1$ satisfies $\tau_{i_0} \leq J$ then the intervals $[\sigma_i, \tau_i - \gamma]$ with $i\leq i_0$ are disjoint subsets of $[0,J]$. 
    Then, comparing \eqref{eq:NewNormalLid} and \eqref{eq:SmartPen}, we have $N_+ \geq \sum_{i=1}^{i_0}Y_i$. 
    Hence, it holds for every $i_0 \geq 1$ that  
    \begin{align} 
        \bbP\Bigl(N_+ \geq c_3 \eta \sqrt{J} \Bigr) \geq \bbP\Bigl(\tau_{i_0} \leq J\Bigr) - \bbP\Bigl(\sum_{i=1}^{i_0} Y_i < c_3 \eta \sqrt{J} \Bigr).\label{eq:PinkKite} 
    \end{align}
    We next estimate the two terms on the right-hand side of \eqref{eq:PinkKite}, starting with the first one. 
    
    One can represent $\tau_{i_0}$ as a sum of $2i_0 - 1$ independent copies of $\tau_1$ as follows:  
    \begin{align} 
        \tau_{i_0} = \tau_1 + (\sigma_2 -\tau_1 )+ (\tau_2 - \sigma_2) + \cdots + (\sigma_{i_0}- \tau_{i_0-1}) + (\tau_{i_0} - \sigma_{i_0}).\label{eq:IcySun}
    \end{align}
    On the other hand, the hitting time $\inf\{t>0: \cB_t = 2(2i_0 - 1) \}$ also admits a natural decomposition as a sum of $2i_0-1$ independent copies of $\tau_1$. 
    Consequently, these random variables have the same distribution. 
    In particular,  
    \begin{align} 
        \bbP\bigl(\tau_{i_0} \leq J\bigr) &= \bbP\bigl(\inf\{t>0: \cB_t = 2(2i_0 - 1)  \} \leq J \bigr).\label{eq:UneasyKey}  
    \end{align}
    Hence, using the scaling principle as well as the reflection principle, 
    \begin{align} 
        \bbP\bigl(\tau_{i_0} \leq J\bigr) = \bbP(\sup\{\cB_t: t\leq J \} \geq 2(2i_0 - 1)) = 1 - \bbP(\lvert G \rvert < 2(2i_0 - 1)/\sqrt{J} )
    \end{align}
    with $G$ a standard Gaussian random variable. 
    In particular, since the density of the Gaussian distribution is bounded, there exists an absolute constant $c_4>0$ such that 
    \begin{align} 
        \bbP(\tau_{i_0} \leq J)\geq 1 - \eta/2\ \  \text{ if }\ \ i_0  \leq c_4 \eta \sqrt{J}.\label{eq:GladDew} 
    \end{align}
    From here on, let $i_0 \de \lfloor c_4 \eta \sqrt{J}\rfloor$.  

    Note that the $Y_i$ depend on $\cB_t$ through non-overlapping periods of time due to the substraction of $\gamma$ in the time interval in \eqref{eq:SmartPen}.
    The strong Markovianity of one-dimensional Brownian motion now implies that there exists $q>0$ depending only on $\zeta$ with   
    \begin{align} 
        \bbP(Y_i = 1 \mid Y_1,\ldots,Y_{i-1}) \geq \inf_{s\geq 0} \bbP(Y_i = 1 \mid \sigma_i = s) \geq q.
    \end{align}
    The value $q$ can be taken to be independent of $\gamma$ since the worst case would be $\gamma = 1$. 
    \Cref{lem: DriftExponential} now yields $c_5,c_6>0$ depending only on $\zeta$ such that  
    \begin{align} 
        \bbP\Bigl(\sum_{i=1}^{i_0}Y_i < (q/2) i_0\Bigr) \leq c_5 \exp\Bigl(-c_6 i_0 \Bigr) < \eta/2\ \ \text{ if }\ \ i_0 \geq \ln(2c_5/\eta)/c_6.\label{eq:CleverOtter}      
    \end{align}
    Recall that $i_0 = \lfloor c_4 \eta \sqrt{J}\rfloor$ and $J \geq c_1 \ln(c_2/\eta)^2 /\eta^2$. 
    Taking $c_1,c_2>0$ sufficiently large, it can hence be ensured that the condition in \eqref{eq:CleverOtter} is satisfied.
    Then, taking $c_3$ sufficiently small in \eqref{eq:ZenSquid}, the combination of \eqref{eq:PinkKite} with \eqref{eq:GladDew} and \eqref{eq:CleverOtter} concludes the proof.     
\end{proof}

\begin{remark}\label{rem: tau_heavy}
    The random variable $\tau_{i_0}$ in \eqref{eq:IcySun} is L\'evy-distributed. 
    In particular, it is heavy-tailed suggesting that optimal rates for $\dH(\hat{\cX}, \cup_{i=0}^k \cX_i)$ may be significantly slower than for typical points. 
    The latter involves many points simultaneously, giving heavy tails many opportunities to realize.   
\end{remark}
\subsection{Proof of \texorpdfstring{\Cref{prop: Main_PathDependentRecovery_Algorithm}}{Proposition}}\label{sec: Main_PathDependentRecovery_Algorithm}
The key ingredients are now in place. 
We here give an outline of the key ideas and defer the detailed calculations to \Cref{apx: Proof_Main_PathDependentRecovery_Algorithm}.
Recall that \Cref{prop: Main_PathDependentRecovery_Algorithm} refers to \Cref{thm: Main_PathDependentRecovery} and \Cref{cor: Xuniform}. 
Hence, it has to be shown why the performance guarantees in \eqref{eq: HighFrequency_FalsePositives}--\eqref{eq: HighFrequency_Uniform} hold.

First, using \eqref{eq:TallSax} from \Cref{cor: NumTransitions} with $n_0 = C\ln(T/t)$ for some large $C>0$ together with a union bound over all $j,k,n$ for which $R_+(p_{j,k},\vec{v}_n)$ was visited at least once, it can be established that \Cref{alg: HighFrequency} does not add any $p_{j,k}$ which is distant from all barriers to $\hat{\cX}$ with high probability. 
This yields \eqref{eq: HighFrequency_FalsePositives} regarding the nonoccurrence of false positives. 
The reason to run the union bound over those rectangular regions which were visited at least once, instead of simply all $R_+(p_{j,k},\vec{v}_n)\subseteq D$, is that this allows for an estimate which does not degrade when $\Area(D)$ is large. 
This is more efficient as $T$ may be small.   

Conversely, if $p_{j,k}$ is close to a barrier, then taking $\vec{v}_n$ to be an approximation of the normal vector to the barrier ensures that \eqref{eq:ZippyCity} from \Cref{cor: NumTransitions} is applicable so that \Cref{alg: HighFrequency} adds $p_{j,k}$ to $\hat{\cX}$ if $N(p_{j,k}, \vec{v}_n) \geq \mathfrak{n}_0$.
Hence, \eqref{eq: HighFrequency_RandomPoint} follows if we show that $X_{\tau}$ is close to some $p_{j,k}$ with at least logarithmically many visits to $R_+(p_{j,k},\vec{v}_n)$.
The latter follows by using the strong Markovianity together with the results in \Cref{sec: ManyVisits} if we take $J$ of order $\ln(T/t)^2$. 
The condition that $J$ is of order $\leq \ell^2/t$ in \Cref{lem: NewManyHits} then requires that we take $\ell$ of order $\ln(T/t)\sqrt{t}$.    

Finally, the uniform recovery of $\cup_{i=0}^k\cX_i$ in \eqref{eq: HighFrequency_Uniform} is a relatively straightforward consequence of \eqref{eq: HighFrequency_RandomPoint}. 
Specifically, considering some large integer $M\geq 1$, it can be ensured that 
\begin{align} 
    \sup\{\Vert X_t - X_{jT/M} \Vert : t \in [jT/M, (j+1)T/M]\} \leq \varepsilon/2\ \text{ for every }\ j=0,\ldots,M - 1\label{eq:LuckyMouse}
\end{align}
with probability greater than $1- \eta/2$. 
In particular, it then holds with   
$
    \tau_j \de \inf\{t \geq jT/M: X_t \in \cup_{i=0}^{\nb} B_i \}
$
that $\Vert X_{\tau_j} - X_{t} \Vert \leq \varepsilon/2$ for every $t\in [jT/M, (j+1)T/M]$ with $X_t \in \cup_{i=0}^{\nb} B_i$.
We can use \eqref{eq: HighFrequency_RandomPoint} to ensure that $X_{\tau_j}$ is at distance $\leq \varepsilon/2$ from $\hat{\cX}$ with probability $\geq 1 - \eta/2M$. 
The desired result then follows from the triangle inequality and a union bound over all $j\leq M-1$.      

\section{Proof of \texorpdfstring{\Cref{thm: ExponentialConvergence}}{Theorem}}\label{sec: ProofExponential}
Throughout this section, we adopt the assumptions of \Cref{thm: ExponentialConvergence}. 
In particular, we assume that $m=0$. 
Then, \Cref{def: ReflectedBrownianMotion} reduces to the definition of reflected Brownian motion (without semipermeable barriers) in $D$.
Recall from \eqref{eq:SafeSun} that $\cX_0 = B_0 \cap \{X_t:t\in [0,T] \}$. 
\begin{lemma}\label{lem: CoverTime}
    For every $\varepsilon>0$ define a random time $
        \cT(\varepsilon) \de \inf\{T\geq 0: \dH(\cX_0, B_0) \leq \varepsilon \}$.  
    Then,  
    $
        \limsup_{\varepsilon \to 0}\bbE[\cT(\varepsilon)/ \ln(\varepsilon)^2 ] \leq  2\Area(D)/\pi.  
    $
\end{lemma}
\begin{proof}
    Let $\{A_i^{\bc{\varepsilon}}:1 \leq i \leq N(\varepsilon)\}$ be a cover for $B_0$ of minimal cardinality consisting of arcs of length $\varepsilon$.
    Here, $N(\varepsilon)$ denotes the number of arcs in the cover. 
    Then, 
    $
    \cT(\varepsilon) \leq \max_{i\leq N(\varepsilon)}\inf\{t\geq 0: X_t \in  A_{i}^{\bc{\varepsilon}} \}$.  
    Hence, by \cite[Eq.(2.2) \& (2.7)]{matthews1988covering}, 
    \begin{align} 
        \bbE[\cT(\varepsilon)] \leq \mu_+^{\bc{\varepsilon}} \sum_{i=1}^{N(\varepsilon)}\frac{1}{i}\ \text{ where }\ \mu_+^{\bc{\varepsilon}} \leq \max_{i \leq N(\varepsilon)}\sup_{y\in D}\bbE[\inf\{t\geq 0: X_t \in  A_{i}^{\bc{\varepsilon}} \}\mid X_0 = y].\label{eq:HauntedZap}  
    \end{align} 

    The time $\inf\{t\geq 0: X_t \in A_i^{\bc{\varepsilon}} \}$ has been studied in the context of narrow escape problems in the physics literature \cite{holcman2004escape} and also in the mathematical literature \cite{chen2011asymptotic}. 
    In particular, with $A \subseteq B_0$ an arc of length $\varepsilon$, it is shown in \cite[Eq.\ (1.2) and/or Theorem 5.2]{chen2011asymptotic} that  
    \begin{align} 
        \lim_{\varepsilon\to 0} \frac{1}{\ln(1/\varepsilon)}\bbE[\inf\{t\geq 0: X_t \in A \} \mid X_0 \sim \Unif(D)] =  \frac{2}{\pi}\Area(D).\label{eq:MadLog} 
    \end{align}   
    Moreover, the convergence in \eqref{eq:MadLog} is uniform\footnote{
    Specifically, this follows from the proof of \cite[Theorem 5.2]{chen2011asymptotic} as the conformal map $X$ introduced preceding (5.11) in \cite{chen2011asymptotic} extends $C^\infty$ smoothly to the boundary by the Kellog--Warschawski theorem \cite[Theorem 3.6]{pommerenke2013boundary}.
    This implies that all error terms in Section 5.2 of \cite{chen2011asymptotic} are uniform on the domain. 
    } in the choice of $A$. 

    To replace the initial condition by a deterministic one, we can exploit that the transition density of reflected Brownian motion on a smooth domain converges to the uniform distribution at an exponential rate (see \eg \cite[Eq.\ (1.4)]{burdzy2006traps}). 
    By applying \cite[Theorem A]{aldous1997mixing} to the discrete-time process $(X_{it})_{i=1}^\infty$ for some arbitrary fixed $t>0$, it then follows that there exists a stopping time $\tau$ with $\sup_{x_0 \in D}\bbE[\tau\mid X_0 = x_0] < \infty$ such that $X_{\tau}\sim \Unif(D)$. 
    Hence,  
    \begin{align} 
        \sup_{x_0 \in D}{}&{} \bbE[\inf\{t\geq 0: X_t \in A \} \mid X_0 = x_0]\label{eq:PaleCamel} \\ 
        &\leq  \sup_{x_0 \in D}\bbE[\tau \mid X_0 = x_0] + \bbE[\inf\{t\geq 0: X_t \in A^{\bc{\varepsilon}} \} \mid X_0 \sim \Unif(D)].\nonumber  
    \end{align}   
    Since the first term is bounded, it is negligible relative to the second term in the limit $\varepsilon\to 0$. 
    Hence, combining \eqref{eq:MadLog}--\eqref{eq:PaleCamel}, we have that 
    $
        \limsup_{\varepsilon\to 0} \mu_+^{\bc{\varepsilon}}/\ln(1/\varepsilon) \leq 2\Area(D)/\pi. 
    $

    Recall that the cover $\{A_i^{\bc{\varepsilon}}: i=1,\ldots,N(\varepsilon) \}$ was chosen to have minimal cardinality. 
    It follows that $\lim_{\varepsilon\to 0}\varepsilon N(\varepsilon) = C$ with $C>0$ the arc length of $B_0$.
    In particular, $\lim_{\varepsilon\to 0} \ln(1/\varepsilon)/ \ln(N(\varepsilon)) = 1$. 
    Consequently, since $\lim_{n\to \infty}\ln(n)^{-1}\sum_{i=1}^n 1/i = 1$,
    \begin{align} 
        \lim_{\varepsilon\to 0}\frac{1}{\ln(1/\varepsilon)}\sum_{i=1}^{N(\varepsilon)}\frac{1}{i}= 1.\label{eq:DryFace} 
    \end{align} 
    Combine \eqref{eq:DryFace} with the preceding estimate on $\mu_+^{\bc{\varepsilon}}$ and \eqref{eq:HauntedZap} to conclude the proof. 
\end{proof}

The desired result is now essentially immediate by rewriting the definition of $\cT(\varepsilon)$. 
Further, we also get an explicit value for the constant $c$ in \Cref{thm: ExponentialConvergence}:   
\begin{proof}[Proof of \texorpdfstring{\Cref{thm: ExponentialConvergence}}{Theorem}]
    By Markov's inequality, it holds for any $\varepsilon>0$ that 
    \begin{align} 
        \bbP\bigl(\dH(\cX_0, B_0)>  \varepsilon \bigr) = \bbP\bigl(\cT(\varepsilon) > T\bigr) \leq \bbE[\cT(\varepsilon)]/T.\label{eq:AlertForce} 
    \end{align}  
    Let $\varepsilon_T \de \exp(-\sqrt{cT/\Area(D)})$. 
    Then, for any fixed $\delta >0$ it follows from \Cref{lem: CoverTime} that for $T$ sufficiently large, or equivalently $\varepsilon_T$ sufficiently small, 
    \begin{align} 
        \bbE[\cT(\varepsilon_T)] \leq (2+\delta)\Area(D)\ln(\varepsilon_T)^2/\pi = c(2+\delta)T /\pi.\label{eq:IvoryKey} 
    \end{align}  
    Combine \eqref{eq:AlertForce} with \eqref{eq:IvoryKey} and take $c \leq \eta \pi /(2+\delta )$ to conclude that \eqref{eq:KindMom} holds. 
\end{proof}

\section{Case study with animal movement data}\label{sec: PotentialApplications}
To conclude, we illustrate the applicability of the model and algorithms using data due to Loe et al.\ \cite{loe2016behavioral}. 
The data contains tracks of reindeer in the period from 2009 to 2022 in the high-Arctic environment of Nordenski\"old Land, Svalbard.
The study area is dominated by two valleys and we will focus our discussion on the southern valley which is the largest of the two; see the left-hand side of \Cref{fig: AnimalMovements2}. 
To give an indication of scale, the sides of the highlighted square region are approximately 20 kilometers long. 

\newpage
We decided to apply \Cref{alg: KernelDiscontinuity} because this algorithm is the most user-friendly, only requiring a small number of parameters.  
These parameters were here picked on an ad-hoc basis; see \Cref{apx: Case} for the details, and for similar findings in the northern valley.
The dataset comes with 116 distinct animal tracks, of which we removed three in preprocessing due to missing data. 
As each of the tracks is of nontrivial length, most spanning a period of multiple years, it may be reasonable to here consider ourselves to be in a regime with large effective observation period and to hope for complete recovery; recall \Cref{rem: effective_obs}. 
We aggregated the remaining 113 tracks when estimating the empirical transition distributions and further used \Cref{alg: KernelDiscontinuity} unchanged.    
The output is displayed on the right-hand side of \Cref{fig: AnimalMovements2} together with one of the employed tracks.   

We recognize both impermeable and semipermeable barriers in the algorithm's output.
The coastline and the slopes bordering the valley yield impermeable barriers, constraining the animals, while semipermeable barriers arise from rivers that are crossed with nontrivial frequency. 
One of these rivers, the rightmost one, is so wide that the algorithm separately identifies the two sides of the river, leaving the middle open as it lacks recorded animal positions. 
Other rivers are less wide and are recovered in one piece; see \eg the vertical river on the center-left. 
At any rate, we conclude that the algorithms can successfully recover barriers from real-life data.  

\begin{figure}[h]
    \centering      
    \includegraphics[width = 1\textwidth]{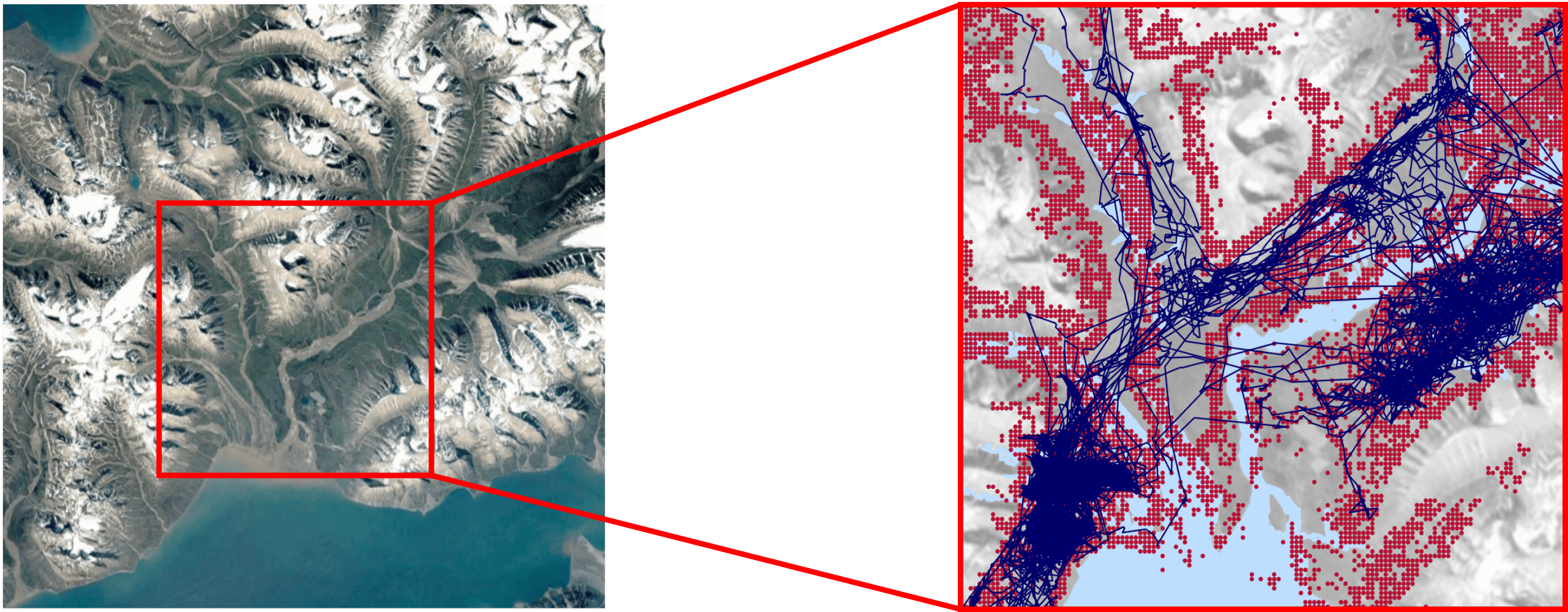}
    \caption{
    On the left: a satellite image of Nordenski\"old Land. 
    On the right: one of the 113 used animal tracks in dark blue, the barriers detected by \Cref{alg: KernelDiscontinuity} as the red points, and a satelite image with manually highlighted rivers in light blue. 
    }
    \label{fig: AnimalMovements2}
\end{figure}

\subsection*{Acknowledgements}
This work is part of the project Clustering and Spectral Concentration in Markov Chains with project number OCENW.KLEIN.324 of the research programme Open Competition Domain
Science -- M, which is partly financed by the Dutch Research Council (NWO).

\subsection*{Code and data availability} Source code for \Cref{alg: KernelDiscontinuity} as well as the simulation scheme used for \Cref{fig: Simulations} have been made available at \url{https://github.com/Alexander-Van-Werde/Brownian-barriers.git}. 
The animal movement data used in \Cref{sec: PotentialApplications} is available at \url{https://www.movebank.org/cms/webapp?gwt_fragment=page=studies,path=study2608802883}. 
\newpage
\bibliographystyle{plainurl}

\newpage
\appendix

\section{Proofs for the preliminaries of \texorpdfstring{\Cref{sec: GeneralPurposePrelim}}{Section}}\label{apx: Proofs_GeneralPurposePrelim}

\subsection{Proof of \texorpdfstring{\Cref{prop: ProcessExistsAndIsUnique}}{Proposition} and \texorpdfstring{\Cref{prop: Fadapted}}{Proposition}}  \label{apx: ProofExistUnique}
We start with the proof of \Cref{prop: ProcessExistsAndIsUnique}.
Recall that this result concerns the existence and uniqueness of processes satisfying \Cref{def: ReflectedBrownianMotion}.  

\begin{figure}[t]
    \includegraphics[width = 0.95\textwidth]{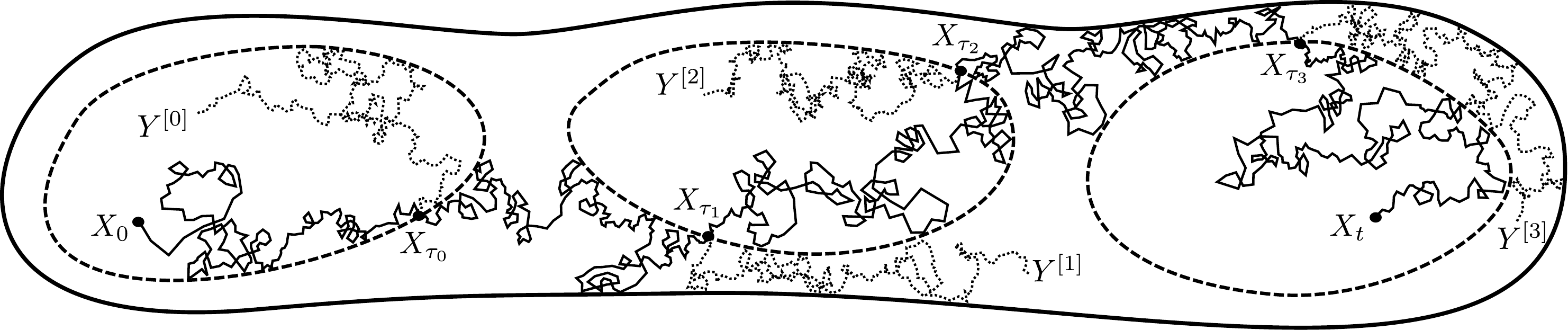}
    \caption{Visualization of the construction which is used in the proof of existence for \Cref{prop: ProcessExistsAndIsUnique}. 
    The process $X$ starts by following a classical reflected Brownian motion with initial condition $X_0$, which we denote by $Y^{\bk{0}}$. 
    When a random variable $\tau_0$ indicates that the local time exceeds a threshold, $X$ switches to a classical reflected Brownian motion $Y^{\bk{1}}$ which lives on the other side of the barrier and has initial condition $Y^{\bk{0}}_{\tau_0}$.
    One can continue similarly, defining $X_t$ in terms of $Y^{\bk{{\nn}}}$ when $t\in (\tau_{{\nn}-1},\tau_{\nn}]$.    
    }
    \label{fig: ConstructionReflectedBM}
\end{figure}
\begin{proof}[Proof of existence] 
    We rely on the classical fact that reflected Brownian motion in a smooth domain (without semipermeable barriers) exists and use a gluing procedure to accomplish the semipermeable barriers. 
    A visualization of this construction may be found in \Cref{fig: ConstructionReflectedBM}.

    Denote $C^{\bk{0}}$ for the closure of the connected component of $D\setminus (\cup_{i=0}^{\nb} B_i)$ indicated by the $s_i(0)$. 
    That is, the set $C^{\bk{0}}$ consists of all points $x\in D$ such that for every $i\leq\nb$ the point $x$ is on the positive (resp.\ negative) side of $B_i$ if $s_i(0) = +1$ (resp.\ $s_i(0) = -1$). 
    Here, recall from \Cref{sec: DomainsAndBarriers} that $x$ is both on the positive and negative sides if $x\in B_i$.
    
    The set $C^{\bk{0}}$ is a smooth domain.  
    Hence, classical reflected Brownian motion in $C^{\bk{0}}$ and its local time exist uniquely; see \cite[Proposition 4]{anderson1976small} or \cite[Theorem 4.3]{lions1984stochastic}.  
    That is, there exist unique continuous processes $Y_t^{\bk{0}}$ and $\cL_t^{\bk{0}}$ which are adapted to the filtration generated by $X_0$, the $s_i(0)$ with $1 \leq i\leq {\nb}$, and $W_s$ with $s\leq t$ and satisfy the following properties:
    \begin{enumerate}[leftmargin = 2em,label = (\arabic*)]
        \item\label{item:JollyUser} The process $Y_t^{\bk{0}}$ takes values in $C^{\bk{0}}$ and satisfies $Y_0^{\bk{0}} = X_0$. The process $\cL_t^{\bk{0}}$ takes values in $\bbR_{\geq 0}$ and satisfies $\cL_0^{\bk{0}} = 0$. 
        \item\label{item:QuickRag} The process $\cL_t^{\bk{0}}$ is nondecreasing and increases at time $t$ if and only if $Y_t^{\bk{0}} \in \partial C^{\bk{0}}$. 
        Equivalently, since $Y_t^{\bk{0}} \in C^{\bk{0}}$ for every $t\geq 0$ and $C^{\bk{0}} \cap \cup_{i=0}^{\nb} B_i = \partial C^{\bk{0}}$, the process $\cL_t^{\bk{0}}$ only increases when $Y_t^{\bk{0}} \in \cup_{i=0}^{\nb} B_i$.  
        \item\label{item:IdleWok} The following stochastic differential equation is satisfied for all $t\geq 0$: 
        \begin{align} 
            \intd Y_t^{\bk{0}} = \intd W_t + \sum_{i=0}^{\nb} s_i(0) \vec{n}_i(Y_t^{\bk{0}}) \bb1\{Y_t^{\bk{0}} \in B_i \}\intd \cL_t^{\bk{0}}. \label{eq:ElegantPig}
        \end{align}
    \end{enumerate}
    We next define $X_t$ and $\Lc{i}_t$ for sufficiently small values of $t$.  

    The local time $\cL_t^{\bk{0}}$ of $Y_t^{\bk{0}}$ increases every time $Y_t^{\bk{0}}$ is in $\cup_{i=0}^{\nb} B_i$. 
    To disambiguate the contributions corresponding to the different curves $B_i$, we may consider the process $\int_0^t \bb1\{Y_r^{\bk{0}} \in B_i \} \intd \cL_r^{\bk{0}}$
    for every $t\geq 0$ and $i\in \{0,1,\ldots, {\nb} \}$.
    Define a $\cG_t$-stopping time\footnote{Recall that definition of the filtration $\cG_t$ was given in \Cref{prop: ProcessExistsAndIsUnique}.} by 
    \begin{align} 
        \tau_0 \de \inf\bigl\{t\geq 0: s_i\bigl({\textstyle\int_0^t} \bb1\{Y_r^{\bk{0}} \in B_i \} \intd \cL_r^{\bk{0}}\bigr) \neq s_i(0) \text{ for some }i \leq {\nb}\bigr\}.\label{eq:AvidDragon}
    \end{align}
    Then, we define 
    \begin{align} 
        X_t \de Y_t^{\bk{0}}\ \text{ and }\ \Lc{i}_t \de {\textstyle\int_0^t} \bb1\{Y_r^{\bk{0}} \in B_i \} \intd \cL_r^{\bk{0}}\ \text{ for }t\leq \tau_0.  \label{eq:RoyalBox}
    \end{align} 
    
    We continue inductively. 
    Consider an integer $n\geq 0$ as well as a $\cG_t$-stopping time $\tau_{{\nn}}$, and suppose that $X_t$ and $\Lc{i}_t$ have been defined for all $t\leq \tau_{{\nn}}$.
    Denote $C^{\bk{{\nn}+1}}$ for the closure of the connected component of $D\setminus (\cup_{i=0}^k B_i)$ indicated by the $s_i(\tau_{\nn})$.
    We define $Y_t^{\bk{{\nn}+1}}$ to be the reflected Brownian motion on $C^{\bk{{\nn}+1}}$ with local time $\cL_t^{\bk{{\nn}+1}}$ and initial condition $Y_0^{\bk{{\nn}+1}} = X_{\tau_{\nn}}$.
    That is, $Y^{\bk{{\nn}+1}}_t$ and $\cL_t^{\bk{{\nn}+1}}$ are the unique $\sigma(s_i(\Lc{i}_{\tau_{\nn}}):i\leq {\nb})\vee \sigma(X_{\tau_{\nn}})\vee \sigma(W_{\tau_{\nn} + r}: r\leq t)$-adapted processes with properties as in items \ref{item:JollyUser}--\ref{item:IdleWok}, \eg   
    \begin{align} 
        \intd Y_t^{\bk{{\nn}+1}} = \intd W_{\tau_{\nn} + t} + \sum_{i=0}^{\nb} s_i(\Lc{i}_{\tau_{\nn}}) \vec{n}_i(Y_t^{\bk{{\nn}+1}}) \bb1\{Y_t^{\bk{{\nn}+1}} \in B_i \}\intd \cL_t^{\bk{{\nn}+1}}.  \label{eq:ShyGoose}
    \end{align}    
    Further, define a $\cG_t$-stopping time by   
    \begin{align} 
        \tau_{{\nn}+1} \de \inf\bigl\{t\geq \tau_{\nn}:s_i\bigl(\Lc{i}_{\tau_{\nn}}+ {\textstyle \int_0^{t-\tau_{\nn} } }\bb1\{Y_r^{\bk{{\nn}+1}} \in B_i \} \intd \cL_r^{\bk{{\nn}+1}}\bigr) \neq s_i\bigl(\Lc{i}_{\tau_{\nn}}\bigr) \text{ for some }i \leq {\nb} \bigr\}.\label{eq:HugeGem} 
    \end{align}
    Then, we define 
    \begin{align} 
        X_t \de Y_{t- \tau_{\nn}}^{\bk{{\nn}+1}}\ \text{ and }\ \Lc{i}_t \de   \Lc{i}_{\tau_{\nn}}+ {\textstyle \int_0^{t-\tau_{\nn} } }\bb1\{Y_r^{\bk{{\nn}+1}} \in B_i \} \intd \cL_r^{\bk{{\nn}+1}} \text{ for }t\in (\tau_{\nn}, \tau_{{\nn}+1}]. \label{eq:OccultFish}
    \end{align}

    Technically, to ensure that the processes are defined for all $t\geq 0$, it also has to be verified that $\lim_{{\nn}\to \infty}\tau_{\nn} = \infty$ almost surely. 
    In this context, let us note that essentially the same arguments as were given for \Cref{cor: TauBound} can be used to show that there exists some $t>0$ such that $\bbP(\tau_{\nn}-\tau_{{\nn}-1} >t \mid \tau_0,\ldots,\tau_{{\nn}-1}) >1/2$ almost surely for every $n$.
    The latter implies that $\tau_n = \tau_0 + \sum_{i=1}^n (\tau_i - \tau_{i-1})$ tends to infinity almost surely, as required.

    That the properties in \Cref{def: ReflectedBrownianMotion} are satisfied is now essentially immediate. 
    \Cref{item: Def_ReflectedBrownianMotion_i} in \Cref{def: ReflectedBrownianMotion}, concerning the stochastic differential equation, follows from \eqref{eq:ElegantPig} and \eqref{eq:ShyGoose}. 
    The claim that $\Lc{i}_t$ only increases when $X_t\in B_i$ in \cref{item: Def_ReflectedBrownianMotion_ii} follows from \cref{item:QuickRag} and the way $\Lc{i}_t$ was defined in \eqref{eq:RoyalBox} and \eqref{eq:OccultFish}. 
    Finally, the claim in \cref{item: Def_ReflectedBrownianMotion_iii} that $X_t$ is on the positive (resp.\ negative) side of $B_i$ if $s_i(\Lc{i}_t) = +1$ (resp.\ $s_i(\Lc{i}_t) = -1$) follows from the definition that $Y^{\bk{{\nn}+1}}$ takes values in $C^{\bk{n+1}}$ and the definition of $\tau_{{\nn}}$. 
\end{proof}
\begin{proof}[Proof of uniqueness]
    This follows from the pathwise uniqueness of reflected Brownian motion in a smooth domain (without semipermeable barriers).
    The proof has a similar structure to the proof that was used in the existence part. 
    We next make this precise. 
    
    Suppose that we are given $\cG_t$-adapted $(X_t,\Lc{0}_t,\ldots,\Lc{m}_t)$ and $(\tiX_t,\tiLc{0}_t,\ldots,\tiLc{m}_t)$ with $X_0 = \tiX_0$ satisfying \Cref{def: ReflectedBrownianMotion} with respect to the same $W_t$ and $s_i$.
    We show that it then holds that $X_t = \tiX_t$ and $\Lc{i}_t = \tiLc{i}_t$ for all $t\geq 0$, almost surely.    
    Define $\cG_t$-stopping times by 
    \begin{align} 
        \tau_0 &\de \inf\bigl\{t\geq 0: s_i(\Lc{i}_t) \neq s_i(0) \text{ for some } i\in \{0,1,\ldots,{\nb} \} \bigr\},\label{eq:SafeHen}\\ 
        \tilde{\tau}_0 &\de \inf\bigl\{t\geq 0: s_i(\tiLc{i}_t) \neq s_i(0) \text{ for some } i\in \{0,1,\ldots,{\nb} \}\bigr\}.\label{eq:ZenKnot} 
    \end{align} 
    Recall the definition of the sets $C^{\bk{n}}$ from the preceding proof of existence and, using \cite[Proposition 4.1]{anderson1976small}, pick a reflected Brownian motion $Z_t^{\bk{0}}$ in $C^{\bk{0}}$ with local time $K_t^{\bk{0}}$ and initial condition $Z_0^{\bk{0}} = X_{\tau_0}$ satisfying 
    \begin{align} 
        \intd Z_t^{\bk{0}} = \intd W_{\tau_0 + t} + \sum_{i=0}^{\nb} s_i(0) \vec{n}_i(Z_t^{\bk{0}})\bb1\{Z_t^{\bk{0}} \in B_i\} \intd K_t^{\bk{0}}.\label{eq:AngryMusic} 
    \end{align} 
    Here, note that while the sum in \eqref{eq:AngryMusic} makes reference to all barriers $B_i$, only the summands with $B_i \cap C^{\bk{0}} \neq \emptyset$ can have nonzero contribution. 
    Then, the processes defined by 
    \begin{align} 
        Y_t^{\bk{0}} \de 
        \begin{cases}
            X_t &\text{ if }t \leq \tau_0, \\ 
            Z_{t-\tau_0}^{\bk{0}}&\text{ if }t>\tau_0 
        \end{cases} 
        \quad \text{ and }\quad 
        \cL_t^{\bk{0}} \de   
        \begin{cases}
            \sum_{i=0}^{\nb} \Lc{i}_t &\text{ if }t \leq \tau_0, \\ 
             \sum_{i=0}^{\nb} \Lc{i}_{\tau_0} + K_{t-\tau_0}^{\bk{0}} &\text{ if }t>\tau_0 
        \end{cases} 
        \label{eq:NormalWand}
    \end{align}
    define a reflected Brownian motion in $C^{\bk{0}}$ with local time $\cL_t^{\bk{0}}$. 
    One can similarly define $\tiY_t^{\bk{0}}$ and $\tilde{\cL}_t^{\bk{0}}$ in terms of $\tiX_t$ and $\tiLc{i}_t$. 

    The pathwise uniqueness of reflected Brownian motion, proofs of which may be found in \cite[Proposition 4]{anderson1976small} or \cite[Theorem 4.3]{lions1984stochastic}, now implies that 
    \begin{align} 
        \bbP(Y_t^{\bk{0}} = \tiY_t^{\bk{0}}, \ \forall t\geq 0 ) = 1 \quad \text{ and }\quad \bbP(\cL_t^{\bk{0}} = \tilde{\cL}_t^{\bk{0}}, \ \forall t\geq 0 )=1. 
    \end{align}
    In particular, $\int_{0}^t \bb1\{Y_r^{\bk{0}} \in B_i \} \intd \cL_r^{\bk{0}} = \int_{0}^t \bb1\{\tiY_r^{\bk{0}} \in B_i \} \intd \tilde{\cL}_r^{\bk{0}}$ for every $i\in \{0,1,\ldots, {\nb} \}$. 
    Note that $\Lc{i}_t =\int_{0}^t \bb1\{Y_r^{\bk{0}} \in B_i \} \intd \cL_r^{\bk{0}}$ for $t\leq \tau_0$ and similarly for $\tiLc{i}_t$.
    Hence, since definitions \eqref{eq:SafeHen} and \eqref{eq:ZenKnot} refer to the same $s_i$, we have $\tau_0 = \tilde{\tau}_0$. 
    It further follows from \eqref{eq:NormalWand} that  
    \begin{align} 
        \bbP(X_t = \tiX_t,\ \forall t\leq \tau_0)=1 \quad \text{ and }\quad \bbP(\Lc{i}_t = \tiLc{i}_t,\ \forall t\leq \tau_0) =1.  
    \end{align} 
    One can now repeat this argument for the processes $t\mapsto X_{t +\tau_0}$ and $t\mapsto \Lc{i}_{t+\tau_0}$ by similarly defining processes $Y^{\bk{1}}_t$ and $\cL_t^{\bk{1}}$ to conclude that equality holds for all $t\leq \tau_1$, and so on.     
\end{proof}
Finally, we consider the adaptedness statement in \Cref{prop: Fadapted}. 
\begin{proof}[Proof of \texorpdfstring{\Cref{prop: Fadapted}}{Proposition}]
    The uniqueness statement in \Cref{prop: ProcessExistsAndIsUnique} yields that the $\cG_t$-adapted processes $X_t$ and $\Lc{i}_t$ are uniquely defined up to null sets.   
    Hence, recalling from \Cref{sec: GeneralPurposePrelim} that the considered $\sigma$-algebras were all completed so that these null sets do not matter, the statement of \Cref{prop: Fadapted} is well-defined. 
    In particular, it suffices to prove that $\{\Lc{i}_t \leq \ell_i,\, \forall i\leq m  \} \in \cF_{t,\ell_0,\ldots,\ell_m}$ and that $X_t, \Lc{i}_t$, and $s_i(\Lc{i}_t)$ are $\cF_t$-adapted for one specific construction of these processes.

    The latter follows by inspecting the proof of existence for \Cref{prop: ProcessExistsAndIsUnique}. 
    The constructions for $X_t$ and $\Lc{i}_t$ given there namely only depend on $(W_r)_{r \geq 0}$ and $(s_i(r))_{r\geq 0}$ through $(W_r)_{r\leq t}$ and $(s_i(\Lc{i}_r))_{r\leq t}$. 
    More precisely, recall from \eqref{eq:RoyalBox} and \eqref{eq:OccultFish} that the definition of $X_t$ and $\Lc{i}_t$ for $t\in (\tau_{\nn}, \tau_{{\nn}+1}]$ was stated in terms of $Y^{\bk{{\nn}+1}}_{t-\tau_{\nn}}$ and $\{\cL^{\bk{{\nn}+1}}_{s-\tau_{\nn}} :s\leq t\}$ and recall from the paragraphs preceding \eqref{eq:ElegantPig} and \eqref{eq:ShyGoose} that the definition of $Y^{\bk{{\nn}+1}}_{t}$ and $\cL^{\bk{{\nn}+1}}_t$ only refers to $W$ and $s_i$ through $(W_{r})_{r\leq \tau_{\nn} + t}$ and $(s_i(\Lc{i}_r))_{r\leq \tau_{\nn}}$. 
    Further, recall the definition of $\tau_{{\nn}}$ from \eqref{eq:AvidDragon} and \eqref{eq:HugeGem} and note that the event $t\in (\tau_{\nn}, \tau_{{\nn}+1}]$ only refers to $(s_i(r))_{r\geq 0}$ through $(s_i(\Lc{i}_r))_{r\leq t}$.  

    Combining the foregoing observations, the event $\{\Lc{i}_t \leq \ell_i,\, \forall i \leq {\nb}  \}$ only depends on $(W_r)_{r\geq 0}$ and $(s_i(r))_{r\geq 0}$ through $(W_r)_{r\leq t}$ and $(s_i(r))_{r\leq \ell_i}$.
    This shows that $\{\Lc{i}_t \leq \ell_i,\, \forall i \leq {\nb}  \}\in \cF_{t,\ell_0,\ldots,\ell_{\nb}}$, as desired. 
    Further, restricted to this event the definitions of $X_t$ and $\Lc{i}_t$ only depend on $W$ and $s_i$ through $(W_r)_{r\leq t}$ and $(s_i(r))_{r\leq \ell_i}$.
    This shows that $X_t$, $\Lc{i}_t$, and $s_i(\Lc{i}_t)$ define $\cF_t$-adapted processes, as desired.    
\end{proof}

\subsection{Proof of \texorpdfstring{\Cref{prop: X_Markov}}{Proposition}}   
The main goal is to establish strong Markovianity for the process $\sX_t$ defined in \Cref{prop: X_Markov}.
The weak Markovianity for $X_t$ then follows with a small additional argument; see \Cref{sec: ProofMarkov}.

We will exploit that the processes $W_t$ and $s_i(t)$ in \Cref{def: ReflectedBrownianMotion} are Markovian.
A possible delicate point is that \Cref{def: ReflectedBrownianMotion} does not refer to $s_i(t)$ directly, but rather to $s_i(\Lc{i}_t)$. 
Some care is hence required to rigorously show that no problems enter through the local times $\Lc{i}_t$. 
Preparatory work related to this is done in \Cref{sec: GeneralizedStrongMarkov}.

\subsubsection{Generalized strong Markovianity}\label{sec: GeneralizedStrongMarkov}   
If $\tau$ is a stopping time for $\cF_t$, then this means by definition of the filtration $\cF_t$ that $\{ \tau\leq t, \Lc{i}_t \leq \ell_i,\, \forall i \leq {\nb}  \}\in \cF_{t,\ell_0,\ldots,\ell_{\nb}}$. 
This suggests that $\tau$ and $\Lc{i}_t$ together again behave like a stopping time in a suitably generalized sense:
\begin{definition}
    Fix an integer $d\geq 1$ and consider a collection of $\sigma$-algebras $\cH_{t_1,t_2,\ldots,t_d}$ indexed by $\bbR_{\geq 0}^{d}$. 
    Then, $(\cH_{t_1,t_2,\ldots,t_d})_{t_1,\ldots,t_d\geq 0}$ is called a \emph{multivariate filtration} if it holds that $\cH_{t_1,t_2,\ldots,t_d} \subseteq \cH_{r_1,r_2,\ldots,r_d}$ whenever $t_i \leq r_i$ for every $i\leq d$.         
\end{definition}
\begin{definition}\label{def: GeneralizedStoppingTime}
    An $\bbR^{d}_{\geq 0}$-valued random variable $\sT \de (\tau_1,\ldots,\tau_d)$ is said to be a \emph{multivariate stopping time} for a multivariate filtration $(\cH_{t_1,t_2,\ldots,t_d})_{t_1,\ldots,t_d\geq 0}$ if $\{\tau_i \leq t_i,\, \forall i\leq d  \} \in \cH_{t_1,t_2,\ldots,t_d}$ for every $t_1,\ldots,t_d \geq 0$. 

    If $\sT$ is a multivariate stopping time, then we denote $\cH_{\sT}$ for the $\sigma$-algebra of all events $E$ with $E\cap \{\tau_i \leq t_i,\, \forall i\leq d  \} \in \cH_{t_1,t_2,\ldots,t_d}$ for every $t_1,\ldots,t_d \geq 0$. 
\end{definition}
\Cref{lem: GeneralizedStrongMarkov} below establishes a multivariate generalization of the classical result that Feller processes are strongly Markovian.
Indeed, the classical result is recovered as the case $d=1$. 
The usefulness for our purposes, however, is exactly in the multivariate nature of the lemma as this is what will allow us to rigorously show that no problems enter through the local times; see the proof of \Cref{lem: StrongMarkov}. 
\begin{definition}\label{def: Feller}
    Let $\sP$ be a Polish space and fix an integer $d\geq 1$.  
    \begin{enumerate}[leftmargin=2.5em,label = (\arabic*)]
        \item\label{item: rightcont} A function $M:\bbR_{\geq 0}^d \to \sP$ is said to be \emph{marginally right-continuous} if for any $\bbR_{\geq 0}^d$-valued sequence $(t^{\bc{1}}_n,\ldots, t^{\bc{d}}_n)_{n=1}^\infty$ with nonincreasing coordinates,  
        \begin{align} 
            \lim_{n\to \infty} M\bigl(t^{\bc{1}}_n,\ldots, t^{\bc{d}}_n\bigr) = M\bigl(\lim_{n\to \infty}t^{\bc{1}}_n,\ldots, \lim_{n\to \infty}t^{\bc{d}}_n \bigr). 
        \end{align} 
        \item\label{item: weak_Markov} Consider a random function $M$ from $\bbR_{\geq 0}^d$ to $\sP$ and let $\cH_{t_1,\ldots,t_d}\de \sigma\{M(r_1,\ldots,r_d): r_i\leq t_i \}$. 
        Then, we call $M$ \emph{weakly multivariate Markovian} if for every fixed $t, \delta \in \bbR_{\geq 0}^d$ and every measurable $E \subseteq \sP$ it holds with $P^{\delta}_y(E) \de  \bbP(M(\delta)\in E \mid M(0) = y)$ that   
        \begin{align} 
            \bbP\bigl(M(t + \delta) \in E \mid \cH_{t_1,\ldots,t_d}\bigr) = P^{\delta}_{M(t)}(E).   
        \end{align}    
        \item\label{item: Feller} Suppose that $M$ is weakly multivariate Markovian with marginally right-continuous sample paths. 
        Then, $M$ is said to have the \emph{Feller property} if for every $p_0 \in \sP$, every sequence $(y_n)_{n=1}^\infty$ with $\lim_{n\to \infty} y_n = p_0$, and every fixed $t \in \bbR_{\geq 0}^d$ it holds that 
        \begin{align} 
            \lim_{n\to \infty}\bbE\bigl[f(M(t)) \mid M(0) = y_n\bigr] = \bbE\bigl[f(M(t)) \mid M(0) = p_0\bigr] 
        \end{align} 
        for every bounded continuous function $f:\sP \to \bbR$.  
    \end{enumerate}
\end{definition}
\begin{lemma}[Strong multivariate Markovianity]\label{lem: GeneralizedStrongMarkov}
    Suppose that $M$ is weakly multivariate Markovian, marginally right-continuous, and has the Feller property.
    Then, for every measurable $E\subseteq \sP$, multivariate stopping time $\sT$, and $t \de (t_1,\ldots,t_d)$, 
    \begin{align} 
        \bbP\bigl(M\bigl(\sT + t\bigr) \in E \mid \cH_\sT\bigr) = \bbP\bigl(M\bigl(\sT + t\bigr) \in E \mid M(\sT)\bigr) \label{eq:IvoryRam}   
    \end{align}
    where $\cH_{\sT}$ is the $\sigma$-algebra resulting from \Cref{def: GeneralizedStoppingTime} applied to the multivariate filtration in \cref{item: weak_Markov} of \Cref{def: Feller}.  
\end{lemma}
\begin{proof}
    This is a direct generalization of the classical proof for $d=1$, but we give details for completeness.
    We start with a reduction to a statement about continuous functions; see \eqref{eq:GladTrunk} below.     
    By the characterizing property of conditional probability, \eqref{eq:IvoryRam} is equivalent to showing that for every $A \in \cH_{\sT}$, 
    \begin{align} 
        \bbE\bigl[\bb1_A \bb1\{M( \sT + t )  \in E \} \bigr] = \bbE\bigl[\bb1_A \bbE\bigl[\bb1\{M( \sT + t) \in E \} \mid M(\sT)\bigr] \bigr].\label{eq:GoodLobster}
    \end{align} 
    Further, it suffices to consider the special case where $E\subseteq \sP$ is an open set. 
    Then, there exists a nondecreasing sequence of continuous and bounded functions $f_n:\sP \to E$ converging pointwise to $\bb1_E$. 
    So, applying the monotone convergence theorem twice, 
    \begin{align} 
        \bbE\bigl[\bb1_A \bbE\bigl[\bb1\{M( \sT +{}&{} t) \in E \} \mid M(\sT)\bigr] \bigr] 
        = \lim_{n\to \infty} \bbE\bigl[\bb1_A \bbE\bigl[f_n\bigl(M( \sT + t)\bigr) \mid M(\sT)\bigr] \bigr].\label{eq:DampBat}
    \end{align}
    Similarly, $\bbE[\bb1_A \bb1\{M( \sT + t) \in E \} ] = \lim_{n\to \infty} \bbE[\bb1_A f_n(M( \sT + t))]$. 
    It hence suffices to show that for every bounded and continuous function $f:\sP \to \bbR$ and every event $A \in \cH_{\sT}$,  
    \begin{align} 
        \bbE\bigl[\bb1_{A} f\bigl(M(\sT + t)\bigr) \bigr]
        &=
        \bbE\bigl[\bb1_A \bbE\bigl[f\bigl(M(\sT + t)\bigr)  \mid M(\sT) \bigr]\bigr].\label{eq:GladTrunk}  
    \end{align}

    To prove \eqref{eq:GladTrunk}, we start with the special case where $\sT$ takes values in a countable set $\cI \subseteq \bbR_{\geq 0}^d$.
    For any fixed $(i_1,\ldots,i_d) \in \cI$ one can readily verify that $A \cap \{\sT = (i_1,\ldots,i_d) \} \in \cH_{i_1,\ldots,i_d}$ using that $A\in \cH_{\sT}$ and that $\sT$ is a multivariate stopping time.  
    Then, also using that $A = \cup_{i\in \cI} A \cap \{\sT = i \}$ together with the tower property, 
    \begin{align} 
        \bbE\bigl[\bb1_{A} f\bigl(M(\sT +t)\bigr) \bigr]
        &= \sum_{i\in \cI} \bbE\bigl[\bb1_{A\cap \{\sT = i \}} f\bigl(M(i+t)\bigr) \bigr]\label{eq:AvidFish}\\ 
        &= \sum_{i\in \cI} \bbE\bigl[\bb1_{A\cap \{\sT = i \}} \bbE\bigl[ f\bigl(M(i+t)\bigr)\mid \cH_{i_1,\ldots,i_d} \bigr] \bigr].  \nonumber  
    \end{align}  
    Here, since $M$ was assumed to be weakly multivariate Markovian, 
    \begin{align} 
        \bbE\bigl[ f\bigl(M(i +t)\bigr)\mid \cH_{i_1,\ldots,i_d} \bigr]\label{eq:TrustyDad} 
        & = \bbE\bigl[ f\bigl(M(i+t)\bigr)\mid M(i) \bigr].
    \end{align}
    Combine \eqref{eq:AvidFish} and \eqref{eq:TrustyDad} to conclude that \eqref{eq:GladTrunk} holds when $\sT$ takes values in a countable set. 

    We recover the general case using a discretization argument.
    For any ${\nn}\geq 0$, consider the partition of $\bbR_{\geq 0}$ into intervals $I_{j,{\nn}} \de ((j-1)2^{-{\nn}}, j2^{-{\nn}}]$ for $j\geq 0$. 
    Then, writing $\tau_1,\ldots,\tau_d$ for the components of $\sT$, we define  
    \begin{align} 
        \tau_i^{\bc{{\nn}}} \de \sum_{j=0}^\infty j 2^{-{\nn}} \bb1\{\tau_i \in I_{j,{\nn}}\}\quad \text{ and }\quad \sT^{\bc{n}} \de \bigl(\tau_1^{\bc{{\nn}}},\tau_2^{\bc{{\nn}}},\ldots,\tau_n^{\bc{{\nn}}}  \bigr).\label{eq:SafeHat}    
    \end{align} 
    Note that $\sT^{\bc{n}}$ is then again a multivariate stopping time, but now has values in a countable set. 
    Hence, by the previously established case, 
    \begin{align} 
        \bbE\bigl[\bb1_{A} f\bigl(M(\sT^{\bc{n}} + t)\bigr) \bigr]
        &=
        \bbE\bigl[\bb1_A \bbE\bigl[f\bigl(M(\sT^{\bc{n}} + t)\bigr)  \mid M(\sT^{\bc{n}}) \bigr]\bigr] \label{eq:ValidBoy}
    \end{align} 
    
    Note that $\tau_i^{\bc{{\nn}}}$ is a nonincreasing in ${\nn}$ and that $\lim_{{\nn}\to \infty}\tau_i^{\bc{{\nn}}} = \tau_i$ with probability one. 
    Consequently, using that $M$ is marginally right-continuous, 
    $
        \lim_{n\to \infty} M(\sT^{\bc{n}}) = M(\sT).   
    $
    Hence, by the dominated convergence theorem and the Feller property, which are applicable since the function $f$ was assumed to be bounded and continuous,   
    \begin{align} 
        \lim_{n\to \infty}\bbE\bigl[\bb1_A \bbE\bigl[f\bigl(M(\sT^{\bc{n}} + t)\bigr)  \mid M(\sT^{\bc{n}}) \bigr]\bigr]   = \bbE\bigl[\bb1_A \bbE\bigl[f\bigl(M(\sT + t)\bigr)  \mid M(\sT) \bigr]\bigr].  \label{eq:ZenBat}
    \end{align} 
    Similarly, by the dominated convergence theorem,   
    \begin{align} 
        \lim_{n\to \infty}\bbE\bigl[\bb1_{A} f\bigl(M(\sT^{\bc{n}} + t)\bigr) \bigr] = \bbE\bigl[\bb1_{A} f\bigl(M(\sT + t)\bigr) \bigr]. \label{eq:WittyBird}
    \end{align}
    Combine \eqref{eq:ValidBoy}, \eqref{eq:ZenBat}, and \eqref{eq:WittyBird} to conclude that \eqref{eq:GladTrunk} also holds in the general case. 
    This concludes the proof. 
\end{proof}

\subsubsection{Proof of Markovianity}\label{sec: ProofMarkov}
Recall the definition of the filtrations $\cF_{t, \ell_1,\ldots,\ell_{\nb}}$ and $\cF_{t}$ from \Cref{prop: Fadapted}.  
\begin{lemma}\label{lem: StrongMarkov}
    Define an $\cF_t$-adapted process $\sX_t \de (X_t, (s_i(\Lc{i}_t))_{i=0}^{\nb})$ as in \Cref{prop: X_Markov}. 
    Then, for every measurable $E \subseteq D\times \{+1,-1 \}^{m+1}$, every $\cF_t$-stopping time $\tau$, and every fixed $u\geq 0$, 
    \begin{align} 
        \bbP(\sX_{\tau + u} \in E \mid \cF_\tau) = \bbP(\sX_{\tau + u} \in E \mid \sX_\tau). 
    \end{align} 
\end{lemma}
\begin{proof} 
    Let $u \geq 0$ and define $\tiX_u \de X_{\tau + u}$, $\tiLc{i}_u \de \Lc{i}_{\tau+ u} - \Lc{i}_\tau$, $\tiW_u \de W_{\tau+u} - W_\tau$, $\tis_i(u) \de s_i(\Lc{i}_\tau + u)$. 
    Then, direct verification of \Cref{def: ReflectedBrownianMotion} shows that $\tiX_u$ is a reflected Brownian motion with semipermeable barriers and local times $\tiLc{i}_u$ driven by $\tiW_u$ and $\tis_i(u)$. 
    We next show that $\tiW_u$ and $\tis_i(u)$ are again Markovian processes from the same distribution. 
    
    Let $d\de {\nb}+2$ and $\sT \de (\Lc{0}_\tau,\ldots,\Lc{{\nb}}_{\tau}, \tau)$. 
    Then, using that $\tau$ is a $\cF_t$-stopping time together with \Cref{prop: Fadapted}, expanding the definitions shows that $\sT$ is a multivariate stopping time for the multivariate filtration $(\cF_{t,\ell_0,\ldots,\ell_{\nb}})_{t,\ell_0,\ldots,\ell_{\nb} \geq 0}$ in the sense of \Cref{def: GeneralizedStoppingTime}.  
    Further, the $\sigma$-algebra $\cF_{\sT}$ defined using \Cref{def: GeneralizedStoppingTime} is equal to $\cF_{\tau}$.
    Consider the process defined by 
    \begin{align} 
        M(t_1,\ldots,t_d) \de (s_0(t_1),s_1(t_2),\ldots,s_m(t_{d-1}), W_{t_d}).   
    \end{align}
    This process has values in the Polish space $\sP \de \{-1,1 \}^{m+1} \times \bbR$ and it is readily verified that it has the Feller property and marginally right-continuous sample paths.
    Further, weak multivariate Markovianity follows using that the processes $s_i$ and $W$ are independent and individually weakly (even strongly) Markovian.    
    
    Hence, the strong multivariate Markovianity from \Cref{lem: GeneralizedStrongMarkov} is applicable to the process $M$. 
    Using this, it follows that conditional on $\cF_{\tau}$ the processes $\tiW$ and $\tis_i$ are again a Wiener process and $\{-1,+1 \}$-valued Markov chain with the same transition rates as they had initially. 
    In particular, the law of the processes $\tiW$ and $\tis_i$ only depends on the initial conditions $\tis_i(0)$. 

    \Cref{prop: ProcessExistsAndIsUnique} implies pathwise uniqueness for $\tiX_u$. 
    Hence, since pathwise uniqueness implies uniqueness in law, the law of $\tiX_u$ conditional on $\cF_\tau$ is uniquely determined by the initial condition $\tiX_0$ and the processes $\tis_i(u)$ and $\tiW_u$. 
    Recall that the law of the latter processes is determined by the initial condition $\tis_i(0)$. 
    This shows that the law of $\tiX_u$ is uniquely determined in terms of $X_\tau$ and $s_i(\tau)$, as desired.      
\end{proof}

\begin{proof}[Proof of \texorpdfstring{\Cref{prop: X_Markov}}{Proposition}]
    The strong Markovianity of $\sX_t$ is shown in \Cref{lem: StrongMarkov} so it remains to establish weak Markovianity for $X_t$.  

    When $u=0$ there is nothing to prove, so assume that $u>0$.    
    Then, for every $\eta \in (0,1)$ conditioning on $X_{u-\varepsilon}$ and $s_i(\Lc{i}_{u-\varepsilon})$ for some sufficiently small $\varepsilon$ depending on $\eta$ and using the strong Markovianity, \Cref{lem: DistantBoundary} yields that $X_u \not\in \cup_{i=0}^{\nb} B_i$ with probability $\geq 1-\eta$. 
    (This is not circular as the arguments for \Cref{lem: DistantBoundary} do not specifically require the weak Markovianity.)
    Hence, since $\eta$ is arbitrary, we have that $X_u \not\in \cup_{i=0}^{\nb} B_i$ with probability $1$. 

    When $X_u \not\in \cup_{i=0}^{\nb} B_i$, then \cref{item: Def_ReflectedBrownianMotion_iii} in \Cref{def: ReflectedBrownianMotion} implies that one can deduce the value of $s_i(\Lc{i}_u)$ from $X_u$. 
    Hence, the $\sigma$-algebra generated by $X_u$ is equal to the $\sigma$-algebra generated by $\sX_u = (X_t, (s_i(\Lc{i}_t)_{i=0}^{\nb}))$.  
    The weak Markovianity of $X_t$ hence follows from the strong Markovianity of $\sX_t$. 
\end{proof}

\subsection{Expression for the local time}\label{sec: ExpressionLocal}
Recall that \eqref{eq: LocalTime} claimed a more explicit description for the local time. 
It is now a convenient location to give a formal proof. 

Let us remark that this formula was only provided for the definiteness and will not be required for any of the other results of this paper. 
The reader may hence feel free to skip this section.  
\begin{proposition}\label{prop: FormulaLocalTime}
    Fix some $T>0$. 
    Then, almost surely, for every $i\in \{0,1\ldots,{\nb} \}$ the local time at the $i$th barrier satisfies  
    \begin{align} 
        \lim_{\varepsilon \to 0} \sup_{t\in [0,T]} \Bigl\lvert \Lc{i}_t - \frac{1}{2\varepsilon} \int_{0}^t \bb1\{ \exists y \in B_i: \Vert X_s - y \Vert < \varepsilon \}\, \intd s\Bigr\rvert = 0.  \label{eq:FastToy}
    \end{align}
\end{proposition}
\begin{proof}
    That \eqref{eq:FastToy} holds for reflected Brownian motion on a smooth domain without semipermeable barriers is classical; see, \eg \cite[Theorem 2.6]{burdzy2004heat} for a proof.
    Our case with semipermeable barriers can be reduced to this classical setting by a stopping time argument. 
    
    More precisely, recall that the construction from the proof of existence in \Cref{apx: ProofExistUnique} used an increasing sequence of stopping times $(\tau_{\nn})_{{\nn}=0}^\infty$ to define $X_t$ and the local times $\Lc{i}_t$ in terms of a classical reflected Brownian motion $Y^{\bk{{\nn}}}$ and its local time $\cL^{\bk{{\nn}}}$. 
    In particular, by the expression for the local time in \eqref{eq:RoyalBox} and \eqref{eq:OccultFish}, 
    \begin{align} 
        \Lc{i}_t = 
        \begin{cases}
            \int_0^t \bb1\{Y_r^{\bk{0}} \in B_i \} \intd \cL_r^{\bk{0}}& \text{ for }t\leq \tau_0,\\ 
            \Lc{i}_{\tau_{\nn}}+  \int_0^{t-\tau_{\nn} } \bb1\{Y_r^{\bk{{\nn}+1}} \in B_i \} \intd \cL_r^{\bk{{\nn}+1}}& \text{ for }t\in (\tau_{\nn}, \tau_{{\nn}+1}].  
        \end{cases}
        \label{eq:JollyGym}
    \end{align}
    Expanding \eqref{eq:JollyGym} recursively when $t\in (\tau_{\nn}, \tau_{{\nn}+1}]$ yields a sum of ${\nn}+1$ integrals. 
    To match this, one can similarly decompose $\frac{1}{2\varepsilon}\int_0^t  \bb1\{ \exists y \in B_i: \Vert X_s - y \Vert < \varepsilon \}\, \intd s$ into a sum of integrals.  
    Then, using that $X_t = Y_{t- \tau_{{\nn}-1}}^{\bk{{\nn}}}$ for $t\in (\tau_{{\nn}-1}, \tau_{\nn}]$ by \eqref{eq:RoyalBox} and \eqref{eq:OccultFish}, we have the following coarse upper bound: 
    \begin{align} 
        &\sup_{t\in [0,T]} \Bigl\lvert \Lc{i}_t - \frac{1}{2\varepsilon} \int_{0}^t \bb1\{ \exists y \in B_i: \Vert X_s - y \Vert < \varepsilon \}\, \intd s\Bigr\rvert\label{eq:RoyalFan} \\ 
        &\leq \sum_{{\nn}=0}^\infty  \bb1\{T \leq \tau_{{\nn}}\} \sup_{t\in [0,T]} \Bigl\lvert \int_0^t \bb1\{Y_r^{\bk{{\nn}}} \in  B_i \}\,\intd \cL_r^{\bk{{\nn}}} - \frac{1}{2\varepsilon} \int_{0}^t \bb1\{ \exists y \in B_i: \Vert Y_r^{\bk{{\nn}}} - y \Vert < \varepsilon \}\, \intd s\Bigr\rvert.\nonumber
    \end{align}
    Here, $Y_r^{\bk{n}}$ is a classical reflected Brownian motion in a smooth domain.
    By using the aforementioned classical result, it may then be shown that for every ${n} \geq 0$, 
    \begin{align} 
        \lim_{\varepsilon \to 0}\sup_{t\in [0,T]} \Bigl\lvert \int_0^{t} \bb1\{Y_r^{\bk{n}} \in B_i \} \, \intd \cL_r^{\bk{n}} - \frac{1}{2\varepsilon} \int_{0}^t \bb1\{ \exists y \in B_i: \Vert Y_s^{\bk{n}} - y \Vert < \varepsilon \}\, \intd s\Bigr\rvert = 0, \label{eq:ZippyMom}
    \end{align}
    almost surely.
    Now, since there are almost surely finitely many terms on the right-hand side of \eqref{eq:RoyalFan}, the combination of \eqref{eq:RoyalFan} and \eqref{eq:ZippyMom} concludes the proof.  
\end{proof}

\section{Proof of \texorpdfstring{\Cref{lem: BarrierLocallyStraight}}{Lemma}}\label{apx: ProofBarrierLocallyStraight}
\begin{proof}
    Recall from \Cref{sec: DomainsAndBarriers} that it was assumed that the barriers $B_i$ do not intersect. 
    Consequently, since $(\cup_{i=0}^{\nb} B_i)\cap \sB(x_0,r)$ is connected for every $r\leq \rho$ by definition of $\rho$ in \eqref{eq: Def_rho}, there exists at most one $B_i$ with $B_i \cap \sB(x_0,r) \neq \emptyset$. 
    In particular, this holds for $r = r_\delta$ since $\delta<1/2 \leq 1$.  

    Assume that a barrier $B_i$ intersecting $\sB(x_0,r_\delta)$ exists. 
    Pick some $y_0 \in  B_i$ with $\Vert y_0 - x_0 \Vert = \min\{\Vert y - x_0 \Vert: y \in B_i \}$. 
    The minimality of the distance to $x_0$ together with the assumption that $B_i$ intersects $\sB(x_0,r_\delta)$ then yields $y_0 \in \sB(x_0, r_\delta)$.
    Define $ \hatn \de \vec{n}_i(y_0)$ and $\mathfrak{c} \de \langle y_0, \hatn \rangle$. 

    Pick an arc-length parametrization $\phi: \bbR \to B_i$ with $\phi(0) = y_0$. 
    Here, since $B_i$ is a closed curve, it should be understood that $\phi$ is periodic with period given by the arc length of $B_i$. 
    We next construct a small neighborhood of $0$ such that the restriction of $\phi$ to the neighborhood parametrizes $B_i \cap \sB(x_0, r_\delta)$ and subsequently verify items \ref{item:GhostlyQuip}--\ref{item:JollyDragon}.    
    
    The minimality of the distance of $y_0$ to $x_0$ implies that $\langle \dd{t}\phi(0), y_0 - x_0 \rangle = 0$.
    Recall from \eqref{eq: Def_kappa} that $\kappa$ is a bound on the curvature of $B_i$. 
    Hence, the first Frenet--Serret formula \cite[p.70, Eq.\ (7.1)]{dineen2014multivariate} implies that $\Vert \frac{\intd^2 }{\intd t^2}\phi(t) \Vert \leq \kappa$ and consequently $\Vert \dd{t}\phi(s) - \dd{t}\phi(0) \Vert \leq \kappa s$ for every $s>0$. 
    Then, by the fundamental theorem of calculus,  
    \begin{align} 
        \langle \dd{t}\phi(0) , \phi(s) - \phi(0)\rangle  = {\textstyle \int_0^s} \langle \dd{t}\phi(0) , \dd{t}\phi(t)  \rangle \, \intd t \geq s - \frac{\kappa}{2}s^2.\label{eq:MythicalGem}  
    \end{align} 
    It follows that $\langle \dd{t}\phi(0) , \phi(2r_\delta) - \phi(0)\rangle \geq 2 (1-\kappa r_\delta)r_\delta$.
    Here, it holds that $2(1-\kappa r_\delta) >1$ due to the assumption that $\delta < 1/2$ and the fact that $\kappa r_\delta \leq \delta$ by \eqref{eq: Def_R_eps}.      
    Considering that $\langle \dd{t}\phi(0), y_0 - x_0 \rangle = 0$ and $\phi(0) = y_0$, it now follows that $\langle \dd{t}\phi(0), \phi(2r_\delta) - x_0 \rangle > r_\delta$. 
    Consequently, since $\dd{t}\phi(0)$ is a unit vector by $\phi$ being an arc-length parametrization, we have $\Vert \phi(2r_\delta) - x_0  \Vert>r_\delta$. 
 
    Recall that $\phi(0) = y_0$ is an element of $B_i \cap \sB(x_0,r_{\delta})$. 
    Let $s_{+} \de \min\{t >0: \phi(t) \not\in  \sB(x_0,r_{\delta})\}$ and note that the preceding shows that $s_+ < 2r_\delta$. 
    It can similarly be shown that $s_- > -2r_\delta$ for $s_{-} \de \min\{t <0: \phi(t) \not\in \sB(x_0,r_\delta) \}$. 
    Since $\phi$ is surjective and $B_i \cap \sB(x_0,r_\delta)$ is connected, the restriction of $\phi$ to $(s_-,s_+)$ then yields a parametrization for $B_i \cap \sB(x_0,r_\delta)$. 
    
    Let us proceed to the verification of \cref{item:GhostlyQuip}.
    Recall that $\mathfrak{c} = \langle \phi(0), \hatn \rangle$.
    Hence, it suffices to show that $\lvert \langle \phi(s)- \phi(0), \hatn \rangle \rvert < 4\delta r_\delta$ for every $s\in (s_-, s_+)$.  
    For every such $s$, 
    \begin{align} 
        \phi(s) - \phi(0) = {\textstyle \int_0^s} \dd{t}\phi(t)\,  \intd t = s \dd{t}\phi(0) + {\textstyle \int_0^s} (\dd{t}\phi(t) - \dd{t}\phi(0))\, \intd t. 
    \end{align}
    Recall that $ \hatn$ is the normal vector to $B_i$ at $y_0 = \phi(0)$ and hence orthogonal to the tangent vector $\dd{t}\phi(0)$.
    Consequently,  using that  $\Vert \dd{t}\phi(s) - \dd{t}\phi(0) \Vert \leq \kappa \lvert s \rvert$ together with the fact that $\lvert s \rvert< 2r_\delta$ for $s \in (s_-,s_+)$,
    \begin{align} 
        \lvert \langle  \phi(s)- \phi(0), \hatn \rangle \rvert \leq 2r_\delta \negsp \sup_{s\in [s_{-},s_{+}]}  \Vert \dd{t}\phi(s) - \dd{t}\phi(0)   \Vert < 4\kappa r_\delta^2 \leq 4\delta r_\delta.      
    \end{align}
    This concludes the proof of \cref{item:GhostlyQuip}. 

    The second Frenet--Serret formula \cite[p.74, Eq.\ (7.1')]{dineen2014multivariate} yields that $\Vert \dd{t}\vec{n}_i( \phi(t)) \Vert \leq \kappa$. 
    Hence, since $ \hatn = \vec{n}_i(\phi(0))$, it holds that $\Vert \vec{n}_i(\phi(s))-  \hatn \Vert \leq \lvert s \rvert\kappa$ for every $s$. 
    In particular, for every $s \in (s_-, s_+)$ we have $\Vert \vec{n}_i(\phi(s))-  \hatn \Vert < 2  \delta$. 
    This proves \cref{item:IdleBat}.   

    Finally, let us prove \cref{item:JollyDragon}. 
    For brevity, we focus on the first part of the claim. 
    (The second part follows similarly if one replaces the roles of positive and negative sides in the arguments.) 
    We start by considering points of the form $y_0 +  \alpha \hatn$ for some $\alpha \in \bbR$ and subsequently employ a reduction argument.   
    
    As a preparatory result, we claim that the only point in $\sB(x_0,r_\delta)\cap B_i$ of the form $y_0 + \alpha\hatn$ is $y_0$ itself. 
    Recalling that $y_0 = \phi(0)$ and that $\hatn = \vecn_i(y_0)$ is orthogonal to $\dd{t}\phi(0)$, this claim is equivalent to showing that $\langle \dd{t}\phi(0), \phi(0) - \phi(s) \rangle = 0$ for $s\in (s_-, s_+)$ if and only if $s=0$. 
    The latter holds true. 
    Indeed, one can check that the right-hand side of \eqref{eq:MythicalGem} is $>0$ for $s \in (0,s_+)$, and it may similarly be shown that $\langle \dd{t}\phi(0), \phi(0) - \phi(s) \rangle < 0$ for $s \in (s_-,0)$.

    Recall from \Cref{sec: DomainsAndBarriers} that $\vecn_i$ points to the positive side of $B_i$.
    This means that $y_0 + \varepsilon\hat{n}$ (resp.\ $y_0 - \varepsilon\hat{n}$) is strictly on the positive (resp.\ negative) side of $B_i$ for $\varepsilon>0$ sufficiently small.    
    However, the side can only change when crossing $B_i$, so, using the foregoing preparatory result, it follows that a point of the form $y_0 + \alpha\hatn$ in $\sB(x_0, r_\delta)$ lies on the positive side of $B_i$ if and only if $\alpha \geq 0$.

    Now, consider an arbitrary $z\in \sB(x_0, r_\delta)$ with $\langle z,\hatn \rangle \leq \mathfrak{c} - 4\delta r_\delta$.  
    Then, by considering the contrapositive, the desired claim from \cref{item:JollyDragon} follows if we show that such a $z$ is \emph{not} on the positive side.

    We will apply the foregoing result to the point $\hat{z} \de y_0 + \alpha \hatn$ with $\alpha \de \langle z - y_0 , \hatn\rangle$. 
    First, let us check that this point is again in $\sB(x_0, r_\delta)$.   
    Recall from paragraph preceding \eqref{eq:MythicalGem} that $x_0 - y_0$ is orthogonal to $\dd{t}\phi(0)$. 
    Equivalently $x_0 - y_0 = \langle x_0  - y_0, \hatn \rangle\hatn$, and hence  
    \begin{align} 
        \hat{z}= y_0 + \langle z - y_0, \hatn \rangle \hatn = y_0 +  (x_0 - y_0) + \langle z - x_0, \hatn \rangle \hatn = x_0 + \langle z - x_0, \hatn \rangle \hatn.    
    \end{align}
    Consequently, we have $\Vert \hat{z} - x_0 \Vert \leq \Vert z - x_0 \Vert$ which implies that $\hat{z} \in \sB(x_0, r_\delta)$, as desired. 
    Further, since $\mathfrak{c} = \langle y_0, \hatn \rangle$, we have that $\alpha  \leq -4\delta r_\delta$. 
    This value is strictly negative implying that $\hat{z}$ is not on the positive side of $B_i$. 
    
    Finally, to transfer this to $z$, consider the line segment connecting $z$ to $\hat{z}$. 
    All points in this line segment are again elements of $\sB(x_0, r_\delta)$ because balls are convex sets. 
    Now recall from \cref{item:GhostlyQuip} that every $y\in B_i \cap \sB(x_0, r_\delta)$ satisfies $\langle y, \hatn \rangle > \mathfrak{c} - 4\delta r_\delta$.  
    Considering that $\langle z , \hatn \rangle  = \langle \hat{z}, \hatn \rangle \leq \mathfrak{c} - 4\delta r_\delta$ it hence follows that the line segment can not intersect $B_i$. 
    This implies that $z$ is also not on the positive side, concluding the proof.    
\end{proof}

\section{Proof that \texorpdfstring{\Cref{thm: Main_GlobalRecovery}}{Theorem} follows from \texorpdfstring{\Cref{prop: Main_GlobalRecovery_Algorithm,prop: Main_GlobalRecovery_Improvement}}{Propositions}}\label{apx: Proof_Main_GlobalRecovery}
\begin{proof}[Proof of \texorpdfstring{\Cref{thm: Main_GlobalRecovery}}{Theorem}]
    To avoid ambiguity with other parts of the proofs, let us denote $\cC_1,\ldots,\cC_3>0$ for the constants appearing in the statement of \Cref{thm: Main_GlobalRecovery}. 
    Then, the assumption is that 
    \begin{align} 
        t \leq \cC_1 \min\{1/\kappa^2, 1/\lambda_{\max}^2,\rho^2\}, \ \ \varepsilon \leq \cC_2 \kappa t, \ \ T \geq \cC_3 \frac{\tmix}{\pi_{\min}}\sqrt{\frac{\kappa}{\varepsilon^3}}\ln\Bigl(\frac{\Area(D)}{\eta} \sqrt{\frac{\kappa}{\varepsilon^3}} \Bigr).\label{eq:UnripeBoy}   
    \end{align}
    Let $\sE_0 \de \sqrt{2} (\varepsilon^3/\kappa)^{1/4}$. 
    Then, in particular,   
    \begin{align} 
        \sE_0 \leq \sqrt{2}\cC_1^{1/4} \cC_2^{3/4} \sqrt{t}\quad \text{ and }\quad T \geq \cC_3 \frac{ \tmix}{\pi_{\min}}\frac{2}{\sE_0^2} \ln\Bigl(\frac{2}{\eta}\frac{\Area(D)}{\sE_0^2}\Bigr).    
    \end{align}  
    Consequently, taking $\cC_1,\cC_2$ sufficiently small and $\cC_3$ sufficiently large, we may assume that \Cref{prop: Main_GlobalRecovery_Algorithm} is applicable when $\varepsilon$ is replaced by $\sE_0$ and $\eta$ is replaced by $\eta/2$. 
    That is, the parameters in \Cref{alg: KernelDiscontinuity} can be chosen such that the output $\fB_0$ satisfies 
    \begin{align} 
        \bbP\bigl(\dH(\fB_0, \cup_{i=0}^{\nb} B_i ) \leq \sE_0 \bigr) \geq 1 - \eta/2.\label{eq:HauntedRock}  
    \end{align}
    If $\sE_0 < \varepsilon$ then we are done, so let us assume that $\sE_0 \geq \varepsilon$.  
    
    A preliminary estimate is required. 
    Recall that $B_0 = \partial D$ and that the curvature of $B_0$ is bounded by $\kappa$. 
    This implies that $\Area(D) \geq \pi/\kappa^2$; see \eg \cite{pankrashkin2015inequality}.
    In particular, taking $\cC_1, \cC_2 \leq 1$ and using \eqref{eq:UnripeBoy}, we have $\Area(D) \geq \pi \sqrt{\varepsilon^3/\kappa}$, and hence   
    \begin{align} 
        \ln\Bigl(\frac{\Area(D)}{\eta}\sqrt{\frac{\kappa}{\varepsilon^3}} \Bigr) 
        = \ln\Bigl(\frac{2\Area(D)}{\eta}\sqrt{\frac{\kappa}{\varepsilon^3}} \Bigr) - \ln(2) 
        \geq c\ln\Bigl(\frac{2\Area(D)}{\eta}\sqrt{\frac{\kappa}{\varepsilon^3}} \Bigr)  \label{eq:ValidVolt} 
    \end{align}
    with $c>0$ an absolute constant which is sufficiently small such that $(1-c)\ln(2\pi) \geq \ln(2)$. 

    We will apply \Cref{prop: Main_GlobalRecovery_Algorithm} with $\eta$ replaced by $\eta/2$ and next verify that the constraints can be satisfied by choosing $\cC_1,\cC_2$, and $\cC_3$ appropriately. 
    First, using \eqref{eq:ValidVolt}, the constraint on $T$ referring to \eqref{eq: Main_GlobalRecovery_n} may be satisfied by taking $\cC_3$ sufficiently large in \eqref{eq:UnripeBoy}. 
    The conditions on $t$ and $\varepsilon$ can further also be satisfied by taking $\cC_1$ and $\cC_2$ sufficiently small; specifically smaller than the constants $c_1$ and $c_2$ from \Cref{prop: Main_GlobalRecovery_Algorithm}, respectively.
    Finally, the required upper bound on $\sE_0$ may also be assumed to hold by taking $\cC_1$ and $\cC_2$ sufficiently small: recall the definition of $\ell$ in \eqref{eq:UnripeNose} and note that $\sqrt{2}(\varepsilon^3/\kappa)^{1/4} \leq \sqrt{2} \cC_1^{1/4}\cC_2^{1/4} \sqrt{\varepsilon/\kappa}$ due to \eqref{eq:UnripeBoy}.

    It follows that there exists an event $\cE$ with $\bbP(\cE) \geq 1 -\eta/2$ such that the following holds. 
    Define scalars $\sE_1,\sE_2, \ldots$ by $\sE_{j} = \max\{\sE_{j-1}/2,  \varepsilon \}$ for every $j\geq 1$, and recursively define $\fB_1,\fB_2,\ldots\subseteq \bbR^2$ by applying \Cref{alg: Improve} on the previous estimate. 
    Then, for all $j\geq 1$,     
    \begin{align}
        \dH(\fB_j, \cup_{i=0}^{{\nb}}B_i) \leq \max\{\sE_0/2^{j},\varepsilon \}\ \text{ if }\cE \text{ occurs and }\dH(\fB_0, \cup_{i=0}^{\nb} B_i) \leq \sE_0.  
    \end{align}
    In particular, considering \eqref{eq:HauntedRock} and that $\cE$ occurs with probability $\geq 1 - \eta/2$, 
    \begin{align} 
        \bbP\bigl(\dH(\fB_J, \cup_{i=0}^{\nb} B_i) \leq \varepsilon \bigr) \geq 1 - \eta \text{ for }J  \de \lceil \log_2(\sE_0/\varepsilon) \rceil. 
    \end{align}    
    Let $\hat{B} \de \fB_J$ to conclude the proof. 
\end{proof}

\section{Proof of \texorpdfstring{\Cref{lem: DistantBoundary}}{Lemma}}\label{apx: ProofDistantBoundary}

\begin{proof}[Proof of \texorpdfstring{\Cref{lem: DistantBoundary}}{Lemma}]
    We rely on the tools from \Cref{sec: Local}. 
    Let $c_1,c_2$ be as in \Cref{lem: DistantBoundary} and set $
        \delta \de c_3\sqrt{t}/\min\{ 1/\kappa, 1/\lambda_{\max},\rho \}  
    $ with $c_3$ a constant, depending on $\eta$, to be chosen later.

    Let us start with the case where $\sB(x_0, r_\delta) \cap B_i \neq \emptyset$ for some $i\leq {\nb}$. 
    By \cref{item:GhostlyQuip} from \Cref{lem: BarrierLocallyStraight}, every point in $\sB(x_0, r_\delta) \cap B_i$ satisfies $\langle y, s_i(0)\hatn \rangle \leq s_i(0)\mathfrak{c} + 4\delta r_\delta$. 
    Hence, using that the distance to any point outside $\sB(x_0, r_\delta)$ is $\geq r_\delta - \Vert X_t - X_0 \Vert$, it suffices to show that 
    \begin{align} 
        \bbP\Bigl(\langle X_t, s_i(0)\hatn \rangle >s_i(0)\mathfrak{c} + 4\delta r_\delta + c_2 \sqrt{t} \ \text{ and }\ \Vert X_t - X_0  \Vert \leq  r_\delta - c_2\sqrt{t} \Bigr)\geq 1 - \eta.\label{eq:ShyEel}  
    \end{align}
    A visualization of the event in \eqref{eq:ShyEel} may be found in \Cref{fig: Intersect2}. 
        
    \begin{figure}[t]
        \includegraphics[width = 0.8\textwidth]{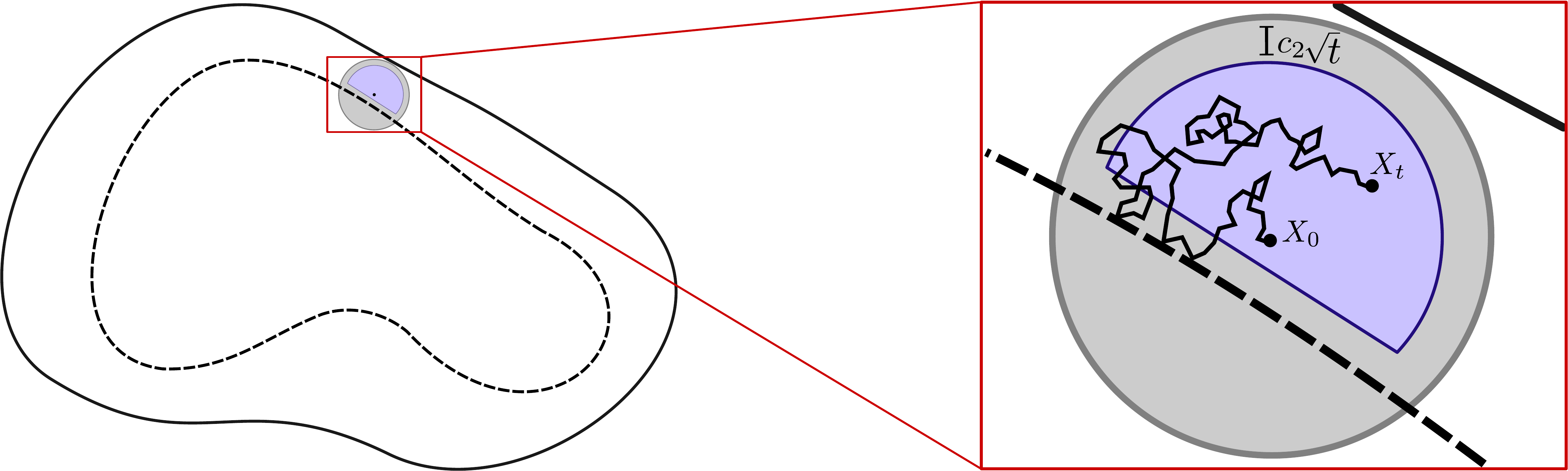}
        \caption{Visualization of the event from \eqref{eq:ShyEel} in the proof of \Cref{lem: DistantBoundary}. 
        Note that every point in the blue spherical cap is at distance $\geq c_2 \sqrt{t}$ from $\cup_{i=0}^{\nb} B_i$.       
        } 
        \label{fig: Intersect2}
    \end{figure}

    Let $\delta_0, C_1,C_2$ be as in \Cref{cor: X_uniform_Y}.  
    By definition of $\delta$, we have $r_\delta = c_3 \sqrt{t}$. 
    Further, by the assumption on $t$, we have $\delta \leq c_3 \sqrt{c_1}$. 
    Hence, using \Cref{cor: X_uniform_Y}, 
    \begin{align} 
        {}&{}\bbP\Bigl(\langle X_t, s_i(0)\hatn \rangle \leq s_i(0)\mathfrak{c} + 4\delta r_\delta + c_2 \sqrt{t}\Bigr)\label{eq:ValidImp} \\ 
        & \leq  \bbP\Bigl(\langle Y_t^+, s_i(0)\hatn \rangle\leq  s_i(0)\mathfrak{c} +  4\delta r_\delta  + (C_1 c_3^2 \sqrt{c_1} + c_2) \sqrt{t}\Bigr) + c_3 \sqrt{c_1} + C_2/c_3^2 \ \text{ if }\ c_3\sqrt{c_1}\leq \delta_0.\nonumber        
    \end{align} 
    Here, note that $\langle Y_t^+ , s_i(0)\hatn \rangle -(s_i(0)\mathfrak{c} +  4\delta r_\delta)$ is a one-dimensional reflected Brownian motion with initial condition $z_0 \de \langle y_0^+ , s_i(0)\hatn \rangle -(s_i(0)\mathfrak{c}+ 4\delta r_\delta)$.
    In particular, for every $x \geq 0$, 
    \begin{align} 
         \bbP\bigl(\langle Y_t^+, s_i(0)\hatn \rangle  \leq s_i(0)\mathfrak{c} + 4\delta r_\delta + x \bigr) &= \bbP\bigl(\lvert \sqrt{t}N +z_0\rvert  \leq x \bigr)\label{eq:YellowWisp} 
    \end{align} 
    with $N$ a standard normal random variable. 
    Here, since the density of the standard normal distribution is bounded, $\bbP(N \in [(z_0-x)/\sqrt{t}, (z_0 + x)/\sqrt{t}]) \leq C_4 x/\sqrt{t}$ for some absolute constant $C_4>0$.
    Combining \eqref{eq:ValidImp} and \eqref{eq:YellowWisp} then yields that 
    \begin{align} 
        \bbP\Bigl(\langle X_t, s_i(0)\hatn \rangle \leq s_i(0)\mathfrak{c} + 4\delta r_\delta {}&{} + c_2 \sqrt{t}\Bigr) \\ 
        &\leq   C_4 \bigl(C_1 c_3^2 \sqrt{c_1} + c_2\bigr) + c_3 \sqrt{c_1} + C_2/c_3^2\  \text{ if }\ c_3\sqrt{c_1} \leq \delta_0.\nonumber  
    \end{align}
    As for the second part of the event in \eqref{eq:ShyEel}, using that $r_\delta = c_3\sqrt{t}$,   
    \begin{align} 
        \bbP\Bigl(\Vert X_t - X_0 \Vert > r_\delta - c_2 \sqrt{t}\Bigr) \leq \bbP\Bigl(\Vert X_t - X_0 \Vert > r_\delta / 2 \Bigr)\quad \text{ if }\ c_2 \leq c_3/2.\label{eq:GreenVolt}   
    \end{align} 
    Recall \eqref{eq: Def_R_eps} and note that $r_\delta / 2 = r_{\delta/2}$. 
    Hence, by \Cref{cor: TauBound}, 
    \begin{align} 
        \bbP\Bigl(\Vert X_t - X_0 \Vert > r_\delta - c_2 \sqrt{t}\Bigr) \leq c_3\sqrt{c_1}/2 + 4C_5/c_3^2 \quad \text{ if } c_2\leq c_3/2 \text{ and }c_3\sqrt{c_1}/2 \leq \delta_0'  \label{eq:DarkWisp}  
    \end{align}
    with $\delta_0',C_5$ the absolute constants from \Cref{cor: TauBound}.

    The combination of \eqref{eq:ValidImp} and \eqref{eq:DarkWisp} yields \eqref{eq:ShyEel} by choosing $c_1,c_2,c_3$ appropriately, similarly to the argument after \eqref{eq:DarkGoose}. 
    That is, one takes $c_3$ sufficiently large to ensure that $C_2/c_3 \leq \eta/3$ and $4C_5/c_3^2 \leq \eta/3$, and subsequently takes $c_1,c_2$ to be sufficiently small to ensure that the contribution of all remaining terms is $\leq \eta/3$ and that the conditions hold.     

    It remains to consider the case where $\sB(x_0, r_\delta)\cap (\cup_{i=0}^{\nb} B_i) = \emptyset$. 
    Then, $X_t$ is of distance $\geq c_2\sqrt{t}$ from all barriers if $\Vert X_t - X_0 \Vert \leq r_\delta - c_2\sqrt{t}$, so one can proceed similarly to \eqref{eq:GreenVolt}. 
    This concludes the proof.
\end{proof}

\section{Proof of \texorpdfstring{\Cref{lem: Concentration_hatPS}}{Lemma}} \label{apx: Concentration}
Fix some measurable subset $S \subseteq D$ with $\Area(S)>0$ and recall that our goal is to show that $\hat{P}_S\approx P_S$ with $\hat{P}_S$ as in \eqref{eq: Def_hatP} and $P_S$ given by  
\begin{align} 
    P_S(E) = \pi(S)^{-1}\int_{S}\bbP(X_t\in E \mid X_0 = x_0)\ \intd \pi(x_0) \ \  \text{ for every measurable }E\subseteq D.  \label{eq: Def_PS}  
\end{align}  
Here, $\pi$ denotes the stationary distribution; recall \Cref{sec: Parameters}.  

Proof-technically, it is convenient to introduce an additional discretization.
Fix a small constant $\mathfrak{d} >0$ and, for any probability measure $Q$ on $\bbR^2$, let $\bin_\mathfrak{d}[Q]$ be the probability measure found by averaging over the bins of the grid $\mathfrak{d}\bbZ^2$. 
That is, the measure whose density is constant on $B_{\mathfrak{d}}(x,y) \de [x, x+  \mathfrak{d} )\times [y, y +\mathfrak{d} )$ for every $x,y\in \mathfrak{d}\bbZ$ and which satisfies $\bin_\mathfrak{d}[Q](B_{\mathfrak{d}}(x,y)) = Q(B_{\mathfrak{d}}(x,y))$.
A visualization may be found in \Cref{fig: Binned}. 
Then, it suffices to show that $\bin_{\mathfrak{d}}[\hat{P}_S] \approx \bin_{\mathfrak{d}}[P_s]$:

\begin{figure}[t]
    \centering
    \begin{subfigure}{.45\textwidth}
      \centering
      \small 
      \textbf{A probability distribution $Q$}
      \par\bigskip
      \includegraphics[width=.55\linewidth]{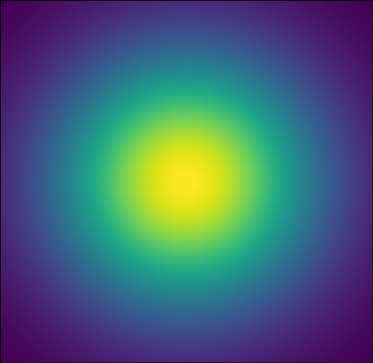}
    \end{subfigure}%
    \begin{subfigure}{.45\textwidth}
      \centering
      \small 
      \textbf{The binned version $\bin_{\mathfrak{d}}[Q]$} 
      \par\bigskip
      \includegraphics[width=.55\linewidth]{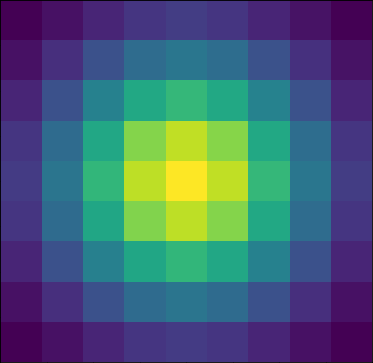}
    \end{subfigure}
    \caption{
        Visualization of the binning operation used in \Cref{apx: Concentration}. 
    }
    \label{fig: Binned}
\end{figure}  
\begin{lemma}\label{lem: WuBin_approxerror}
    Fix constants $\mathfrak{d}, \mathfrak{u}>0$. 
    Then, for every probability measure $Q$ on $\bbR^2$, 
    \begin{align} 
        \Wu \bigl(Q,\bin_{\mathfrak{d}}[Q] \bigr) \leq \min\{\sqrt{2}\mathfrak{d}, \mathfrak{u} \}.
    \end{align}
\end{lemma}
\begin{proof}
    Consider a $Q$-distributed random variable $X$ and let $Y$ be a $\Unif(B_{\mathfrak{d}}(x,y))$-distributed random variable conditional on the event $X\in B_{\mathfrak{d}}(x,y)$. 
    Then, recalling the definition of truncated Wasserstein from \eqref{eq: Def_W1u} and using that $Y$ is $\bin_{\mathfrak{d}}[Q]$-distributed, 
    \begin{align} 
        \Wu(Q,\bin_{\mathfrak{d}}[Q]) \leq \bbE\bigl[d_{\mathfrak{u}}(X,Y)\bigr] = \bbE\bigl[\min\{\Vert X - Y \Vert, \mathfrak{u} \} \bigr] \leq \min\{\sqrt{2}\mathfrak{d}, \mathfrak{u}\}
    \end{align}
    where the final inequality used that $\Vert a - b \Vert \leq \sqrt{2}\mathfrak{d}$ for every $a,b \in B_{\mathfrak{d}}(x,y)$. 
\end{proof}

Given two finite measures $\mu, \nu$ defined on the same measurable space $(\Omega, \cF)$, the \emph{total variation distance} is defined by 
\begin{align} 
    \Vert \mu - \nu \Vert_{\TV} \de \sup_{E \in \cF} \lvert \mu(E) - \nu(E) \rvert. \label{eq: Def_TV_Distance} 
\end{align}
The following result, whose proof is given in \Cref{sec: ProofConcentrationBinned}, establishes concentration of $\bin_{\mathfrak{d}}[\hat{P}_S]$ for the total variation distance and, as a consequence, also for truncated Wasserstein. 
\begin{lemma}\label{lem: ConcentrationBinned}
    For every $\eta \in (0,1)$ and $\gamma >0$ there exist constants $c_1, c_2, c_3>0$ depending only on $\eta$ and $\gamma$ such that the following holds when $\mathfrak{d} \de \gamma \sqrt{t}$. 
    
    Assume that $t\leq c_1 \min\{1/\kappa^2, 1/\lambda_{\max}^2,\rho^2 \}$ and assume that $\Vert a-b \Vert \leq \sqrt{t}$ for every $a,b \in S$. 
    Further, suppose that $X_0 \sim \pi$ starts from the stationary distribution. 
    Then,  
    \begin{align} 
        \bbP\bigl(\Vert \bin_{\mathfrak{d}}[\hat{P}_S] - \bin_{\mathfrak{d}}[P_S]  \Vert_{\TV} > \eta\bigr) \leq   c_2\exp\bigl(-c_3 \pi(S) T/ \tmix \bigr). \label{eq:IvoryQuote}
    \end{align}
    In particular, for every truncation level $\mathfrak{u}>0$, 
    \begin{align} 
        \bbP\bigl(\Wu(\bin_{\mathfrak{d}}[\hat{P}_S], \bin_{\mathfrak{d}}[P_S]) > \eta \mathfrak{u} \bigr) \leq   c_2\exp\bigl(-c_3 \pi(S) T/ \tmix \bigr). \label{eq:SlowEel}
    \end{align} 
\end{lemma}

The technical advantage of the $\bin_{\mathfrak{d}}[\cdot]$-operation is now also clear: concentration with respect to total variation would be blatantly false without this operation. 
Indeed, it holds that $\Vert \hat{P}_S -P_S  \Vert_{\TV} = 1$ with probability one since $\hat{P}_S$ is finitely supported and $P_S$ assigns no mass to finite sets. 

Combining \Cref{lem: WuBin_approxerror} and \Cref{lem: ConcentrationBinned} with the triangle inequality now yields \Cref{lem: Concentration_hatPS}:
\begin{proof}[Proof of \texorpdfstring{\Cref{lem: Concentration_hatPS}}{Lemma}]
    Using that $\Wu$ is a metric, 
    \begin{align} 
        \Wu(\hat{P}_S, P_s) \leq \Wu(\hat{P}_S, \bin_{\mathfrak{d}}[\hat{P}_S]) +\Wu(\bin_{\mathfrak{d}}[\hat{P}_S], \bin_{\mathfrak{d}}[P_S]) +\Wu(\bin_{\mathfrak{d}}[P_S] , P_S).\label{eq:GladRug}   
    \end{align}
    The first and final term can here be controlled by choosing $\mathfrak{d}$ appropriately in \Cref{lem: WuBin_approxerror}:   
    \begin{align} 
        \Wu(\hat{P}_S, \bin_{\mathfrak{d}}[\hat{P}_S]) +\Wu(\bin_{\mathfrak{d}}[P_S] , P_S)\leq \beta \sqrt{t}/2\  \ \text{ if }\ \ \sqrt{2}\mathfrak{d} = \beta\sqrt{t}/4. 
    \end{align}
    Further, recall the assumption that $\mathfrak{u}\leq \alpha \sqrt{t}$ in \Cref{lem: Concentration_hatPS}. 
    Hence, using \eqref{eq:SlowEel} with $\gamma = \beta/(4\sqrt{t})$ and $\eta = \beta/(2\alpha)$, it holds for $t\leq c_1 \min\{ 1/\kappa^2, 1/\lambda_{\max}^2 ,\rho^2 \}$ that 
    \begin{align} 
        \bbP\bigl(\Wu(\bin_{\mathfrak{d}}[\hat{P}_S], \bin_{\mathfrak{d}}[P_S]) > \beta \sqrt{t}/2 \bigr) \leq c_2 \exp\bigl(-c_3 \pi(S)T/\tmix \bigr)\label{eq:PoliteXray}
    \end{align}
    where $c_1,c_2$, and $c_3$ only depend on $\alpha$ and $\beta$. 
    Combine \eqref{eq:GladRug}--\eqref{eq:PoliteXray} and note that $\pi(S) \geq \pi_{\min} \Area(S)$ by \eqref{eq: Def_pimin} to conclude the proof. 
\end{proof}

\subsection{Proof of \texorpdfstring{\Cref{lem: ConcentrationBinned}}{Lemma}}\label{sec: ProofConcentrationBinned}
Throughout this section, we assume that $X_0$ starts from stationarity and let $S\subseteq D$ be a fixed subset with nonzero volume. 
\begin{lemma}\label{lem: Paulin}
    Consider an integer $n\geq 1$ with $n t \geq \tmix$ and a measurable set $E\subseteq \bbR^2$. 
    Denote $N(E) \de  \#\{0 \leq i \leq n-1: X_{it}\in S, X_{(i+1)t}\in E \}$.
    Then, for every $\zeta \leq 1$,  
    \begin{align} 
        \bbP\bigl(\lvert N(E) - n\pi(S)P_S(E) \rvert  > n\zeta \pi(S)\bigr)\leq 2 \exp\bigl(-c_S n\zeta^2\bigr) 
    \end{align} 
    where $c_S \de \pi(S)/(90 \lceil\tmix/t\rceil)$. 
\end{lemma}
\begin{proof}
    The quantity of interest can be represented as 
    $
        N(E)= \sum_{i=0}^{n-1} \bb1\{X_{it}\in S, X_{(i+1)t}E \}. 
    $ 
    This brings us to a setting where general-purpose concentration inequalities such as \cite[Theorem 3.4]{paulin2015concentration}\footnote{For readers consulting the arXiv version (\ie arXiv:1212.2015v5), let us note that the numbering is there different from the published version. 
    The cited result is there Theorem 3.11.} are applicable.  
    That result uses some parameters which we next estimate.

    The mixing properties of the chain are quantified in \cite{paulin2015concentration} by the \emph{pseudo-spectral gap} $\gamma_{\operatorname{ps}}$. 
    For our purposes, a bound on this quantity will be sufficient.
    Denote $\imix$ for the mixing time of the discrete-time Markov chain $(X_{it})_{i=0}^\infty$:
    \begin{align} 
        \imix \de \min\{i\geq 1:  \lvert \bbP(X_{it} \in A \mid X_0 = x_0) - \pi(E) \rvert \leq 1/4,\  \forall x_0 \in D,\, A\subseteq D \} \label{eq:LazyFan}
    \end{align}
    where it is to be understood that $A$ runs only over measurable subsets of $D$. 
    Then, by \cite[Proposition 3.4]{paulin2015concentration}, we have $1/\gamma_{\operatorname{ps}} \leq 2 \imix$.

    Since the chain is assumed to start from its stationary distribution, the expectation and variance of the terms is fairly explicit: 
    \begin{align} 
        \bbE\bigl[\bb1\{ X_{it}\in S, X_{(i+1)t} \in E \}\bigr] &= \bbP\bigl(\bb1\{ X_{it}\in S, X_{(i+1)t} \in E \}\bigl) =  \pi(S)P_S(E),\\  
        \Var\bigl[\bb1\{ X_{it}\in S, X_{(i+1)t} \in E \}\bigr]&=\bigl(1 - \pi(S)P_S(E)\bigr) \pi(S)P_S(E) \leq \pi(S)P_S(E). 
    \end{align} 
    Further, since the terms are $\{0,1 \}$-valued, 
    \begin{align} 
        \bigl\lvert \bb1\{ X_{it}\in S, X_{(i+1)t} \in E \} - \bbE[\bb1\{ X_{it}\in S, X_{(i+1)t} \in E \}] \bigr\rvert\leq 1. 
    \end{align}
    Now, by \cite[Theorem 3.4]{paulin2015concentration}, 
    for every $x>0$,  
    \begin{align} 
        \bbP\bigl( \lvert N(E) - n\pi(S)P_S(E) \rvert > x\bigr)  
        &\leq 2\exp\Bigl(-\frac{x^2/(2\imix) }{8(n + 2\imix)\pi(S)P_S(E) + 20 x } \Bigr).\label{eq:AlertMan}
    \end{align}
    Here, comparing \eqref{eq:LazyFan} with \eqref{eq: Def_tmix}, we have $\imix \leq \lceil \tmix/t\rceil$.
    Consequently, by the assumption that $nt \geq \tmix$, we have $8(n + 2\imix) \leq  24(n-1)$. 
    Now, taking $x = \zeta n\pi(S)$,  
    \begin{align} 
        \bbP\bigl(\lvert N(E) - n\pi(S)&P_S(E) \rvert > \zeta n\pi(S) \bigr) \label{eq:TediousNet} 
        \leq 2 \exp\Bigl(-\frac{\zeta^2n^2\pi(S)^2/(2\lceil\tmix/t\rceil)}{24n\pi(S)P_S(E) + 20 \zeta n \pi(S)} \Bigr).
    \end{align}
    Simplify \eqref{eq:TediousNet} using that $24P_S(E) + 20\zeta \leq 45$ to conclude the proof.   
\end{proof}

Recall the estimator $\hat{P}_S$ from \eqref{eq: Def_hatP}.  
Then, we have the following concentration inequality: 
\begin{corollary}\label{cor: PSE_concentration}
    There exist absolute constants $c_1,c_2 >0$ such that the following holds for every $t\leq c_1\min\{1/\kappa^2, 1/\lambda_{\max}^2,\rho^2\}$, $T>0$, measurable $E\subseteq \bbR^2$, and $\zeta \leq 1$, 
    \begin{align} 
        \bbP\bigl(\lvert \hat{P}_S(E) - P_S(E) \rvert  > \zeta\bigr) \leq 4\exp\bigl(-c_2\pi(S)T \zeta^2/\tmix \bigr) \label{eq:IvoryGym}
    \end{align}
\end{corollary}
\begin{proof}
    Note that $\pi(S)\zeta^2 \leq 1$. 
    Hence, taking $c_2$ sufficiently small to ensure that the bound is trivial otherwise, we may assume that $T\geq 2\tmix$.
    Further, it is not difficult to deduce from \Cref{lem: BarrierLocallyStraight} that $\tmix \geq c\min\{ 1/\kappa^2, 1/\lambda_{\max}^2 ,\rho^2\}$ for some $c>0$. 
    Hence, taking $c_1$ sufficiently small, we may assume that $t \leq \tmix$.     
    
    The preceding discussion implies that $\lfloor T/t \rfloor \geq T/2t \geq 1$ and $\lfloor T/t \rfloor t \geq \tmix$.   
    Hence, by \Cref{lem: Paulin} with $E = D$ and $n = \lfloor T/t \rfloor$, also using that $P_S(D) = 1$ since $X_t$ is $D$-valued, 
    \begin{align} 
         \bbP\bigl(\lvert \#\{0 \leq i\leq \lfloor T/t \rfloor -1: X_{it}\in S \} - \lfloor T/t \rfloor \pi(S) \rvert > \lfloor {}&{}  T/t \rfloor \pi(S) \zeta/2 \bigr)\label{eq:SilentRock}\\ 
         &\leq 2\exp\bigl(- c' \pi(S) \lfloor T/t \rfloor \zeta^2/ \lceil \tmix / t \rceil  \bigr) \nonumber 
    \end{align}
    for some absolute constant $c'>0$. 
    Hence, using that $\lceil \tmix /t \rceil \leq 2\tmix /t$ since $t \leq \tmix$ and recalling that $\lfloor T/t \rfloor \geq T/2t$, also rewriting the event on the left-hand side of \eqref{eq:SilentRock},     
    \begin{align} 
        \bbP\Bigl(\Bigl\lvert \frac{\#\{ i\leq \lfloor T/t \rfloor -1: X_{it}\in S \} }{ \lfloor T/t \rfloor \pi(S)} - 1 \Bigr\rvert > \frac{\zeta}{2}\Bigr) \leq 2\exp\Bigl(-c'' \pi(S) T \zeta^2 /\tmix \Bigr)\label{eq:ZenNut} 
    \end{align}
    with $c'' \de c'/4$. 
    Similarly, by another application of \Cref{lem: Paulin}, using the definition of $\hat{P}_S$, 
    \begin{align} 
        \bbP\Bigl(\Bigl\lvert  \frac{\#\{  i\leq \lfloor T/t \rfloor -1: X_{it}\in S \} }{ \lfloor T/t \rfloor \pi(S)} \hat{P}_S(E) - P_S(E) \Bigr\rvert > \frac{\zeta}{2}\Bigr) \leq 2\exp\Bigl(-c'' \pi(S) T \zeta^2 /\tmix \Bigr).\label{eq:BlueQuip} 
    \end{align}
    Combine \eqref{eq:ZenNut} and \eqref{eq:BlueQuip} using that $\hat{P}_S(E) \leq 1$ to conclude that \eqref{eq:IvoryGym} holds. 
\end{proof}

Recall from the discussion preceding \Cref{lem: WuBin_approxerror} that $\bin_{\mathfrak{d}}[\hat{P}_S]$ was defined by averaging over the bins $B_{\mathfrak{d}}(x,y)$ of the grid $\mathfrak{d}\bbZ^2$. 
The idea is now to apply \Cref{cor: PSE_concentration} with $E = B_{\mathfrak{d}}(x,y)$.
However, it is too inefficient use a union bound over all (infinitely many) such sets. 
To reduce the combinatorial costs, the following result will be used:
\begin{lemma}\label{lem: grid}
    For every $\eta \in (0,1)$ and $\zeta >0$ there exist constants $c_1, c_2 >0$ depending only on $\eta$ and $\gamma$ such that the following holds for every $t\leq c_1 \min\{ 1/\kappa^2, 1/\lambda_{\max}^2,\rho^2\}$ if $\mathfrak{d}\de \gamma \sqrt{t}$. 
    
    Assume that $\Vert a-b \Vert \leq \sqrt{t}$ for every $a,b\in S$.
    Then, there exists a set of gridpoints $\cG \subseteq \mathfrak{d} \bbZ^2$ with cardinality $\#\cG \leq c_2$ such that 
    \begin{align} 
        P_S\Bigl(\cup_{(x,y)\in \cG} B_{\mathfrak{d}}(x,y) \Bigr) \geq 1 - \eta. \label{eq:TinyPen}
    \end{align} 
\end{lemma}
\begin{proof}
    It follows from \Cref{cor: TauBound} that there exist constants $c_1,C_1>0$ depending only on $\eta$ such that for every $t\leq c_1\min\{1/\kappa^2, 1/\lambda_{\max}^2, \rho^2\}$ and every $x_0\in D$  
    \begin{align} 
        \bbP\bigl(\Vert X_t  - x_0 \Vert\leq C_1 \sqrt{t}\mid X_0 = x_0\bigr) \geq 1 -\eta.\label{eq:CrimsonCrow} 
    \end{align}  
    In view of this, let us define 
    \begin{align} 
        \cG \de \bigl\{(x,y)\in \mathfrak{d}\bbZ^2: B_{\mathfrak{d}}(x,y)\cap \sB(x_0, C_1\sqrt{t}) \neq \emptyset \text{ for some }x_0 \in S\bigr\}. \label{eq:GladQuip}
    \end{align}
    Then, recalling the definition of $P_S$ from \eqref{eq: Def_PS} and using the law of total probability, 
    \begin{align} 
        P_S\Bigl(\cup_{(x,y)\in \cG} B_{\mathfrak{d}}(x,y) \Bigr) &\geq \inf_{x_0 \in S} \bbP\Bigl( X_t \in \cup_{(x,y)\in \cG} B_{\mathfrak{d}}(x,y)\mid X_0 = x_0\Bigr)\label{eq:EmptyWok}\\ 
        & 
        \geq \bbP\bigl( \Vert X_t  - x_0 \Vert\leq  C_1\sqrt{t} \mid X_0 = x_0\bigr). \nonumber   
    \end{align}
    Combine \eqref{eq:CrimsonCrow} and \eqref{eq:EmptyWok} to find \eqref{eq:TinyPen}. 

    It remains to show that $\#\cG \leq c_2$ with $c_2$ only depending on $\eta$ and $\gamma$. 
    Fix some arbitrary $a_0 \in S$. 
    Then, by the assumption that $\Vert x_0 - a_0 \Vert \leq \sqrt{t}$ for every $x_0 \in S$ and $\mathfrak{d} = \gamma \sqrt{t}$, any point $(x,y)\in \bbR^2$ for which $B_{\mathfrak{d}}(x,y)$ intersects $\sB(x_0, C_1 \sqrt{t})$ for some $x_0 \in S$ must satisfy that $\Vert (x,y) - a_0 \Vert \leq (C_1 + \sqrt{2}\gamma + 1)\sqrt{t}$.
    Rescaling by $\sqrt{t}$, it follows that $\#\cG$ is bounded by the number of points in $\gamma \bbZ^2$ contained in a ball of radius $C_1 + \sqrt{2}\gamma + 1$. 
    Since $C_1$ only depends on $\eta$,  the resulting upper bound only depends on $\gamma$ and $\eta$.    
\end{proof}

\begin{proof}[Proof of \texorpdfstring{\Cref{lem: ConcentrationBinned}}{Lemma}]
    Using that total variation distance is half of the $L^1$ distance \cite[Proposition 4.2]{levin2017markov} together with the definition of the $\bin_{\mathfrak{d}}[\cdot]$-operation,    
    \begin{align} 
        \bigl\Vert \bin_{\mathfrak{d}}[\hat{P}_S] - \bin_{\mathfrak{d}}[P_S]  \bigr\Vert_{\TV} = \frac{1}{2}\sum_{x,y \in  \mathfrak{d}\bbZ}\bigl\lvert \hat{P}_S\bigl(B_{\mathfrak{d}}(x,y)\bigr) - P_S\bigl(B_{\mathfrak{d}}(x,y)\bigr) \bigr\rvert.\label{eq:RestlessDad}  
    \end{align}
    We next rewrite the right-hand side of \eqref{eq:RestlessDad} to prepare for an application of \Cref{cor: PSE_concentration}.

    \Cref{lem: grid} provides constants $c_1>0$ and $c_2^{\bc{1}}>0$ depending only on $\eta$ and $\gamma$ such that there exists a set of gridpoints $\cG \subseteq \mathfrak{d}\bbZ^2$ with $\#\cG \leq c_2^{\bc{1}}$ and 
    \begin{align} 
        P_S\bigl(E_{\cG} \bigr) <  \eta/2 \quad \text{ where }\quad E_{\cG} \de \bbR^2\setminus \cup_{(x,y)\in \cG} B_{\mathfrak{d}}(x,y)\label{eq:KindUrn}
    \end{align}
    whenever $t\leq c_1\min\{ 1/\kappa^2, 1/\lambda_{\max}^2,\rho^2  \}$. 
    Now, by \eqref{eq:RestlessDad} together with \eqref{eq:KindUrn} and two application of the triangle inequality, 
    \begin{align} 
        \bigl\Vert \bin_{\mathfrak{d}}[{}&{}\hat{P}_S] - \bin_{\mathfrak{d}}[P_S]  \bigr\Vert_{\TV} \leq \frac{1}{2}\sum_{(x,y) \in \cG}\bigl\lvert \hat{P}_S\bigl(B_{\mathfrak{d}}(x,y)\bigr) - P_S\bigl(B_{\mathfrak{d}}(x,y)\bigr) \bigr\rvert + \frac{1}{2}\hat{P}_S(E_{\cG}) + \frac{1}{2}P_S(E_{\cG})\nonumber\\ 
        &\
        \leq  \frac{1}{2}\sum_{(x,y) \in \cG}\bigl\lvert \hat{P}_S\bigl(B_{\mathfrak{d}}(x,y)\bigr) - P_S\bigl(B_{\mathfrak{d}}(x,y)\bigr) \bigr\rvert +\frac{1}{2}\lvert \hat{P}_S(E_\cG)  - P_S(E_\cG) \rvert+ \frac{\eta}{2}.\label{eq:SlowRoad}   
    \end{align}
    
    Let $E\in \{B_{\mathfrak{d}}(x,y):(x,y) \in \cG \}\cup \{E_\cG \}$ be one of the measurable sets which occurs on the right-hand side of \eqref{eq:SlowRoad}. 
    Then, by \Cref{cor: PSE_concentration} there exists some $c_3>0$ depending only on $\eta$ and $\gamma$ such that,    
    \begin{align}
        \bbP\Bigl(\bigl\lvert \hat{P}_S\bigl(E\bigr) - P_S\bigl(E\bigr) \bigr\rvert > \frac{\eta}{c_2^{\bc{1}} + 1}  \Bigr) \leq 4\exp\Bigl(-c_3 \pi(S) T/ \tmix \Bigr).  
    \end{align}
    Let $c_2 = 4(c_2^{\bc{1}}+1)$ and apply the union bound to conclude that \eqref{eq:IvoryQuote} holds. 

    Finally, \eqref{eq:SlowEel} follows from \eqref{eq:IvoryQuote}. 
    Indeed, note that $d_{\mathfrak{u}}(x,y) \leq \mathfrak{u}$ for every $x,y \in \bbR^2$. 
    This implies that $\Wu(\mu, \nu) \leq \mathfrak{u}\Vert \mu -\nu \Vert_{\TV}$ for all probability measures $\mu,\nu$; see \cite[p.103]{villani2009optimal}\footnote{Note that the definition of total variation distance in \cite{villani2009optimal} differs by a factor two from \eqref{eq: Def_TV_Distance}.}.
\end{proof}

\section{Proof of \texorpdfstring{\Cref{prop: Main_GlobalRecovery_Algorithm}}{Proposition}}\label{apx: ProofGlobalRecovery}
The following preparatory lemma follows by approximating the barrier with straight lines using \Cref{lem: BarrierLocallyStraight}; see \Cref{sec: Proof_BoxPosBoxNeg} for the details.   
\begin{lemma}\label{lem: BoxPosBoxNeg}
    There exists an absolute constant $c>0$ such that the following holds whenever the discretization scale in \Cref{alg: KernelDiscontinuity} satisfies $\epsilon \leq c \min\{ 1/\kappa, 1/\lambda_{\max},\rho \}$.
    
    Fix some $i\leq {\nb}$ and integers $j,k$ with $\cS(j,k) \cap B_i \neq \emptyset$. 
    Then, there exist $-2\leq h_j^+, h_k^+\leq 2$ such that every point in $\cS(j + h_j^+, k + h_k^+)$ is on the positive side of $B_i$. 
    Moreover, one can ensure that $\cS(j + h_j^+, k + h_k^+)$ does not intersect $\cup_{l=0}^{\nb} B_l$.  
 
    Similarly, there exist $-2 \leq h_j^- , h_k^- \leq 2$ such that every point in $\cS(j + h_j^-, k + h_k^-)$ is on the negative side of $B_i$, and such that $\cS(j + h_j^-, k+ h_k^-)$ does not intersect $\cup_{l=0}^{\nb} B_l$.  
\end{lemma}

Further, the following combines \Cref{lem: Concentration_hatPS} with the union bound; see \Cref{apx: ExpLog}.    
\begin{lemma}\label{lem: ConcentrationUnionBoxes}
    For every $\alpha, \beta >0$ there exist constants $c_1,c_2,c_3 >0$ depending only on $\alpha$ and $\beta$ such that the following holds for every $t\leq c_1 \min\{ 1/\kappa^2, 1/\lambda_{\max}^2,\rho^2\}$, $\varepsilon \leq c_2 \sqrt{t}$, $\mathfrak{u}\leq \alpha \sqrt{t}$, and $\eta \in (0,1)$.

    Assume that $X_0 \sim \pi$ starts in stationarity and that the observation time $T$ satisfies \eqref{eq:HighUser} with respect to $c_3$. 
    Then, if the discretization scale in \Cref{alg: KernelDiscontinuity} satisfies $\epsilon = \varepsilon/(3\sqrt{2})$, 
    \begin{align} 
        \bbP\bigl(\Wu(\hat{P}_{\cS(j,k)}, P_{\cS(j,k)}) \leq \beta \sqrt{t}  \text{ for all integers } j,k \text{ with }\cS(j,k)\subseteq D \bigr) \geq 1 - \eta.\label{eq:CalmWok} 
    \end{align}    
\end{lemma}

\Cref{prop: Main_GlobalRecovery_Algorithm} now follows by a direct computation with the results from \Cref{sec: PrelimGlobal}:
\begin{proof}[Proof of \texorpdfstring{\Cref{prop: Main_GlobalRecovery_Algorithm}}{Proposition}]
    To avoid notational ambiguity with other parts of the proof, let us denote $\cC_1,\ldots, \cC_5>0$ for the constants in \Cref{prop: Main_GlobalRecovery_Algorithm}. 
    Then, 
    \begin{align} 
        t &\leq \cC_1\min\{ 1/\kappa^2, 1/\lambda_{\max}^2, \rho^2\},\qquad \varepsilon \leq \cC_2 \sqrt{t},\  \qquad T \geq \cC_3\frac{\tmix }{ \pi_{\min} \varepsilon^2}\ln\Big(\frac{1}{\eta}\frac{\Area(D)}{\varepsilon^2} \Big),\nonumber\\ 
        \mathfrak{s} &= \cC_4 \sqrt{t}, \hphantom{\qquad \min\{ 1/\kappa^2, 1/\lambda^2,\rho^2 \}}  \epsilon = \varepsilon/(3\sqrt{2}), \, \quad \mathfrak{u} = \cC_5 \sqrt{t}. \label{eq:CozyQuote}   
    \end{align} 
    The values of $\cC_4$ and $\cC_5$ will soon be fixed. 
    Subsequently, we show that the result follows if $\cC_1$ and $\cC_2$ are taken sufficiently small, and $\cC_3$ sufficiently large.

    Recall the definition of $P_{\cS(j , k )}$ from \Cref{lem: Concentration_hatPS} and let it be understood that $P_{\cS(j, k)}$ is the zero measure if $\cS(j, k) \cap D = \emptyset$. 
    Then, if $\cC_1$ is sufficiently small and $\cC_5$ is sufficiently large, \Cref{cor: Discontinuity} provides an absolute constant $c>0$ such that 
    \begin{align} 
        \Wu(P_{\cS(j + h_j^+,k + h_k^+)} , P_{\cS(j + h_j^-, k + h_k^-)}) \geq c\sqrt{t} \label{eq:PaleOak}
    \end{align}
    when all points in $\cS(j + h_j^+,k + h_k^+)$ are on the positive side of some barrier $B_i$ and all points of $\cS(j + h_j^-, k + h_k^-)$ are on the negative side of $B_i$.   
    From here on, we fix $\cC_5$ and let $\cC_4 \de c/3$.

    For every $j,k$ with $\cS(j,k)\cap (\cup_{i=0}^{\nb} B_i) = \emptyset$ we either have $\cS(j,k) \subseteq D$ or $\cS(j,k) \subseteq \bbR^2 \setminus D$.
    In the second case, $\hat{P}_{\cS(j, k)}$ and $P_{\cS(j, k)}$ are both equal to the zero measure.  
    In the first case, \Cref{lem: ConcentrationUnionBoxes} is applicable.
    Hence, it holds with probability $\geq 1 - \eta$ that   
    \begin{align} 
        \Wu(\hat{P}_{\cS(j,k)}, P_{\cS(j,k)})  \leq  \frac{c}{100} \sqrt{t} \ \text{ for all integers }\ j,k\ \text{ with }\ \cS(j,k)\cap (\cup_{i=0}^{\nb} B_i) = \emptyset \label{eq:OddJet} 
    \end{align}
    whenever $\cC_1$ and $\cC_2$ are sufficiently small and $\cC_3$ is sufficiently large.  
    From here on, we assume that the event in \eqref{eq:OddJet} occurs. 
    We will prove that the event implied in \eqref{eq: Main_GlobalRecovery_dH} follows.   

    Consider some $j,k$ such that there exists some point in $\cS(j,k)$ at distance $> \varepsilon$ from every point in $\cup_{i=0}^{\nb} B_i$. 
    Pythagoras' theorem implies that $\Vert x_0 - \tix_0 \Vert \leq 3\sqrt{2}\epsilon = \varepsilon$ for every $x_0\in \cS(j,k)$ and $\tix_0 \in \cS(j + h_j, k + h_k)$ with $\lvert h_j \rvert,\lvert h_k \rvert\leq 2$. 
    Hence, the square region $\cup_{-2 \leq h_j, h_k\leq 2}\cS(j+h_j, k + h_k)$ does not intersect $\cup_{i=0}^{\nb} B_i$, implying that all points in this region are on the same side of every barrier. 
    Hence, by \Cref{prop: Continuity} with $\zeta = c/(100(1+\cC_5))$, also using \eqref{eq:OddJet} and the triangle inequality,    
    \begin{align} 
        \max\bigl\{\Wu\bigl(\hat{P}_{\cS(j,k)},\hat{P}_{\cS(j+h_j, k+h_k)}\bigr): \lvert h_j \rvert, \lvert h_k \rvert \leq 2 \bigr\}  \leq 2\frac{c}{100} + \zeta(\sqrt{t} + \mathfrak{u}) \leq \frac{3c}{100}\sqrt{t}\label{eq:SilentFan}      
    \end{align}
    if $\cC_1$ and $\cC_2$ are sufficiently small such that \Cref{prop: Continuity} is applicable. 

    Recalling that we took $\cC_4 = c/3$ and considering \eqref{eq:CozyQuote}, the right-hand side of \eqref{eq:SilentFan} does not exceed the sensitivity threshold $\mathfrak{s}$ in \Cref{alg: KernelDiscontinuity}.
    This shows that no region $\cS(j,k)$ which contains a point at distance $>\varepsilon$ from $\cup_{i=0}^{\nb} B_i$ is added to $\hat{B}$ by the algorithm. 
    Hence, 
    \begin{align} 
        \sup_{a \in \hat{B}}\inf_{b \in \cup_{i=0}^{\nb} B_i}\Vert a - b \Vert \leq \varepsilon. \label{eq:YellowBox}
    \end{align}

    In the other direction, consider $j,k$ with $\cS(j,k)\cap B_i \neq \emptyset$.  
    Then, if $\cC_1$ and $\cC_2$ are sufficiently small, \Cref{lem: BoxPosBoxNeg} allows us to find $\lvert h_j^+  \rvert,\lvert h_k^+ \rvert \leq 2$ with $\cS(j + h_j^+, k+h_k^+ )$ on the positive side of $B_i$ not intersecting $\cup_{l = 0}^{\nb} B_l$. 
    Similarly, we can find $\lvert h_j^-  \rvert,\lvert h_k^- \rvert \leq 2$ with $\cS(j + h_j^-, k+h_k^- )$ on the negative side of $B_i$ not intersecting $\cup_{l = 0}^{\nb} B_l$.
    Then, by \eqref{eq:PaleOak} and \eqref{eq:OddJet},
    \begin{align} 
        \Wu(\hat{P}_{\cS(j+h_j^+,k + h_k^+)},\hat{P}_{\cS(j+h_j^-, k+h_k^-)}) \geq \Bigl( c - \frac{2c}{100} \Bigr) \sqrt{t}.\label{eq:HighBall}        
    \end{align}
    In particular, since the triangle inequality would be violated otherwise,  
    \begin{align} 
        \max\bigl\{\Wu\bigl(\hat{P}_{\cS(j,k)},\hat{P}_{\cS(j+h_j, k+h_k)}\bigr):  \lvert h_j \rvert, \lvert h_k \rvert \leq 2 \Bigr\} \geq \Bigl(\frac{c}{2} - \frac{c}{100}\Bigr)\sqrt{t}
    \end{align}
    This exceeds the sensitivity threshold $\mathfrak{s}$ in \Cref{alg: KernelDiscontinuity}, and hence every $\cS(j,k)$ intersecting $\cup_{i=0}^{\nb} B_i$ is added to $\hat{B}$ by the algorithm. 
    Then, $\cup_{i=0}^{\nb} B_i \subseteq \hat{B}$ and hence 
    \begin{align} 
        \sup_{b \in \cup_{i=0}^{\nb} B_i} \inf_{a \in \hat{B}} \Vert a - b \Vert = 0.\label{eq:ZombieError} 
    \end{align}  
    Combine \eqref{eq:YellowBox} with \eqref{eq:ZombieError} and recall the definition of the Hausdorff metric from \eqref{eq: Def_dH} to conclude the proof. 
\end{proof}

\subsection{Proof of \texorpdfstring{\Cref{lem: BoxPosBoxNeg}}{Lemma}}\label{sec: Proof_BoxPosBoxNeg}
We start with a preparatory result: 
\begin{lemma}\label{lem:UsefulForce}
    There exists an absolute constant $c>0$ such that the following holds.
    Consider a point $x\in B_i$ for some $i\leq {\nb}$ and some $\epsilon \leq c\min\{ 1/\kappa, 1/\lambda_{\max} ,\rho\}$. 

    Then, there exists $\xi^+ \in \bbR^2$ with $\Vert x - \xi^+ \Vert \leq 2\epsilon$ such that every point $z$ with $\Vert z -\xi^+ \Vert \leq \sqrt{2}\epsilon$ is on the positive side of $B_i$.
    Moreover, it may be assumed that any such $z$ is not in $\cup_{l=0}^{\nb} B_l$.  

    Similarly, there exists $\xi^- \in \bbR^2$ with $\Vert x - \xi^- \Vert \leq 2\epsilon$ such that every point $z$ with $\Vert z - \xi^- \Vert\leq \sqrt{2}\epsilon$ is on the negative side of $B_i$, and it may be assumed that any such $z$ is not in $\cup_{l=0}^{\nb} B_l$. 
\end{lemma}
\begin{proof}
    Let $\delta \de 4\epsilon /\min\{ 1/\kappa, 1/\lambda_{\max},\rho \}$ and let $\hatn$ be the vector which arises from \Cref{lem: BarrierLocallyStraight} applied to $\sB(x, r_\delta)$. 
    We claim that the following points have the desired properties:  
    \begin{align} 
        \xi^+ \de x + 2\epsilon \hatn\qquad \text{ and }\qquad \xi^- \de x - 2\epsilon \hatn.  
    \end{align}   
    Let us verify this for $\xi^+$. 
    The proof for $\xi^-$ proceeds similarly. 

    That $\Vert \xi^+ - x \Vert \leq 2\epsilon$ is immediate from the fact that $\hatn$ is a unit vector, so it remains to show that every $z$ with $\Vert z - \xi^+ \Vert \leq \sqrt{2}\epsilon$ is on the positive side of $B_i$ and not in $\cup_{l=0}^{\nb} B_l$. 
    By the triangle inequality, every such $z$ satisfies 
    \begin{align} 
        \Vert z - x \Vert \leq (2 + \sqrt{2})\epsilon < 4\epsilon  = r_\delta 
    \end{align} 
    which is to say that $z\in \sB(x, r_\delta)$.
    
    Let $\mathfrak{c}$ be as in \Cref{lem: BarrierLocallyStraight}.
    We claim that it suffices to show that $\langle z , \hatn \rangle \geq \mathfrak{c} + 4\delta r_\delta$. 
    Indeed, \cref{item:GhostlyQuip} from \Cref{lem: BarrierLocallyStraight} then implies that $z\not\in B_i \cap \sB(x, r_\delta)$ which implies that $z\not\in \cup_{l=0}^{\nb} B_l$ since $z\in \sB(x, r_\delta)$. 
    Further, \cref{item:JollyDragon} from \Cref{lem: BarrierLocallyStraight} then implies that $z$ is not on the negative side of $B_i$, which implies that $z$ is on the positive side.

    Using bilinearity of the inner product as well as \cref{item:GhostlyQuip} from \Cref{lem: BarrierLocallyStraight}, 
    \begin{align} 
        \langle z, \hatn \rangle \geq \langle \xi^+, \hatn \rangle - \sqrt{2}\epsilon = \langle x, \hatn \rangle + (2-\sqrt{2})\epsilon \geq \mathfrak{c} +(2 - \sqrt{2})\epsilon  - 4\delta r_\delta . 
    \end{align}
    The assumed upper bound on $\epsilon$ and the definition of $\delta$ imply that $\delta \leq 4c$ and $r_\delta = 4\epsilon$. 
    Hence, it can be ensured that  
    $(2 - \sqrt{2})\epsilon  - 4\delta r_\delta \geq 4\delta r_\delta$ by taking $c$ to be sufficiently small. 
    This yields the desired sufficient condition and consequently concludes the proof.  
\end{proof}

\begin{proof}[Proof of \texorpdfstring{\Cref{lem: BoxPosBoxNeg}}{Lemma}]
    Let $c$ be as in \Cref{lem:UsefulForce}.   
    We prove the existence of $-2\leq h_j^+, h_k^+\leq 2$ with $\cS(j+h_j^+, k + h_k^+)$ on the positive side of $B_i$; the existence of a box on the negative side proceeds similarly.  
    
    By assumption, there exists some $x\in \cS(j, k)$ with $x \in B_i$. 
    Let $\xi^+$ be as in \Cref{lem:UsefulForce}. 
    Then, in particular, $\Vert x - \xi^+ \Vert \leq 2\epsilon$. 
    This implies that there exist $-2 \leq h_j^+, h_k^+ \leq 2$, not necessarily unique if $\xi^+$ is on the boundary between two squares, with $\xi^+ \in \cS(j+ h_j^+, k+h_k^+)$.  

    Recall that $\cS(j+h_j^+, k+h_k^+)$ is an $\epsilon\times \epsilon$ box. 
    Hence, every $z\in \cS(j+h_j^+, k+h_k^+)$ satisfies $\Vert \xi^+ - z \Vert \leq \sqrt{2}\epsilon$. 
    The desired properties that every $z\in \cS(j+h_j^+, k+h_k^+)$ is on the positive side and not on any barrier now follow from the conclusion of \Cref{lem:UsefulForce}.  
\end{proof}

\subsection{Proof of \texorpdfstring{\Cref{lem: ConcentrationUnionBoxes}}{Lemma}}\label{apx: ExpLog}

\begin{proof}[Proof of \texorpdfstring{\Cref{lem: ConcentrationUnionBoxes}}{Lemma}]
    Each of the regions $\cS(j,k)$ is an $\epsilon \times \epsilon$ square and consequently satisfies $\Area(\cS(j,k)) = \epsilon^2$. 
    It follows that the number of integers $j,k$ with $\cS(j,k)\subseteq D$ is at most $\Area(D)/\epsilon^2$. 
    Hence, combining \Cref{lem: Concentration_hatPS} with the union bound,    
    \begin{align} 
        \bbP\bigl(\Wu(\hat{P}_{\cS(j,k)}, P_{\cS(j,k)}) \leq \beta \sqrt{t} {}&{} \text{ for all integers } j,k \text{ with }\cS(j,k)\subseteq D \bigr)\label{eq:CalmTea} \\ 
        &\quad \geq 1 - C \bigl(\Area(D)/\epsilon^2\bigr)\exp\bigl(- c \pi_{\min} \epsilon^2 T/ \tmix \bigr)\nonumber 
    \end{align}
    with $c,C>0$ constants depending only on $\alpha$ and $\beta$. 
    We next use the assumed lower bound on $T$ from \eqref{eq:HighUser} to estimate the right-hand side of \eqref{eq:CalmTea}. 
    
    Recall from the discussion preceding \eqref{eq:ValidVolt} that $\Area(D) \geq \pi/\kappa^2$ due to the assumption that the curvature of $B_0$ is bounded by $\kappa$.    
    In particular, taking $c_1$ and $c_2$ sufficiently small, we may assume that $\Area(D) \geq 2\varepsilon^2 \geq 2 \varepsilon^2 \eta$. 
    Then, using that $\epsilon = \varepsilon / (3\sqrt{2})$, 
    \begin{align} 
        \ln\Bigl(\frac{C\Area(D)}{\epsilon^2 \eta} \Bigr) = \ln\bigl(C (3\sqrt{2})^2 \bigr) +  \ln\Bigl(\frac{\Area(D)}{\varepsilon^2 \eta} \Bigr) \leq C' \ln\Bigl(\frac{\Area(D)}{\varepsilon^2 \eta} \Bigr)
    \end{align}  
    for some constant $C' >0$ which should be sufficiently large to satisfy $\ln(C(3\sqrt{2})^2)\leq (C'-1) \ln(2)$. 
    Hence, for $c_3$ sufficiently large, 
    \begin{align} 
        c_3\frac{1}{ \pi_{\min} \varepsilon^2 }\ln\Bigl(\frac{\Area(D)}{\varepsilon^2 \eta } \Bigr) \geq  \frac{1}{ c\pi_{\min} \epsilon^2}\ln\Bigl(C\frac{\Area(D)}{\epsilon^2 \eta } \Bigr). \label{eq:BraveKnot}
    \end{align}
    Combine \eqref{eq:CalmTea} and \eqref{eq:BraveKnot} with the assumption \eqref{eq:HighUser} to find \eqref{eq:CalmWok}.
\end{proof}

\section{Proof of \texorpdfstring{\Cref{prop: Main_GlobalRecovery_Improvement}}{Proposition}}\label{apx: Proof_Main_GlobalRecovery_Improvement}
The following two preparatory lemmas state the direction of the barriers is estimated sufficiently well by \Cref{alg: Improve} to ensure that the rectangular regions $\{\cR(j,k,{\nn},h): \lvert h \rvert \leq 2 \}$ are all on a fixed side of the barriers if and only if $p_{j,k}$ is distant from all barriers: 
\begin{lemma}\label{lem: Rectangle_Close}
    There exists an absolute constant $c_1>0$ such that the following holds. 
    Assume that the parameters in \Cref{alg: Improve} satisfy the following for some $\sE \geq 0$:
    \begin{align} 
        \ell \leq c_1 \min\{ 1/\kappa, 1/\lambda_{\max}, \rho\}, \  \epsilon = \kappa\ell^2, \  \text{ and }\ \dH(\fB, \cup_{i=0}^{\nb} B_i)\leq \sE. \label{eq:GladEgg}
    \end{align} 

    Then, for every $y \in   B_i$ there exists some $0 \leq {\nn} \leq \lfloor 2\pi \ell / \epsilon\rfloor$ as well as integers $j,k$ such that the following properties are satisfied: 
    \begin{enumerate}[leftmargin=2.5em]
        \item The point $p_{j,k}$ has distance $\leq \epsilon/\sqrt{2}$ from $y$. 
        Further, the points $p_{j,k} + \ell \vec{w}_{\nn}$ and $p_{j,k} - \ell \vec{w}_{\nn}$ both have distance $\leq \sE + 2\epsilon$ from $\fB$. 
        \item There exists $-2 \leq h_+ \leq 2$ such that every point in $\cR(j,k,{\nn},h_+)$ is on the positive side of $B_i$, and such that $\cR(j,k,{\nn},h_+)$ does not intersect $\cup_{l=0}^{\nb} B_l$. 
        \item Similarly, there exists $-2 \leq h_-\leq 2$ such that every point in $\cR(j,k,{\nn},h_-)$ is on the negative side of $B_i$, and such that $\cR(j,k,{\nn},h_-)$ does not intersect $\cup_{l=0}^{\nb} B_l$.    
    \end{enumerate} 
\end{lemma}

\begin{lemma}\label{lem: Rectangle_Distant}
    There exist absolute constants $c_1,c_2,c_3>0$ such that the following holds. 
    Assume that the parameters in \Cref{alg: Improve} satisfy \eqref{eq:GladEgg} with respect to $c_1$ and assume that $c_2\epsilon \leq \sE \leq c_3\ell$. 
    Further, consider integers $j,k$ with $\inf\{\Vert p_{j,k} - y \Vert: y \in \cup_{i=0}^{\nb} B_i \} >  \sE/2$. 
    
    Then, for every $0 \leq {\nn} \leq \lfloor 2\pi \ell / \epsilon\rfloor$ with $p_{j,k} + \ell \vec{w}_{\nn}$ and $p_{j,k} - \ell \vec{w}_{\nn}$ both at distance $\leq \sE + 2\epsilon$ from $\fB$ it holds that the rectangular region $\cup_{h=-2}^2\cR(j,k,{\nn},h)$  does not intersect $\cup_{i=0}^{\nb} B_i$. 
\end{lemma}

Proofs are given in \Cref{apx: ProofRectangleClose,apx: ProofRectangleDistant}. 
The following estimate will play a similar role in the proof of \Cref{prop: Main_GlobalRecovery_Improvement} as \Cref{lem: ConcentrationUnionBoxes} did in the proof of \Cref{prop: Main_GlobalRecovery_Algorithm}: 
\begin{lemma}\label{lem: Rectangle_Concentration}
    For every $\alpha, \beta,\gamma >0$ there exist constants $c_1,c_2, c_3 >0$ depending only on $\alpha, \beta$ and $\gamma$ such that the following holds for every $t\leq c_1 \min\{ 1/\kappa^2, 1/\lambda_{\max}^2, \rho^2 \}$, $\varepsilon \leq c_2 \kappa t$, $\mathfrak{u} \leq \alpha\sqrt{t}$, and $\eta \in (0,1)$. 

    Assume that $X_0 \sim \pi$ and that \eqref{eq: Main_GlobalRecovery_n} holds with respect to $c_3$. 
    Then, if the parameters in \Cref{alg: Improve} satisfy $\ell = \gamma \sqrt{\varepsilon/ \kappa}$ and $\epsilon = \kappa \ell^2$, 
    \begin{align} 
        \bbP\bigl(\Wu(\hat{P}_{\cR(j,k,{\nn},h)}, P_{\cR(j,k,{\nn},h)}) \leq \beta \sqrt{t},\ \forall  j,k,{\nn},h \text{ with }\cR(j,k,{\nn},h) \subseteq D  \bigr) \geq 1 - \eta. 
    \end{align}  
\end{lemma}

The proof is given in \Cref{apx: ProofRectangleConcentration}. 
\Cref{prop: Main_GlobalRecovery_Improvement} now follows by combining these preliminaries with the results from \Cref{sec: PrelimGlobal}.
This proceeds almost exactly like the proof of \Cref{prop: Main_GlobalRecovery_Algorithm}, but let us give the details for completeness: 
\begin{proof}[Proof of \texorpdfstring{\Cref{prop: Main_GlobalRecovery_Improvement}}{Proposition}]
    To avoid ambiguity with other results, let us denote $\cC_1,\ldots,\cC_7 >0$ for the constants in the statement of \Cref{prop: Main_GlobalRecovery_Improvement}. 
    Then, the assumption is that 
    \begin{align} 
        t \leq \cC_1 {}&{}\min\{ 1/\kappa^2, 1/\lambda_{\max}^2, \rho^2\},\quad  \varepsilon \leq \cC_2 \kappa t, \quad T \geq \cC_3 \frac{\tmix}{\pi_{\min}}\sqrt{\frac{\kappa}{\varepsilon^3}}\ln\Bigl(\frac{\Area(D)}{\eta} \sqrt{\frac{\kappa}{\varepsilon^3}}\Bigr), \nonumber \\ 
        &\mathfrak{s} = \cC_4 \sqrt{t},\quad  \ell = \cC_5 \sqrt{\varepsilon/\kappa},\quad  \epsilon = \kappa \ell^2, \quad \mathfrak{u} = \cC_6 \sqrt{t}, \ \text{ and }\  \varepsilon \leq \sE  \leq \cC_7 \ell.   \label{eq:PalePig} 
    \end{align}
    The values of $\cC_4, \cC_5,$ and $\cC_6$ will soon be fixed. 
    Subsequently, we show that the result follows if $\cC_1,\cC_2$, and $\cC_7$ are taken sufficiently small, and $\cC_3$ is taken sufficiently large. 
    
    Recall from \Cref{lem: Concentration_hatPS} that $P_S = \bbP(X_t\in \cdot \mid X_0 \in S)$ for $S \subseteq D$ and $X_0 \sim \pi$. 
    \Cref{cor: Discontinuity} hence ensures that there is an absolute constant $c>0$ such that, when $\cC_1$ is sufficiently small and $\cC_6$ is sufficiently large, 
    \begin{align} 
        \Wu(P_{\cR(j,k,{\nn},h^+)}, P_{\cR(j,k,{\nn},h^-)}) \geq c \sqrt{t}\label{eq:OilySquid}  
    \end{align}
    when all points of $\cR(j,k,{\nn},h^+)$ are on the positive side of some barrier $B_i$, and all points of $\cR(j,k,{\nn},h^-)$ are on the negative side of $B_i$. 
    From here on, fix $\cC_6$ and let $\cC_4 \de c/3$. 
    
    Let us take $\cC_1$ and $\cC_2$ sufficiently small so that they are bounded by an absolute constant, say $\cC_1, \cC_2 \leq 1$. 
    Then, combining the inequalities in \eqref{eq:PalePig} we have that $\ell \leq \cC_5 \min\{ 1/\kappa, 1/\lambda_{\max},\rho \}$. 
    Further, note that $\sE \geq \varepsilon = \epsilon/\cC_5^2$.   
    Let us fix $\cC_5$ at a value which is $\leq 1/\sqrt{2}$ and sufficiently small to ensure that the upper bounds on $\ell$ in \Cref{lem: Rectangle_Close} and \Cref{lem: Rectangle_Distant} are satisfied, and that the lower bound on $\sE$ in \Cref{lem: Rectangle_Distant} is satisfied.    

    Let $\cE$ be the event where concentration occurs: 
    \begin{align} 
        \cE \de \{\omega: \Wu(\hat{P}_{\cR(j,k,{\nn},h)}, P_{\cR(j,k,{\nn},h)}) \leq \frac{c}{100} \sqrt{t},\,  \forall j,k,{\nn},h \text{ with }\cR(j,k,{\nn},h) \subseteq D\}. \label{eq:HauntedSquid}
    \end{align}
    Then, considering that $\hat{P}_S$ and $P_S$ are both equal to the zero measure when $S \subseteq \bbR^2\setminus D$, 
    \begin{align} 
        \Wu(\hat{P}_{\cR(j,k,{\nn},h)}, P_{\cR(j,k,{\nn},h)}) \leq \frac{c}{100} \sqrt{t},\,  \forall j,k,{\nn},h \text{ with }\cR(j,k,{\nn},h)\cap (\cup_{i=0}^{\nb} B_i) = \emptyset\label{eq:MagicKite}  
    \end{align}
    whenever $\cE$ occurs.   

    \Cref{lem: Rectangle_Concentration} implies that $\bbP(\cE)\geq 1 - \eta$ if the constants $\cC_1$ and $\cC_2$ are taken sufficiently small, and the constant $\cC_3$ is taken sufficiently large. 
    From here on out, we work on the event where $\cE$ occurs. 
    It remains to show that the output of \Cref{alg: Improve} then satisfies $\dH(\cB', \cup_{i=0}^{\nb} B_i) \leq \sE/2$ for every $\sE>0$ and $\fB \subseteq \bbR^2$ with $\sE$ subject to the constraints in \eqref{eq:PalePig} and $\fB$ satisfying $\dH(\fB, \cup_{i=0}^{\nb} B_i) \leq \sE$.

    Consider some $p_{j,k}$ at distance $> \sE/2$ from $\cup_{i=0}^{\nb} B_i$. 
    Then, taking $\cC_7$ sufficiently small and recalling the previous discussion regarding the choice of $\cC_5$ ensures that the conditions of \Cref{lem: Rectangle_Distant} are satisfied. 
    Hence, $\cup_{h=-2}^2 \cR(j,k,{\nn},h)$ does not intersect $\cup_{i=0}^{\nb} B_i$ for every $0\leq {\nn} \leq \lfloor 2\pi \ell /\epsilon \rfloor$ with $p_{j,k} + \ell\vec{w}_{\nn}$ and $p_{j,k} - \ell\vec{w}_{\nn}$ sufficiently close to $\fB$ to satisfy the constraint in \Cref{alg: Improve}. 
    In particular, the rectangular sets $\cR(j,k,{\nn},h)$ for varying $\lvert h \rvert \leq 2$ are then all on the same side of every barrier. 
    Also note that the diameter of $\cup_{h=-2}^2\cR(j,k,{\nn},h)$ can be made an arbitrarily small multiple of $\sqrt{t}$ by taking $\cC_1$ and $\cC_2$ sufficiently small in \eqref{eq:PalePig}.       
    
    The preceding implies that the conditions of \Cref{prop: Continuity} with $\zeta = (c/100)/(\cC_6 + 1)$ can be satisfied by taking $\cC_1,\cC_2$ and $\cC_7$ sufficiently small.
    Combining the conclusion \eqref{eq:LoudYarn} with \eqref{eq:MagicKite} and the triangle inequality, 
    \begin{align} 
        \max\bigl\{\Wu(\hat{P}_{\cR(j,k,{\nn},0)}), \Wu(\hat{P}_{\cR(j,k,{\nn},h)}: \lvert h \rvert \leq 2  \bigr\} \leq \frac{3c}{100}\sqrt{t}. \label{eq:TallCat}
    \end{align}  
    Considering the definition of the sensitivity threshold $\mathfrak{s}$ in \eqref{eq:PalePig} and recalling that we took $\cC_4 = c/3$, the right-hand side of \eqref{eq:TallCat} does not exceed the sensitivity threshold in \Cref{alg: Improve}. 
    This shows that no point $p_{j,k}$ at distance $>\sE/2$ from $\cup_{i=0}^{\nb} B_i$ is added to $\fB'$ by the algorithm.
    Hence, 
    \begin{align} 
        \sup_{p_{j,k} \in \fB'} \inf_{y \in \cup_{i=0}^{\nb} B_i} \Vert p_{j,k} - y \Vert \leq \sE /2. \label{eq:RedOwl}
    \end{align} 
    
    In the other direction, consider some arbitrary point $y\in B_i$. 
    Then, recalling that $\cC_5$ was chosen such that $\ell$ satisfies the constraint in \eqref{eq:GladEgg}, \Cref{lem: Rectangle_Close} ensures that there will be some $p_{j,k} \in \fB'$ with $\Vert y - p_{j,k} \Vert \leq \epsilon/\sqrt{2}$ and some $0 \leq {\nn}\leq \lfloor 2\pi \ell/\epsilon \rfloor$ such that $p_{j,k} + \ell \vec{w}_{\nn}$ and $p_{j,k} - \ell\vec{w}_{\nn}$ are sufficiently close to $\fB$ to satisfy the constraint in \Cref{alg: Improve}, and such that there exist $\lvert h_+ \rvert, \lvert h_- \rvert \leq 2$ with $\cR(j,k,{\nn},h_+)$ on the positive side of some barrier $B_i$, and $\cR(j,k,{\nn},h_-)$ on the negative side. 
    Then, combining \eqref{eq:OilySquid} and \eqref{eq:MagicKite}, 
    \begin{align} 
        \Wu(\hat{P}_{\cR(j,k,{\nn},h_+)}, \hat{P}_{\cR(j,k,{\nn},h_-)}) \geq \Bigl(c - \frac{2c}{100} \Bigr)\sqrt{t}. 
    \end{align}  
    In particular, since the triangle inequality would be violated otherwise, 
    \begin{align} 
        \max\bigl\{\Wu(\hat{P}_{\cR(j,k,{\nn},0)},\hat{P}_{\cR(j,k,{\nn},h)}): \lvert h \rvert \leq 2  \bigr\} \geq \Bigl(\frac{c}{2} - \frac{c}{100} \Bigr)\sqrt{t}. 
    \end{align}
    This exceeds the sensitivity threshold implying that $p_{j,k}$ will be added to $\mathfrak{B}'$ by \Cref{alg: Improve}. 
    
    Recalling that we took $\cC_5 \leq 1/\sqrt{2}$, it follows from \eqref{eq:PalePig} that $\epsilon \leq \varepsilon /2$. 
    In particular, since $\Vert p_{j,k} - y \Vert\leq \epsilon/\sqrt{2} < \epsilon$, this implies that $\Vert p_{j,k} - y \Vert \leq \varepsilon /2$.
    Recall that $\varepsilon \leq \sE$ from \eqref{eq:PalePig} to conclude that 
    \begin{align} 
        \sup_{y \in \cup_{i=0}^{\nb} B_i} \inf_{p_{j,k} \in \fB'} \Vert y - p_{j,k} \Vert \leq \sE/2.  \label{eq:ShyZero}
    \end{align}
    Combine \eqref{eq:RedOwl} with \eqref{eq:ShyZero} and recall the definition of the Hausdorff metric from \eqref{eq: Def_dH} to conclude the proof. 
\end{proof}

\subsection{Preliminary: a variant of \texorpdfstring{\Cref{lem: BarrierLocallyStraight}}{Lemma}}\label{apx: Variant}
Recall from \Cref{rem: LocalStraight} that \Cref{lem: BarrierLocallyStraight} admits a stronger variant which was inconvenient in the main text for notational reasons. 
The variant will be convenient in the proofs of \Cref{lem: Rectangle_Close,lem: Rectangle_Distant}.
Let us give a precise statement for easy reference. 
Denote $r_{\delta}' \de \min\{ \delta /\kappa, \rho \}.$ 
\begin{lemma}\label{lem: VariantLocalStraight}
    For every $x_0 \in \bbR^2$ and $\delta >0$, there is at most one barrier $B_i$ which intersects $\sB(x_0, r_{\delta}')$.
    Moreover, if such a $B_i$ exists and $\delta <1/2$, then there exists a unit vector $\hat{n} \in \bbR^2$ depending on $x_0$ and $B_i$ such that the following properties hold: 
    \begin{enumerate}[leftmargin=2.5em,label = (\alph*)]
        \item\label{item:r'1} There exists some $\mathfrak{c}\in \bbR$ such that $\lvert \langle  y, \hatn\rangle - \mathfrak{c} \rvert <  4\delta r_\delta'$ for every $y\in B_i\cap \sB(x_0, r_\delta')$.    
        \item\label{item:r'2} One has $\Vert \vec{n}_i(y)-  \hatn \Vert < 2\delta$ for every $y\in B_i \cap \sB(x_0, r_\delta')$.
        \item\label{item:r'3} Every point $z \in \sB(x_0, r_\delta')$ on the positive side of $B_i$ satisfies $ \langle  z, \hatn\rangle > \mathfrak{c} -  4\delta r_\delta'$. 
        Similarly, every $z\in \sB(x_0, r_\delta)$ the negative side of $B_i$ satisfies $ \langle  z, \hatn\rangle < \mathfrak{c}  +  4\delta r_\delta'$.    
    \end{enumerate}    
\end{lemma}
The proof is identical to that of \Cref{lem: BarrierLocallyStraight}; see \Cref{apx: ProofBarrierLocallyStraight}.  

\subsection{Proof of \texorpdfstring{\Cref{lem: Rectangle_Close}}{Lemma}}\label{apx: ProofRectangleClose}
A direct computation shows that the properties claimed in \Cref{lem: Rectangle_Close} hold if $p_{j,k}$ and $\vec{w}_{\nn}$ are chosen appropriately:
\begin{lemma}\label{lem: w_tangent_items}
    There exists an absolute constant $c_1>0$ such that the following is satisfied. 
    Assume that \eqref{eq:GladEgg} holds with respect to $c_1$. 
    Then, for every $y\in B_i$ there exists a unit vector $\vec{W}\in \bbR^2$ such that the following properties are satisfied for every unit vector $\vec{w} \in \bbR^2$ with $\Vert \vec{w} - \vec{W} \Vert \leq \epsilon/2 \ell$ and every point $p\in \bbR^2$ with $\Vert p - y \Vert \leq \epsilon/\sqrt{2}$: 
    \begin{enumerate}[leftmargin=2.5em,label = (\arabic*)]
        \item\label{item: w1} The points $p + \ell \vec{w}$ and $p - \ell \vec{w}$ have distance $\leq \sE + 2\epsilon$ from $\fB$. 
        \item\label{item: w2} Let $\vec{w}^{\perp}$ be a unit vector orthogonal to $\vec{w}$, and for every $-2 \leq h\leq 2$ define  
        \begin{align} 
            \fR(h) \de \{x \in \bbR^2: \lvert \langle x - p, \vec{w}^{\perp} \rangle - h\epsilon \rvert \leq \epsilon/2, \lvert \langle x - p, \vec{w} \rangle \rvert \leq \ell/10   \}.   
        \end{align} 
        Then, there is some $-2 \leq h_+ \leq 2$ such that every point in $\fR(h_+)$ is on the positive side of $B_i$, and such that $\fR(h_+)$ does not intersect $\cup_{l=0}^{\nb} B_l$. 
        \item\label{item: w3} Similarly, there is some $-2 \leq h_- \leq 2$ such that every point in $\fR(h)$ is on the negative side of $B_i$, and such that $\fR(h_-)$ does not intersect $\cup_{l=0}^{\nb} B_l$.  
    \end{enumerate}       
\end{lemma}
\begin{proof} 
    Recall the assumption that $\dH(\fB, \cup_{i=0}^{\nb} B_i)\leq \sE$ from \eqref{eq:GladEgg}. 
    \Cref{item: w1} hence follows if we prove that $p + \ell \vec{w}$ and $p - \ell \vec{w}$ have distance $\leq 2\epsilon$ from $B_i$. 
    Let us show this for $p+\ell \vec{w}$, the proof for $p-\ell \vec{w}$ proceeds similarly.

    We proceed as in the proof of \Cref{lem: BarrierLocallyStraight}.
    Pick an arc-length parametrization $\phi: \bbR \to B_i$ with $\phi(0) = y$ and let $\vec{W}\de \dd{t}\phi(0)$.
    In particular, it then holds that $\phi(\ell) \in B_i$, so it suffices to show that $\Vert (p+\ell \vec{w}) - \phi(\ell) \Vert\leq 2\epsilon$. 
    By the triangle inequality and the assumed upper bounds on $\Vert p - y \Vert$ and $\Vert \vec{w} - \vec{W}  \Vert$, 
    \begin{align} 
        \Vert (p+\ell \vecw) - \phi(\ell) \Vert &\leq \Vert p - \phi(0) \Vert +   \Vert \ell\vec{w} - \ell \vec{W} \Vert + \Vert \phi(\ell) - \phi(0) - \ell\vec{W} \Vert\label{eq:SafeCell}\\ 
        &\leq (1/\sqrt{2} + 1/2)\epsilon + \Vert \phi(\ell) - \phi(0) - \ell\vec{W} \Vert.\nonumber     
    \end{align}  
    Recall from the discussion preceding \eqref{eq:MythicalGem} that $\Vert \dd{t}\phi(t) - \dd{t}\phi(0) \Vert\leq \kappa t$ by the first Frenet--Serret formula.   
    Hence, also using that $\vec{W} = \dd{t}\phi(0)$ and that $\epsilon = \kappa\ell^2$ by \eqref{eq:GladEgg}, 
    \begin{align} 
        \Vert \phi(\ell) - \phi(0) - \ell\vec{W} \Vert \leq \int_0^\ell \Vert \dd{t}\phi(t) - \dd{t}\phi(0)  \Vert\,  \intd t \leq \kappa \ell^2/2 = \epsilon/2.\label{eq:UsefulRoad} 
    \end{align}
    Use that $1/\sqrt{2} + 1/2 + 1/2 < 2$ to conclude that \cref{item: w1} holds true.

    We next prove \cref{item: w2}. 
    The idea is to rely on \Cref{lem: VariantLocalStraight}  applied to a ball with a radius slightly greater than $\ell/10$, say a ball of $\ell/9$. 
    Let us start with a preparatory estimate to ensure that $\fR(h) \subseteq \sB(y,\ell/9)$.  
    Fix some arbitrary $z \in \fR(h)$ and $\lvert h \rvert \leq 2$. 
    Then, using the definition of $\fR(h)$ with the assumed upper bound on $\Vert p-y \Vert$ and the triangle inequality,   
    \begin{align} 
        \Vert z - y \Vert \leq \Vert p - y  \Vert + \Vert z - p \Vert \leq \epsilon/\sqrt{2} + (\lvert h \rvert +1/2)\epsilon + \ell/10 \leq   C\epsilon + \ell/10\label{eq:LoudTea}
    \end{align}
    with $C>0$ an absolute constant.
    Recall that $\epsilon = \kappa\ell^2$. 
    Hence, taking $c_1$ in \eqref{eq:GladEgg} to be sufficiently small so that $C \kappa\ell < 1/9 - 1/10$, \eqref{eq:LoudTea} yields $\fR(h) \subseteq \sB(y, \ell/9)$, as desired.

    Set $\delta \de \kappa\ell/9$. 
    Then, taking $c_1$ in \eqref{eq:GladEgg} sufficiently small ensures that $r_{\delta}' = \min\{ \delta/\kappa, \rho \} = \ell/9$ and $\delta < 1/2$.     
    Combining $\fR(h)\subseteq \sB(y, \ell/9)$ with \Cref{lem: VariantLocalStraight}, we certainly have $\fR(h) \cap B_j = \emptyset$ for every $j\neq i$. 
    \Cref{item: w2} hence follows if we ensure that one can pick $h_+$ such that $\fR(h_+)\cap B_i = \emptyset$ and that every $z\in\fR(h_+)$ is on the positive side.
    To show this, we exploit \cref{item:r'3} from \Cref{lem: VariantLocalStraight}.

    Assume that $\langle \vec{w}^{\perp}, \vec{n}_i(y) \rangle\geq 0$; the other case follows similarly except that $h_+$ should be replaced by $-h_+$.
    We claim that $h_+ = 2$ then does the job.   
    Pick some arbitrary $z\in \fR(2)$ and let $\hat{n}$ and $\mathfrak{c}$ be as in \Cref{lem: VariantLocalStraight} applied to $\sB(y, r_{\delta}')$. 
    Then, by bilinearity,    
    \begin{align} 
        \langle z, \hatn  \rangle &= \langle z -y, \vec{w}^{\perp}  \rangle + \langle z -y, \hatn - \vec{w}^{\perp}  \rangle  + \langle  y, \hatn  \rangle.   \label{eq:IcyLid} 
    \end{align}       
    Here, by the assumption that $\Vert p - y \Vert \leq \epsilon/\sqrt{2}$ and the definition of $\mathfrak{R}(2)$,
    \begin{align} 
        \langle z -y, \vec{w}^{\perp}  \rangle \geq \langle z -p, \vec{w}^{\perp}  \rangle - \epsilon/\sqrt{2} \geq (2 - 1/\sqrt{2} - 1/2)\epsilon.   
    \end{align}   
    Note that $\vec{W} = \dd{t}\phi(0)$ is the tangent vector to $B_i$ at $y$ and recall that $\langle \vec{w}^{\perp}, \vec{n}_i(y) \rangle\geq 0$.   
    The assumption that $\Vert \vec{w} - \vec{W} \Vert \leq \epsilon/(2\ell)$ then implies that also the normal vectors satisfy $\Vert \vec{w}^{\perp} -  \vec{n}_i(y) \Vert \leq \epsilon/(2\ell)$. 
    Hence, since $\Vert \vec{n}_i(y) - \hat{n} \Vert < 2 \delta$ by \cref{item:r'2} from \Cref{lem: VariantLocalStraight}, we have $\Vert \vec{w}^{\perp} - \hat{n}  \Vert \leq \epsilon/(2 \ell) + 2\delta$.  
    Further, note that $\fR(h) \subseteq \sB(y,r_{\delta}')$ implies $\Vert z-y \Vert \leq r_{\delta}' = \ell/9$. 
    Combining these estimates with the Cauchy--Schwarz inequality,   
    \begin{align} 
        \lvert \langle z -y, \hatn - \vec{w}^{\perp}  \rangle \rvert \leq (\epsilon/(2 \ell) + 2\delta) r_{\delta}' = (1/18 + 2/81)\epsilon  
    \end{align}
    where the equality used that $r_{\delta}' = \ell/9$, $\delta = \kappa \ell/9$, and $\kappa \ell^2 = \epsilon$. 
    As for the final term in \eqref{eq:IcyLid}, using \cref{item:r'1} from \Cref{lem: VariantLocalStraight} on $y\in \sB(y,r_{\delta}')\cap B_i$,
    \begin{align} 
        \langle y, \hat{n} \rangle \geq \mathfrak{c} - 4\delta r_{\delta}' = \mathfrak{c} - (4/81)\epsilon.\label{eq:AlertFox}  
    \end{align}   
    Combine \eqref{eq:IcyLid}--\eqref{eq:AlertFox} to find that 
    \begin{align} 
        \langle z, \hat{n} \rangle \geq \mathfrak{c} + (2 - 1/\sqrt{2} - 1/2 - 1/18 - 2/81 - 4/81)\epsilon > \mathfrak{c} + \epsilon/2 > \mathfrak{c} + 4\delta r_{\delta}'.   
    \end{align} 
    Hence, due to \cref{item:r'3} from \Cref{lem: VariantLocalStraight}, we have that $z$ is not on the negative side of $B_i$. 
    Equivalently, we have $z \not\in B_i$ and $z$ is on the positive side of $B_i$. 
    Considering that $z$ was an arbitrary element of $\fR(h_+)$, this proves \cref{item: w2}. 
    The same arguments can be used to establish \cref{item: w3}.  
\end{proof}
\begin{proof}[Proof of \texorpdfstring{\Cref{lem: Rectangle_Close}}{Lemma}]
    This follows from \Cref{lem: w_tangent_items}.
    Indeed, note that vectors of the form $\vec{w}_{\nn} = (\cos(\epsilon {\nn}/\ell), \sin(\epsilon {\nn}/\ell))$ with $0 \leq {\nn} \leq \lfloor 2\pi \ell/\epsilon \rfloor $ divide the unit circle in arcs of length $\leq \epsilon/\ell$. 
    In particular, for every unit vector $\vec{W}\in \bbR^2$, there exist some ${\nn}$ with $\Vert \vec{W}- \vec{w}_{\nn} \Vert \leq \epsilon/2\ell$.
    Further, since every $y \in \bbR^2$ is at distance $\leq \epsilon /\sqrt{2}$ from some point in the lattice $\epsilon \bbZ^2$, there always exists some $j,k$ with $\Vert p_{j,k} - y \Vert \leq \epsilon /\sqrt{2}$.  
\end{proof}

\subsection{Proof of \texorpdfstring{\Cref{lem: Rectangle_Distant}}{Lemma}}\label{apx: ProofRectangleDistant}
Again, the proof amounts to an application of \Cref{lem: VariantLocalStraight}. 
More specifically, we apply that lemma to a ball of radius somewhat greater than $\ell$: 
\begin{proof}[Proof of \texorpdfstring{\Cref{lem: Rectangle_Distant}}{Lemma}]
    Let $\delta \de 2\kappa\ell$ and recall the definition of $r_\delta'$ from \Cref{apx: Variant}. 
    Taking $c_1$ sufficiently small in the assumption \eqref{eq:GladEgg} then ensures that $r_{\delta}' = 2\ell$ and $\delta < 1/2$.
    Further, recall that $\sE \leq c_3\ell$ and note that $\epsilon = \kappa\ell^2 \leq c_1\ell$ by \eqref{eq:GladEgg}. 
    Hence, taking $c_1$ and $c_3$ sufficiently small, it can additionally be ensured that 
    \begin{align} 
        r_{\delta}' > \ell/10 + (2 + 1/2)\epsilon \quad \text{ and }\quad r_{\delta}' > \ell + 2\sE + 2 \epsilon. \label{eq:GladYoga} 
    \end{align}
    In particular, recalling the definition of $\cR(j,k,{\nn},h)$ from \Cref{alg: Improve}, the first inequality implies that $\cR(j,k,{\nn},h) \subseteq \sB(p_{j,k}, r_{\delta}')$ for every $\lvert h \rvert \leq 2$.
    For future use, also note that 
    \begin{align} 
        \delta r_\delta' = 4\kappa \ell^2= 4\epsilon. \label{eq:QuickRug}
    \end{align}   
     
    The assumption that $\dH(\fB, \cup_{i=0}^{\nb} B_i)\leq \sE$ in \cref{eq:GladEgg} together with the assumption that $p_{j,k} \pm \ell \vec{w}_{\nn}$ have distance $\leq \sE + 2\epsilon$ to $\fB$ then implies that there exist $y_+, y_- \in \cup_{i=0}^{\nb} B_i$ with 
    \begin{align} 
        \bigl\Vert (p_{j,k} + \ell \vec{w}_{\nn}) -y_+\bigr\Vert\leq 2\sE +2\epsilon \quad \text{ and }\quad \bigl\Vert (p_{j,k} - \ell \vec{w}_{\nn}) -y_-\bigr\Vert\leq 2 \sE + 2\epsilon.\label{eq:IvoryMom}
    \end{align}
    The triangle inequality then implies that the points $y_+,y_-$ are at distance $\leq \ell + 2\sE + 2 \epsilon < r_{\delta}'$ from $p_{j,k}$.
    Thus, at least one barrier $B_i$ intersects $\sB(p_{j,k},r_{\delta}')$ and \Cref{lem: VariantLocalStraight} implies that this barrier is unique.  

    Let $\hatn$ and $\mathfrak{c}$ be as in \Cref{lem: VariantLocalStraight} applied to $\sB(p_{j,k},r_{\delta}')$. 
    Then, using \eqref{eq:IvoryMom} together with the fact that \cref{item:r'1} yields $\lvert \langle y_+,\hat{n} \rangle -\mathfrak{c} \rvert \leq 4\delta r_{\delta}'$ and $\lvert \langle y_-,\hat{n} \rangle -\mathfrak{c} \rvert \leq 4\delta r_{\delta}'$,
    \begin{align} 
        2\ell \lvert \langle \vec{w}_{\nn}, \hatn \rangle \rvert = \lvert \langle p_{j,k} +\ell \vec{w}_{\nn},\hatn   \rangle - \langle p_{j,k} -\ell \vec{w}_{\nn},\hatn \rangle \rvert 
        \leq 8 \delta r_{\delta}'  + 4\sE + 4\epsilon. \label{eq:SilentTown}    
    \end{align}
    The assumption that $p_{j,k}$ is at distance $> \sE/2$ from all barriers implies that all points in a ball of radius $\sE/2$ lie on the same side of $B_i$. 
    In particular, if $p_{j,k}$ lies on the positive side then so should $p_{j,k} - (\sE/2)\hat{n}$, and hence \cref{item:r'3} in \Cref{lem: VariantLocalStraight} yields $\langle p_{j,k}, \hat{n} \rangle - \sE/2 > \mathfrak{c} - 4\delta r_{\delta}'$. 
    Similarly, if $p_{j,k}$ lies on the negative side then $\langle p_{j,k}, \hat{n} \rangle + \sE/2 < \mathfrak{c} + 4\delta r_{\delta}'$
    . 
    Taking $c_2$ sufficiently large in the assumption that $\sE \geq c_2 \epsilon$ and using \eqref{eq:QuickRug}, it may here be assumed that $\sE/2 - 4\delta r_\delta' >0$.
    Hence, regardless of what side $p_{j,k}$ lies on, 
    \begin{align} 
        \lvert \langle p_{j,k}, \hatn \rangle - \mathfrak{c} \rvert \geq \sE/2 - 4\delta r_{\delta}'.  \label{eq:HauntedDragon}
    \end{align}
    Now note that every point $z\in \cup_{h=-2}^2\cR(j,k,{\nn},h)$ can be represented as $z = p_{j,k} + a\vec{w}_{\nn} + b\vec{w}_{\nn}^{\perp}$ for $a \in [-\ell/10, \ell/10]$ and $b\in [-(2+1/2)\epsilon, (2+1/2)\epsilon]$.
    Consequently, for every such $z$, using  \eqref{eq:HauntedDragon}, the bound on $\lvert \langle \vec{w}_{\nn}, \hatn \rangle \rvert$ resulting from \eqref{eq:SilentTown}, and the fact that $\lvert \langle \vec{w}_{\nn}^{\perp}, \hatn \rangle \rvert\leq 1$ by $\vec{w}_{\nn}^{\perp}$ being a unit vector,
    \begin{align} 
        \lvert \langle z, \hat{n} \rangle -\mathfrak{c}\rvert &\geq \lvert \langle p_{j,k} , \hat{n} \rangle - \mathfrak{c} \rvert - (\ell/10)\lvert \langle \vec{w}_{\nn}, \hatn \rangle \rvert - (2+1/2)\epsilon \lvert \langle \vec{w}_{\nn}^{\perp}, \hatn \rangle \rvert \label{eq:CrimsonRaven}\\ 
        & \geq (1/2 - 4/20)\sE - (4 + 8/20)\delta r_{\delta}' - (4/20 + 2 + 1/2)\epsilon. \nonumber      
    \end{align}  
    In particular, also using \eqref{eq:QuickRug} to rewrite $\delta r_\delta'$ in terms of $\epsilon$, there exist absolute constants $c,C>0$ such that 
    \begin{align} 
        \lvert \langle z, \hat{n} \rangle  -\mathfrak{c}\rvert = 4\delta r_{\delta}' + (
        \lvert \langle z, \hat{n} \rangle  -\mathfrak{c}\rvert  - 4\delta r_{\delta}') \geq  4\delta r_{\delta}' + (c\sE - C \epsilon).\label{eq:ZenFan}      
    \end{align}
    Here, taking $c_2$ sufficiently large, it can be ensured that $c\sE  - C \epsilon>0$. 
    Then, \cref{item:r'1} from \Cref{lem: VariantLocalStraight} implies that $z\not\in B_i$. 
    The latter lemma is applicable since $\cup_{h=-2}^2 \cR(j,k,n,h) \subseteq \sB(p_{j,k}, r_{\delta}')$ by the discussion after \eqref{eq:GladYoga}. 

    Considering that $z$ was an arbitrary element, this shows that $\cup_{h=-2}^2\cR(j,k,{\nn},h)$ does not intersect the unique barrier $B_i$ with $B_i \cap \sB(p_{j,k}, r_{\delta}') \neq \emptyset$. 
    Equivalently, since $\cR(j,k,{\nn},h) \subseteq \sB(p_{j,k}, r_{\delta}')$, we have that $\cup_{h=-2}^2 \cR(j,k,{\nn},h)$ does not intersect any barrier. 
    This concludes the proof.   
\end{proof}

\subsection{Proof of \texorpdfstring{\Cref{lem: Rectangle_Concentration}}{Lemma}}\label{apx: ProofRectangleConcentration}
This proceeds similarly to the proof of \Cref{lem: ConcentrationUnionBoxes}:    
\begin{proof}
    Temporarily fix some $0 \leq {\nn}\leq \lfloor 2\pi \ell / \epsilon \rfloor$ and $\lvert h \rvert\leq 2$. 
    Taking $c_1,c_2$ sufficiently small in the assumption, we may assume that $\ell >\epsilon$. 
    Then, $\cR(j,k,{\nn},h) \subseteq D$ necessitates that $\sB(p_{j,k} + h\epsilon \vec{w}_{\nn}^{\perp}, \epsilon/2) \subseteq D$. 
    Considering that the $p_{j,k}$ associated with different $j,k$ are at distance $\geq \epsilon$ from each other, the balls $\sB(p_{j,k} + h\epsilon\vec{w}_{\nn}^{\perp}, \epsilon /2)$ are all disjoint. 
    Hence, there can be at most $\Area(D) / ( \pi\epsilon^2/4)$ integers $j,k$ with $\cR(j,k,{\nn},h) \subseteq D$.

    There are additionally five possible choices for $h$, and at most $2\pi \ell/\epsilon + 1$ possible choices for ${\nn}$.
    Recall that we may assume that $\ell >\epsilon$. 
    In particular, this yields that $2\pi \ell/\epsilon + 1 \leq 3\pi\ell/\epsilon$.  
    The number of rectangular regions with  $\cR(j,k,{\nn},h) \subseteq D$ is hence at most $60\Area(D)\ell/\epsilon^3 = 60 \Area(D)/(\kappa^3\ell^5)$. 

    Note that $\Area(\cR(j,k,{\nn},h)) =\epsilon\ell = \kappa\ell^3$.  
    Further, taking $c_2$ sufficiently small, we may assume that the diameter $\sup_{a,b\in \cR(j,k,{\nn},h)}\Vert a-b  \Vert$ is at most $\sqrt{t}$.  
    Hence, by \Cref{lem: Concentration_hatPS} and the union bound, 
    \begin{align} 
        \bbP\bigl(\Wu(\hat{P}_{\cR(j,k,{\nn},h)}, P_{\cR(j,k,{\nn},h)}) \leq \beta \sqrt{t},\ &\forall  j,k,{\nn},h \text{ with }\cR(j,k,{\nn},h) \subseteq D \bigr)\label{eq:GhostlyAnt}\\ 
        & \geq 1 - C\frac{\Area(D)}{\kappa^3\ell^5} \exp\Bigl(- c\frac{T \pi_{\min} \kappa \ell^3 }{\tmix } \Bigr)\nonumber   
    \end{align}
    for certain constants $C,c>0$ depending only on $\alpha, \beta$ and $\gamma$.  
    We will rely on the assumption \eqref{eq: Main_GlobalRecovery_n} regarding $T$ to estimate the right-hand side of \eqref{eq:GhostlyAnt}.
    First, however, we consider some preliminary estimates to rewrite the logarithmic factor in \eqref{eq: Main_GlobalRecovery_n}.

    As in the discussion preceding \eqref{eq:ValidVolt}, the assumption that the curvature of $B_0$ is bounded by $\kappa$ implies that $\Area(D) \geq \pi/\kappa^2$.
    Consequently, also using that $\eta <1$, 
    \begin{align} 
        \ln\Bigl(\pi \sqrt{\frac{1}{\kappa^3 \varepsilon^3}} \Bigr) \leq \ln\Bigl(\frac{\Area(D)}{\eta} \sqrt{\frac{\kappa}{\varepsilon^3}} \Bigr).  
    \end{align}
    Hence, first using that $\ell = \gamma \sqrt{\varepsilon/\kappa}$ and some direct calculations using the properties of the logarithm,   
    \begin{align} 
        \ln\Bigl(\frac{C\Area(D)}{\eta} \frac{1}{\kappa^3\ell^5}\Bigr) &=  \ln\Bigl(\frac{\Area(D)}{\eta} \sqrt{\frac{\kappa}{\varepsilon^3}} \Bigr) 
        + \frac{3}{2} \ln\Bigl(\pi \sqrt{\frac{1}{\kappa^3 \varepsilon^3} }\Bigr) + \ln\Bigl(\pi^{-2/3}\frac{C}{\gamma^5}\Bigr)
        \label{eq:CalmCobra}\\ 
        & \leq    
        \frac{5}{2}\ln\Bigl(\frac{\Area(D)}{\eta} \sqrt{\frac{\kappa}{\varepsilon^3}} \Bigr) + \ln\Bigl(\pi^{-2/3} \frac{C}{ \gamma^5}\Bigr). \nonumber 
    \end{align}
    Taking $c_1$ and $c_2$ sufficiently small, it can be ensured that $\sqrt{\varepsilon^3/\kappa} \leq \pi/(2\kappa^2) \leq  \Area(D)/(2\eta)$ so that the first term on the right-hand side of \eqref{eq:CalmCobra} is $\geq (5/2)\ln(2)\geq 1$.
    Hence, picking some large constant $C'>0$ depending on $\gamma$ to also cover the second term in \eqref{eq:CalmCobra},   
    \begin{align} 
        \ln\Bigl(\frac{\Area(D)}{\eta} \frac{1}{\kappa^3\ell^5}\Bigr) \leq C'\ln\Bigl(\frac{\Area(D)}{\eta} \sqrt{\frac{\kappa}{\varepsilon^3}} \Bigr).\label{eq:GhostlyBox} 
    \end{align}
    
    Now, taking $c_3 \geq C'/c$ sufficiently large in the assumption \eqref{eq: Main_GlobalRecovery_n} and using \eqref{eq:GhostlyBox} as well as the assumption that $\ell = \gamma \sqrt{\varepsilon/\kappa}$, it can be ensured that    
    \begin{align} 
        T \geq \frac{1}{c}\frac{\tmix}{\pi_{\min}} \frac{1}{\kappa \ell^3} \ln\Bigl(\frac{C\Area(D)}{\eta}\frac{1}{\kappa^3 \ell^5}\Bigr). \label{eq:NewRock}
    \end{align}    
    Combine \eqref{eq:GhostlyAnt} and \eqref{eq:NewRock} to conclude the proof. 
\end{proof}

\section{Proof of \texorpdfstring{\Cref{lem: DriftExponential}}{Lemma}}\label{apx: ProofDriftExponential}
\begin{proof}
    Consider the sequence of independent and identically distributed $\{0,1 \}$-valued random variables $Z_1, Z_2, \ldots $ whose joint law with the $Y_n$ is defined by  
    \begin{align} 
        \bbP(Z_n = 1 \mid Y_1,\ldots,Y_n, Z_1,\ldots,Z_{n-1}) \de 
        \begin{cases}
            0 & \text{ if }Y_n =0,\\ 
            q/\bbP(Y_n =1 \mid Y_1,\ldots, Y_{n-1}) & \text{ if } Y_n = 1.  
        \end{cases}
    \end{align} 
    Then, we have $Z_n \leq Y_n$ with probability one, so it suffices to prove \eqref{eq:BlueOtter} for $\sum_{i=1}^n Z_i$. 
    Note that $\bbP(Z_i = 1) = q$. 
    Hence, by the union bound and Hoeffding's inequality,  
    \begin{align} 
        \bbP\Bigl(\sum_{i=1}^{n}Z_i \leq n (q-\eta) \text{ for some }n\geq n_0 \Bigr) \leq \sum_{n=n_0}^\infty \bbP\Bigl(\sum_{i=1}^{n}Z_i \leq n (q-\eta) \Bigr) \leq  \sum_{n=n_0}^\infty\exp(-2\eta^2 n).\label{eq:KindTin}
    \end{align}
    \Cref{lem: DriftExponential} now follows by summing the series and using that $\exp(2\eta^2)/(\exp(2\eta^2)-1) \leq 2\eta^{-2}$ for all $\eta \leq 1$. 
\end{proof} 

\section{Proof of \texorpdfstring{\Cref{cor: NumTransitions}}{Corollary}}\label{apx: ProofNumTransition}
The following lemma makes it rigorous that transitions from $R_+(p,\vecv)$ to $H_{-}(p,\vecv)$ are unlikely if and only if there is a barrier in the way.    
\begin{lemma}\label{lem: NumberOfTransitions}
    There exist absolute constants $q,c_1,\ldots,c_4>0$ such that for every $\sqrt{t} \leq \ell \leq c_1(t \min\{ 1/\kappa^2, 1/\lambda_{\max}^2, \rho^2\})^{1/4}$ and unit vector $\vecv \in \bbR^2$, the following holds. 
    
    For every point $p\in D$ satisfying $\Vert p -y \Vert \geq c_2\ell $ for every $y\in \cup_{i=0}^{\nb} B_i$,  
    \begin{align} 
        \inf\{\bbP(X_t \in H _{-}(p,\vecv) \mid X_0 = x_0) : x_0 \in R_+(p,\vecv) \} \geq q.\label{eq:HappyRam} 
    \end{align}    
    On the other hand, if $\Vert p-y \Vert\leq c_3 \sqrt{t}$ and $\min_{\pm \in \{+,- \}}\Vert \vec{v} \pm \vec{n}_i(y) \Vert \leq c_4 \sqrt{t}/\ell$ for some $y\in B_i$, 
    \begin{align} 
        \sup\{\bbP(X_t \in H_{-}(p,\vecv) \mid X_0 = x_0) : x_0 \in R_+(p,\vecv) \} \leq q/2. \label{eq:JollyHat}
    \end{align}
\end{lemma}
\begin{proof}
    Let us start with the proof of \eqref{eq:HappyRam}. 
    Recall \eqref{eq: Def_Splus} and \eqref{eq: Def_Smin} and note that there exists an absolute constant $c >0$ such that
    $
        \Area\bigl(\sB(x_0, 4\sqrt{t}) \cap  H_{-}(p,\vecv)\bigr) \geq c t 
    $
    for every $x_0 \in R_+(p,\vecv)$. 
    Consequently, since $x_0 + W_t$ has density of order $\geq 1/t$ on $\sB(x_0, 4\sqrt{t})$, there exists an absolute constant $c' >0$ such that 
    \begin{align} 
        \bbP\bigl(x_0 + W_t \in H_{-}(p,\vecv)\bigr) \geq \bbP\bigl(x_0 + W_t \in \sB(x_0, 4\sqrt{t}) \cap  H_{-}(p,\vecv)\bigr) \geq c'. \label{eq:WeepyVine}
    \end{align}

    Recall the assumption that $p$ is at distance $c_2\ell$ from all barriers and note that $x_0 \in R_+(p,\vecv)$ implies that $\Vert x_0 - p \Vert \leq \ell + 2\sqrt{2}$.
    It hence follows that $x_0$ is at distance $\geq (c_2-1)\ell - 2\sqrt{t} \geq (c_2 - 3)\sqrt{t}$ from all barriers due to the triangle inequality.  
    Then, by \eqref{eq:BlueCity} from \Cref{cor: X_uniform_Y}, if $c_1$ is sufficiently small and $c_2$ is sufficiently large,  
    \begin{align} 
        \bbP(X_t = x_0 + W_t ) \geq 1 - c'/2.\label{eq:DangerousGoose}
    \end{align}     
    In other words, we have $X_t = x_0 + W_t$ with high probability. 
    The combination of \eqref{eq:WeepyVine} and \eqref{eq:DangerousGoose} yields \eqref{eq:HappyRam} with $q \de c'/2$.
    
    We next prove \eqref{eq:JollyHat}. 
    Let us consider the case where $\Vert \vecv - \vec{n}_i(y) \Vert \leq c_4 \sqrt{t}/\ell$; the case with $\Vert \vecv + \vec{n}_i(y) \Vert \leq c_4 \sqrt{t}/\ell$ proceeds similarly. 
    Denote $\delta \de C\ell/\min\{1/\kappa,1/\lambda_{\max},\rho\}$ for some constant $C \geq 3$ which will soon be fixed.

    The assumed inequalities on $\ell$ and $t$ imply, in particular, that $t\leq c_1^4 \min\{1/\kappa^2,1/\lambda_{\max}^2,\rho^2 \}$.
    Hence, by a second application of the assumed bounds on $\ell$, we have $\delta \leq c_1^2 C$ and $r_{\delta} = C\ell \geq C\sqrt{t}$. 
    Now, first fixing $C$ at a sufficiently large value and subsequently taking $c_1$ sufficiently small, it can be ensured using \Cref{cor: TauBound} that for every $x_0 \in R_+(p,\vecv)$,  
    \begin{align} 
        \bbP\bigl(X_t \not\in \sB(x_0,r_\delta)\ \text{ or }\ s_i(\Lc{i}_t) \neq s_i(0)\mid X_0 = x_0 \bigr)\leq q/2.\label{eq:CozyDog} 
    \end{align} 
    
    Recall \eqref{eq: Def_Splus} and note that the distance from points in $R_+(p,\vecv)$ to $p$ is at most $\ell + 2\sqrt{t} \leq 3\ell$. 
    In particular, recalling that $r_{\delta} = C\ell \geq 3\ell$ since $C\geq 3$, we may assume that $\sB(x_0, r_\delta) \subseteq \sB(p, 2r_\delta)$ for every $x_0 \in R_+(p,\vecv)$. 
    Hence, for every $x_0\in R_+(p,\vecv)$,    
    \begin{align} 
        x_0  \in {}&{}\{z_+\in \sB(p, 2r_\delta): \langle z_+ - p, \vecv \rangle \geq \sqrt{t} \},\label{eq:LuckyIce} \\ 
        H_-(p,\vecv) \cap \sB(x_0, r_\delta) \subseteq   {}&{}\{z_-\in \sB(p, 2r_\delta): \langle z_- - p, \vecv \rangle \leq - \sqrt{t} \}.\label{eq:JollyLobster} 
    \end{align} 
    Hence, also using \eqref{eq:CozyDog}, the desired result in \eqref{eq:JollyHat} follows if we show that every $z_+$ as in \eqref{eq:LuckyIce} is strictly on the positive side of $B_i$ and that every $z_-$ as in \eqref{eq:JollyLobster} is strictly on the negative side. 
    Let us prove this for $z_+$.
    The proof for $z_-$ proceeds similarly. 

    Recall that $\delta \leq c_1^2C$. 
    In particular, taking $c_1$ sufficiently small, we may assume that $\delta <1/2$ so that \Cref{lem: BarrierLocallyStraight} is applicable.
    Then, \cref{item:JollyDragon} in that lemma implies that it suffices to show that $z_+$ satisfies $\langle z_+, \hatn  \rangle \geq \mathfrak{c} + 4\delta r_{\delta}$.
    By the bilinearity of the inner product, 
    \begin{align} 
        \langle z_+  ,\hatn \rangle = \langle y,\hatn  \rangle  +  \langle z_+ - p, \hatn \rangle + \langle p  - y,\hatn\rangle.    \label{eq:YellowTrade}
    \end{align}
    Recall that $\Vert \hatn - \vec{v} \Vert \leq c_4 \sqrt{t}/\ell$ and $r_\delta =C\ell$. 
    Consequently, using \eqref{eq:LuckyIce}, the Cauchy--Schwarz inequality, and taking $c_4$ sufficiently small, 
    \begin{align} 
        \langle z_+ - p, \hatn \rangle  = \langle z_+ - p,\vecv \rangle +  \langle z_+ - p, \hatn - \vecv \rangle  \geq (1 -  2c_4 C)\sqrt{t}  \geq \sqrt{t}/2.       
    \end{align}
    Further recall that $\Vert p - y \Vert \leq c_3\sqrt{t}$. 
    Hence, taking $c_3$ sufficiently small and using that $\hat{n}$ is a unit vector, 
    \begin{align} 
        \langle p - y , \hatn\rangle \geq -c_3\sqrt{t} \geq -\sqrt{t}/4.\label{eq:QuickHermit}
    \end{align} 
    Now combining \eqref{eq:YellowTrade}--\eqref{eq:QuickHermit} and using that $ \langle y , \hatn \rangle  \geq \mathfrak{c} - 4\delta r_\delta$ by \cref{item:GhostlyQuip} in \Cref{lem: BarrierLocallyStraight},
    \begin{align} 
        \langle z_+ , \hatn \rangle \geq \langle y , \hatn \rangle + \sqrt{t}/4 \geq \mathfrak{c} - 4\delta r_{\delta} + \sqrt{t}/4.  \label{eq:AlertBird} 
    \end{align}
    Here, note that $\delta r_{\delta} = C^2 \ell^2/\min\{ 1/\kappa, 1/\lambda_{\max},\rho\} \leq c_1^2  C^2 \sqrt{t}$ by definition of $\delta$ and the assumed bound on $\ell$.
    Hence, taking $c_1$ sufficiently small ensures that the right-hand side of \eqref{eq:AlertBird} is $\geq \mathfrak{c} +  4\delta r_{\delta}$. 
    This concludes the proof.     
\end{proof}

\begin{proof}[Proof of \texorpdfstring{\Cref{cor: NumTransitions}}{Corollary}]
    We start with \eqref{eq:TallSax}.
    Denote the $n$th index with $X_{it} \in R_+(p,\vecv)$ by 
    \begin{align} 
        I_{n} \de \inf\{i\geq 0: \#\{0 \leq j \leq i: X_{jt} \in R_+(p,\vecv) \} \geq n \}. \label{eq:EasyGym}
    \end{align}
    Then, a random variable which indicates whether the $n$th visit to $R_+(p,\vecv)$ was followed by a transition to $H_{-}(p,\vecv)$ may be defined by 
    $
        Y_n \de \bb1\{X_{(I_n + 1) t} \in H_{-}(p,\vecv) \}.  
    $
    Recall \eqref{eq:VividImp} and \eqref{eq:ValidOrb} and observe that $M(p,\vecv) = \sum_{i=1}^{N(p,\vecv)} Y_i$. 
    Consequently,   
    \begin{align} 
        \bbP\bigl(M(p,\vecv) / N(p,\vecv) < c_3 \text{ and }  N(p,\vecv) \geq n_0  \bigr) 
        &=  \bbP\Bigl(\sum_{i=1}^{N(p,\vecv)} Y_i <  c_3N(p,\vecv) \text{ and }  N(p,\vecv) \geq n_0  \Bigr)\nonumber\\ 
        &\leq \bbP\Bigl(\sum_{i= 1}^{n} Y_i < c_3 n \text{ for some }n \geq n_0   \Bigr).\label{eq:HighSong} 
    \end{align}  
    Here, applying the strong Markovianity of \Cref{prop: X_Markov} with the stopping time $I_n t$ and using \Cref{lem: NumberOfTransitions}, there exists an absolute constant $q>0$ such that 
    \begin{align} 
        \bbP(Y_n=1 \mid Y_1,\ldots,Y_{n-1}) \geq \inf\{\bbP(X_t \in H_{-}(p,\vecv) \mid X_0 = x_0') : x_0' \in R_+(p,\vecv) \}\geq q \label{eq:TrustyCrow}
    \end{align}
    provided that $c_1$ is taken sufficiently small and $c_2$ sufficiently large.  
    Now, by \Cref{lem: DriftExponential} with $\eta = q/4$, it holds with $c_3 \de (3/4)q$ that   
    \begin{align} 
        \bbP\bigl(M(p,\vecv) / N(p,\vecv) < c_3 \text{ and }  N(p,\vecv) \geq n_0  \bigr) \leq 32q^{-2} \exp(-q^2n_0/8). 
    \end{align} 
    This proves \eqref{eq:TallSax}. 

    As for \eqref{eq:ZippyCity}, let $I_n$ be as in \eqref{eq:EasyGym} and define 
    $
        Z_n \de \bb1\{X_{(I_n+1)t} \not\in H_{-}(p,\vecv)\}.  
    $ 
    Then, similarly to \eqref{eq:HighSong}, we have that 
    \begin{align} 
        \bbP\bigl(M(p,\vecv) / N(p,\vecv) \geq  {}&{}c_3 \text{ and }  N(p,\vecv) \geq n_0  \bigr)\label{eq:FairVine}\\
        &= 
        \bbP\bigl((N(p,\vecv)  - M(p,\vecv))/ N(p,\vecv) \leq 1-c_3 \text{ and }  N(p,\vecv) \geq n_0  \bigr) \nonumber\\ 
        &= \bbP\Bigl(\sum_{i=1}^{N(p,\vecv)} Z_i \leq  (1-c_3) N(p,\vecv) \text{ and }  N(p,\vecv) \geq n_0  \Bigr)\nonumber\\ 
        &\leq \bbP\Bigl(\sum_{i=1}^{n} Z_i \leq  (1-c_3) n \text{ for some }  n \geq n_0  \Bigr).\nonumber   
    \end{align}
    Further, similarly to \eqref{eq:TrustyCrow}, now using \eqref{eq:JollyHat} from \Cref{lem: NumberOfTransitions}, 
    \begin{align} 
        \bbP(Z_n=1 \mid Z_1,\ldots,Z_{n-1}) &\geq 1  - \sup\{\bbP(X_t \in H_{-}(p,\vecv) \mid X_0 = x_0') : x_0' \in R_+(p,\vecv) \} \label{eq:JustCamel} \\ 
        &\geq 1 - q/2.\nonumber  
    \end{align}
    Recall that $c_3 = (3/4)q$ and fix some $\eta < q/4$. 
    Then, we have that $1 - q/2 - \eta > 1 - c_3$. 
    Hence, the combination of \Cref{lem: DriftExponential} with \eqref{eq:FairVine} and \eqref{eq:JustCamel} yields \eqref{eq:ZippyCity}.  
\end{proof}

\section{Proof of \texorpdfstring{\Cref{lem: NewManyHits}}{Lemma}}\label{apx: ReductionToBrownianCount}
We rely on the tools from \Cref{sec: Local} with 
\begin{align} 
    \delta \de \ell/ (2\min\{ 1/\kappa, 1/\lambda_{\max},\rho \}). \label{eq: Def_delta}
\end{align}  
Then, with $r_{\delta}$ as in \eqref{eq: Def_R_eps}, it follows from \eqref{eq:NewMoistKnot} that   
\begin{align} 
    r_\delta = \ell/2, \qquad \qquad \delta \leq c_2/2, \qquad\qquad \delta r_{\delta}  \leq  (c_1^2/4) \sqrt{t}.     \label{eq: rdel_and_delrdel}  
\end{align}
As a preliminary reduction, we replace $R_+(p,\vecv)$ by another set which is more convenient for our purposes; see \Cref{fig: Splus_Shat} for a visualization. 
\begin{lemma}\label{lem: Splus_Shat}
    For every $\zeta < 1/2$ there exist $c_1,\ldots, c_4>0$ such that the following holds for every $t,\ell >0$, $x_0\in B_i$, $p\in \bbR^2$, and unit vector $\vecv \in \bbR^2$ satisfying \eqref{eq:NewMoistKnot} for these constants.

    Let $R_+(p,\vecv)$ be as in \eqref{eq: Def_Splus} and $\hatn$ as in \Cref{lem: BarrierLocallyStraight}. 
    Then, one has $\widehat{R}_+ \subseteq R_+(p,\vecv)$ for
    \begin{align} 
        \widehat{R}_+ \de\Bigl\{q \in \sB(x_0, r_\delta): (1+ \zeta)\sqrt{t} \leq \langle q - x_0, s_i(0)\hatn \rangle   \leq (2-\zeta)\sqrt{t} \Bigr\}.  \label{eq:DarkBee}
    \end{align}
\end{lemma}
\begin{figure}[t]
    \includegraphics[width = 0.75\textwidth]{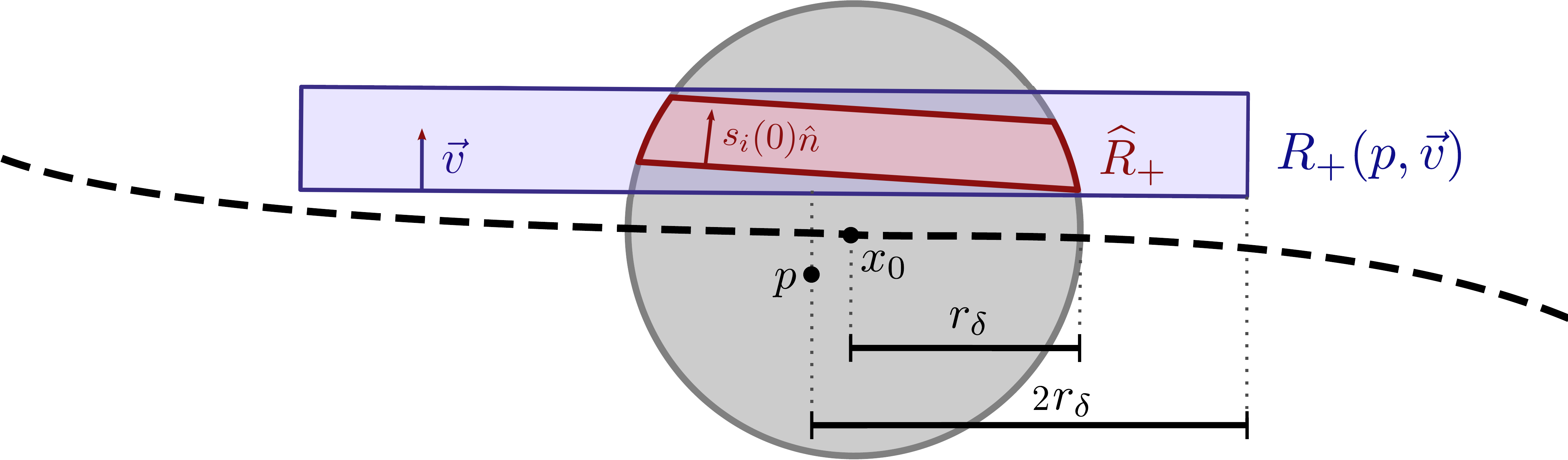}
    \caption{
    Visualization of the sets $R_+(p,\vecv)$ and $\widehat{R}_+$ considered in \Cref{lem: BarrierLocallyStraight}. 
    }
    \label{fig: Splus_Shat}
\end{figure}
\begin{proof}
    Pick some $q\in \widehat{R}_+$ and let us show that $q \in R_+(p,\vecv)$. 
    By bilinearity,
    \begin{align} 
        \langle q - p, \vecv \rangle  
        &= \langle q - x_0, s_i(0)\hatn \rangle + \langle q - x_0, \vecv- s_i(0)\hatn  \rangle + \langle  x_0 - p, \vecv \rangle.\label{eq:GreenImp}   
    \end{align}
    Regarding the first term, the assumption that $q\in \widehat{R}_+$ yields that 
    \begin{align} 
        (1 + \zeta)\sqrt{t} \leq \langle q - x_0, s_i(0)\hatn \rangle   \leq (2-\zeta)\sqrt{t}.\label{eq:AlertAir}   
    \end{align}  
    We next estimate the remaining two terms in \eqref{eq:GreenImp}. 

    By the Cauchy--Schwarz inequality and the fact that $q\in \sB(x_0,r_\delta)$,    
    \begin{align} 
        \lvert \langle q - x_0, \vecv- s_i(0)\hatn  \rangle  \rvert \leq r_\delta (\Vert \vecv - s_i(0)\vec{n}_i(x_0) \Vert +\Vert s_i(0)\hatn - s_i(0)\vec{n}_i(x_0) \Vert). \label{eq:AncientSun}  
    \end{align} 
    Taking $c_2$ sufficiently small in \eqref{eq: rdel_and_delrdel}, it can be ensured that $\delta < 1/2$.  
    Then, \cref{item:IdleBat} in \Cref{lem: BarrierLocallyStraight} is applicable and yields $\Vert s_i(0)\hatn - s_i(0)\vec{n}_i(x_0) \Vert\leq 2\delta$.
    Further, recall from \eqref{eq:NewMoistKnot} that $\Vert \vecv - s_i(0)\vec{n}_i(x_0) \Vert \leq c_4 \sqrt{t} / \ell$. 
    Hence, using \eqref{eq: rdel_and_delrdel} and taking $c_1,c_4$ sufficiently small depending on $\zeta$, we may bound the right-hand side of \eqref{eq:AncientSun} as
    \begin{align} 
        \lvert \langle q - x_0, \vecv- s_i(0)\hatn  \rangle \rvert \leq r_\delta (c_4\sqrt{t}/\ell + 2\delta ) \leq (c_4/2+ c_1^2/2 )\sqrt{t} \leq (\zeta/2) \sqrt{t}. \label{eq:CrimsonJinx}  
    \end{align}    

    Regarding the final term in \eqref{eq:GreenImp}, recall that $\vecv$ is a unit vector and that $\Vert x_0 - p \Vert \leq c_3\sqrt{t}$ by \eqref{eq:NewMoistKnot}. 
    Hence, for $c_3$ sufficiently small depending on $\zeta$,   
    \begin{align} 
        \lvert \langle x_0 - p , \vecv \rangle \rvert \leq \Vert x_0 - p \Vert \leq c_3\sqrt{t} \leq (\zeta/2)\sqrt{t}. \label{eq:PoliteHotel}
    \end{align} 
    Combine \eqref{eq:CrimsonJinx} and \eqref{eq:PoliteHotel} to see that the total contribution of the final two terms in \eqref{eq:GreenImp} is at most $\zeta\sqrt{t}$.
    Hence, by \eqref{eq:AlertAir}, we have $\langle q - p, \vecv \rangle \in [\sqrt{t}, 2\sqrt{t}]$.

    Recalling the definition of $R_+(p,\vecv)$ from \eqref{eq: Def_Splus}, it remains to show that $\lvert \langle q-p, \vec{v}^{\perp} \rangle \rvert\leq \ell$.     
    Recall that $\vec{v}^\perp$ is a unit vector and that $q \in \sB(x_0, r_\delta)$. 
    Hence, using \eqref{eq:NewMoistKnot} and \eqref{eq: rdel_and_delrdel} and choosing $c_3$ sufficiently small,
    \begin{align} 
        \lvert \langle q-p, \vec{v}^{\perp} \rangle \rvert \leq \Vert q - x_0 \Vert + \Vert x_0 - p \Vert \leq r_\delta + c_3\sqrt{t} \leq (1/2 + c_3)\ell \leq \ell. \label{eq:MadPizza} 
    \end{align}
    This concludes the proof.  
\end{proof}

Let $Y_t^+$ and $y_0^+$ be as in \Cref{sec: XtYtApprox}.
Then, using the preliminary reduction of \Cref{lem: Splus_Shat} together with \Cref{cor: X_uniform_Y} reduces us to studying the number of times that $Y_t^+$ visits a horizontal strip:  
\begin{lemma}\label{lem: Nplus}
    For every $\xi < 1/2$ and $\eta \in (0,1)$ there exist $c_1,\ldots,c_5>0$ such that the following holds for every $\ell, t, x_0, p, \vecv$ which satisfy \eqref{eq:NewMoistKnot} with respect to $c_1,\ldots,c_4$. 
    For every $1 \leq J  \leq c_5 \ell^2/t$, every $n_0 \geq 0$, and every $0 \leq \gamma < 1$,  
    \begin{align} 
        \bbP\bigl(\#\{0 \leq {}&{} j \leq J: X_{(\gamma +j)t} \in R_+(p,\vecv)   \} \geq n_0 \mid X_0 = x_0 \bigr)\label{eq:PinkNut}\\ 
        &\geq \bbP\bigl(\#\{0 \leq j \leq J: \langle Y_{(\gamma+j)t}^+ - y_0^+, s_i(0)\hatn \rangle \in [(1+\xi)\sqrt{t}, (2-\xi)\sqrt{t} ]   \} \geq n_0 \bigr) - \eta. \nonumber  
    \end{align}
\end{lemma}
\begin{proof} 
    Let $\zeta \de \xi/2$. 
    Then, taking $c_1,\ldots,c_4$ such that \Cref{lem: Splus_Shat} is applicable and defining $\widehat{R}_+$ as in \eqref{eq:DarkBee}, 
    \begin{align}
        \bbP\bigl(\#\{j \leq J: X_{(\gamma +j)t} \in {}&{} R_+(p,\vecv)   \} \geq n_0  \bigr)\label{eq:KindDew} \geq  
        \bbP\bigl(\#\{j\leq J: X_{(\gamma +j)t} \in \widehat{R}_+  \} \geq n_0  \bigr).
    \end{align}
    Recall the stopping time $\tau$ from \eqref{eq: Def_tau}. 
    Then, by \Cref{cor: TauBound}, also using that $\gamma+J \leq 2J$ by the assumption that $J\geq 1$ and $\gamma < 1$, there is an absolute constant $C>0$ such that  
    \begin{align} 
        \bbP(X_{(\gamma +j)t} \in \sB(x_0,r_\delta),\, \forall j\leq J) \geq 1 - \bbP(\tau \leq (\gamma +J)t)\geq 1 - \delta  - C(Jt/r_\delta^2). \label{eq:OddRam}   
    \end{align}
    Recall that $r_\delta = \ell/2$ and $J \leq c_5 \ell^2/t$. 
    Further, we have $\delta \leq c_2/2$ by \eqref{eq: rdel_and_delrdel}.
    Hence, taking $c_2,c_5$ to be sufficiently small,
    \begin{align} 
        \bbP(X_{(\gamma +j)t} \in \sB(x_0,r_\delta),\, \forall j\leq J) \geq 1 - \eta/2. \label{eq:BoldJet}
    \end{align} 

    Similarly, using \Cref{cor: X_uniform_Y} and taking $c_2,c_5$ to be sufficiently small, there exists an absolute constant $C'>0$ such that 
    \begin{align} 
        \bbP(\Vert X_{(\gamma +j)t} - Y_{(\gamma +j)t}^+ \Vert  \leq C'\delta r_\delta,\, \forall j\leq J) \geq 1 - \eta/2. \label{eq:KnownVine} 
    \end{align}
    By \eqref{eq: rdel_and_delrdel}, we can ensure that $C' \delta r_\delta \leq \zeta \sqrt{t}/2$ by taking $c_1$ sufficiently small. 
    Further, recalling \eqref{eq:RoyalLobster}, we may assume that $\Vert x_0 - y_0^+ \Vert \leq 8\delta r_\delta \leq \zeta \sqrt{t}/2$.    
    Hence, by \eqref{eq:KnownVine} and the triangle inequality, 
    \begin{align} 
       \bbP( \lvert \langle X_{(\gamma + j)t} - x_0, s_i(0)\hatn \rangle - \langle Y_{(\gamma + j)t}^+ - y_0^+, s_i(0)\hatn \rangle \rvert \leq \zeta\sqrt{t},\ \forall j\leq J) \geq 1-\eta/2.\label{eq:TallJinx}  
    \end{align}
    
    Considering that we defined $\zeta = \xi/2$ and recalling the definition of $\widehat{R}_+$ from \eqref{eq:DarkBee}, combining \eqref{eq:BoldJet} and \eqref{eq:TallJinx} in \eqref{eq:KindDew} now yields the desired result \eqref{eq:PinkNut}.  
\end{proof} 

\begin{proof}[Proof of \texorpdfstring{\Cref{lem: NewManyHits}}{Lemma}]
    The assumption that $x_0\in B_i$ implies that $y_0^+$ lies on the reflection barrier $A^+$ for $Y^+$; recall \eqref{eq:SmartOrb} and \eqref{eq:RoyalLobster}.  
    Consequently, the process $\langle Y_r^+  - y_0^+, s_i(0)\hatn \rangle$ is a one-dimensional reflected Brownian motion starting from zero, meaning that it has the same distribution as $\lvert \cB_r \rvert$ for $\cB_r$ a one-dimensional Brownian motion. 
    Hence, \eqref{eq:NewQuickArm} follows from \Cref{lem: Nplus} and the scaling property.
\end{proof}

\section{Proof of \texorpdfstring{\Cref{prop: Main_PathDependentRecovery_Algorithm}}{Proposition}}\label{apx: Proof_Main_PathDependentRecovery_Algorithm}
Recall that an outline was given \Cref{sec: Main_PathDependentRecovery_Algorithm}.
Details for the proofs of the performance guarantee in \Cref{thm: Main_PathDependentRecovery} are given in \Cref{apx: ProofEqFalseRandom}. 
Subsequently, we prove \Cref{cor: Xuniform} in \Cref{apx: ProofUniform}.

\subsection{Proof of \texorpdfstring{\Cref{thm: Main_PathDependentRecovery}}{Proposition}} \label{apx: ProofEqFalseRandom}
Following the pattern from \Cref{apx: ProofGlobalRecovery,apx: Proof_Main_GlobalRecovery_Improvement}, we start by stating some preparatory lemmas. 

By combining \Cref{cor: NumTransitions} with a union bound, the following result shows that the expected behavior of the number of transitions holds for all $p_{j,k}$ and $\vec{v}_n$ in \Cref{alg: HighFrequency}:
\begin{lemma}\label{lem: UnionNumTransitions}
    There exist $c_1,\ldots,c_6>0$ such that the following holds for every $\sqrt{t} \leq \ell \leq c_1 (t \min\{ 1/\kappa^2, 1/\lambda_{\max}^2, \rho^2 \})^{1/4}$, $\epsilon \leq \sqrt{t}/2$, $\mathfrak{n}_0 \geq 1$, and initial condition $x_0 \in D$.

    For every $p\in \bbR^2$ and unit vector $\vec{v}\in \bbR^2$, define a $\{0,1 \}$-valued random variable $\cE_{\mathrm{good}}(p,\vecv)$ as follows. 
    If $\Vert p -y\Vert \geq c_2 \ell$ for every $y \in \cup_{i=0}^m B_i$, then with $M(p,\vecv)$ and $N(p,\vecv)$ as in \Cref{cor: NumTransitions} we define 
    \begin{align} 
        \cE_{\mathrm{good}}(p,\vecv) \de 
        \bb1\bigr\{M(p,\vecv)/N(p,\vecv)\geq c_3 \text{ or }N(p,\vecv) < \mathfrak{n}_0 \bigl\}.\label{eq:ZenMan} 
    \end{align}
    On the other hand, if  $\Vert p-y \Vert \leq c_4 \sqrt{t}$  and $\min_{\pm\in \{+,- \}}\Vert \vecv \pm \vec{n}_i(y) \Vert \leq c_5 \sqrt{t}/\ell$ for some $y\in B_i$,  
    \begin{align} 
        \cE_{\mathrm{good}}(p,\vecv)\de
            \bb1\bigr\{M(p,\vecv)/N(p,\vecv) < c_3 \text{ or }N(p,\vecv) < \mathfrak{n}_0\bigl\}.\label{eq:YoungBall}
    \end{align}
    And finally, define $\cE_{\mathrm{good}}(p,\vecv) \de 1$ if $p$ and $\vecv$ do not satisfy either of the preceding conditions.

    Consider some arbitrary $\eta \in (0,1)$ and assume that    
    $
        \mathfrak{n}_0 \geq c_6\ln(\ell T / \eta \epsilon t) 
    $.  
    Then, it holds with $p_{j,k}$ and $\vecv_n$ as in \Cref{alg: HighFrequency} that 
    \begin{align} 
        \bbP\bigl(\cE_{\mathrm{good}}(p_{j,k}, \vec{v}_n) = 1 \text{ for all }j,k \in \bbZ \text{ and }0 \leq n \leq \lfloor 2\pi \ell / \epsilon \rfloor\bigr) \geq 1 - \eta. 
    \end{align}
\end{lemma}

It could be difficult to recover hits of the barrier which only occurred at the very end of the observation period as we may not get to observe much associated data. 
\Cref{lem: HittingEnd} shows that this difficulty is irrelevant as it does not occur with high probability:  
\begin{lemma}\label{lem: HittingEnd}
    For every $\eta \in (0,1)$ there exists $c>0$ such that the following holds for every $T>0$ and initial condition $x_0 \in D$. 
    Consider some $\Delta \leq c\min\{T,1/\kappa^2,1/\lambda_{\max}^2, \rho^2\}$. 
    Then,  
    \begin{align} 
        \bbP(X_s \not\in \cup_{i=0}^m B_i, \ \forall s \in [ T- \Delta, T]) \geq 1- \eta.\label{eq:DangerousHermit} 
    \end{align}     
\end{lemma}

The proofs of \Cref{lem: UnionNumTransitions,lem: HittingEnd} are given in \Cref{sec: UnionNumTransitions,sec: HittingEnd}, respectively.
It now remains to combine the preceding ingredients with those in \Cref{sec: ProofPathDependent}. 
This amounts to a straightforward but notationally cumbersome computation:   
\begin{proof}[Proof of \texorpdfstring{\Cref{thm: Main_PathDependentRecovery}}{Theorem}]
    To avoid ambiguity with other parts, let us denote $\cC_1,\cC_2,\cC_3$ for the constants in \Cref{thm: Main_PathDependentRecovery}, and denote $\cC_4,\ldots, \cC_7$ for the constants in \Cref{prop: Main_PathDependentRecovery_Algorithm}. 

    The assumption preceding \eqref{eq: HighFrequency_FalsePositives} and \eqref{eq: HighFrequency_RandomPoint} is that 
    \begin{align} 
        t \leq \cC_1 T, \ \text{ and }\ \ln(T/t)^4 t \leq \cC_2\min\{T, 1/\kappa^2, 1/\lambda_{\max}^2, \rho^2 \}. \label{eq:CozyOwl}   
    \end{align}
    Further, \Cref{prop: Main_PathDependentRecovery_Algorithm} posits that the parameters in \Cref{alg: HighFrequency} satisfy
    \begin{align} 
        \mathfrak{s} = \cC_4,\ \ \epsilon = \cC_5 \sqrt{t}, \ \ \ell = \cC_6\ln(T/t)\sqrt{t}, \ \text{ and }\  \mathfrak{n}_0 = \cC_{7}\ln(T/t).\label{eq:UnripeRug} 
    \end{align}
    It now has to be shown that these constants can be chosen such that the output of \Cref{alg: HighFrequency} satisfies the performance guarantees in \eqref{eq: HighFrequency_FalsePositives} and \eqref{eq: HighFrequency_RandomPoint} with respect to $\cC_3$.

    Let us start by determining appropriate values for $\cC_4,\ldots,\cC_7$. 
    Let $c_1,\ldots,c_6$ be as in \Cref{lem: UnionNumTransitions}, let $c_1',\ldots,c_5'$ be the constants arising in \Cref{lem: NewManyHits} with some fixed $\zeta < 1/2$ and with $\eta$ replaced by $\eta/4$, and let $\tic_1,\ldots,\tic_3$ be the constants arising in \Cref{lem: NplusLarge} with the same $\zeta$. 
    Then, appropriate values for $\mathfrak{s}$ and $\epsilon$ may be found by taking   
    \begin{align} 
        \cC_4 \de c_3\quad \text{ and }\quad \cC_5 \de \min\{c_4,c_5, c_3', c_4', 1/2\}.\label{eq:OddHome} 
    \end{align} 
    The choice of $\cC_6$ and $\cC_7$ is more tedious.
    More specifically, having an application of \Cref{lem: NewManyHits,lem: NplusLarge,lem: UnionNumTransitions} in mind, the following inequalities will be required:  
    \begin{align} 
        \frac{\tic_3 \eta}{4}  \sqrt{\Bigl\lfloor  \frac{c_5'\ell^2}{t} \Bigr\rfloor } \geq  \mathfrak{n}_0 \geq c_6 \ln\Bigl( \frac{4}{\eta}\frac{T}{t}\frac{\ell}{\epsilon} \Bigr).\label{eq:DampXray} 
    \end{align} 
    Equivalently, by \eqref{eq:UnripeRug} $\cC_6$ and $\cC_7$ will have to satisfy the following constraints:  
    \begin{align} 
         \frac{\tic_3 \eta}{4}  \sqrt{\bigl\lfloor c_5'  \cC_6^2 \ln(T/t)^2 \bigr\rfloor} \geq  \cC_{7}\ln(T/t)  \geq c_6 \ln\Bigl(\frac{4}{\eta}\frac{T}{t}\frac{\cC_6}{\cC_5}   \ln(T/t) \Bigr) .\label{eq:DangerousGum} 
    \end{align}
    An inconvenience is that $\cC_6$ occurs on both the left and right of \eqref{eq:DangerousGum}. 
    To deal with this, one can exploit that the dependence of the left on $\cC_6$ is essentially linear while the dependence on the right is logarithmic, suggesting that the inequalities can be satisfied if $\cC_6$ is taken large enough. 
    More precisely, taking $\cC_1$ sufficiently small in \eqref{eq:CozyOwl}, we may assume that $4\ln(T/t)/(\eta \cC_5) \leq T/t$. 
    Then, 
    \begin{align} 
        c_6 \ln\Bigl(\frac{4}{\eta}\frac{T}{t}\frac{\cC_6}{\cC_5}   \ln(T/t) \Bigr) \leq c_6 \ln(\cC_6) + 2c_6\ln(T/t). \label{eq:CozyZero}
    \end{align}
    Further, taking $\cC_6 \geq 1$ and taking $\cC_1$ sufficiently small, we may assume that $\lfloor c_5'  \cC_6^2 \ln(T/t)^2\rfloor \geq c_5' \cC_6^2 \ln(T/t)^2/2$. 
    Then, with $\sC \de   \tic_3 \eta \sqrt{c_5'}  / (4 \sqrt{2})$,  
    \begin{align} 
        \frac{\tic_3 \eta}{4}  \sqrt{\bigl\lfloor c_5'  \cC_6^2 \ln(T/t)^2 \bigr\rfloor} \geq \sC \cC_6 \ln(T/t). 
    \end{align}
    Further, taking $\cC_1$ sufficiently small and taking $\cC_6$ sufficiently large, it can be ensured that
    \begin{align} 
        \sC \cC_6 \ln(T/t) \geq c_6 \ln(\cC_6) + 2c_6\ln(T/t). \label{eq:SlowCat}
    \end{align}
    Now, fix $\cC_6$ and take $\cC_7$ such that $\cC_{7}\ln(T/t) =\sC \cC_6 \ln(T/t)$. 
    Then, combining \eqref{eq:CozyZero}--\eqref{eq:SlowCat} yields \eqref{eq:DangerousGum}. 
    Equivalently, \eqref{eq:DampXray} holds. 

    From here on, $\cC_4,\ldots,\cC_7$ will remain fixed, and the goal is to show that \eqref{eq: HighFrequency_FalsePositives} and \eqref{eq: HighFrequency_RandomPoint} are satisfied if $\cC_1$ and $\cC_2$ are taken sufficiently small, and $\cC_3$ sufficiently large.  

    We start with the nonoccurrence of false positives from \eqref{eq: HighFrequency_FalsePositives}. 
    A direct calculation using the previously fixed values of $\cC_5$, $\cC_6$, and $\cC_7$ shows that the assumptions in the first paragraph of \Cref{lem: UnionNumTransitions} are satisfied if $\cC_1$ and $\cC_2$ are taken sufficiently small.\footnote{For instance, the assumption $\ell \leq c_1 (t \min\{ 1/\kappa^2, 1/\lambda_{\max}^2, \rho^2 \})^{1/4}$ can be equivalently be phrased as $\cC_6^4 \ln(T/t)^4 t \leq c_1^4 \min\{ 1/\kappa^2, 1/\lambda_{\max}^2, \rho^2 \}$ and hence holds if $\cC_2$ is taken sufficiently small in \eqref{eq:CozyOwl}.} 
    Further, recall from \eqref{eq:DampXray} that $\mathfrak{n}_0 \geq c_6 \ln(4T\ell / \eta \epsilon t )$. 
    Hence, by \Cref{lem: UnionNumTransitions},  
    \begin{align} 
        \bbP\bigl(\cE_{\mathrm{good}}(p_{j,k}, \vec{v}_n) = 1 \text{ for all }j,k \in \bbZ \text{ and }0 \leq n \leq \lfloor 2\pi \ell / \epsilon \rfloor\bigr) \geq 1 - \eta/4.\label{eq:HighLeaf} 
    \end{align}
    In particular, the probability in \eqref{eq:HighLeaf} is $\geq 1 -\eta$. 
    Recalling that $\mathfrak{s} = \cC_4 = c_3$ and recalling \eqref{eq:ZenMan}, it holds on the event in the left-hand side of \eqref{eq:HighLeaf} that for every $j,k,n$ such that $\Vert p_{j,k} - y \Vert \geq c_2 \ell$ for all $y\in \cup_{i=0}^m B_i$,  
    \begin{align} 
        M(p_{j,k}, \vec{v}_n)/N(p_{j,k}, \vec{v}_n) \geq \mathfrak{s} \text{ or }N(p_{j,k}, \vec{v}_n) < \mathfrak{n}_0.     
    \end{align}  
    Recall the definition of $M(p_{j,k},\vec{v})$ and $N(p_{j,k},\vec{v}_n)$ from \Cref{cor: NumTransitions}.
    It follows that \Cref{alg: HighFrequency} will not add any point $p_{j,k}$ which has distance $\geq c_2\ell$ from the barriers to $\hat{\cX}$. 
    Equivalently, every point in $\hat{\cX}$ has distance $< c_2 \ell$ from some barrier.   
    This proves \eqref{eq: HighFrequency_FalsePositives} if we take $\cC_3$ sufficiently large to satisfy $\cC_3 \geq c_2 \cC_6$.  
    
    Let us proceed to the proof of \eqref{eq: HighFrequency_RandomPoint} regarding the recovery of typical points. 
    Consider a stopping time $\tau$ with $X_{\tau} \in \cup_{i=0}^m B_i$ almost surely.
    We start with some preliminary reductions. 
    First, write $\sE_{\text{good}}$ for the event in \eqref{eq:HighLeaf} where $\cE_{\text{good}}(p_{j,k}, \vec{v}_n) = 1$ for every $j,k,n$, and recall that $\bbP(\sE_{\text{good}}) \geq 1 -\eta/4$. 
    Further, write 
    $ 
        \Delta \de  (\lfloor c_5' \ell^2/t \rfloor +  2)t 
    $
    and note that \Cref{lem: HittingEnd}  ensures that $\tau \not\in [T-\Delta,T]$ with probability $\geq 1 - \eta/4$ if $\cC_1$ and $\cC_2$ are taken sufficiently small.
    Hence, 
    \begin{align} 
        \bbP{}&{} \bigl(\inf\{\Vert X_\tau - p \Vert: p\in \hat{\cX} \} >\cC_3 \ln(T/t)\sqrt{t}\text{ and }\tau \leq T \bigr)\label{eq:OldWisp}\\  
        &\leq \bbP\bigl(\inf\{\Vert X_\tau - p \Vert: p\in \hat{\cX} \} >\cC_3 \ln(T/t)\sqrt{t},\  \tau \leq T - \Delta,\ \sE_{\mathrm{\text{good}}}  \bigr) + \eta/2. \nonumber   
    \end{align}
    Here, by the law of total probability,   
    \begin{align} 
        &\bbP\bigl(\inf\{\Vert X_\tau - p \Vert: p\in \hat{\cX} \} >\cC_3 \ln(T/t)\sqrt{t},\  \tau \leq T - \Delta,\ \sE_{\mathrm{\text{good}}}  \bigr) \leq \label{eq:KindLog}\\ 
        &\sup_{\substack{i \leq m\\ y \in B_i}} \sup_{\substack{\cS_i = -1,+1\\\cT \in [0, T-\Delta]}  } \bbP\bigl(\inf\{\Vert y - p \Vert: p\in \hat{\cX} \} >\cC_3 \ln(T/t)\sqrt{t}, \, \sE_{\mathrm{\text{good}}} \mid X_{\tau}  = y, s_i(\Lc{i}_\tau) = \cS_i, \tau = \cT  \bigr). \nonumber
    \end{align}
    From here on, let us fix some $y\in B_i$, $\cS_i \in \{-1,+1 \}$, and $\cT \in [0,T-\Delta]$ and work on the event described in the conditioning of \eqref{eq:KindLog}.

    The choice of $\cC_5$ in \eqref{eq:OddHome} ensures that $\epsilon$ is sufficiently small to guarantee that there exist $j,k,n$ with $\Vert y - p_{j,k}  \Vert \leq \cC_5\sqrt{t} \leq \min\{c_4 , c_3' \} \sqrt{t}$ and $\Vert \vec{v}_n - \cS_i \vec{n}_i(y) \Vert \leq \min\{c_5,c_4' \} \sqrt{t}/\ell$.
    Then, recalling \eqref{eq:YoungBall}, \Cref{alg: HighFrequency} will add $p_{j,k}$ to $\hat{\cX}$ if  $\cE_{\text{good}}(p_{j,k}, \vec{v}_n) = 1$ and $N(p_{j,k}, \vec{v}_n) \geq \mathfrak{n}_0$. 
    Hence, since $\sE_{\text{good}}$ implies that $\cE_{\text{good}}(p_{j,k}, \vec{v}_n) = 1$,    
    \begin{align} 
        \bbP\bigl(\inf\{\Vert y - p \Vert: p\in \hat{\cX} \} >\cC_3 {}&{} \ln(T/t)\sqrt{t}, \, \sE_{\mathrm{\text{good}}}   \mid X_{\tau}  = y, s_i(\Lc{i}_\tau) = \cS_i, \tau = \cT  \bigr)\label{eq:MoistFan}\\ 
        &\leq \bbP(N(p_{j,k}, \vec{v}_{n})  < \mathfrak{n}_0 \mid X_{\tau }= y,  s_i(\Lc{i}_{\tau}) = \cS_i, \tau = \cT) \nonumber 
    \end{align}
    provided that $\cC_3\ln(T/t)\sqrt{t} \geq\Vert y-p_{j,k} \Vert$ which can be ensured by taking $\cC_3$ sufficiently large so that $\cC_3\ln(T/t)\sqrt{t} \geq \cC_5 \sqrt{t}$. 

    Let $\gamma \in [0,1)$ be the least value such that $\cT/t + \gamma$ is an integer.  
    (That is, $\gamma \de \lceil \cT /t \rceil - \cT/t$.) 
    Then, recalling the definition of $N(p_{j,k}, \vec{v}_n)$ from \eqref{eq:VividImp},   
    \begin{align} 
        N(p_{j,k}, \vec{v}_n) \geq \#\{j\geq 0: X_{\cT +(\gamma +j) t } \in R_+(p_{j,k}, \vec{v}_n) \text{ and }  \cT/t + \gamma +j  \leq \lfloor T/t \rfloor -1 \}. \label{eq:VividBox} 
    \end{align}      
    Let $J \de \lfloor c_5' \ell^2/t \rfloor$ and note that we then have $\Delta = (J + 2)t$. 
    Hence, the assumption that $\cT \leq T-\Delta$ ensures that $\cT/t + \gamma + J \leq T/t - 2 + \gamma $. 
    Consequently, since $\cT/t + \gamma + J$ is an integer and $- 2 + \gamma \leq -1$, we have $\cT/t + \gamma + j \leq \lfloor T/t \rfloor -1$ for every $j\leq J$, and hence  
    \begin{align} 
        N(p_{j,k}, \vec{v}_n) \geq \#\{0 \leq j \leq J: X_{\cT +(\gamma +j) t } \in R_+(p_{j,k}, \vec{v}_n)\}.\label{eq:LazyCup} 
    \end{align}
    Now, first combining \eqref{eq:MoistFan}--\eqref{eq:LazyCup} and subsequently using the strong Markovianity from \Cref{prop: X_Markov}, 
    \begin{align} 
        {}&{}\bbP\bigl(\inf\{\Vert y - p \Vert: p\in \hat{\cX} \} >\cC_3 \ln(T/t)\sqrt{t}, \, \sE_{\mathrm{\text{good}}}   \mid X_{\tau}  = y, s_i(\Lc{i}_\tau) = \cS_i, \tau = \cT  \bigr)\label{eq:CalmPizza}\\ 
        &\leq \bbP\bigl(\#\{0 \leq j \leq J: X_{\cT +(\gamma +j) t } \in R_+(p_{j,k}, \vec{v}_n)\} < \mathfrak{n}_0\mid  X_{\tau}  = y, s_i(\Lc{i}_\tau) = \cS_i, \tau = \cT \bigr)\nonumber\\ 
        &= \bbP\bigl(\#\{0 \leq j \leq J: X_{(\gamma +j) t } \in R_+(p_{j,k}, \vec{v}_n)\} < \mathfrak{n}_0\mid  X_{0}  = y, s_i(0) = \cS_i \bigr).\nonumber 
    \end{align}  
    The right-hand side of \eqref{eq:CalmPizza} can be controlled using \Cref{lem: NewManyHits}. 
    The assumptions in \eqref{eq:NewMoistKnot} of that lemma are satisfied due to the choice of $\cC_5,\cC_6$, and $\cC_7$ together with the choice of $p_{j,k},\vec{v}_n$ and $J$, at least if $\cC_1$ and $\cC_2$ are taken sufficiently small.
    Hence, with $N_+$ as in \eqref{eq:NewNormalLid},
    \begin{align} 
        \bbP\bigl(\#\{0 \leq j \leq J: X_{(\gamma +j) t } \in R_+(p_{j,k}, \vec{v}_n)\} <{}&{} \mathfrak{n}_0\mid  X_{0}  = y, s_i(0) = \cS_i \bigr)\label{eq:KindHermit}\\ 
        & \leq  \bbP(N_+ < \mathfrak{n}_0) + \eta/4. \nonumber  
    \end{align} 
    We next use \Cref{lem: NplusLarge} whose assumptions can be satisfied by taking $\cC_1$ and $\cC_2$ sufficiently small.\footnote{In particular, the assumption that $J \geq \tic_1 \ln(\tic_2/\eta)^2 /\eta^2$ can be equivalently phrased as $\lfloor c_5' \cC_6^2 \ln(T/t)^2 \rfloor \geq \tic_1 \ln(\tic_2/\eta)^2 /\eta^2$ and will hence be satisfied if $\cC_1$ is sufficiently small in \eqref{eq:CozyOwl}.} 
    Hence, recalling from \eqref{eq:DampXray} that $\mathfrak{n}_0 \leq (\tic_3\eta/4) \sqrt{J}$,  
    \begin{align} 
        \bbP\Bigl(N_+ < \mathfrak{n}_0\Bigr) \leq \bbP\Bigl(N_+ < \frac{\tic_3\eta}{4}\sqrt{J}\Bigr) \leq \eta/4. \label{eq:MadFace} 
    \end{align}
    
    Combine \eqref{eq:OldWisp}, \eqref{eq:KindLog}, \eqref{eq:CalmPizza}, \eqref{eq:KindHermit}, and \eqref{eq:MadFace} to conclude that 
    \begin{align} 
        \bbP\bigl(\inf\{\Vert X_\tau - p \Vert: p\in \hat{\cX} \} >\cC_3 \ln(T/t)\sqrt{t}\text{ and }\tau \leq T \bigr) \leq \eta. 
    \end{align}
    This is equivalent to the desired performance guarantee \eqref{eq: HighFrequency_RandomPoint} by passing to the complementary event, concluding the proof. 
\end{proof}

\subsubsection{Proof of \texorpdfstring{\Cref{lem: UnionNumTransitions}}{Lemma}}\label{sec: UnionNumTransitions}
As in the proofs of \Cref{lem: ConcentrationUnionBoxes,lem: Rectangle_Concentration} in \Cref{apx: ExpLog,apx: ProofRectangleConcentration}, the argument amounts to an application of the union bound. 
Different from the setting of those proofs, however, we are now concerned with a regime where $T$ may be small. 
This will be exploited to get an estimate which does not degrade when $\Area(D)$ is large.
\begin{proof}[Proof of \texorpdfstring{\Cref{lem: UnionNumTransitions}}{Lemma}]
    By the union bound, 
    \begin{align} 
        \bbP\bigl(\exists j,k, n: \cE_{\mathrm{good}}(p_{j,k}, \vec{v}_n) \neq 1 \bigr) \leq \sum_{j,k \in \bbZ} \sum_{n=0}^{\lfloor 2\pi \ell / \epsilon \rfloor} \bbP\bigl(\cE_{\mathrm{good}}(p_{j,k}, \vec{v}_n) \neq 1  \bigr).\label{eq:KindCat} 
    \end{align}
    Fix some $j,k,n$ and note from \eqref{eq:ZenMan} and \eqref{eq:YoungBall} that it is only possible to have $\cE_{\text{good}}(p,\vec{v}) \neq 1$ if $N(p,\vecv) \geq \mathfrak{n}_0 \geq 1$.
    Recall from \eqref{eq:VividImp} that this means that $X_{it}\in R_+(p_{j,k},\vec{v}_n)$ for some $i\in \{0,1,\ldots,\lfloor T/t \rfloor -1\}$.
    Hence, conditioning on the first visit to $R_+(p_{j,k},\vec{v}_n)$, 
    \begin{align} 
        \bbP\bigl(\cE_{\mathrm{good}}(p_{j,k}, \vec{v}_n) \neq 1  \bigr) = {}&{}\sum_{i=0}^{\lfloor T/t \rfloor -1} \bbP\bigl(\inf\{u\in\bbZ_{\geq 0}: X_{ut} \in R_+(p_{j,k}, \vec{v}_n) \} = i\bigr)\\ 
        {}&{}\ \times \bbP\bigl(\cE_{\mathrm{good}}(p_{j,k}, \vec{v}_n) \neq 1  \mid   \inf\{u\in\bbZ_{\geq 0}: X_{ut} \in R_+(p_{j,k}, \vec{v}_n) \} = i \bigr). \nonumber 
    \end{align}
    
    Note that $\tau = t\inf\{u\in\bbZ_{\geq 0} : X_{ut} \in R_+(p_{j,k}, \vec{v}_n) \}$ is a stopping time. 
    Hence, the strong Markovianity of \Cref{prop: X_Markov} implies that $(X_{s + \tau})_{s\geq 0}$ is again a reflected Brownian motion with semipermeable barriers. 
    Applying \Cref{cor: NumTransitions} to that process, there exist absolute constants $c,C>0$ such that  
    \begin{align} 
        \bbP\bigl(\cE_{\mathrm{good}}(p_{j,k}, \vec{v}_n) \neq 1  \mid   \inf\{u\in\bbZ_{\geq 0}: X_{ut} \in R_+(p_{j,k}, \vec{v}_n) \} = i \bigr) \leq C \exp(-c \mathfrak{n}_0)\label{eq:RipeDew}  
    \end{align}
    if $c_1,\ldots,c_5$ are chosen appropriately in the statement in \Cref{lem: UnionNumTransitions}. 
    Further, note that $\inf\{u\in\bbZ_{\geq 0}: X_{ut} \in R_+(p_{j,k}, \vec{v}_n) \} = i$ can only occur if $X_{it} \in R_+(p_{j,k},\vecv)$. 
    In particular, noting in \eqref{eq: Def_Splus} that every point in $R_+(p_{j,k}, \vecv)$ is at distance $\leq \ell + 2\sqrt{t} \leq 3\ell$ from $p_{j,k}$,
    \begin{align} 
        \bbP\bigl(\inf\{u\in\bbZ_{\geq 0}: X_{ut} \in R_+(p_{j,k}, \vec{v}_n) \} = i\bigr)\leq \bbP\bigl(\Vert X_{it} - p_{j,k} \Vert \leq 3\ell\bigr).\label{eq:DarkWand}  
    \end{align}
    
    Given $X_{it}$, there can be at most $(6\ell/\epsilon + 1)^2$ points $p_{j,k} = (\epsilon j , \epsilon k)$ with $\Vert X_{it} - p_{j,k} \Vert \leq 3\ell$. 
    Hence, also using that $\epsilon \leq \sqrt{t}/2 \leq \ell /2$, there exists an absolute constant $C'>0$ such that for every fixed $i$, 
    \begin{align} 
        \sum_{n=0}^{\lfloor 2\pi \ell /\epsilon \rfloor }\sum_{j,k\in \bbZ} \bbP\bigl(\Vert X_{it} - p_{j,k} \Vert \leq 3\ell \bigr) &= (\lfloor 2\pi \ell /\epsilon \rfloor  + 1)\bbE\bigl[\# \{j,k \in \bbZ: X_{it} \in \sB(p_{j,k}, 3\ell) \}  \bigr]\nonumber\\ 
        & \leq   C' (\ell/\epsilon)^3. \label{eq:DarkYarn}
    \end{align}
    Now, combining \eqref{eq:KindCat}--\eqref{eq:DarkYarn} there exists an absolute constant $C''>0$ such that 
    \begin{align} 
        \bbP\bigl(\exists j,k, n: \cE_{\mathrm{good}}(p_{j,k}, \vec{v}_n) \neq 1 \bigr) \leq C''  (T/t) (\ell/\epsilon)^3 \exp(-c \mathfrak{n}_0). \label{eq:JumpyRaven} 
    \end{align}  
    Taking $c_6>0$ sufficiently large in the assumed lower bound on $\mathfrak{n}_0$ now ensures that the right-hand side of \eqref{eq:JumpyRaven} is $\leq \eta$, concluding the proof. 
\end{proof}

\subsubsection{Proof of \texorpdfstring{\Cref{lem: HittingEnd}}{Lemma}}\label{sec: HittingEnd}
\begin{proof}[Proof of \texorpdfstring{\Cref{lem: HittingEnd}}{Lemma}]
    Let $T' \de \min\{T,1/\kappa^2, 1/\lambda_{\max}^2,\rho^2 \}$ and take the constant $c$ in the statement of \Cref{lem: HittingEnd} sufficiently small to ensure that $\Delta + T'/2 < T$.   
    Then, first using the weak Markovianity of \Cref{prop: X_Markov} to condition on $X_{T - \Delta - c_1T'}$ for some sufficiently small $c_1>0$ and subsequently applying \Cref{lem: DistantBoundary}, there exists $c_2>0$ such that    
    \begin{align} 
        \bbP\bigl(\forall y\in \cup_{i=0}^m B_i : \Vert X_{T - \Delta} -y\Vert \geq  c_2 \sqrt{T'} \bigr) \geq 1 - \eta/2.\label{eq:MildNose} 
    \end{align}  
    Further, if the constant $c$ in \Cref{lem: HittingEnd} is taken sufficiently small depending on $c_2$ so that $\Delta$ is a sufficiently small multiple of $T'$, 
    \begin{align} 
        \bbP\bigl(\forall s \leq \Delta: \Vert W_{T-\Delta + s} - W_{T-\Delta} \Vert < c_2 \sqrt{T'}  \bigr) \geq 1-\eta/2. \label{eq:JustBeetle}
    \end{align}
    Recalling that the stochastic differential equation in \Cref{def: ReflectedBrownianMotion} implies that the increments of $X_{T-\Delta +s}$ agree with those of $W_{T-\Delta +s}$ whenever the process is not on a barrier, the combination of \eqref{eq:MildNose} and \eqref{eq:JustBeetle} implies \eqref{eq:DangerousHermit}.
    This concludes the proof.  
\end{proof}

\subsection{Proof of \texorpdfstring{\Cref{cor: Xuniform}}{Corollary}}\label{apx: ProofUniform}
As was alluded to in \eqref{eq:LuckyMouse} of \Cref{sec: Main_PathDependentRecovery_Algorithm}, the idea is to cut the trajectory into $M\geq 1$ pieces such that $\sup\{\Vert X_{t} - X_{jT/M} \Vert: t \in [jT/M, (j+1)T/M] \}$ is small for every $j$ with high probability.
That such an $M$ exists is immediate from $X_{t}$ being a continuous process, but it requires a small additional computation to verify that the dependence on the geometry of the barriers is only through the parameters of \Cref{sec: Parameters}:  
\begin{lemma}\label{lem: M}
    For every $\varepsilon >0$, $\eta\in (0,1)$, and $T>0$ there exists an integer $M\geq 1$ depending only on $\varepsilon, \eta, \min\{1/\kappa, 1/\lambda_{\max}, \rho\}$, and $T$ such that 
    \begin{align} 
        \bbP\bigl(\sup\{\Vert X_{t} - X_{jT/M} \Vert: t \in [jT/M, (j+1)T/M] \} \leq \varepsilon, \ \forall j \leq M-1 \bigr) \geq 1 - \eta. \label{eq:JumpyYoga} 
    \end{align}   
\end{lemma}
\begin{proof}
    Define a sequence of nonnegative random variables $\tau_1,\tau_2,\ldots$ by 
    \begin{align} 
        \tau_1 &\de \inf\{t \geq 0 : \Vert X_t - X_0 \Vert > \varepsilon/2 \text{ or }s_i(\Lc{i}_t) \neq s_i(0) \text{ for some }i\leq m \},\\ 
        \tau_{j+1} &\de  \inf\{t \geq 0 : \Vert X_{\tau_j + t} - X_{\tau_j}  \Vert > \varepsilon/2 \text{ or }s_i(\Lc{i}_{\tau_j +t}) \neq s_i(\Lc{i}_{\tau_j }) \text{ for some }i\leq m \}.  
    \end{align} 
    \Cref{cor: TauBound} then yields some $t>0$ depending on $\varepsilon$ and $\min\{ 1/\kappa^2, 1/\lambda_{\max}^2, \rho^2 \}$ such that $\bbP(\tau_{j} > t \mid \tau_1,\ldots, \tau_{j-1}) \geq  1/2$ for every $j$.
    In particular, one can then find some $J\geq 1$ depending on $\varepsilon, \eta, \min\{1/\kappa, 1/\lambda_{\max}, \rho\}$, and $T$ such that 
    \begin{align} 
        \bbP\Bigl(\sum_{j=1}^J \tau_j > T \Bigr) \geq 1 - \eta/2. \label{eq:CozyOrb}
    \end{align}
    (For instance, \eqref{eq:CozyOrb} can be deduced from 
    \Cref{lem: DriftExponential}.)
    Again by \Cref{cor: TauBound}, there exists some $t' >0$ depending on $J, \varepsilon, \eta,$ and $\min\{1/\kappa, 1/\lambda_{\max}, \rho\}$ such that $\bbP(\tau_j \geq t') \leq \eta/(2J)$. 
    Then, by the union bound, 
    \begin{align} 
        \bbP(\tau_j \geq t', \forall j\leq J) \geq 1 - \eta/2. \label{eq:ElegantOtter}
    \end{align}
    The combination of \eqref{eq:CozyOrb} and \eqref{eq:ElegantOtter} yields \eqref{eq:JumpyYoga} if we let $M$ be sufficiently large to ensure that $T/M < t'$. 
    (Indeed, note that every interval of length $T/M$ can then intersect at most two intervals of the form $[\tau_{j},\tau_{j+1}]$.) 
    This concludes the proof.       
\end{proof}

\begin{proof}[Proof of \texorpdfstring{\Cref{cor: Xuniform}}{Corollary}]
    \Cref{lem: M} provides some $M\geq 1$ with 
    \begin{align} 
        \bbP\bigl(\sup\{\Vert X_{t} - X_{jT/M} \Vert: t \in [jT/M, (j+1)T/M] \} \leq \varepsilon/4, \ \forall j \leq M-1 \bigr) \geq 1 - \eta/2. \label{eq:ElegantDog}
    \end{align}
    On the other hand, if we define stopping times $\cT_0,\ldots,\cT_{M-1}$ by 
    $
        \cT_j \de \inf\{t\geq jT/M: X_t \in \cup_{i=0}^k B_i \}
    $, 
    then \eqref{eq: HighFrequency_RandomPoint} from \Cref{thm: Main_PathDependentRecovery} yields $c_1',c_2',c_3'>0$ depending on $\eta$ and $M$ such that for every $j\leq M-1$, 
    \begin{align} 
        \bbP\bigl(\inf\{\Vert X_{\cT_j} - p \Vert: p\in \hat{\cX} \} \leq c_3' \ln(T/t)\sqrt{t}  \text{ or }\cT_j >T \bigr) \geq 1 - \eta/2M \label{eq:RedOtter}
    \end{align} 
    if $t\leq c_1' T$ and $\ln(T/t)^4  t \leq c_2'\min\{T, 1/\kappa^2, 1/\lambda_{\max}^2, \rho^2 \}$. 

    Define the constant $c_1$ in the statement of \Cref{cor: Xuniform} by $c_1 \de c_1' T$ and pick the constant $c_2$ such that $c_3'\sqrt{c_2} \leq \varepsilon/2$ and $c_2 \leq c_2' \min\{T,1/\kappa^2,1/\lambda_{\max}^2, \rho^2 \}$. 
    Then, in particular, $c_3'\ln(T/t)\sqrt{t} \leq \varepsilon/2$.   
    We claim that the desired result now follows by combining \eqref{eq:ElegantDog} with \eqref{eq:RedOtter} and using the definition of $\cup_{i=0}^k \cX_i$ from \eqref{eq:SafeSun} together with the triangle inequality.
    
    To be precise, by the union bound, we can assume with probability $\geq 1-\eta$ that the event in \eqref{eq:ElegantDog} occurs and that the event in \eqref{eq:RedOtter} occurs for every $j\leq M-1$. 
    It remains to show that the event described on the left-hand side of \eqref{eq: HighFrequency_Uniform} then follows. 
    Pick some arbitrary $x\in \cup_{i=0}^m \cX_i$. 
    We have to show that there exists some $p\in \hat{\cX}$ with $\Vert x - p \Vert \leq \varepsilon$. 

    The assumption that $x\in \cup_{i=0}^m \cX_i$ means that there exists some $i\leq m$ and $t_x \in [0,T]$ with $x\in B_i$ and $x = X_{t_x}$. 
    Let $j_x \leq M-1$ be the least integer with $t_x \geq j_xT/M$. 
    Then, the interval $[j_xTM, (j_x+1)TM]$ contains at least one time for which the process $X_t$ is in $\cup_{i=0}^k \cX_i$, implying that $\cT_{j_x} \leq (j_x+1)T/M \leq T$.
    In particular, by \eqref{eq:RedOtter} and the subsequent discussion regarding the choice of $c_2$, there exists some $p_x \in \hat{\cX}$ with 
    \begin{align} 
        \Vert p_x - X_{\cT_{j_x}} \Vert \leq \varepsilon/2.\label{eq:SmartPig} 
    \end{align}  
    On the other hand, since $t_x$ and $\cT_{j_x}$ are both times in the interval $[j_xT/M, (j_x +1)T/M]$, it follows from \eqref{eq:ElegantDog} that 
    \begin{align} 
        \Vert x - X_{\cT_{j_x}} \Vert = \Vert X_{t_x} -X_{\cT_{j_x}}  \Vert \leq \Vert X_{t_x} -X_{jT/M}  \Vert + \Vert X_{\cT_{j_x}}  - X_{jT/M}   \Vert \leq \varepsilon /2. \label{eq:MadFish}
    \end{align}
    Combine \eqref{eq:SmartPig} and \eqref{eq:MadFish} to conclude that $\Vert x-p_x \Vert \leq \varepsilon$, as desired.  
    
\end{proof}

\section{Additional details for \texorpdfstring{\Cref{sec: PotentialApplications}}{Section}}\label{apx: Case}
We finally provide some supplemental details for the case study. 
First, recall that the right-hand side of \Cref{fig: AnimalMovements2} zoomed in on a region of the southern valley to keep the key features discussed in the text legible.
The complete picture is displayed in \Cref{fig: AnimalMovements_Full}. 

Comparing to the satellite image on the left of \Cref{fig: AnimalMovements2}, the main findings in the northern valley are similar to those in the southern valley: impermeable barriers arise from the coastline and the slopes of the valley, while a permeable barrier arises from a river.  
One can further see some scattered loose points in the far north and in the west. 
These stem from rare excursions into low-density areas and do not necessarily correspond to barriers. 

As was mentioned in \Cref{sec: PotentialApplications}, the parameters were chosen on an ad-hoc basis. 
Specifically, the discretization scale $\epsilon$ was approximately 120 meters, and the sensitivity parameter used was $\mathfrak{s} \de 6 \epsilon$.    
This choice was arbitrary and experimentation suggests that the qualitative findings are robust to variations of these parameters. 
The truncation parameter $\mathfrak{u}$ was mainly a technical convenience to simplify the proofs and is not important for this dataset. 
This parameter may be useful if the data suffers from outliers, but otherwise one can simply fix it at a large value or use the non-truncated Wasserstein distance with $\mathfrak{u} = \infty$. 

\begin{figure}[h]
    \includegraphics[width = 0.63\textwidth]{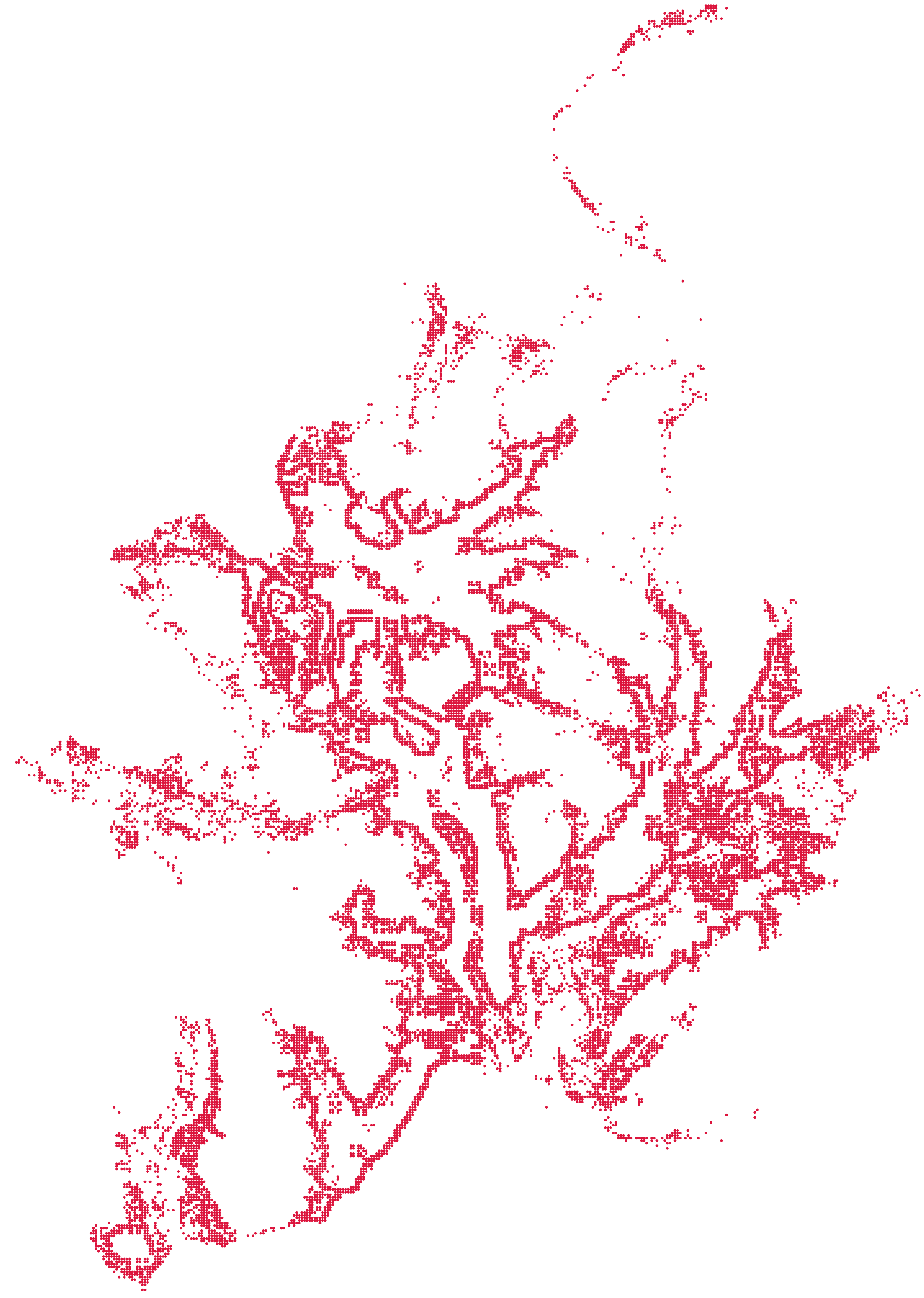}
    \caption{Full image for the barriers uncovered by \Cref{alg: KernelDiscontinuity}.}
    \label{fig: AnimalMovements_Full}
\end{figure}
\vfill 


\begin{thebibliography}{10}

    \bibitem{aamari2023minimax}
    E.~Aamari, C.~Aaron, and C.~Levrard.
    \newblock Minimax boundary estimation and estimation with boundary.
    \newblock {\em Bernoulli}, 2023.
    \newblock \href {https://doi.org/10.3150/23-BEJ1585}
      {\path{doi:10.3150/23-BEJ1585}}.
    
    \bibitem{albrecht2016nanoscopic}
    D.~Albrecht, C.M. Winterflood, M.~Sadeghi, T.~Tschager, F.~No{\'e}, and
      H.~Ewers.
    \newblock Nanoscopic compartmentalization of membrane protein motion at the
      axon initial segment.
    \newblock {\em {J}ournal of {C}ell {B}iology}, 2016.
    \newblock \href {https://doi.org/10.1083/jcb.201603108}
      {\path{doi:10.1083/jcb.201603108}}.
    
    \bibitem{aldous1997mixing}
    D.~Aldous, L.~Lov{\'a}sz, and P.~Winkler.
    \newblock {M}ixing times for uniformly ergodic {M}arkov chains.
    \newblock {\em {S}tochastic Processes and their {A}pplications}, 1997.
    \newblock \href {https://doi.org/10.1016/S0304-4149(97)00037-9}
      {\path{doi:10.1016/S0304-4149(97)00037-9}}.
    
    \bibitem{anderson1976small}
    R.F. Anderson and S.~Orey.
    \newblock Small random perturbation of dynamical systems with reflecting
      boundary.
    \newblock {\em {N}agoya {M}athematical {J}ournal}, 1976.
    \newblock \href {https://doi.org/10.1017/S0027763000017232}
      {\path{doi:10.1017/S0027763000017232}}.
    
    \bibitem{bai2022bayesian}
    Y.~Bai, Y.~Wang, H.~Zhang, and X.~Zhuo.
    \newblock Bayesian estimation of the skew {O}rnstein-{U}hlenbeck process.
    \newblock {\em {C}omputational {E}conomics}, 2022.
    \newblock \href {https://doi.org/10.1007/s10614-021-10156-z}
      {\path{doi:10.1007/s10614-021-10156-z}}.
    
    \bibitem{barahona2016simulation}
    M.~Barahona, L.~Rifo, M.~Sep{\'u}lveda, and S.~Torres.
    \newblock A simulation-based study on {B}ayesian estimators for the skew
      {B}rownian motion.
    \newblock {\em Entropy}, 2016.
    \newblock \href {https://doi.org/10.3390/e18070241}
      {\path{doi:10.3390/e18070241}}.
    
    \bibitem{bardou2010statistical}
    O.~Bardou and M.~Martinez.
    \newblock Statistical estimation for reflected skew processes.
    \newblock {\em {S}tatistical {I}nference for {S}tochastic {P}rocesses}, 2010.
    \newblock \href {https://doi.org/10.1007/s11203-010-9047-6}
      {\path{doi:10.1007/s11203-010-9047-6}}.
    
    \bibitem{beyer2016you}
    H.L. Beyer, E.~Gurarie, L.~B{\"o}rger, M.~Panzacchi, M.~Basille, I.~Herfindal,
      B.~Van~Moorter, S.~R.~Lele, and J.~Matthiopoulos.
    \newblock ‘{Y}ou shall not pass!’: quantifying barrier permeability and
      proximity avoidance by animals.
    \newblock {\em {J}ournal of {A}nimal {E}cology}, 2014.
    \newblock \href {https://doi.org/10.1111/1365-2656.12275}
      {\path{doi:10.1111/1365-2656.12275}}.
    
    \bibitem{bishwal2007parameter}
    J.P.N. Bishwal.
    \newblock {\em Parameter estimation in stochastic differential equations}.
    \newblock Springer, 2007.
    
    \bibitem{bressloff2022probabilistic}
    {P}.{C}. Bressloff.
    \newblock A probabilistic model of diffusion through a semi-permeable barrier.
    \newblock {\em {P}roceedings of the {R}oyal {S}ociety {A}}, 2022.
    \newblock \href {https://doi.org/10.1098/rspa.2022.0615}
      {\path{doi:10.1098/rspa.2022.0615}}.
    
    \bibitem{bressloff2023renewal}
    {P}.{C}. {B}ressloff.
    \newblock Renewal equations for single-particle diffusion in multilayered
      media.
    \newblock {\em {S}{I}{A}{M} {J}ournal on {A}pplied {M}athematics}, 2023.
    \newblock \href {https://doi.org/10.1137/23M1545835}
      {\path{doi:10.1137/23M1545835}}.
    
    \bibitem{brunel2018methods}
    V.-E. Brunel.
    \newblock Methods for estimation of convex sets.
    \newblock {\em {S}tatistical {S}cience}, 2018.
    \newblock \href {https://doi.org/10.1214/18-STS669}
      {\path{doi:10.1214/18-STS669}}.
    
    \bibitem{burdzy2006traps}
    K.~Burdzy, Z.-Q. Chen, and D.E. Marshall.
    \newblock Traps for reflected {B}rownian motion.
    \newblock {\em {M}athematische {Z}eitschrift}, 2006.
    \newblock \href {https://doi.org/10.1007/s00209-005-0849-y}
      {\path{doi:10.1007/s00209-005-0849-y}}.
    
    \bibitem{burdzy2004heat}
    K.~{B}urdzy, Z.-Q. Chen, and J.~Sylvester.
    \newblock The heat equation and reflected {B}rownian motion in time-dependent
      domains.
    \newblock {\em {A}nnals of {P}robability}, 2004.
    \newblock \href {https://doi.org/10.1214/aop/1079021464}
      {\path{doi:10.1214/aop/1079021464}}.
    
    \bibitem{chen2011asymptotic}
    X.~Chen and A.~Friedman.
    \newblock Asymptotic analysis for the narrow escape problem.
    \newblock {\em {S}{I}{A}{M} journal on mathematical analysis}, 2011.
    \newblock \href {https://doi.org/10.1137/090775257}
      {\path{doi:10.1137/090775257}}.
    
    \bibitem{cholaquidis2023home}
    A.~Cholaquidis, R.~Fraiman, and M.~Hern{\'a}ndez-Banadik.
    \newblock {H}ome-range estimation under a restricted sample scheme.
    \newblock {\em {J}ournal of {N}onparametric {S}tatistics}, 2024.
    \newblock \href {https://doi.org/10.1080/10485252.2023.2280003}
      {\path{doi:10.1080/10485252.2023.2280003}}.
    
    \bibitem{cholaquidis2016set}
    A.~Cholaquidis, R.~Fraiman, G.~Lugosi, and B.~Pateiro-L{\'o}pez.
    \newblock {S}et estimation from reflected {B}rownian motion.
    \newblock {\em {J}ournal of the {R}oyal {S}tatistical {S}ociety {S}eries {B}:
      {S}tatistical {M}ethodology}, 2016.
    \newblock \href {https://doi.org/10.1111/rssb.12149}
      {\path{doi:10.1111/rssb.12149}}.
    
    \bibitem{cholaquidis2021level}
    A.~Cholaquidis, R.~Fraiman, E.~Mordecki, and C.~Papalardo.
    \newblock Level set and drift estimation from a reflected {B}rownian motion
      with drift.
    \newblock {\em {S}tatistica {S}inica}, 2021.
    \newblock \href {https://doi.org/10.5705/ss.202018.0211}
      {\path{doi:10.5705/ss.202018.0211}}.
    
    \bibitem{cuevas2009set}
    A.~Cuevas.
    \newblock {S}et estimation: {A}nother bridge between statistics and geometry.
    \newblock {\em {B}olet\'{i}n de {E}stad\'{i}stica e {E}nvestigaci\'{o}n
      {O}perativa}, 2009.
    
    \bibitem{dineen2014multivariate}
    S.~Dineen.
    \newblock {\em Multivariate calculus and geometry}.
    \newblock Springer, 2014.
    \newblock \href {https://doi.org/10.1007/978-1-4471-6419-7}
      {\path{doi:10.1007/978-1-4471-6419-7}}.
    
    \bibitem{erhard2021slow}
    D.~Erhard, T.~Franco, and D.S. da~Silva.
    \newblock The slow bond random walk and the snapping out {B}rownian motion.
    \newblock {\em {A}nnals of {A}pplied {P}robability}, 2021.
    \newblock \href {https://doi.org/10.1214/20-AAP1584}
      {\path{doi:10.1214/20-AAP1584}}.
    
    \bibitem{forien2019gene}
    R.~Forien.
    \newblock Gene flow across geographical barriers---scaling limits of random
      walks with obstacles.
    \newblock {\em {S}tochastic {P}rocesses and their {A}pplications}, 2019.
    \newblock \href {https://doi.org/10.1016/j.spa.2018.10.006}
      {\path{doi:10.1016/j.spa.2018.10.006}}.
    
    \bibitem{harrison1981skew}
    J.M. Harrison and L.A. Shepp.
    \newblock {O}n skew {B}rownian motion.
    \newblock {\em {T}he {A}nnals of {P}robability}, 1981.
    \newblock \href {https://doi.org/10.1214/aop/1176994472}
      {\path{doi:10.1214/aop/1176994472}}.
    
    \bibitem{hofling2013anomalous}
    F.~H{\"o}fling and T.~Franosch.
    \newblock Anomalous transport in the crowded world of biological cells.
    \newblock {\em {R}eports on {P}rogress in {P}hysics}, 2013.
    \newblock \href {https://doi.org/10.1088/0034-4885/76/4/046602}
      {\path{doi:10.1088/0034-4885/76/4/046602}}.
    
    \bibitem{holcman2004escape}
    D.~Holcman and Z.~Schuss.
    \newblock Escape through a small opening: receptor trafficking in a synaptic
      membrane.
    \newblock {\em {J}ournal of {S}tatistical {P}hysics}, 2004.
    \newblock \href {https://doi.org/10.1007/s10955-004-5712-8}
      {\path{doi:10.1007/s10955-004-5712-8}}.
    
    \bibitem{karr2017point}
    A.~Karr.
    \newblock {\em {P}oint processes and their statistical inference}.
    \newblock Marcel Dekker, New York, 1991.
    
    \bibitem{kusumi2005paradigm}
    A.~Kusumi, C.~Nakada, K.~Ritchie, K.~Murase, K.~Suzuki, H.~Murakoshi, R.S.
      Kasai, J.~Kondo, and T.~Fujiwara.
    \newblock Paradigm shift of the plasma membrane concept from the
      two-dimensional continuum fluid to the partitioned fluid: high-speed
      single-molecule tracking of membrane molecules.
    \newblock {\em Annual Review of Biophysics}, 2005.
    \newblock \href {https://doi.org/10.1146/annurev.biophys.34.040204.144637}
      {\path{doi:10.1146/annurev.biophys.34.040204.144637}}.
    
    \bibitem{kutoyants2013statistical}
    Y.A. Kutoyants.
    \newblock {\em Statistical inference for ergodic diffusion processes}.
    \newblock Springer Science \& Business Media, 2013.
    
    \bibitem{lejay2006constructions}
    A.~Lejay.
    \newblock {O}n the constructions of the skew {B}rownian motion.
    \newblock {\em {P}robability {S}urveys}, 2006.
    \newblock \href {https://doi.org/10.1214/154957807000000013}
      {\path{doi:10.1214/154957807000000013}}.
    
    \bibitem{lejay2016snapping}
    A.~Lejay.
    \newblock The snapping out {B}rownian motion.
    \newblock {\em {T}he {A}nnals of {A}pplied {Probability}}, 2016.
    \newblock \href {https://doi.org/10.1214/15-AAP1131}
      {\path{doi:10.1214/15-AAP1131}}.
    
    \bibitem{lejay2018estimation}
    A.~Lejay.
    \newblock Estimation of the bias parameter of the skew random walk and
      application to the skew {B}rownian motion.
    \newblock {\em {S}tatistical {I}nference for {S}tochastic {P}rocesses}, 2018.
    \newblock \href {https://doi.org/10.1007/s11203-017-9161-9}
      {\path{doi:10.1007/s11203-017-9161-9}}.
    
    \bibitem{lejay2018monte}
    A.~Lejay.
    \newblock A {M}onte {C}arlo estimation of the mean residence time in cells
      surrounded by thin layers.
    \newblock {\em {M}athematics and {C}omputers in {S}imulation}, 2018.
    \newblock \href {https://doi.org/10.1016/j.matcom.2017.05.008}
      {\path{doi:10.1016/j.matcom.2017.05.008}}.
    
    \bibitem{lejay2014brownian}
    A.~Lejay, E.~Mordecki, and S.~Torres.
    \newblock Is a {B}rownian motion skew?
    \newblock {\em {S}candinavian {J}ournal of {S}tatistics}, 2014.
    \newblock \href {https://doi.org/10.1111/sjos.12033}
      {\path{doi:10.1111/sjos.12033}}.
    
    \bibitem{lejay2019two}
    A.~Lejay, E.~Mordecki, and S.~Torres.
    \newblock {T}wo consistent estimators for the skew {B}rownian motion.
    \newblock {\em {E}{E}{A}{I}{M}: {P}robability and {S}tatistics}, 2019.
    \newblock \href {https://doi.org/10.1051/ps/2018018}
      {\path{doi:10.1051/ps/2018018}}.
    
    \bibitem{lejay2018statistical}
    A.~Lejay and P.~Pigato.
    \newblock Statistical estimation of the {O}scillating {B}rownian {M}otion.
    \newblock {\em Bernoulli}, 2018.
    \newblock \href {https://doi.org/10.3150/17-BEJ969}
      {\path{doi:10.3150/17-BEJ969}}.
    
    \bibitem{lejay2020maximum}
    A.~Lejay and P.~Pigato.
    \newblock Maximum likelihood drift estimation for a threshold diffusion.
    \newblock {\em {S}candinavian {J}ournal of {S}tatistics}, 2020.
    \newblock \href {https://doi.org/10.1111/sjos.12417}
      {\path{doi:10.1111/sjos.12417}}.
    
    \bibitem{levin2017markov}
    D.A. Levin and Y.~Peres.
    \newblock {\em {M}arkov chains and mixing times}.
    \newblock American Mathematical Society, second edition, 2017.
    
    \bibitem{lions1984stochastic}
    P.-L. Lions and A.-S. Sznitman.
    \newblock Stochastic differential equations with reflecting boundary
      conditions.
    \newblock {\em {C}ommunications on {P}ure and {A}pplied {M}athematics}, 1984.
    \newblock \href {https://doi.org/10.1002/cpa.3160370408}
      {\path{doi:10.1002/cpa.3160370408}}.
    
    \bibitem{loe2016behavioral}
    L.E. Loe, B.B. Hansen, A.~Stien, S.D. Albon, R.~Bischof, A.~Carlsson, R~J.
      Irvine, M.~Meland, I.M. Rivrud, E.~Ropstad, V.~Verbj{\o}rn, and A.~Mysterud.
    \newblock {B}ehavioral buffering of extreme weather events in a high-{A}rctic
      herbivore.
    \newblock {\em Ecosphere}, 2016.
    \newblock \href {https://doi.org/10.1002/ecs2.1374}
      {\path{doi:10.1002/ecs2.1374}}.
    
    \bibitem{mandrekar2016brownian}
    V.~Mandrekar and A.~Pilipenko.
    \newblock {O}n a {B}rownian motion with a hard membrane.
    \newblock {\em {S}tatistics \& {P}robability {L}etters}, 2016.
    \newblock \href {https://doi.org/10.1016/j.spl.2016.02.005}
      {\path{doi:10.1016/j.spl.2016.02.005}}.
    
    \bibitem{matthews1988covering}
    P.~Matthews.
    \newblock {C}overing problems for {B}rownian motion on spheres.
    \newblock {\em {T}he {A}nnals of {P}robability}, 1988.
    \newblock \href {https://doi.org/10.1214/aop/1176991894}
      {\path{doi:10.1214/aop/1176991894}}.
    
    \bibitem{pankrashkin2015inequality}
    {K}. {P}ankrashkin.
    \newblock An inequality for the maximum curvature through a geometric flow.
    \newblock {\em {A}rchiv der {M}athematik}, 2015.
    \newblock \href {https://doi.org/10.1007/s00013-015-0804-z}
      {\path{doi:10.1007/s00013-015-0804-z}}.
    
    \bibitem{paquette2009statistical}
    S.R. Paquette and F.-J. Lapointe.
    \newblock A statistical procedure to assess the significance level of barriers
      to gene flow.
    \newblock {\em {J}ournal of {G}enetics and {G}enomics}, 2009.
    \newblock \href {https://doi.org/10.1016/S1673-8527(08)60161-7}
      {\path{doi:10.1016/S1673-8527(08)60161-7}}.
    
    \bibitem{paulin2015concentration}
    D.~Paulin.
    \newblock Concentration inequalities for {M}arkov chains by {M}arton couplings
      and spectral methods.
    \newblock {\em {E}lectronic {J}ournal of {Probability}}, 2015.
    \newblock \href {https://doi.org/10.1214/EJP.v20-4039}
      {\path{doi:10.1214/EJP.v20-4039}}.
    
    \bibitem{paviolo2020nanoscale}
    C.~Paviolo, F.N. Soria, J.S. Ferreira, A.~Lee, L.~Groc, E.~Bezard, and
      L.~Cognet.
    \newblock Nanoscale exploration of the extracellular space in the live brain by
      combining single carbon nanotube tracking and super-resolution imaging
      analysis.
    \newblock {\em Methods}, 2020.
    \newblock \href {https://doi.org/10.1016/j.ymeth.2019.03.005}
      {\path{doi:10.1016/j.ymeth.2019.03.005}}.
    
    \bibitem{peyre2019computational}
    {G} Peyr{\'e} and {M}. Cuturi.
    \newblock Computational optimal transport: With applications to data science.
    \newblock {\em {F}oundations and {T}rends in {M}achine {L}earning}, 2019.
    \newblock \href {https://doi.org/10.1561/2200000073}
      {\path{doi:10.1561/2200000073}}.
    
    \bibitem{pommerenke2013boundary}
    C.~Pommerenke.
    \newblock {\em Boundary behaviour of conformal maps}.
    \newblock Springer Science \& Business Media, 2013.
    \newblock \href {https://doi.org/10.1007/978-3-662-02770-7}
      {\path{doi:10.1007/978-3-662-02770-7}}.
    
    \bibitem{remon2018estimating}
    J.~Remon, E.~Chevallier, J.G. Prunier, M.~Baguette, and S.~Moulherat.
    \newblock Estimating the permeability of linear infrastructures using recapture
      data.
    \newblock {\em {L}andscape {E}cology}, 2018.
    \newblock \href {https://doi.org/10.1007/s10980-018-0694-0}
      {\path{doi:10.1007/s10980-018-0694-0}}.
    
    \bibitem{ringbauer2018estimating}
    {H}. {R}ingbauer, {A}. {K}olesnikov, {D}.{L}. {F}ield, and {N}.{H}. {B}arton.
    \newblock Estimating barriers to gene flow from distorted
      isolation--by--distance patterns.
    \newblock {\em Genetics}, 2018.
    \newblock \href {https://doi.org/10.1534/genetics.117.300638}
      {\path{doi:10.1534/genetics.117.300638}}.
    
    \bibitem{sadegh2017plasma}
    S.~Sadegh, J.L. Higgins, P.C. Mannion, M.M. Tamkun, and D.~Krapf.
    \newblock Plasma membrane is compartmentalized by a self-similar cortical actin
      meshwork.
    \newblock {\em {P}hysical {R}eview {X}}, 2017.
    \newblock \href {https://doi.org/10.1103/PhysRevX.7.011031}
      {\path{doi:10.1103/PhysRevX.7.011031}}.
    
    \bibitem{sawyer2013framework}
    H.~Sawyer, Matthew~J. Kauffman, A.D. Middleton, T.A. Morrison, R.M. Nielson,
      and T.B. Wyckoff.
    \newblock A framework for understanding semi-permeable barrier effects on
      migratory ungulates.
    \newblock {\em {J}ournal of {A}pplied {E}cology}, 2013.
    \newblock \href {https://doi.org/10.1111/1365-2664.12013}
      {\path{doi:10.1111/1365-2664.12013}}.
    
    \bibitem{schumm2023numerical}
    {R}.{D}. {S}chumm and {P}.{C}. {B}ressloff.
    \newblock A numerical method for solving snapping out {B}rownian motion in
      {2}{D} bounded domains.
    \newblock {\em {J}ournal of {C}omputational {P}hysics}, 2023.
    \newblock \href {https://doi.org/10.1016/j.jcp.2023.112479}
      {\path{doi:10.1016/j.jcp.2023.112479}}.
    
    \bibitem{slkezak2021diffusion}
    J.~{\'S}l{\k{e}}zak and S.~Burov.
    \newblock From diffusion in compartmentalized media to non-{G}aussian random
      walks.
    \newblock {\em {S}cientific {R}eports}, 2021.
    \newblock \href {https://doi.org/10.1038/s41598-021-83364-0}
      {\path{doi:10.1038/s41598-021-83364-0}}.
    
    \bibitem{su2015quasi}
    F.~Su and K.-S. Chan.
    \newblock Quasi-likelihood estimation of a threshold diffusion process.
    \newblock {\em Journal of econometrics}, 2015.
    \newblock \href {https://doi.org/10.1016/j.jeconom.2015.03.038}
      {\path{doi:10.1016/j.jeconom.2015.03.038}}.
    
    \bibitem{villani2009optimal}
    C.~Villani.
    \newblock {\em Optimal transport: old and new}.
    \newblock Springer, 2009.
    \newblock \href {https://doi.org/10.1007/978-3-540-71050-9}
      {\path{doi:10.1007/978-3-540-71050-9}}.
    
    \bibitem{zhao2024voter}
    L.~Zhao and X.~Xue.
    \newblock {T}he {V}oter {M}odel with a {S}low {M}embrane.
    \newblock {\em {J}ournal of {T}heoretical {P}robability}, 2024.
    \newblock \href {https://doi.org/10.1007/s10959-024-01321-9}
      {\path{doi:10.1007/s10959-024-01321-9}}.
    
    \end{thebibliography}
\end{document}